\documentclass{dalthesis}
\usepackage{amsfonts}
\usepackage{amssymb}
\usepackage{amsthm}
\usepackage[leqno]{amsmath}
\usepackage{mathrsfs}
\usepackage{calligra}
\usepackage{faktor}
\usepackage{combgames}
\usepackage{graphs}
\usepackage{url}
\usepackage{verbatim}

\usepackage{slashbox}
\usepackage{colortbl}

\usepackage{multirow}

\allowdisplaybreaks

\usepackage{enumerate}

\newcommand{\ideal}[1]{\left\langle #1 \right\rangle}

\newcommand{\nat}{\mathbb{N}}

\renewcommand{\l}{\ell}

\makeatletter
\def\rightharpoonupfill@{\arrowfill@\relbar\relbar\rightharpoonup}
\newcommand{\overrightharpoonup}{\mathpalette{\overarrow@\rightharpoonupfill@}}
\makeatother

\newcommand{\cl}[1]{\mathit{c}\ell\left({#1}\right)}

\newcommand{\quotient}{\faktor}
\newcommand{\comments}[1]{}

\newcommand{\combgame}[1]{#1}

\newcommand{\Mis}{Mis\`ere }
\newcommand{\mis}{mis\`ere }
\newcommand{\ms}{mis\`ere}
\newcommand{\then}{\hookrightarrow}
 
\newcommand{\Next}{\mathcal{N}}
\newcommand{\Prev}{\mathcal{P}}
\newcommand{\Left}{\mathcal{L}}
\newcommand{\Right}{\mathcal{R}}

\newcommand{\monoid}[1]{\mathscr{#1}}
\newcommand{\rhob}{\overline{\rho}}
\newcommand{\sigmab}{\overline{\sigma}}
\newcommand{\taub}{\overline{\tau}}
\newcommand{\Genus}[1]{\Gamma\left(#1\right)}

\newcommand{\xib}{\overline{\xi}}

\newcommand{\dand}{\bigtriangleup}
\newcommand{\dor}{\bigtriangledown}

\newcommand{\ab}[1]{\mathrm{ab}\mathit{#1}}

\renewcommand{\L}{\mathsf{L}}
\newcommand{\R}{\mathsf{R}}

\newcommand{\Gr}{\mathcal{G}^{+}}
\newcommand{\Gm}{\mathcal{G}^{-}}
\newcommand{\mex}{\text{mex}}

\makeatletter
\def\imod#1{\allowbreak\mkern3mu({\operator@font mod}\,\,#1)}
\makeatother

\newtheorem{theorem}{Theorem}[section]
\newtheorem{lemma}[theorem]{Proposition}
\newtheorem{proposition}[theorem]{Proposition}
\newtheorem{corollary}[theorem]{Corollary}
\newtheorem{conjecture}[theorem]{Conjecture}
\newtheorem{openproblem}[theorem]{Open Problem}
\newtheorem{notation}[theorem]{Notation}
\newtheorem{definition}[theorem]{Definition}
\newtheorem{construction}[theorem]{Construction}
\newtheorem{example}[theorem]{Example}

\numberwithin{theorem}{section}
\numberwithin{figure}{section}
\numberwithin{table}{section}

\begin{document}

\phd

\title{AN INVESTIGATION OF PARTIZAN MIS\`ERE GAMES}
\author{M.R. Allen}

\submitdate{July 17, 2009}
\defencedate{July 17, 2009}
\convocation{October}{2009}

\supervisor{\text{}}
\examiner{\text{}}
\firstreader{\text{}}
\secondreader{\text{}}

\dedicate{nnnnnnnnnn \\ \small{(she'll know what it means)}}


\frontmatter

\begin{abstract}
	Combinatorial games are played under two different play conventions: normal play, where the last player to move wins, and \mis play, where the last player to move loses.   Combinatorial games are also classified into impartial positions and partizan positions, where a position is impartial if both players have the same available moves and partizan otherwise.  
	
	\Mis play games lack many of the useful calculational and theoretical properties of normal play games.  Until Plambeck's indistinguishability quotient and \mis monoid theory were developed in 2004, research on \mis play games had stalled.  This thesis investigates partizan combinatorial \mis play games, by taking Plambeck's indistinguishability and \mis monoid theory for impartial positions and extending it to partizan ones, as well as examining the difficulties in constructing a category of \mis play games in a similar manner to Joyal's category of normal play games. 
	
	This thesis succeeds in finding an infinite set of positions which each have finite \mis monoid, examining conditions on positions for when $* + *$ is equivalent to 0, finding a set of positions which have Tweedledum-Tweedledee type strategy, and the two most important results of this thesis: giving necessary and sufficient conditions on a set of positions $\Upsilon$ such that the \mis monoid of $\Upsilon$ is the same as the \mis monoid of $*$ and giving a construction theorem which builds all positions $\xi$ such that the \mis monoid of $\xi$ is the same as the \mis monoid of $*$.
\end{abstract}

\begin{acknowledgements}
	I would like to thank my supervisor, Richard J. Nowakowski, for his support during the past three years, especially as I took a rather meandering path through Ethiopia, pregnancy, and raising a baby to get here.
	
	I would like to thank the Natural Sciences and Engineering Research Council of Canada, the Killam Foundation, and Patrick Lett for their monetary support.  I also received money from my parents and my grandparents to put towards my education.
	
	Neil McKay, Jason Brown, and David Wolfe from Dalhousie University each proofread this thesis and offered invaluable comments, criticism, and advice.
	
	On a personal note, many thanks must be given to Geoff, and I would be remiss if I did not also mention Tesfa and Sunshine.
\end{acknowledgements}

\mainmatter


\chapter{Introduction}\label{chapter-intro}

\section{Introduction}\label{sec-intro}

Combinatorial games are played under two different play conventions, normal play, where the last player to move wins, and \mis play, where the last player to move loses.  Combinatorial games are also classified into impartial games, where both players have the same move set, and partizan games, where each player has a different move set.  As such, a combinatorial game can be categorized as one of four types based on their move set and their play convention: impartial normal play, partizan normal play, impartial \mis play, and partizan \mis play.  

The theory for games played under the normal play convention is well understood, with the theory of impartial normal play games having been developed in the 1930s independently by Sprague and Grundy \cite{SPRAGUE, GRUNDY}, and the theory of partizan normal play games initially explored by Conway, Berlekamp, and Guy in the 1970s \cite{WW1, ONAG}.  However, with the exception of work regarding a {\sc chess}-based problem posed by Dawson in 1935 \cite{DAWSON}, very little effort was devoted to the study of \mis play games.   A chapter in both \cite{ONAG} and \cite{WW1} was devoted to impartial \mis play games, but after that, the theory seemed to have stalled due to complications.  As Conway stated, these
	\begin{quote}
		complications so produced persist indefinitely, and make the \mis play theory much more complicated than the normal one \cite{ONAG}.
	\end{quote}
Many of these obstacles arose from game theorists trying to take normal play ideas and applying them directly to \mis play.  In developing the theory for normal play games, positions from one game are played with positions from any other game using disjunctive sum.  While this does give rise to a very strong normal play theory, it leads to problems when applied to \mis play, as there are difficulties in easily determining the outcome of a disjunctive sum of positions played under the \mis play convention \cite{MESDAL}.  In 2004, Plambeck developed a new theory for \mis play games.  He kept disjunctive sum, but only examined positions of a game with respect to other positions of the same game \cite{TAMING}.  In comparing positions from within the same game, rather than comparing positions from arbitrary games, Plambeck was able to overcome many of the complications that seemed insurmountable with \mis play games.  Plambeck's theory is considered to be ``the biggest advance in \mis play theory in the past 50 years" \cite{BANFF2005}.  Partnering with Siegel, the two produced a theory which allows for the analysis of \mis play games \cite{ADVANCES, TAMING, MISQUOTIENT, NOTES, STRUCTURE}.  While their work was generally concerned with impartial \mis play games. there is nothing about their theory which restricts it solely to impartial games.  This thesis takes Plambeck and Siegel's theory and applies it to partizan \mis play games, an area of combinatorial game theory which has not been examined much until now.  We present a number of preliminary examples and initial theoretical results that will serve as a starting point into further investigations of this subject.

This chapter is divided as follows:
	\begin{itemize}
		\item Section \ref{sec-bg} gives some background into Combinatorial Game Theory for those who are not as familiar with the subject as they would like to be.  
		
		\item Section \ref{sec-history} gives a history on \mis play games from Conway's Genus theory to Plambeck's \mis monoids.    It also includes the definitions and basic theory on \mis theory that we will be using in the rest of the thesis.
		
		\item Section \ref{sec-poset} discusses how to construct the partial order on positions, as we do this for all our examples in Chapter \ref{chapter-examples}.
		
		\item Section \ref{sec-thesis} gives an overview of the remaining chapters of this thesis.
	\end{itemize}

\section{Introducing Combinatorial Game Theory}\label{sec-bg}
In this section, we will give an overview of the main ideas of combinatorial game theory.  This is merely to give the reader a quick introduction.  Full details are available from the following sources: \cite{LIP, WW1, ONAG}.

Most readers are familiar with the idea of a game; some familiar ones are {\sc tic-tac-toe}, {\sc poker}, {\sc chess}, and {\sc monopoly}.  This thesis concerns itself with \emph{combinatorial games}, which are games satisfying the following conditions:
		\begin{enumerate}
			\item There are two players, generally denoted as \emph{Left} and \emph{Right}.  These two players have genders assigned to them.  In this thesis, Left will be female, while Right is male.  This agrees with the Louise/Left and Richard/Right convention of \cite{WW1}.  It was also chosen since the author of this thesis is both female and left-handed. 
			
			\item There is a clearly defined rule set which states which moves are legal moves and which moves are not.
			
			\item There is complete information.  That is, all information is available to both Left and Right at all points during the game.
			
			\item There are no elements of chance which affect the outcome of the game.  That is, for example, there are no dice, spinners, or revealing of cards.
			
			\item Each game has only a finite number of moves.
			
			\item Each game ends with one winner and one loser.  There are no draws.   During the game, each player \emph{plays perfectly}. That is, both players make the optimal moves available to them. If a player can make a move that guarantees a win, the player will make that move.  If no such move is available, a non-winning move will be made. 
			
			\item Determining the winner and loser depends on whether we are playing under the \emph{normal} play convention or the \emph{\mis} play convention.  In normal play, a player loses if there is no move available on the player's turn.  In \mis play, a player wins if there is no move available on the player's turn.  
		\end{enumerate}

\begin{example}
	Let us review the four games given at the start of this section to see whether they are combinatorial games.
		\begin{enumerate}
			\item {\sc tic-tac-toe}: {\sc tic-tac-toe} satisfies conditions 1, 2, 3, 4, and 5.  However, {\sc tic-tac-toe} can end in a draw, so it does not satisfy condition 6.  Thus {\sc tic-tac-toe} is not a combinatorial game.
			
			\item {\sc poker}: {\sc poker} satisfies condition 2.  However {\sc poker} can have multiple players, there is not complete information as a player is unaware of the cards in the other players' hands or remaining in the deck, there are elements of chance in the revealing of the cards, it is possible for play to continue indefinitely, and since the play may continue forever, there may never be a winner or a loser.  Thus {\sc poker} is not a combinatorial game.
			
			\item {\sc chess}: Like {\sc tic-tac-toe}, {\sc chess} can end in a draw.  As well, the last player to move in {\sc chess} does not necessarily determine the winner or loser of the game. Therefore {\sc chess} is not a combinatorial game.
			
			\item {\sc monopoly}: {\sc monopoly} fails to satisfy the same conditions as {\sc poker}, and hence {\sc monopoly} is not a combinatorial game. 
		\end{enumerate}
\end{example}

With the exclusion of so many familiar games, one might wonder if there are any games which satisfy the conditions required to be a combinatorial game.  Of course, the answer is yes.  The following gives an example of one of the most useful combinatorial games.

\begin{example}\label{example-nim}
	An example of a combinatorial game is the game of {\sc nim}.  One plays {\sc nim} as follows:  given several heaps of tokens, a player's move is to pick a heap and remove some number of tokens from that heap.  Play continues until no heaps remain.  Played under the normal play convention, the player who takes the last token is the winner.  Played under the \mis play convention, the player who takes the last token is the loser.
\end{example}

In Section \ref{sec-history}, we see how important this seemingly innocuous game is.

In this thesis, all games are \emph{combinatorial games}.  As such, we often just 
refer to them as \emph{games} and drop the \emph{combinatorial}. 

In game theory, we often also refer to a position in a certain game as a game.  For example, we would say the \emph{game of {\sc nim}} when talking about {\sc nim} and its rule set in general.  However, we may also say that two heaps of size three and a heap of size six is a \emph{game} where we are playing {\sc nim} on those heaps.  This can often be confusing to a non-game theorist. This thesis endeavours to avoid this confusion by using \emph{game} to denote a rule-set (in our example, {\sc nim}) and a \emph{position} to be a position in a game (in our example, the two heaps of size three and a heap of size six where we are to play {\sc nim} on those heaps).  We have tried to ensure that this convention is followed throughout the thesis. 

With the exception of the named positions $0$, $1$, $\overline{1}$, and $*$, this thesis denotes positions by lower-case Greek letters.  This is to ensure that there is no confusion between positions in a game and elements of a \mis monoid, for which we generally use Roman letters.

We now continue our game theory definitions.

\begin{definition}
	Suppose we are given a position $\xi$. A \textbf{left option} of $\xi$ is a new position which arises after one move from Left. Similarly, a \textbf{right option} of $\xi$ is a new position which arises after one move from Right. The \textbf{set of options of a position} is the union of the left options and the right options.  We let $\xi^L$ denote the set of left options of $\xi$ and $\xi^R$ denote the set of right options of $\xi$.  
\end{definition}

In an abuse of notation, for a position $\xi$, we often use $\xi^L$ or $\xi^R$ to denote an element of the set of left or right options respectively.  However, it is always clear from context whether we mean an element of the set or the set itself.

Given a set of left options, $\xi^L$, we can again, for example, determine the left options of $\xi^L$, which we denote by $\xi^{LL}$.  For those starting their investigations into combinatorial game theory, this feels decidedly strange, as earlier we stated that the players Left and Right alternate moves and $\xi^{LL}$ records the results of Left moving twice in the position $\xi$.  However, it is important to keep track of such things.  Suppose, for example, that we have the position $\xi + \gamma$, for some other position $\gamma$.  Moreover, suppose Left moves first to $\xi^L + \gamma$, to which Right responds with $\xi^L + \gamma^R$.  Then Left can move to $\xi^{LL} + \gamma^R$, and it becomes clear why we would need to record those occurrences where one player moves multiple times in the same component.

To define a position, we define it recursively in terms of its options, as follows.

\begin{definition}
	A \textbf{position} $\xi$ is a bipartite set, written as 
		\[ \xi = \combgame{\{\xi^L \mid \xi^R\}},\]
	where $\xi^L$ and $\xi^R$ are the sets of positions of left and right options of $\xi$ respectively.
\end{definition}

That is, we define positions recursively based on their left and right options.  If either $\xi^L$ or $\xi^R$ is empty, we use a $\cdot$ in the above notation to denote that there are no options.  When using the above notation, again game theorists often refer to it interchangeably as either a \emph{game} or a \emph{position}; we will endeavour to refer to it as a \emph{position}.

\begin{example}\label{example-0-1-*}
	Suppose we have a position where neither Left nor Right has an option.  We call this position 0 and, using the above notation, we write
		\[ 0 = \combgame{\{\cdot\mid\cdot\}}.\]
	The position in which Left can move to 0 but Right has no move is called 1, and we write
		\[ 1 = \combgame{\{0\mid\cdot\}}.\]
	The position in which move Left and Right can move to 0 is called $*$, and we write
		\[ * = \combgame{\{0\mid0\}}.\]
\end{example}

As we saw in Example \ref{example-0-1-*}, we constructed new positions by using positions we had constructed earlier as options.  We formalize this with the notion of a \emph{birthday}.

\begin{definition}
	The \textbf{birthday} of a position $\xi = \combgame{\{\xi^L\mid\xi^R\}}$ is 1 plus the maximum birthday of any position in $\xi^L \cup \xi^R$, where we take the birthday of 0 to be 0.
\end{definition}

\begin{example}
	By definition, the birthday of the position 0 is 0.  The birthday of the positions $1$ and $*$ is 1.
\end{example}

In proving things in combinatorial game theory, we often use induction on birthday, since, by definition, the options of a position have birthday strictly less than the position itself.  In doing such, we either say \emph{induction on the birthday} or \emph{induction on the options}.

Sometimes we would like to represent a position in a different, more visual fashion.  In this case, we draw the game tree of a position.  In doing such, we draw the left options below and to the left and the right options below and to the right.

\begin{example}
	We wish to draw the game trees of the positions given in Example \ref{example-0-1-*}.  Since 0 has neither Left nor Right options, the game tree of 0 is simply a vertex.  For the game tree of 1, Left can move to 0 and Right has no moves, so the game tree of 1 is two vertices connected by a line.  The line slopes down and to the left since 0 is a left option.  For the game tree of $*$, there are three vertices which are connected to form an upside down V, representing that both Left and Right can move to 0.  All these game trees are drawn in Figure \ref{fig-gt-0-1*}.

		\begin{figure}[htb]
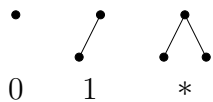

		\unitlength 8pt
		\begin{center}
		\begin{graph}(9,3.5)(-7,-1.5)
		\graphnodesize{0.4}
		\fillednodestrue

		\roundnode{A}(1,2)
		\roundnode{B}(0,0)
		\roundnode{C}(2,0)
		
		\edge{A}{C}
		\edge{A}{B}
		
		\freetext(1,-1.5){$*$}

		\roundnode{a1}(-3,2)
		\roundnode{b1}(-4,0)
		
		\edge{a1}{b1}
		
		\freetext(-3.5,-1.5){1}

		\roundnode{c1}(-7,2)
		
		\freetext(-7,-1.5){0}
		
		\end{graph}
		\caption{The game trees of 0, 1, and $*$.}
		\label{fig-gt-0-1*}
		\end{center}
		\end{figure}		
\end{example}
Using game trees, we have an equivalent definition for birthday, namely the birthday of a position is the height of its game tree.

In Section \ref{sec-intro}, we introduced impartial and partizan.  We now have the tools to define these terms formally.

\begin{definition}
	A position $\xi$ is \textbf{impartial} if the left options and right options of $\xi$ are equal as sets.   Otherwise the position is \textbf{partizan}
	
	A game is \textbf{impartial} if every position in the game is impartial.   Otherwise the game is \textbf{partizan}.  
\end{definition}

When we are considering whether a game is impartial or partizan, we must consider all positions, including such positions as $\xi^{LLL}$ which may never occur during allowed play.

We again return to our three example positions.

\begin{example}
	In 0 and $*$, Left and Right have the same options.  Thus 0 and $*$ are impartial positions.  This is not true for 1, and so 1 is a partizan position.
\end{example}

\begin{example}
	The game {\sc nim} (Example \ref{example-nim}) is an impartial game.  If we modified the rules of {\sc nim} so that Left's moves are to take an even number of tokens, while Right's moves are to take an odd number of tokens, then this modified game is partizan.
\end{example}

We also classify games based on whether Left and Right can both move from any position, even if they do not necessarily have the same moves from every position.

\begin{definition}
	A position $\xi$ is \textbf{all-small} if Left can move if and only if Right can, and every option of $\xi$ is all-small.
	
	A game is \textbf{all-small} if every position in the game is all-small.
\end{definition}

\begin{example}
	All impartial games are all-small.  The position $*$ is all-small.
\end{example}

We often wish to play sums of positions.  The following definitions clarifies what we mean by a sum of positions.

\begin{definition}
	Given positions $\xi_1$, $\xi_2$, $\ldots$, $\xi_m$, the \textbf{disjunctive sum} is the position $\xi_1 + \xi_2 +\cdots + \xi_m$ where on a given player's turn, that player picks a position $\xi_i$ and plays in it according to the rules of $\xi_i$. In \mis play, when a player has no moves in every position in the disjunctive sum, that player wins. In normal play, when a player has no moves in every position in the disjunctive sum, that player loses. 
\end{definition}

Notice that the $\xi_i$ need not all be from the same game.  For example, $\xi_1$ could be a {\sc nim} position, $\xi_2$ could be a {\sc domineering} position (for the rules of {\sc domineering}, see \cite{LIP}), $\xi_3$ could be an {\sc amazons} position (for the rules of {\sc amazons}, see \cite{LIP}), etc.  

To represent the disjunctive sum of two positions $\xi$ and $\kappa$ in terms of its options, we write the following:
	\[ \xi + \kappa = \combgame{\{\xi^L+\kappa, \xi+\kappa^L \mid \xi^R + \kappa, \xi+\kappa^R\}},\]
where the commas in the expressions denote union.

Disjunctive sum is not the only way we can play in a set of positions.  In Chapter \ref{chapter-cat}, we look at some other types of sums on positions.  

In our definition of combinatorial game, we stated that in each game there is always a winner, and hence, at least one of the players has a winning strategy.  The following definition separates positions based on who has the winning strategy.

\begin{definition}\label{def-outcomes}
	Let $\xi$ be a position of a combinatorial game.  Then $\xi$ belongs to an \textbf{outcome class} which specifies which player(s) has a winning strategy in $\xi$.  They are
		\begin{enumerate}
			\item $\Left$ if Left has a winning strategy in $\xi$ regardless of moving first or second.
			\item $\Right$ if Right has a winning strategy in $\xi$ regardless of moving first or second.
			\item $\Next$ if the next player to move has a winning strategy in $\xi$.
			\item $\Prev$ if the next player moving does not have a winning strategy (i.e.\ the next player to move will lose) in $\xi$.  The $\Prev$ stands for previous, as if the next player loses, the player who would have played previous will win.
		\end{enumerate}
\end{definition} 

The \emph{Fundamental Theorem of Combinatorial Games} states that every position belongs to precisely one of the four outcome classes listed above \cite{LIP}.

We need to be careful when we are considering the outcome class of a position.  The definition is framed as to who has the winning strategy, but the concept of winning differs between normal play and \mis play.  As such, we need some further definition.

\begin{definition}
	For a position $\xi$, we let $o^+(\xi)$ denote the normal play outcome of $\xi$ while $o^-(\xi)$ denotes the \mis play outcome.  
\end{definition}

\begin{notation}
	In considering outcomes, $\cup$ denotes \emph{or}.  
\end{notation}

\begin{example}
	If $o^-(\xi) = \Next \cup \Prev$, this means that the \mis outcome of $\xi$ is either $\Next$ or $\Prev$.
\end{example}

Knowing the outcomes of the options of a position allows us to determine what the outcome of the position itself is, irrespective of whether we are playing under the normal play or \mis play convention.  Table \ref{table-outcomes} \cite{LIP} shows how the outcome class of a position can be determined from the outcome classes of its options.

\begin{table}[htb]
\begin{center}
\begin{tabular}[b]{rcccc}
&  some $\xi^R \in \Right \cup \Prev$ & all $\xi^R \in \Left \cup \Next$ \\

some $\xi^L \in \Left \cup \Prev$ & 
	\cellcolor[gray]{0.8}$\Next$
	& 
	\cellcolor[gray]{0.8}$\Left$ 
\\ 
all $\xi^L \in \Right \cup \Next$ & 
	$\Right$
	& 
	$\Prev$
\end{tabular}
\end{center}
\caption{Determining the outcome class of a position $\xi$ given information about the outcome classes of $\xi$'s options.}
\label{table-outcomes}
\end{table}

We can also sometimes glean information about the outcomes of the options of a position if we know the outcome of the position itself.  The most important instance of this is if $\xi$ is a $\Prev$ position, then no option of $\xi$ is also a $\Prev$ position, otherwise one of the players would have moved to this position and won, meaning our initial position could not have been a $\Prev$ position.

We examine our three example positions again.

\begin{example}\label{example-o+-o-}
	Consider the position 0.  In this position, neither Left nor Right has any moves available.  Therefore
		\begin{align*}
			o^+(0) &= \Prev \\
			o^-(0) &= \Next.
		\end{align*}
	
	Consider the position 1.  In the position, Left can move to 0, while Right has no moves.  Therefore
		\begin{align*}
			o^+(1) &= \Left \\
			o^-(1) &= \Right.
		\end{align*}
		
	Consider the position $*$.  In this position, both Left and Right can move to 0.  Therefore
		\begin{align*}
			o^+(*) &= \Next \\
			o^-(*) &= \Prev.
		\end{align*}
\end{example}

If $\xi$ is an impartial position, then 
	\begin{align*}
		o^+(\xi) &= \Next \cup \Prev,\\
		o^-(\xi) &= \Next \cup \Prev \text{ } (\cite{WW1}).
	\end{align*}
That is, played under either normal play or \mis play, an impartial position is always in either $\Next$ or $\Prev$ and never $\Left$ or $\Right$.  However, it is important to note that if $\xi$ is an impartial position and $o^+(\xi) = \Next$ does not necessarily imply that $o^-(\xi) = \Prev$, or vice versa (this is discussed further in Section \ref{sec-problems-mis}).  

This concludes our quick overview of combinatorial game theory.  It is by no means exhaustive as the theory is very well developed.  For those interested in delving more into the theory, \cite{LIP, WW1, ONAG} are all excellent resources with which to begin.

\section{All About \Mis Play}\label{sec-history}

We now turn our attention to \mis play and some of the theory we will be employing in this thesis.

\subsection{Problems with \Mis Play}\label{sec-problems-mis}

To be blunt, \mis defies all expectations.  Things that seem to work well in normal play fail (miserably) in \mis play.  For example, initial perusal of Example \ref{example-o+-o-} seems to suggest that the outcome of a position played under the \mis play convention is simply the \emph{opposite} of what its outcome was under normal play.  This is far from always being true.   Table \ref{table-normal-mis-outcomes} gives examples of positions with varying outcomes in normal play and \mis play.

\begin{table}[htb]
\begin{center}
\begin{tabular}[b]{ccccc}
\cellcolor[gray]{0.9}$o^-$ $\backslash$ $o^+$ &  $\Next$ & $\Prev$ & $\Left$ & $\Right$ \\

$\Next$ & 
	\cellcolor[gray]{0.8} $*_2 := \combgame{\{0, * \mid 0, *\}}$
	& 
	\cellcolor[gray]{0.8} $0$ 
	& 
	\cellcolor[gray]{0.8} $\combgame{\{*, 0 \mid \cdot \}}$
	& 
	\cellcolor[gray]{0.8} $\combgame{\{\cdot\mid*, 0\}}$
\\ 
$\Prev$ & 
	$*$
	& $*_2 + *_2$
	&
	$\combgame{\{ \combgame{\{*,0\mid\cdot\}}\mid*_2 \} }$
	&
	$\combgame{\{*_2\mid\combgame{\{\cdot\mid*,0\}}\}}$
\\
$\Left$ & 
	\cellcolor[gray]{0.8}
	$\combgame{\{\combgame{\{*_2+*_2\mid*_2\}}\mid0\}}$
	&
	\cellcolor[gray]{0.8}
	$\combgame{\{0\mid*_2\}}$
	&
	\cellcolor[gray]{0.8}
	$\combgame{\{*_2 + *_2 \mid *_2\}}$
	&
	\cellcolor[gray]{0.8}
	$\{\cdot \mid 0\}$
	\\
$\Right$ & 
	$\combgame{\{0\mid\combgame{\{*_2 + *_2 \mid *_2\}}\}}$
	&
	$\combgame{\{*_2 \mid 0\}}$ 
	&
	$\combgame{\{0\mid\cdot\}}$
	&
	$\combgame{\{*_2 \mid *_2 + *_2\}}$
	\\
\end{tabular}
\end{center}
\caption{Positions with varying outcomes in normal play and \mis play.}
\label{table-normal-mis-outcomes}
\end{table}

Our next failure occurs with outcomes and disjunctive sum.  In normal play, Table \ref{table-normal-outcomes-+} \cite{LIP} gives the possible outcome classes of a disjunctive sum of two positions given the outcome classes of the summands.  

\begin{table}[htb]
\begin{center}
\begin{tabular}[b]{ccccc}
\cellcolor[gray]{0.9}$+$ &  $\Next$ & $\Prev$ & $\Left$ & $\Right$ \\

$\Next$ & 
	\cellcolor[gray]{0.8}?
	& 
	\cellcolor[gray]{0.8}$\Next$ 
	& 
	\cellcolor[gray]{0.8}$\Left \cup \Next$
	& 
	\cellcolor[gray]{0.8}$\Right \cup \Next$
\\ 
$\Prev$ & 
	$\Next$
	& 
	$\Prev$
	&
	$\Left$
	&
	$\Right$
\\
$\Left$ & 
	\cellcolor[gray]{0.8}%
	$\Left \cup \Next$
	&
	\cellcolor[gray]{0.8}%
	$\Left$
	&
	\cellcolor[gray]{0.8}%
	$\Left$
	&
	\cellcolor[gray]{0.8}%
	?
	\\
$\Right$ & 
	$\Right \cup \Next$
	&
	$\Right$ 
	&
	?
	&
	$\Right$
	\\
\end{tabular}
\end{center}
\caption{Normal play outcomes under disjunctive sum where ? denotes that the outcome could be in any of the four outcome classes.}
\label{table-normal-outcomes-+}
\end{table}

We can see that, for the most part, knowing the outcome class of the summands gives us an idea of the outcome class of the sum if we play under the normal play convention.  Compare this with in Table \ref{table-mis-outcomes-+} \cite{MESDAL}.

\begin{table}[htb]
\begin{center}
\begin{tabular}[b]{ccccc}
\cellcolor[gray]{0.9}$+$ &  $\Next$ & $\Prev$ & $\Left$ & $\Right$ \\

$\Next$ & 
	\cellcolor[gray]{0.8}?
	& 
	\cellcolor[gray]{0.8}? 
	& 
	\cellcolor[gray]{0.8}?
	& 
	\cellcolor[gray]{0.8}?
\\ 
$\Prev$ & 
	?
	& 
	?
	&
	?
	&
	?
\\
$\Left$ & 
	\cellcolor[gray]{0.8}%
	?
	&
	\cellcolor[gray]{0.8}%
	?
	&
	\cellcolor[gray]{0.8}%
	?
	&
	\cellcolor[gray]{0.8}%
	?
	\\
$\Right$ & 
	?
	&
	?
	&
	?
	&
	?
	\\
\end{tabular}
\end{center}
\caption{\Mis play outcomes under disjunctive sum where ? denotes that the outcome could be in any of the four outcome classes.}
\label{table-mis-outcomes-+}
\end{table}

Combining the results of Table \ref{table-mis-outcomes-+} with the results of Table \ref{table-normal-mis-outcomes}, it seems that almost all of our intuition regarding \mis play games is incorrect. 

In \mis play, we also lose some useful strategical tools we had in normal play; the most frustrating loss is that of the Tweedledum-Tweedledee strategy.  To describe this strategy, we first need to describe the \emph{conjugate of a position}.

\begin{definition}\label{def-conjugate}
	For a position $\xi = \combgame{\{\xi^L\mid\xi^R\}}$, we recursively define $\overline{\xi}$ as $\overline{\xi} = \combgame{\{\overline{\xi^R}\mid\overline{\xi^L}\}}$ and call $\overline{\xi}$ the \textbf{conjugate of $\xi$}.
\end{definition}

\begin{example}\label{example-conjugate-0-1-*}
	Returning to our three positions $0$, $*$, and $1$, we calculate their conjugates.
	
	\begin{align*}
		\overline{0} &= 0;\\
		\overline{*} &= \combgame{\{\overline{0} \mid \overline{0}\}} \\
		&= \combgame{\{0\mid0\}}\\
		&= *;\\
		\overline{1} &= \combgame{\{\cdot \mid \overline{0}\}} \\
		&= \combgame{\{\cdot \mid 0\}}.
	\end{align*}
\end{example}

Playing in the conjugate of a position $\xi$ means that in $\overline{\xi}$, Left's available moves are those of Right in $\xi$, and Right's available moves are those of Left in $\xi$.  Essentially, the players are switching roles.  With this in mind, we have the following result:

\begin{theorem}\label{theorem-outcome-class-conj}
	For a position $\xi$, we have
		\begin{align*}
			o^-(\xi) = \Left &\implies o^-(\overline{\xi}) = \Right;\\
			o^-(\xi) = \Next &\implies o^-(\overline{\xi}) = \Next;\\
			o^-(\xi) = \Prev &\implies o^-(\overline{\xi}) = \Prev;\\
			o^-(\xi) = \Right &\implies o^-(\overline{\xi}) = \Left.
		\end{align*}
	This result also holds if we replace $o^-$ by $o^+$.
\end{theorem}

Conjugation is generally written as a negative when dealing only with normal play, i.e.\ rather than $\overline{\xi}$, one would write $-\xi$.  However, this leads to confusion in \mis play, as it causes us to write things as unintuitive as $\xi + (-\xi) \not = 0$.  As \mis play is unintuitive enough already, we use conjugation instead.  

We now define Tweedledum-Tweedledee.

\begin{definition}\label{def-tweedle}
	The \textbf{Tweedledum-Tweedledee} strategy is a method of playing in $\xi + \overline{\xi}$ that ensures that the second player makes the last move.  In $\xi + \overline{\xi}$, whatever move the first player makes, the second player makes the symmetric move in the other component, and play continues as such.  
\end{definition}

In normal play, this strategy is excellent for the second player as one wants to be the last player to move.  In fact, this strategy is a basis in forming a non-trivial category of normal play games (see Chapter \ref{chapter-cat} and \cite{CAT}).  However, in \mis play, the second player will never willingly use this strategy through to the end of a game, as to do so guarantees a loss for the second player. 

We have now seen three ways in which \mis play defies our expectations: it is not simply the reversal of outcomes, disjunctive sum no longer works as nicely, and we lose the Tweedledum-Tweedledee strategy.  It seems that almost all of our intuition regarding \mis play games is incorrect.  Because of this, almost all combinatorial game theory research has focused on games played under the normal play convention.

\subsection{Genus}

The first organised attack on \mis play games was led by Conway.  To analyse impartial \mis play games, he developed \emph{genus theory} in the 1970s.  Essentially, genus theory determines, for a position in an impartial game, how closely this position behaves to that of a {\sc nim} heap or a sum of {\sc nim} heaps.  To formally define genus, we first need some definitions.

\begin{definition}
	The \textbf{minimal excludant}, or \textbf{mex}, of a set of ordinals $\mathscr{S}$, is the least ordinal not in the set $\mathscr{S}$.
\end{definition}

\begin{example}
	For the set $\mathscr{S}_1 =\{1,2,4,6\}$,  $\mex(\mathscr{S}_1) = 0$.  
	
	For the set $\mathscr{S}_2 = \{n \mid n \ge 0, n \equiv 0 \imod 2\}$, $\mex(\mathscr{S}_2) = 1$.
	
	For the set $\mathscr{S}_3 = \{\mathbb{N}\}$, $\mex(\mathscr{S}_3) = 0$.	
	
	For the set $\mathscr{S}_4 = \{\mathbb{N} \cup \{0\}\}$, $\mex(\mathscr{S}_4) = \omega$.
\end{example}

Using mex, we now define $\Gr$ and $\Gm$ of a position.

\begin{definition}
	Fix an impartial position $\xi$.  We define
		\[ \Gr(\xi) = \begin{cases}
			0 &\text{if $\xi$ has no options};\\
			\mex\{\Gr(\xi^{\prime}) \mid \xi^{\prime} \text{ is an option of } \xi\} &\text{else},
		\end{cases}\]
	and
		\[ \Gm(\xi) = \begin{cases}
			1 &\text{if $\xi$ has no options};\\
			\mex\{\Gr(\xi^{\prime}) \mid \xi^{\prime} \text{ is an option of } \xi\} &\text{else},
		\end{cases}\]
\end{definition}

Notice that the only difference between the definitions of $\Gr$ and $\Gm$ is the value a position with no options, i.e.\ the value of the position 0.  This arises due to the fact that $o^+(0) = \Prev$ while $o^-(0) = \Next$.

We now have the tools to define genus.

\begin{definition}\label{def-genus}
	The \textbf{genus} of an impartial position $\xi$, denoted by $\Genus{\xi}$ is a list of the form $x^{x_0x_1x_2\ldots}$ where $x \in \mathbb{Z}^{\ge 0}$ and $x_0x_1x_2\ldots$ is a string of non-negative integers.  We determine the values of $x$ and $x_i$ as follows: 	
		\begin{align*}
			x &= \Gr(\xi), \\
			x_0 &= \Gm(\xi), \\
			x_1 &= \Gm(\xi + *_2), \\
			x_2 &= \Gm(\xi + *_2 + *_2), \\
			&\vdots \\
			x_n &= \Gm\left(\xi + \sum_{i=1}^n *_2 \right), \\
			&\vdots
		\end{align*}
	recalling from Table \ref{table-normal-mis-outcomes} that $*_2$ is the position defined by
		\[ *_2 = \combgame{\{0, *\mid 0, *\}}.\]
\end{definition}

The definition of $\Gr$, $\Gm$, and genus are only for \textbf{impartial games}.  As such, we can only use genus for analysing impartial games.  In the following example, we calculate the genera of the three impartial positions we have encountered thus far.

\begin{example}
	In this example, we will calculate the genera of $0$, $*$, and $*_2$.  
		\begin{align*}
			\Gr(0) &= 0, & \Gr(*) &= \mex\{\Gr(0)\} & \Gr(*_2) &= \mex\{\Gr(0), \Gr(*)\}\\
			&&&= \mex\{0\} && = \mex\{0,1\}\\
			&&&=1, &&=2.
		\end{align*}
	Therefore $\Genus{0} = 0^{z_0z_1z_2\cdots}$, $\Genus{*} = 1^{s_0s_1s_2\cdots}$, and $\Genus{*_2} = 2^{t_0t_1t_2\cdots}$.  We now need to find $z_i$, $s_i$, and $t_i$.  We start by finding $z_0$, $s_0$, and $t_0$.
		\begin{align*}
			\Gm(0) &= 1, & \Gm(*) &= \mex\{\Gm(0)\} & \Gm(*_2) &= \mex\{\Gm(0), \Gm(*)\}\\
			&&&= \mex\{1\} && = \mex\{0,1\}\\
			&&&=0, &&=2.
		\end{align*}
	We move onto $z_1$, $s_1$, and $t_1$.
		\begin{align*}
			\Gm(0+*_2) &= \Gm(*_2)\\
			&= 2,\\
			\Gm(* +*_2) &= \mex\{\Gm(*), \Gm(*_2), \Gm(*+*)\} \\
			&= \mex\left\{0,2, \mex\{\Gm(*)\}\right\}\\
			&=\mex\left\{0,2, \mex\{0\}\right\}\\
			&=\mex\{0,2,1\}\\
			&=3,\\
			\Gm(*_2 + *_2) &= \mex\{\Gm(*_2), \Gm(* +*_2)\}\\
			&= \mex\{2,3\} \\
			&= 0.
		\end{align*}
	Thus far we have $\Genus{0} = 0^{12z_2\cdots}$, $\Genus{*} = 1^{03s_3\cdots}$, and $\Genus{*_2} = 2^{20t_3\cdots}$.  We will now calculate $z_2$, $s_2$, and $t_2$.
		\begin{align*}
			\Gm(0 + *_2 + *_2) &= \Gm(*_2 + *_2) \\
			&= 0,\\
			\Gm(* + *_2 + *_2) &= \mex\{\Gm(*_2 + *_2), \Gm(* + *_2), \Gm(* + * + *_2)\} \\
			&= \mex\left\{0, 3, \mex\{\Gm(* + *_2), \Gm(*+*), \Gm(*+*+*)\}\right\} \\
			&= \mex\left\{0, 3, \mex\left\{3, 1, \mex\{\Gm(*+*)\}\right\}\right\} \\
			&= \mex\left\{0, 3, \mex\left\{3, 1, \mex\{1\}\right\}\right\} \\
			&= \mex\{0,3,\mex\{3,1,0\}\} \\
			&= \mex\{0,3,2\}\\
			&= 1,\\
			\Gm(*_2 + *_2 + *_2) &= \mex\{\Gm(* + *_2 + *_2), \Gm(*_2 + *_2)\} \\
			&= \mex\{1,0\}\\
			&= 2.
		\end{align*}
	We now have $\Genus{0} = 0^{120\cdots}$, $\Genus{*} = 1^{031\cdots}$, and $\Genus{*_2} = 2^{202\cdots}$.  Further calculations, which are left to the interested reader, give us
		\begin{align*}
			\Genus{0} &= 0^{1202020202\cdots} \\
			\Genus{*} &= 1^{0313131313\cdots} \\
			\Genus{*_2} &= 2^{2020202020\cdots}.
		\end{align*}
	Therefore, we have calculated the genera of 0, $*$, and $*_2$.  
\end{example}

In each of the three positions in the preceding example, the genus of the position eventually alternates between two digits.  This is true for all genera \cite{MTHESIS,  ONAG}.  Because of this, we usually truncate the genus so that the last two digits in it are the two digits which continue to repeat indefinitely.  In the preceding example, we would then write
		\begin{align*}
			\Genus{0} &= 0^{120}, \\
			\Genus{*} &= 1^{031}, \\
			\Genus{*_2} &= 2^{20}.
		\end{align*}	

\begin{example}\cite{MTHESIS, ONAG}
	Let $h_n$ denote a {\sc nim} heap of $n$ tokens.  Then
		\begin{align*}
			\Genus{h_n} &= \begin{cases}
				0^{120} &\text{if } n=0;\\
				1^{031} &\text{if } n=1;\\
				n^{n(n \oplus 2)} &\text{else}.
			\end{cases}
		\end{align*}
	where
		\begin{align*}
			n \oplus 2 &= \begin{cases}
				n+2 &\text{if } n \equiv 0,1 \imod 4;\\
				n-2 &\text{if } n \equiv 2,3 \imod 4.
			\end{cases}
		\end{align*}
\end{example}

We use genus to determine how similar a position behaves to a {\sc nim} heap.  Our concern with {\sc nim} is a direct consequence of the Sprague-Grundy Theorem for normal play impartial games \cite{GRUNDY, SPRAGUE}, which says that every normal play  impartial position is equivalent to a {\sc nim} heap of a certain size, where we define equivalent as follows.

\begin{definition}\label{def-equivalent}
	Given two positions $\xi$ and $\kappa$, we say that $\xi$ is \textbf{equivalent} to $\kappa$ if for all positions $\gamma$, 
		\[ o^+(\xi + \gamma) = o^+(\kappa + \gamma).\]
\end{definition}

It is worth noting that $\xi$, $\kappa$, and most importantly $\gamma$, can be positions from vastly different games with vastly different rule sets.  

Again letting $h_n$ denote a {\sc nim} heap of $n$ tokens, given any position $\xi$ from some impartial game, the Sprague-Grundy Theorem for normal play impartial games says that there exists some $n \in \mathbb{Z}^{\ge 0}$ such that $\xi$ and $h_n$ are equivalent.   This is the most important result in impartial normal play game theory.  Unfortunately, there is no Sprague-Grundy Theorem for \mis play impartial games.  However, using genus, we can determine which positions behave like \mis play {\sc nim} heaps.  We call such positions \emph{tame}.

\begin{definition}
	A position $\xi$ is \textbf{tame} if
		\begin{itemize}
			\item there exists an $n \in \mathbb{Z}^{\ge 0}$ such that $\Genus{\xi} = \Genus{h_n}$ or $\Genus{\xi} = 0^{02}$ or $\Genus{\xi} = 1^{13}$ (these last two genera are equivalent to the genera of certain sums of {\sc nim} heaps);
			\item all options of $\xi$ are tame.
		\end{itemize}
	If it is not tame, then we say that $\xi$ is \textbf{wild}.  
	
	A game is tame if every position in the game is tame.  Otherwise it is wild.
\end{definition}

The structure of tame games is quite manageable.  Most notably, their outcomes behave nicely under disjunctive sum \cite{MTHESIS, ONAG}.  However, when a game is found to be wild (as many are, see \cite{MTHESIS, WW2, ONAG, TAMING} for some examples of wild games), genus does not easily allow us to say much else about the game.  Thus, while genus theory is a useful tool in the classification of impartial \mis play games, game theorists were still unable to analyse \mis play games as fully as normal play ones.

\subsection{Indistinguishability and \Mis Monoids}

The next breakthrough in \mis play theory was by Plambeck in 2004 \cite{TAMING}.  Recall that in our definition of equivalence (Definition \ref{def-equivalent}), we used positions from any game.  Plambeck found that by restricting himself to positions from the same game, he was able to construct a theory for \mis play games, the indistinguishability quotient and \mis monoid theory.  Working with Siegel, the two further developed the theory, publishing a number of papers which further advanced Plambeck's original idea, focusing on how the theory applied to impartial games \cite{ADVANCES, TAMING, MISQUOTIENT, NOTES, STRUCTURE}.  

To understand indistinguishability, we need some starting definitions.

\begin{definition}\label{definition-closed}
	A set of positions $\Upsilon$ is \textbf{closed} if it is
		\begin{enumerate}
			\item closed under addition, i.e.\ if $\alpha$, $\beta \in \Upsilon$, then $\alpha + \beta \in \Upsilon$, and
			\item option closed, i.e.\ if $\alpha \in \Upsilon$, then every option of $\alpha$ is also in $\Upsilon$.
		\end{enumerate}
\end{definition}

Since a closed set is option closed, the position 0 is always an element of a closed set.

Frequently, the set of positions over which we want to work is not closed.  As such, we are required to take the \textbf{closure of the set of positions}.

\begin{definition}
	Let $\Upsilon$ be a set of positions.  Then the \textbf{closure of $\Upsilon$}, denoted by $\cl{\Upsilon}$, is the smallest set of positions (in terms of set inclusion) such that $\Upsilon \subseteq \cl{\Upsilon}$.
\end{definition}

Given a set of positions $\Upsilon$, $\cl{\Upsilon}$ can be obtained by recursively taking options of all positions in $\Upsilon$, and then taking arbitrary disjunctive sums.

We note that the closure of a set of positions is itself closed.

If we had a set of positions $\{\xi_1, \xi_2, \ldots, \xi_n\}$ and we wish to calculate their closure, we drop the $\{$ and $\}$ in the notation, i.e.\ we write
	\[ \cl{\xi_1, \xi_2, \ldots, \xi_n}\]
rather than
	\[ \cl{\{\xi_1, \xi_2, \ldots, \xi_n\}}.\]
This is solely for ease of notation and to eliminate excess visual noise.

\begin{example}
	For our three positions, $0$, $1$, and $*$, their closures are as follows:
		\begin{itemize}
			\item $\cl{0} = \{0\}$,
			\item $\cl{1} = \{0,1, 1+1, 1+1+1, \ldots\}$
			\item $\cl{*} = \{0, *, *+*, *+*+*, \ldots\}$.
		\end{itemize}
\end{example}

We now define what it means for two positions to be indistinguishable.

\begin{definition}\label{def-indis}
	Suppose $\Upsilon$ is a closed set of positions with $\alpha$, $\beta \in \Upsilon$.  Then \textbf{$\alpha$ and $\beta$ are indistinguishable over $\Upsilon$} if
			\[ o^-(\alpha + \gamma) = o^-(\beta + \gamma) \text{ for all } \gamma \in \Upsilon,\]
	and we write 
		\[ \alpha \equiv \beta \imod{\Upsilon}.\]
	If $\alpha$ and $\beta$ are not indistinguishable, then we say that they are \textbf{distinguishable}.  Moreover, if $\alpha$ and $\beta$ are distinguishable, then there must exist some $\gamma \in \Upsilon$ such that
		\[ o^-(\alpha + \gamma) \not = o^-(\beta+ \gamma),\]
	and we say that \textbf{$\gamma$ distinguishes $\alpha$ and $\beta$}.
\end{definition}

\begin{example}\label{example-1}
	Consider $\cl{1}$.  Let $n 1$ denote $n$ copies of 1 under disjunctive sum.  We see
		\[ o^-(n1) = \begin{cases}
		\Next &\text{if } $n = 0$;\\
		\Right &\text{else}.
		\end{cases}\]
	Then
		\[ o^-(1 + \gamma) = o^-(n1 + \gamma)\]
	for all $n \ge 1$, $\gamma \in \cl{1}$.  Therefore
		\[ 1 \equiv n1 \imod{\cl{1}},\]
	but 
		\[ 1 \not \equiv 0 \imod{\cl{1}}\]
	since
		\[ o^-(1+0) = o^-(1) = \Right,\]
	while
		\[ o^-(0+0) = o^-(0) = \Next,\]
	so the two positions are distinguished by 0.
\end{example}

\begin{example}
	Consider $\cl{1, \overline{1}}$.  Moreover, consider the positions $1$ and $1+1$.  We saw in Example \ref{example-o+-o-} that $o^-(1) = \Right$.  It is easy to see that $o^-(1+1) = \Right$ also.  
	
	We now add $\overline{1}$ to both positions.  In doing such, we can see that $o^-(1 + \overline{1}) = \Next$ while $o^-(1+1+\overline{1}) = \Right$.  That is $\overline{1}$ distinguishes $1$ and $1+1$.  Therefore
		\[ 1 \not \equiv 1+1 \imod{\cl{1,\overline{1}}}.\]
\end{example}

Once two elements are distinguished, they remain distinguished as the following theorem shows.

\begin{proposition}\label{proposition-distinguishable-subseteq}
	Take a closed set of positions $\Upsilon$ with elements $\alpha$, $\beta \in \Upsilon$ such that $\alpha$ and $\beta$ are distinguishable over $\Upsilon$.  If $\Upsilon \subseteq \Gamma$ for some closed set of positions $\Gamma$, then $\alpha$ and $\beta$ remain distinguishable over $\Gamma$.
\end{proposition}

\begin{proof}
	Since $\alpha$ and $\beta$ are distinguishable over $\Upsilon$, this means there exists some $\gamma \in \Upsilon$ such that
		\[ o^-(\alpha + \gamma) \not = o^-(\beta + \gamma).\]
	Since $\Upsilon \subseteq \Gamma$, we have $\alpha$, $\beta$, and $\gamma \in \Gamma$, so $\gamma$ also distinguishes $\alpha$ and $\beta$ over $\Gamma$.
\end{proof}

It is important to note the following:  If $\Upsilon \subseteq \Gamma$ and $\alpha, \beta \in \Upsilon$, then $\alpha$ and $\beta$ may be indistinguishable in $\Upsilon$ but distinguishable in $\Gamma$, or vice versa.

Sometimes, when we want to be clear about the set from which we are building indistinguishability, we write $\stackrel{\Upsilon}{\equiv}$ rather than just $\equiv$.  

For a closed set $\Upsilon$ with indistinguishability relation $\stackrel{\Upsilon}{\equiv}$, Plambeck showed that $\stackrel{\Upsilon}{\equiv}$ is an equivalence relation \cite{TAMING}.  In fact, it is even stronger than that; $\stackrel{\Upsilon}{\equiv}$ is a congruence with respect to disjunctive sum.  That is, for $\alpha$, $\beta$ and $\gamma \in \Upsilon$, 
	\[ \alpha \stackrel{\Upsilon}{\equiv} \beta \implies \alpha + \gamma \stackrel{\Upsilon}{\equiv}\beta + \gamma \text{ } (\cite{ADVANCES}).\]

If we took $\Upsilon$ to be the set of all positions from all games, then we could call indistinguishability over this set \emph{\mis equivalence} (see Definition \ref{def-equivalent}).  However, we rarely take $\Upsilon$ to be this large set; rather $\Upsilon$ is usually taken to be a smaller closed set, such as all the positions which occur in a particular game.

Before we can introduce Plambeck's \emph{indistinguishability quotient}, we must recall some basic algebra definitions.

\begin{definition}
	A \textbf{monoid} is a set $\monoid{M}$ along with a binary operation $\star: \monoid{M} \times \monoid{M} \to \monoid{M}$ such that
		\begin{itemize}
			\item $\star$ is associative, that is for $a$, $b$, $c \in \monoid{M}$, $a \star (b \star c) = (a \star b) \star c$;
			
			\item there exists an element $e \in \monoid{M}$, called the \textbf{identity element of $\monoid{M}$}, such that $a \star e = a = e \star a$. 
		\end{itemize}
	The monoid is \textbf{commutative} if the following is also true:
		\begin{itemize}
			\item for all $a$, $b \in \monoid{M}$, $a \star b = b \star a$.
		\end{itemize}
\end{definition}

In other words, we can consider a monoid as a semigroup with identity, or as group without inverses.  

While we should denote a monoid by $(\monoid{M}, \star, e)$, i.e.\\ the set, the binary operation, and the identity element, in practice, we often just use $\monoid{M}$ to denote the monoid when the binary operation and the identity are in some way obvious.

We now construct the \textbf{indistinguishability quotient of $\Upsilon$}.  We do this as follows:
	\begin{enumerate}
		\item Take a closed set of positions $\Upsilon$ with indistinguishability relation $\stackrel{\Upsilon}{\equiv}$ as given in Definition \ref{def-indis}.  
		
		\item Calculate the quotient of $\Upsilon$ over $\stackrel{\Upsilon}{\equiv}$, denoted by $\quotient{\Upsilon}{\stackrel{\Upsilon}{\equiv}}$.
		
		\item Each element of $\quotient{\Upsilon}{\stackrel{\Upsilon}{\equiv}}$ is written as $\alpha + \Upsilon$ and called \textbf{the equivalence class of $\alpha$ over $\Upsilon$}.
		
		\item Given two classes $\alpha + \Upsilon$ and $\beta + \Upsilon$,
			\[\alpha + \Upsilon = \beta + \Upsilon\]
			if
			\[o^-(\alpha + \gamma) = o^-(\beta + \gamma) \text{ for all } \gamma \in \Upsilon.\]  
		
		\item Given two classes $\alpha + \Upsilon$ and $\beta + \Upsilon$, we have
			\[ (\alpha + \Upsilon) + (\beta + \Upsilon) = (\alpha + \beta) + \Upsilon.\]
		where the addition in $(\alpha + \beta)$ is disjunctive sum.  We say that the sum in $\quotient{\Upsilon}{\stackrel{\Upsilon}{\equiv}}$ is \textbf{inherited from disjunctive sum}.
	\end{enumerate}
Since $\stackrel{\Upsilon}{\equiv}$ is a congruence on $\Upsilon$, all the above are well-defined.  Moreover, 
	\[\left(\quotient{\Upsilon}{\stackrel{\Upsilon}{\equiv}}, +, 0 + \Upsilon\right)\]
forms a monoid \cite{TAMING}.  Since disjunctive sum is commutative, the inherited sum is also commutative, so we have a commutative monoid.   That is, we define the indistinguishability quotient of $\Upsilon$ as follows:

\begin{definition}
	For $\Upsilon$ a closed set of positions, the commutative monoid formed from indistinguishability $\stackrel{\Upsilon}{\equiv}$, with the sum inherited from disjunctive sum, and the  identity being the equivalence class of 0 under $\stackrel{\Upsilon}{\equiv}$ is called the \textbf{indistinguishability quotient of $\Upsilon$}.  
\end{definition}

\begin{example}
	Returning to Example \ref{example-1}, we see that the indistinguishability quotient of $\cl{1}$ has two elements, $0 + \cl{1}$ and $1 + \cl{1}$.
\end{example}

Sometimes we are interested in the map which takes $\Upsilon$ to its indistinguishability quotient.  

\begin{definition}\label{def-quotient-map}
	Let $\Upsilon$ be a closed set of positions with indistinguishability relation $\stackrel{\Upsilon}{\equiv}$.  Then the map
		\begin{center}
			\begin{tabular}{cccc}
			$\mathscr{Q}:$&$\Upsilon$&$\rightarrow$&$\quotient{\Upsilon}{\stackrel{\Upsilon}{\equiv}}$\\
			      &	$\alpha$&	$\mapsto$&	$\alpha+\Upsilon$ \\
			\end{tabular}
		\end{center}
	is called the \textbf{canonical quotient map of $\Upsilon$}.  
\end{definition}

In older \mis monoid papers, most notably \cite{ADVANCES, TAMING}, the map in Definition \ref{def-quotient-map} is referred to as the \emph{pretending function}.  The reason for this was that proofs regarding the validity of this method had yet to be completed.  As we now know the method to be correct, we are no longer \emph{pretending} that this method gives a valid result;  we know it does.  

Much as we divided positions into outcome classes in Definition \ref{def-outcomes}, we divide equivalence classes as well based on outcomes.

\begin{definition}\label{def-outcome-tetra}
	We divide $\quotient{\Upsilon}{\stackrel{\Upsilon}{\equiv}}$ into four \textbf{outcome portions}, $\Left$, $\Right$, $\Next$, and $\Prev$, where an equivalence class $\alpha + \Upsilon$ is placed into an outcome portion such that $o^-(\alpha + \Upsilon) = o^-(\alpha)$.
\end{definition}

Because each position is in exactly one outcome class, each equivalence class is placed into exactly one of the outcome portions.  Using the facts that  outcome classes are disjoint \cite{LIP} and the definition of $\stackrel{\Upsilon}{\equiv}$, we see that the outcome portions of $\quotient{\Upsilon}{\stackrel{\Upsilon}{\equiv}}$ are also disjoint.  

If we think of $o^-$ as a function which takes positions and equivalence classes of positions to elements of $\{\Next, \Prev, \Left, \Right\}$, and $\mathscr{Q}$ the canonical quotient map, then what we have is a commutative diagram given in Figure \ref{fig-commutative-diagram}.
		\begin{figure}[htb]
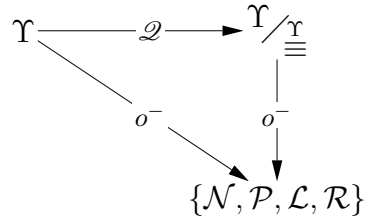

			\unitlength 16pt
			\begin{center}
			\begin{graph}(6,4)(0,0)
			\graphlinecolour{1}
			\graphlinewidth{.01}
			\grapharrowlength{.5}
			\newcommand{\arr}[3]{%
			\edge{#1}{#2}[\graphlinewidth{.1}]
			\diredge{#1}{#2}[\graphlinecolour{0}]
			\edgetext{#1}{#2}{\footnotesize #3}}

			\textnode{A}(6,0){$\{\Next, \Prev, \Left, \Right\}$}
			\textnode{B}(6,4){$\quotient{\Upsilon}{\stackrel{\Upsilon}{\equiv}}$}
			\textnode{E}(0,4){$\Upsilon$}
			
			\arr{B}{A}{$o^-$}
			\arr{E}{A}{$o^-$}
			\arr{E}{B}{$\mathscr{Q}$}

			\end{graph}	
		\caption{Commutativity of $o^-$ between positions in $\Upsilon$ and their equivalence classes under $\Upsilon$.}
		\label{fig-commutative-diagram}
		\end{center}
		\end{figure} 

To make it easier to compare indistinguishability quotients of different sets, Plambeck developed the \mis monoid notation.  To form the \mis monoid of an indistinguishability quotient, addition is replaced by multiplication, and each equivalence class is denoted by (generally) a lower-case Roman letter.   Indistinguishability relations are denoted as relations on the monoid.

\begin{notation}
	We denote the \mis monoid of $\Upsilon$ by $\monoid{M}_{\Upsilon}$.
\end{notation}

Sometimes, rather than write another map which re-labels the elements of the indistinguishability quotient to obtain the \mis monoid, we just say that the canonical quotient map (Definition \ref{def-quotient-map}) goes from $\Upsilon$ to $\monoid{M}_{\Upsilon}$.

As with the indistinguishability quotient, we also divide the \mis monoid into outcome portions. When we write out the \mis monoid, we almost always also write out the outcome tetrapartition, as we are generally concerned with the outcomes of positions from our initial set $\Upsilon$.

\begin{example}\label{example-M-1}
	We now construct the monoid for the \mis monoid of $\cl{1}$.  With the mappings:
		\begin{align*}
			0 &\mapsto 1;\\
			1 &\mapsto a,
		\end{align*}
	we obtain the following monoid:
		\begin{align*}
		\monoid{M}_{\cl{1}} &= \ideal{1,a \mid a^2 = a} \\
		\Next &= \{1\} \\
		\Prev &= \emptyset \\
		\Left &= \emptyset \\
		\Right &= \{a\}.
		\end{align*}
\end{example}

\begin{example}\label{example-*}\cite{TAMING}
	We wish to construct $\monoid{M}_{\cl{*}}$.  With the mappings:
		\begin{align*}
			0 &\mapsto 1;\\
			1 &\mapsto a,
		\end{align*}
	we obtain the following monoid:
		\begin{align*}
		\monoid{M}_{\cl{*}} &= \ideal{1,a \mid a^2=1} \\
		\Next &= \{1\} \\
		\Prev &= \{a\} \\
		\Left &= \emptyset \\
		\Right &= \emptyset.
		\end{align*}

	Examining the relations on $\monoid{M}_{\cl{*}}$, the astute reader may realise that the elements not only form a recognisable monoid; they form a recognisable group, namely $(\mathbb{Z}_2, \oplus)$.  However, it is important to note that $\monoid{M}_{\cl{*}}$ is not the same as $(\mathbb{Z}_2, \oplus)$, as the latter does not have an outcome tetrapartition associated with it and hence is not a \mis monoid.  
\end{example}

If we calculate the \mis monoid of an impartial position, then in the outcome tetrapartition we have $\Left = \emptyset$ and $\Right = \emptyset$, as it does in the preceding example with $\monoid{M}_{\cl{*}}$.  While it is tempting to assume that if we are given a \mis monoid with $\Left = \emptyset$ and $\Right = \emptyset$, then it must have come from an impartial position, this is not true as examples in Chapter \ref{chapter-examples} demonstrate. 

By using this \mis monoid construction, we are now able to analyse games with respect to themselves.  This differs from normal play, in which we analyse games with respect to all other games (see Definition \ref{def-equivalent}).  However, this seems more natural in practise.  In playing a position, it often decomposes into sub-positions (like in {\sc go} where play splits the board into different regions), but the initial position and these sub-positions are all from the same game; how often in practise does one play a game of {\sc go} in conjunction with a game of {\sc chess}?  \Mis monoid theory takes advantage of this in its analysis of \mis play games.

Plambeck and Siegel's work is on impartial \mis play games \cite{ADVANCES, TAMING, MISQUOTIENT, NOTES, STRUCTURE};  however the construction of \mis monoids is not specific to impartial games, as we saw in the construction of $\monoid{M}_{\cl{1}}$ in Example \ref{example-M-1}.  The structure and the theory behind partizan \mis monoids differs from their impartial counterparts.  This thesis takes the \mis monoid construction and extends it to partizan \mis play games.  In doing such, we develop the basis of partizan \mis monoid theory.

\section{The Partial Order on Positions}\label{sec-poset}

In Chapter \ref{chapter-examples}, we calculate the partial order of the \mis monoids given in the examples, in the hopes that we can find connections between the partial order and some other theoretical results. This section reviews the definitions and tools necessary to calculate the partial order of an indistinguishability quotient.

In this section, we will work under the multiplicative notation of the \mis monoid;  however, all the definitions and results could be stated in terms of the additive notation of the indistinguishability quotient, should one so desire.

We define $\ge$ as follows:

\begin{definition}\label{def-<=}
	Let $\Upsilon$ be a closed set of positions with \mis monoid $\monoid{M}_{\Upsilon}$ and $x,y \in \monoid{M}_{\Upsilon}$.  We say that \textbf{$x$ is greater than $y$}, and write $x \ge y$ if
		\[ x \ge y \text{ if } o^-(xz) \ge o^-(yz) \text{ for all } z \in \monoid{M}_{\Upsilon},\]
	where the outcome lattice is still given in Figure \ref{fig-outcome-poset-1}. 
	
	If $x \not \ge y$ and $y \not \ge x$, then we say that $x$ and $y$ are \textbf{incomparable}.

	\begin{figure}[htb]
	\unitlength 15pt
	\begin{center}
	\begin{graph}(4,4)(0,0)
	\graphlinecolour{1}
	\graphlinewidth{.01}
	\grapharrowlength{.5}

	\textnode{L}(2,4){$\Left$}
	\textnode{P}(0,2){$\Prev$}
	\textnode{N}(4,2){$\Next$}
	\textnode{R}(2,0){$\Right$}
	
	\newcommand{\arr}[2]{%
	\edge{#1}{#2}[\graphlinewidth{.01}\graphlinecolour{0}]
	}
	
	\arr{L}{P}
	\arr{L}{N}
	\arr{P}{R}
	\arr{N}{R}
	
	\end{graph}
	\caption{Outcome Class Partial Order.}
	\label{fig-outcome-poset-1}
	\end{center}
	\end{figure}
\end{definition}

Figure \ref{fig-outcome-poset-1} gives two chains of outcome ordering, namely $\Left \ge \Next \ge \Right$ and $\Left \ge \Prev \ge \Right$.  $\Next$ and $\Prev$ are incomparable elements in the outcome ordering.  That is, if for two positions $a, b \in \monoid{M}_{\Upsilon}$, if $o^-(a) = \Next$ and $o^-(b) = \Prev$, then these two elements are incomparable.  

We place the partial order on $\Right$, $\Next$, $\Prev$, and $\Left$ in the way given in Figure \ref{fig-outcome-poset-1} as under normal play, this positions are assigned values which are compatible with this partial ordering (i.e. a position in outcome class $\Left$ is assigned a value which is greater than that of any position in $\Right$, $\Next$, or $\Prev$).  For a further explanation of normal play values, please see \cite{LIP, WW1, ONAG}.  

We will make use of the following results in calculating partial orders in Chapter \ref{chapter-examples}:

\begin{proposition}\label{prop-results-<=}
	For $\Upsilon$ a closed set, in $\monoid{M}_{\Upsilon}$, we have the following:

	\begin{enumerate}
		\item If $o^-(a) \not \ge o^-(b)$, then $a \not \ge b$.
		
		\item If $a$ and $b$ are incomparable, then either
		\begin{enumerate}
			\item there exists $z$ such that $o^-(az) = \Next$ while $o^-(bz) = \Prev$; or
			
			\item there exist $z_1$ and $z_2$ such that:
				\begin{itemize}
					\item$o^-(az_1) \ge o^-(bz_1)$,
					\item$o^-(az_2) \le o^-(bz_2)$,
				\end{itemize}
			with at least one of these two inequalities being strict.
		\end{enumerate}
	\end{enumerate}
\end{proposition}

\begin{proof}\text{}
	\begin{enumerate}
		\item 	If $o^-(a) \not \ge o^-(b)$, then $a \not \ge b$ since $o^-(a1) \not \ge o^-(b1)$.  
		\item This follows from the definition of incomparability. \qedhere
	\end{enumerate}
\end{proof}

We also have a result which we will make use of in Chapter \ref{chapter-examples}.  It is as follows.  Since this result is regarding conjugates, information which is not kept when examining the \mis monoid, this result will be stated in terms of the indistinguishability quotient's additive notation rather than the \mis monoid.

\begin{proposition}\label{prop->-<-conjugate}
	Let $\Upsilon$ be a set of positions such that if $\xi \in \cl{\Upsilon}$, then $\overline{\xi} \in \cl{\Upsilon}$ as well.  Take $\alpha, \beta  \in \cl{\Upsilon}$.  Then we have the following:
		\begin{enumerate}
			\item\label{prop->-<-conjugate-<=} If $\alpha \le \beta$, then $\overline{\alpha} \ge \overline{\beta}$.
			\item\label{prop->-<-conjugate-||} If $\alpha$ and $\beta$ are incomparable, then so are $\overline{\alpha}$ and $\overline{\beta}$.
		\end{enumerate}
\end{proposition}

\begin{proof}\text{ }
	Recall from Theorem \ref{theorem-outcome-class-conj},
		\begin{align*}
			o^-(\xi) = \Left &\implies o^-(\overline{\xi}) = \Right;\\
			o^-(\xi) = \Next &\implies o^-(\overline{\xi}) = \Next;\\
			o^-(\xi) = \Prev &\implies o^-(\overline{\xi}) = \Prev;\\
			o^-(\xi) = \Right &\implies o^-(\overline{\xi}) = \Left.
		\end{align*}	
	
	\begin{enumerate}
		\item Take arbitrary $\gamma \in \cl{\Upsilon}$.  We want 
			\[o^-(\gamma + \overline{\alpha}) \ge o^-(\gamma + \overline{\beta}).\]
			We have $\overline{\gamma} \in \cl{\Upsilon}$, and since $\alpha \le \beta$, this means
				\[ o^-(\overline{\gamma} + \alpha) \le o^-(\overline{\gamma} + \beta).\]
			By how outcomes work under conjugation, we get
				\[ o^-(\overline{\overline{\gamma} + \alpha)} \ge o^-(\overline{\overline{\gamma} + \beta}).\]
			But
				\begin{align*}
					\overline{\overline{\gamma} + \alpha} &= \gamma + \overline{\alpha},\\
					\overline{\overline{\gamma} + \beta} &= \gamma + \overline{\beta},
				\end{align*}
			which gives our desired result.
		\item The proof for this assertion is similar to that of the proof for the previous case. \qedhere
	\end{enumerate}
\end{proof}

Finally, in classifying the partial orders, we use the following definitions.

\begin{definition}
	A partial order $\mathscr{P}$ is 
		\begin{itemize}
			\item \textbf{down directed} if for all elements $x$, $y \in \mathscr{P}$, there exists an element $z \in \mathscr{P}$ such that $x \ge z$ and $y \ge z$.  
			
			\item \textbf{up directed} if for all elements $x$, $y \in \mathscr{P}$, there exists an element $z \in \mathscr{P}$ such that $z \ge x$ and $z \ge y$.
			
			\item a \textbf{lattice} if it contains all binary meets and joins.  That is, it is a lattice if for all elements $x$, $y \in \mathscr{P}$,
				\begin{itemize}
					\item there exists an element $z \in \mathscr{P}$ such that $x \ge z$ and $y \ge z$ and for any other element $w \in \mathscr{P}$ with $x \ge w$, $y \ge w$, we have $z \ge w$,
					
					\item there exists an element $a \in \mathscr{P}$ such that $a \ge x$ and $a \ge y$, and for any other element $b \in \mathscr{P}$ with $b \ge x$ and $b \ge y$, we have $b \ge a$.  
				\end{itemize}
		\end{itemize}
\end{definition}

We conclude with the following proposition.

\begin{proposition}
	If the partially ordered set $\mathscr{P}$ is finite, then:
		\begin{enumerate}
			\item $\mathscr{P}$ is down directed $\iff$ there exists a minimum element.
			\item $\mathscr{P}$ is up directed $\iff$ there exists a maximum element.
		\end{enumerate}
\end{proposition}

Further information on partially ordered sets can be found in \cite{POSET}.

\section{Thesis Layout}\label{sec-thesis}
The remainder of this thesis is laid out as follows:

\begin{itemize}
	\item Chapter \ref{chapter-examples} calculates five indistinguishability quotients and \mis monoids.  Four of these are from partizan positions, while one is from an impartial position.  These examples demonstrate the wide range of \mis monoids which may occur, from finite to infinite, lattice to non-lattice, etc.  After finishing Chapter \ref{chapter-examples}, the reader should be intimately aware of the calculations behind finding \mis monoids and their partial orders.
	
	\item Chapter \ref{chapter-cardinality-left} discusses the cardinality of \mis monoids.  In it, we give some positions which ensure that the monoid is infinite.  A set of positions which always have finite monoids is also given.
	
	\item Chapter \ref{chapter-stars-0} looks at conditions on $\xi$ for when $* + * \equiv 0 \imod{\cl{\xi}}$. In particular, we extend the known result for impartial games to encompass all all-small games.
	
	\item Chapter \ref{chapter-0} constructs a set of positions, which we call $\ab{3}$, such that for $\xi$ an $\ab{3}$ position, $\xi + \xib \equiv 0 \imod{\cl{\text{$\ab{3}$}}}$, a result which mimics that of normal play.   Chapter \ref{chapter-0} also shows how these $\ab{3}$ positions also demonstrate a Tweedledum-Tweedledee type strategy under \mis play.
	
	\item Chapter \ref{chapter-cat} discusses our current difficulties in building a non-trivial category of \mis play games.
	
	\item Chapter \ref{chapter-congruent} focuses on \mis monoids being isomorphic.  Particular emphasis is given on being isomorphic to $\monoid{M}_{\cl{*}}$.  The two most important results of the thesis lie in this chapter, namely we give necessary and sufficient conditions on a set of positions $\Upsilon$ such that $\monoid{M}_{\cl{\Upsilon}} \cong \monoid{M}_{\cl{*}}$, and a construction theorem which builds all positions $\xi$ such that $\monoid{M}_{\cl{\xi}} \cong \monoid{M}_{\cl{*}}$.
	
	\item Chapter \ref{chapter-heap} calculates the \mis monoid of two heap based games using the new method developed by Weimerskirch \cite{MIKE}.  
	
	\item Chapter \ref{chapter-conclusion} concludes the thesis, listing some open problems and future avenues of research for partizan \mis play theory.
	
	\item Appendix \ref{appendixA} gives a list and details of the most used positions in this thesis.
\end{itemize}

This thesis endeavours to give detailed proofs of most of the results in this thesis to ensure that those unfamiliar with the style of game theoretic proofs are able to follow.

\chapter{Examples of \Mis Monoids for Certain Partizan Positions}\label{chapter-examples}

\section{Introduction}
It is vital that we understand how to calculate the \mis monoids of partizan positions.  The role of this chapter is to familiarize the reader with the work required the \mis monoid of a position, as well as its partial order.  While the literature contains examples of this for impartial positions \cite{MTHESIS, ADVANCES, TAMING, MISQUOTIENT, NOTES, STRUCTURE}, there have been no examples of partizan ones.  

This chapter calculates the \mis monoids of five different positions.  Four of these positions are partizan while one is impartial, which is included as it will be of use to us in Chapter \ref{chapter-cardinality-left}.  The partizan examples were chosen as they demonstrate the huge differences of partizan \mis monoids - from finite to infinite, from lattices to sets of incomparable elements, from \mis monoids isomorphic to $\cl{*}$ to \mis monoids which are  beyond simple classification.  

Before we begin, we need a piece of notation.

\begin{notation}
	For $n \in \mathbb{Z}^{\ge 0}$, $\xi$ a position in some game, $n  \xi$ denotes the disjunctive sum of $n$ copies of $\xi$.   
\end{notation}

\begin{example}
	The position $2 *$ is the disjunctive sum of 2 copies of $*$, i.e.\ $* + *$.  
\end{example}

Let the games begin!

\section{The \Mis Monoid of $\cl{1, \bar{1}}$}\label{sec-1-1}

This section gives an example of a \mis monoid with the following properties:
	\begin{itemize}
		\item infinite cardinality,
		\item inverses of elements exist, meaning that the monoid is also a group,
		\item partial order is a lattice.
	\end{itemize}
	
\begin{definition}
	Recall from Example \ref{example-0-1-*}, the position $1$ which we define as $1 = \combgame{\{0\mid\cdot\}}$.  That is, 1 is the position where Left has one move to 0 while Right has no move.
	
	The position $\overline{1}$ is defined as $\overline{1} = \combgame{\{\cdot\mid0\}}$.  That is, $\overline{1}$ is the position where Right has one move to 0 while Left has no move.
\end{definition}

The game trees of $1$ and $\overline{1}$ are given in Figure \ref{fig-gt-1}.

	\begin{figure}[htb]
	\unitlength 8pt
	\begin{center}
	\begin{graph}(6,2)(0,0)
	\graphnodesize{0.4}
	\fillednodestrue

	\roundnode{A}(1,2)
	\roundnode{B}(0,0)
		
	\edge{A}{B}

	\roundnode{A1}(5,2)
	\roundnode{C1}(6,0)
		
	\edge{A1}{C1}
		
	\end{graph}
	\caption{The game trees of $1$ and $\overline{1}$.}
	\label{fig-gt-1}
	\end{center}
	\end{figure}
	
We wish to examine $\cl{1, \overline{1}}$.  Elements of this set are of the form $a1 + b \overline{1}$ for $a,b \in \mathbb{Z}^{\ge 0}$.  

We now determine the outcome classes for all positions in the closure.

\begin{proposition}\label{prop-a1+b(-1)}
	Suppose $a,b \in \mathbb{Z}^{\ge 0}$.  Then
		\[ o^-(a  1 + b  \bar{1}) = \begin{cases}
		\Next &\text{if } a = b;\\
		\Right &\text{if } a > b;\\
		\Left &\text{if } a < b.
		\end{cases}\]
\end{proposition}

\begin{proof}
	If $a=b$, then it is a simple parity argument allowing whoever moves first to win.
	
	If $a > b$, then there are more moves available for Left, so Right will run out of moves before Left does.  If $a<b$, then the argument is reversed.
\end{proof}

We are now concerned with which elements in $\cl{1, \overline{1}}$ are distinguishable and which elements are indistinguishable (the definitions for distinguishable and indistinguishable are given by Definition \ref{def-indis}).  In determining this, we will show that an infinite number of elements of $\cl{1, \overline{1}}$ are distinguishable.  This means that the \mis monoid of $\cl{1, \overline{1}}$ is infinite.  

\begin{proposition}\label{prop-a1-b(-1)}
	The positions $a1$ and $b\overline{1}$ are distinguishable for any $a, b \in \mathbb{Z}^{\ge 0}$ provided $a,b > 0$.
\end{proposition}

\begin{proof}
	Since $a \not = b$, we have $o^-(a1) = \Right$ while $o^-(b \overline{1}) = \Left$, so these two positions are distinguished by 0.
\end{proof}

\begin{proposition}\label{prop-1-a1}
	The positions $1$ and $a  1$ are distinguishable for any $a \in \mathbb{Z}^{\ge 2}$.  
\end{proposition}

\begin{proof}
	$\bar{1}$ distinguishes $1$ and $a  1$ for any $a \in \mathbb{Z}^{\ge 2}$ since by Proposition \ref{prop-a1+b(-1)}, $o^-(1  + \bar{1}) = \Next$ while $o^-(a  1 + \bar{1}) = \Right$.  
\end{proof}

\begin{corollary}\label{cor-a1-b1-distinguishable}
	The positions $a1$ and $b1$ are distinguishable for any $a,b \in \mathbb{Z}^{\ge 2}$ where $a \not = b$.
\end{corollary}

\begin{proof}
	Suppose, without loss of generality, that $a < b$.  Then Proposition \ref{prop-a1+b(-1)} gives $o^-(a1 + a \overline{1}) = \Next$ while $o^-(b1 + a \overline{1}) = \Right$.  Therefore $a \overline{1}$ distinguishes $a1$ and $b1$.  
\end{proof}

\begin{corollary}\label{cor-1,bar1-infinite}
	The \mis monoid of $\cl{1, \bar{1}}$ is infinite.
\end{corollary}

\begin{proof}
	By Corollary \ref{cor-a1-b1-distinguishable}, there is an infinite number of positions which are distinguishable.  Therefore the \mis monoid of $\cl{1, \bar{1}}$ is infinite.
\end{proof}

We also have the following corollary.

\begin{corollary}
	The positions $a\overline{1}$ and $b\overline{1}$ are distinguishable for any $a,b \in \mathbb{Z}^{\ge 0}$ where $a \not = b$.  
\end{corollary}

\begin{proof}
	Take the arguments of Proposition \ref{prop-1-a1} and Corollary \ref{cor-a1-b1-distinguishable} replacing $1$ with $\overline{1}$ and vice versa.
\end{proof}

We must also determine which elements are indistinguishable.

\begin{proposition}\label{prop-a1-b1-b-a}
	Suppose $a, b \in \nat$.  Then
		\[ a1 + b\bar{1} \equiv \begin{cases}
		(a-b)1 \imod{\cl{1,\bar{1}}} &\text{if } a > b;\\
		(b-a)\bar{1} \imod{\cl{1,\bar{1}}} &\text{if } a \le b.
		\end{cases}
		\]
\end{proposition}

\begin{proof}
	Suppose $a > b$.  To show
		\[ a1 + b\overline{1} \equiv (a-b)1 \imod{\cl{1,\bar{1}}},\]
	we must show that for an arbitrary position of $\cl{1, \overline{1}}$, say $c1 + d\overline{1}$, 
		\[ o^-((a1 + b\overline{1})+(c1 + d \overline{1})) = o^-((a-b)1 + (c1 + d\overline{1})).\]
	So, take arbitrary $c1 + d\bar{1}$.  Adding this position to $a1 + b \bar{1}$ gives, by Proposition \ref{prop-a1+b(-1)}, 
		\begin{align*}
			o^-((a+c) 1 + (b+d)  \bar{1}) &= \begin{cases}
			\Next &\text{if } a+c = b+d;\\
			\Right &\text{if } a+c > b+d;\\
			\Left &\text{if } a+c < b+d.
			\end{cases}
		\end{align*}	
	while adding it to $(a-b) 1$ gives
		\begin{align*}
			o^-((a-b+c) 1 + d  \bar{1}) &= \begin{cases}
			\Next &\text{if } a-b+c = d;\\
			\Right &\text{if } a-b+c > d;\\
			\Left &\text{if } a-b+c < d.
			\end{cases}\\
			&= \begin{cases}
			\Next &\text{if } a+c = b+d;\\
			\Right &\text{if } a+c > b+d;\\
			\Left &\text{if } a+c < b+d.
			\end{cases}
		\end{align*}	
	Since $c1 + d\bar{1} \in \cl{1,\bar{1}}$ was arbitrary, this gives that
		\[a1 + b\bar{1} \equiv (a-b)1 \imod{\cl{1,\bar{1}}} \text{ if } a > b,\]
	as required.  The case where $a \le b$ follows similarly.
\end{proof}

We have now checked distinguishability and indistinguishability for all positions in $\cl{1, \overline{1}}$.  We can now explicitly write the \mis monoid of $\cl{1,\overline{1}}$. With the mappings:
	\begin{align*}
		0 &\mapsto 1;\\
		a1 &\mapsto x^{-a};\\
		b\bar{1} &\mapsto x^{b},
	\end{align*}
we obtain the following monoid:
	\begin{align*}
	\monoid{M}_{\cl{1, \bar{1}}} &= \ideal{x^n \text{ where } n \in \mathbb{Z} \mid x^n x^m = x^{n+m}} \\
	\Next &= \{1\} \\
	\Prev &= \emptyset \\
	\Left &=  \{x^n \mid n \in \nat\}\\
	\Right &= \{x^{-n} \mid n \in \nat\}
	\end{align*}
However, $\monoid{M}_{\cl{1, \overline{1}}}$ is more than a monoid;  the relation $x^{n}x^{m} = x^{n+m}$ shows that each element of $\monoid{M}_{\cl{1,\overline{1}}}$ has an inverse, namely the inverse of $x^n$ is $x^{-n}$.  This means that $\monoid{M}_{\cl{1, \overline{1}}}$ is also group.

\begin{example}
	Thus, for example, to determine what element of the monoid the position $7 \cdot 1 + 18 \cdot \bar{1}$ we proceed as follows:
		\begin{enumerate}
			\item $7 \cdot 1 \mapsto x^{-7}$ while 	$18 \cdot \bar{1} \mapsto x^{18}$.
			\item Then $(7 \cdot 1 + 18 \cdot \bar{1}) \mapsto x^{-7}x^{18}$.
			\item Adding $-7$ and $18$ together gives $11$.  
			\item Therefore $(7 \cdot 1  + 18 \cdot \bar{1}) \mapsto x^{11}$.  
		\end{enumerate}
	Equally, we could apply Proposition \ref{prop-a1-b1-b-a} first to get that $7 \cdot 1 + 18 \cdot \bar{1} \equiv 11 \cdot \bar{1} \imod{\cl{1,\bar{1}}}$, and then $11 \cdot \bar{1} \mapsto x^{11}$.  
\end{example}
	
We now examine the partial order of the elements.  Recall from Definition \ref{def-<=} that 
	\[ x \ge y \text{ if } o^-(xz) \ge o^-(yz) \text{ for all } z \in \monoid{M}_{\Upsilon},\]
and that, in terms of outcomes, the outcome lattice is given in Figure \ref{fig-outcome-poset}.

	\begin{figure}[htb]
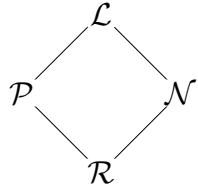

	\unitlength 15pt
	\begin{center}
	\begin{graph}(4,4)(0,0)
	\graphlinecolour{1}
	\graphlinewidth{.01}
	\grapharrowlength{.5}

	\textnode{L}(2,4){$\Left$}
	\textnode{P}(0,2){$\Prev$}
	\textnode{N}(4,2){$\Next$}
	\textnode{R}(2,0){$\Right$}
	
	\newcommand{\arr}[2]{%
	\edge{#1}{#2}[\graphlinewidth{.01}\graphlinecolour{0}]
	}
	
	\arr{L}{P}
	\arr{L}{N}
	\arr{P}{R}
	\arr{N}{R}
	
	\end{graph}
	\caption{Outcome Class Partial Order.}
	\label{fig-outcome-poset}
	\end{center}
	\end{figure}

We are now equipped with enough tools to determine the partial order of the elements.

\begin{proposition}
	The partial order of the elements of $\monoid{M}_{\cl{1,\bar{1}}}$ is
		\[ \cdots < x^{-3} <  x^{-2} < x^{-1} < 1 < x^{1} < x^{2}  < x^{3} < \cdots. \]
\end{proposition}

\begin{proof}
	Take $x^{a}$, $x^{a+1} \in \monoid{M}_{\cl{1,\bar{1}}}$, and arbitrary $x^{b} \in \monoid{M}_{\cl{1,\bar{1}}}$.  To show that $x^{a} \le x^{a+1}$, we want to show that $o^-(x^{a+b}) \le o^-(x^{a+1+b})$. 
	
	Rewriting Proposition \ref{prop-a1-b1-b-a} in terms of the monoid notation of $\monoid{M}_{\cl{1,\overline{1}}}$, we obtain 
		\[ o^-(x^{a+b}) = \begin{cases}
			\Next &\text{if } a=-b;\\
			\Left &\text{if } a+b > 0;\\
			\Right &\text{if } a+b < 0.
		\end{cases}
		\]
	
	If $a=-b$, then $a+1-b > 0$ which gives $o^-(x^{a+1+b}) = \Left$, so $o^-(x^{a+1+b}) > o^-(x^{a+b})$.
	
	If $a+b > 0$, then $a+1+b > 0$ which gives $o^-(x^{a+1+b}) = \Left$, so $o^-(x^{a+1+b}) = o^-(x^{a+b})$.
	
	If $a+b < 0$, then $a+1+b \le 0$, which gives $o^-(x^{a+1+b}) = \Right \cup \Next$, so $o^-(x^{a+b+1}) \ge o^-(x^{a+b})$.
	
	Therefore $x^{a} < x^{a+1}$.
\end{proof}

With these results in hand, we obtain the final proposition of this section.

\begin{proposition}
	Ignoring the outcome class tetrapartition, the \mis monoid $\monoid{M}_{\cl{1, \overline{1}}}$ is, as a partially ordered monoid, isomorphic to the totally ordered group of integers.
\end{proposition}

\section{The \Mis Monoid of $\cl{\sigma, \sigmab}$}\label{sec-sigma}

This section gives an example of a \mis monoid with the following properties:
	\begin{itemize}
		\item finite cardinality,
		\item it is the same as $\monoid{M}_{\cl{*}}$,
		\item its partial order contains two incomparable elements.
	\end{itemize}

\begin{definition}\label{def-sigma}
	The position $\sigma$\index{$\sigma$} is the position $\combgame{\{*\mid \cdot \}}$.  That is, it is the position in which Left has a move to $*$ while Right has no move.  
\end{definition}

The game tree of $\sigma$ is given in Figure \ref{fig-gt-sigma}.
		\begin{figure}[htb]
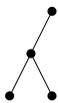

		\unitlength 8pt
		\begin{center}
		\begin{graph}(4,4)(-1,0)
		\graphnodesize{0.4}
		\fillednodestrue

		\roundnode{A}(1,4)
		\roundnode{B}(0,2)
		\roundnode{D}(-1,0)
		\roundnode{E}(1,0)

		\edge{A}{B}
		\edge{B}{D}
		\edge{B}{E}
		
		\end{graph}
		\caption{The game tree of $\sigma$.}
		\label{fig-gt-sigma}
		\end{center}	
		\end{figure}

As with $\cl{1, \overline{1}}$, we wish to calculate the \mis monoid $\cl{\sigma, \sigmab}$.  Our first step is to determine the outcome classes of arbitrary positions in the closure.  

\begin{proposition}\label{prop-n*+m-sigma+m-sigma_L}
	Suppose $n,m,\l \in \mathbb{Z}^{\ge 0}$.  Then
		\[ o^-(n  * + m  \sigma + \l  \sigmab) = \begin{cases}
		\Next &\text{if } n \equiv 0 \imod 2;\\
		\Prev &\text{if } n \equiv 1 \imod 2.
		\end{cases}\]
\end{proposition}

\begin{proof}
	We proceed by induction on the options of a position.  Thus, when we have a position $n * + m \sigma + \l \sigmab$ and we are assuming the induction hypothesis is true for its options, we are assuming that the induction hypothesis is true for the positions 
		\begin{align*}
			(n-1) * + m \sigma + \l \sigmab,\\
			(n+1) * + (m-1) \sigma + \l \sigmab, \\
			(n+1) * + m \sigma + (\l-1) \sigmab.
		\end{align*}
	Other proofs in this section will also make use of the same style of argument without explicitly justifying the lack of circularity as is done here.  This is akin to defining some sort of lexicographical ordering and proceeding by induction on that. 
	
	If $n=m=\l=0$, then the position is 0, which has outcome $\Next$, which agrees with the statement of the proposition.
	
	Consider position $n * + m \sigma + \l \sigmab$ and suppose that the outcomes for its options are as given in the statement of the proposition.
	
	Suppose $n=0$ and Left is moving first.  If $m=0$, then Left has no moves, and so Left wins.  If $m \not = 0$, then Left moves to $* + (m-1) \sigma + \l \sigmab$, which is a $\Prev$ position by induction.  Similarly, Right can win moving first.  Thus $o^-(m \sigma + \l \sigmab) = \Next$.
	
	Suppose $n \equiv 0 \imod 2$ and $n > 0$.  Then both Left and Right can move to $(n-1) * + m \sigma + \l \sigmab$, which is a $\Prev$ position by induction.
	
	Therefore $o^-(n * + m \sigma + \l \sigmab) = \Next$ if $n \equiv 0 \imod 2$.
	
	Suppose $n \equiv 1 \imod 2$.  Left moving first has two options:
		\begin{enumerate}
			\item $(n-1) * + m \sigma + \l \sigmab$, which is an $\Next$ position by induction; or
			\item $(n+1) * + (m-1) \sigma + \l \sigmab$, which is an $\Next$ position by induction.
		\end{enumerate}
	Therefore Left loses moving first.  Similarly, Right loses moving first.  Thus $o^-(n * + m \sigma + \l \sigmab) = \Prev$ for $n \equiv 1 \imod 2$.
\end{proof}

We notice that the outcome of a position in $\cl{\sigma, \sigmab}$ depends only on the parity of the number of $*$ positions.  Using this, we obtain the following indistinguishability relations on $\cl{\sigma, \sigmab}$.

\begin{proposition}\label{prop-sigma-*+*=0}
	The following indistinguishability relationships exist on $\cl{\sigma,\sigmab}$:
		\begin{enumerate}
			\item $* + * \equiv 0 \imod{{\cl{\sigma, \sigmab}}}$,
			\item $a  \sigma + b  \sigmab \equiv 0 \imod{{\cl{\sigma, \sigmab}} }$ for any $a,b \in \mathbb{Z}^{\ge 0}$,
			\item $* + a  \sigma + b  \sigmab \equiv * \imod{\cl{\sigma,\sigmab}}$ for all $a,b \in \mathbb{Z}^{\ge 0}$.
		\end{enumerate}
\end{proposition}

\begin{proof}
	We start with an arbitrary $n * + m \sigma + \l \sigmab$.  
	\begin{enumerate}
		\item\label{item-*+*=0-1} By Proposition \ref{prop-n*+m-sigma+m-sigma_L},
			\[ o^-((n+2) * + m \sigma + \l \sigmab) = o^-(n * + m \sigma + \l \sigmab).\]
		Therefore 
			\[ * + * \equiv 0 \imod{{\cl{\sigma, \sigmab}}}.\]
		
		\item Since we just showed
			\[ * + * \equiv 0 \imod{{\cl{\sigma, \sigmab}}},\]		
		we need only to consider $n=0$ or $n=1$ in our arbitrary position$n * + m \sigma + \l \sigmab$.
		
		We have
		\begin{align*}
			o^-((a  \sigma + b  \sigmab) + (n  * + m  \sigma + \l  \sigmab)) &= o^-(n  * + (m+a)  \sigma + (\l +b)  \sigmab) \\
			&= \begin{cases}
				\Next &\text{if } n =0;\\
				\Prev &\text{if } n = 1.
			\end{cases}
		\end{align*}
		But 
		\begin{align*}
			o^-(0 + (n  * + m  \sigma + \l  \sigmab)) &= \begin{cases}
				\Next &\text{if } n =0;\\
				\Prev &\text{if } n = 1.
			\end{cases}
		\end{align*}
		Therefore the two elements are indistinguishable.
		
		\item This follows from the previous case. \qedhere
	\end{enumerate}
\end{proof}

\begin{corollary}\label{cor-sigma-finite}
	Given a position $n * + m \sigma + \l \sigmab$ in $\cl{\sigma, \sigmab}$, it is indistinguishable from exactly one of either $0$ or $*$.
\end{corollary}

We now explicitly write the \mis monoid.
With the mappings:
	\begin{align*}
		0 &\mapsto 1;\\
		* &\mapsto a;\\
		\sigma &\mapsto 1;\\
		\sigmab &\mapsto 1;
	\end{align*}
we obtain the following monoid:
	\begin{align*}
	\monoid{M}_{\cl{\sigma,\sigmab}} &= \ideal{1,a \mid a^2=1} \\
	\Next &= \{1\} \\
	\Prev &= \{a\} \\
	\Left &= \emptyset \\
	\Right &= \emptyset
	\end{align*}
with the additive notation in $\cl{\sigma,\sigmab}$ becoming a multiplicative notation in $\monoid{M}_{\cl{\sigma,\sigmab}}$.

But, of course, this is the same of $\monoid{M}_{\cl{*}}$ (Example \ref{example-*}).  That is, we have a partizan position with \mis monoid the same as that of an impartial position.  We also recall from Example \ref{example-*} that, as a monoid ignoring the outcome tetrapartition, that $\monoid{M}_{\cl{*}}$ is equivalent to $(\mathbb{Z}_2, \oplus)$.

Since the two elements of $\monoid{M}_{\cl{\sigma, \sigmab}}$ have outcome classes $\Next$ and $\Prev$, they are incomparable.  Therefore, the partial order of $\monoid{M}_{\cl{\sigma, \sigmab}}$ contains two incomparable elements.

\section{The \mis monoid of $\cl{\rho}$}\label{section-rho}
 
 This section gives an example of a \mis monoid with the following properties:
 	\begin{itemize}
 		\item finite cardinality,
 		\item not congruent to that of any impartial game,
 		\item partial order is down-directed but not up-directed.
	\end{itemize}
 
\begin{definition}\label{def-rho}
	Let $\rho$ be the position $\combgame{\{*\mid0\}}$.  That is, $\rho$ is the position whose Left option is to $*$ and whose Right option is to 0. 
\end{definition}

The game tree of $\rho$ is given in Figure \ref{fig-gt-rho}.
		\begin{figure}[htb]
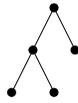

		\unitlength 8pt
		\begin{center}
		\begin{graph}(4,4)(-1,0)
		\graphnodesize{0.4}
		\fillednodestrue

		\roundnode{A}(1,4)
		\roundnode{B}(0,2)
		\roundnode{C}(2,2)
		\roundnode{D}(-1,0)
		\roundnode{E}(1,0)
		
		\edge{A}{C}
		\edge{A}{B}
		\edge{B}{D}
		\edge{B}{E}
		
		\end{graph}
		\caption{The game tree of $\rho$.}
		\label{fig-gt-rho}
		\end{center}	
		\end{figure}

As in our previous examples, we begin by determining the outcome class of an arbitrary position in $\cl{\rho}$, $n  * + m  \rho$.  As the following proposition shows, although one copy of $\rho$ gives a win for Left, once we are given enough copies of $\rho$ (at least four), this game becomes favourable to Right.

\begin{proposition}\label{prop-rho-oc}
	For a position $n * + m \rho$, its outcome is given in Table \ref{table-outcomes-cl-rho}.
\end{proposition}

\begin{table}[htb]
\begin{center}
\begin{tabular}{ccc}
& $n \equiv 0$&  $n \equiv 1$ \\
$m=0$
&\cellcolor[gray]{0.8}$\Next$
&\cellcolor[gray]{0.8}$\Prev$
\\
$m=1$
&$\Left$
&$\Next$
\\
$m=2$
&\cellcolor[gray]{0.8}$\Prev$
&\cellcolor[gray]{0.8}$\Next$
\\
$m=3$
&$\Right$
&$\Next$
\\
$m \ge 4$
&\cellcolor[gray]{0.8}$\Right$
&\cellcolor[gray]{0.8}$\Right$
\end{tabular}
\end{center}
\caption{Outcomes of positions $n * + m \rho$ where $n \equiv 0$ or $1 \imod 2$.}
\label{table-outcomes-cl-rho}
\end{table}	

\begin{proof}
	We wish to show that Table \ref{table-outcomes-cl-rho} is correct.  We will start by filling out the following table, with $n \le 3$ and $m \le 5$.  
		
		\begin{center}
		\begin{tabular}{ccccc}
		& $n = 0$&  $n = 1$ & $n=2$ & $n=3$ \\
		$m=0$
		&\cellcolor[gray]{0.8}
		&\cellcolor[gray]{0.8}
		&\cellcolor[gray]{0.8}
		&\cellcolor[gray]{0.8}
		\\
		$m=1$
		&
		&
		&
		\\
		$m=2$
		&\cellcolor[gray]{0.8}
		&\cellcolor[gray]{0.8}
		&\cellcolor[gray]{0.8}
		&\cellcolor[gray]{0.8}
		\\
		$m=3$
		&
		&
		&
		\\
		$m=4$
		&\cellcolor[gray]{0.8}
		&\cellcolor[gray]{0.8}
		&\cellcolor[gray]{0.8}
		&\cellcolor[gray]{0.8}
		\\
		$m=5$
		&
		&
		&
		\end{tabular}
		\end{center}
	
	When $m=0$, we either have the positions $0$, $*$, $2 *$, and $3 *$, with outcomes $\Next$, $\Prev$, $\Next$, and $\Prev$ respectively.  We add these into our table.

		\begin{center}
		\begin{tabular}{ccccc}
		& $n = 0$&  $n = 1$ & $n=2$ & $n=3$ \\
		$m=0$
		&\cellcolor[gray]{0.8}$\Next$
		&\cellcolor[gray]{0.8}$\Prev$
		&\cellcolor[gray]{0.8}$\Next$
		&\cellcolor[gray]{0.8}$\Prev$
		\\
		$m=1$
		&
		&
		&
		\\
		$m=2$
		&\cellcolor[gray]{0.8}
		&\cellcolor[gray]{0.8}
		&\cellcolor[gray]{0.8}
		&\cellcolor[gray]{0.8}
		\\
		$m=3$
		&
		&
		&
		\\
		$m=4$
		&\cellcolor[gray]{0.8}
		&\cellcolor[gray]{0.8}
		&\cellcolor[gray]{0.8}
		&\cellcolor[gray]{0.8}
		\\
		$m=5$
		&
		&
		&
		\end{tabular}
		\end{center}
	The alternating between $\Next$ and $\Prev$ in the first row will continue, i.e.\ for $m=0$ and $n \equiv 0 \imod 2$, we have outcome $\Next$ and for $m=0$ and $n \equiv 1 \imod 2$, we have outcome $\Prev$.  	
		
	To determine the outcomes of the other positions, we proceed as follows:  From a position in the table, Left's possible moves are
		\begin{enumerate}
			\item to move one position up and to the right (corresponding with taking a position $\rho$ and leaving $*$),
			\item to move one position to the left (corresponding with taking a position $*$ and leaving 0).
		\end{enumerate}
	For a position in the table, Right's possible moves are
		\begin{enumerate}
			\item to move one position up (corresponding with taking a position $\rho$ and leaving 0),
			\item to move one position to the left (corresponding with taking a position $*$ and leaving 0).
		\end{enumerate}
	With this in mind, we fill out the remaining outcomes in our table.
		\begin{center}
		\begin{tabular}{ccccc}
		& $n = 0$&  $n = 1$ & $n=2$ & $n=3$\\
		$m=0$
		&\cellcolor[gray]{0.8}$\Next$
		&\cellcolor[gray]{0.8}$\Prev$
		&\cellcolor[gray]{0.8}$\Next$
		&\cellcolor[gray]{0.8}$\Prev$
		\\
		$m=1$
		&$\Left$
		&$\Next$
		&$\Left$
		&$\Next$
		\\
		$m=2$
		&\cellcolor[gray]{0.8}$\Prev$
		&\cellcolor[gray]{0.8}$\Next$
		&\cellcolor[gray]{0.8}$\Prev$
		&\cellcolor[gray]{0.8}$\Next$
		\\
		$m=3$
		&$\Right$
		&$\Next$
		&$\Right$
		&$\Next$
		\\
		$m=4$
		&\cellcolor[gray]{0.8}$\Right$
		&\cellcolor[gray]{0.8}$\Right$
		&\cellcolor[gray]{0.8}$\Right$
		&\cellcolor[gray]{0.8}$\Right$
		\\
		$m=5$
		&$\Right$
		&$\Right$
		&$\Right$
		&$\Right$
		\end{tabular}
		\end{center}
	We notice two things, namely that the rows $m=4$ and $m=5$ have the same outcomes, and the columns $n=1$ and $n=3$ have the same outcomes.  Since our moves in the table only depend on the previous row and the previous column, this means that our table has ``become periodic", i.e.\ the outcomes for rows $a$ and $b$ are equal if $a, b \ge 4$, and the outcomes for columns $c$ and $c+2$ are equal.  This gives us the result of Table \ref{table-outcomes-cl-rho}.  	
\end{proof}

Using the preceding result, the following corollaries regarding the indistinguishability of certain positions are obtained.

\begin{corollary}\label{cor-cl-rho-indistinguishability}
	The following indistinguishability relations exist on $\cl{\rho}$:
		\begin{enumerate}
			\item $* + * \equiv 0 \imod{\cl{\rho}}$,
			\item $4  \rho \equiv u  \rho \imod{\cl{\rho}}$ for any $u \in \mathbb{Z}^{\ge 4}$,
			\item $4  \rho \equiv * + u  \rho \imod{\cl{\rho}}$ for any $u \in \mathbb{Z}^{\ge 4}$.
		\end{enumerate}
\end{corollary}

\begin{proof}
	For arbitrary $n  * + m  \rho$, Proposition \ref{prop-rho-oc} gives
		\begin{enumerate}
			\item $o^-((n+2)  * + m  \rho) = o^-(n  * + m  \rho)$,
			\item $o^-(n  * + (m+4)  \rho) = o^-(n  * + (m+u)  \rho) = \Right$ for any $u \in \mathbb{Z}^{\ge 4}$,
			\item $o^-(n  * + (m+4)  \rho) = o^-((n+1)  * + (m+u)  \rho) = \Right$ for any $u \in \mathbb{Z}^{\ge 4}$.\qedhere
		\end{enumerate}
\end{proof}

Note that $3  \rho \not \equiv m  \rho \imod{\cl{\rho}}$ for any $m \in \mathbb{Z}^{\ge 4}$, as Proposition \ref{prop-rho-oc} gives that $o^-(* + 3  \rho) = \Next$ while  $o^-(* + m  \rho) = \Right$.  That is, $*$ distinguishes $3  \rho$ and $m  \rho$ for any $m \in \mathbb{Z}^{\ge 4}$.

We claim that Table \ref{table-A-positions} gives all the positions in $\cl{\rho}$ up to indistinguishability.   By Corollary \ref{cor-cl-rho-indistinguishability}, given any other position, it is indistinguishable from one of the positions given in Table \ref{table-A-positions}.  It remains to show that the positions in the table are pairwise distinguishable.  

\begin{proposition}\label{prop-cl-rho-pairwise}
	All positions in Table \ref{table-A-positions} are pairwise distinguishable.
\end{proposition}	

\begin{proof}
	Table \ref{table-A-dist-elements} gives the distinguishing elements.  If two positions have different outcome classes, then they are distinguishable by 0.  Thus, only positions with the same outcome classes are in the table.
	
	Therefore, up to indistinguishability, Table \ref{table-A-positions} completely details the elements of $\cl{\rho}$.  
\end{proof} 

\begin{table}[p]
\begin{center}
\begin{tabular}{rc}
\textbf{Positions}&\textbf{Outcome}\\
\rowcolor[gray]{.8} 0&$\Next$\\
$*$ & $\Prev$\\
\rowcolor[gray]{.8} $\rho$&$\Left$\\
$2  \rho$&$\Prev$\\
\rowcolor[gray]{.8} $3  \rho$&$\Right$\\
$4  \rho$&$\Right$\\
\rowcolor[gray]{.8} $* + \rho$&$\Next$\\
$* + 2  \rho$&$\Next$\\
\rowcolor[gray]{.8} $* + 3  \rho$&$\Next$\\
\end{tabular}
\end{center}
\caption{Positions of $\cl{\rho}$ up to Indistinguishability.}
\label{table-A-positions}
\end{table}

	\begin{table}[p]
	\begin{center}
	\begin{tabular}{rrc}
	\textbf{Position 1}&\textbf{Position2}&\textbf{Distinguishing Element}\\
	\rowcolor[gray]{.8} 0&$*+\rho$&$*$\\
	0&$*+2 \rho$&$\rho$\\
	\rowcolor[gray]{.8} 0&$*+3 \rho$&$\rho$\\
	$*+\rho$&$*+2\rho$&$*$\\
	\rowcolor[gray]{.8} $* +\rho$&$*+3\rho$&$*$\\
	$* + 2  \rho$&$* + 3  \rho$&$*$\\
	\rowcolor[gray]{.8} $*$&$2  \rho$&$2  \rho$\\
	$3  \rho$&$4  \rho$&$*$\\
	\end{tabular}
	\end{center}
	\caption{Positions of $\cl{\rho}$ and the elements which distinguish them.}
	\label{table-A-dist-elements}
	\end{table}

We now determine $\monoid{M}_{\cl{\rho}}$.  With the mappings:
	\begin{align*}
		0 &\mapsto 1;\\
		* &\mapsto a;\\
		\rho &\mapsto p,
	\end{align*}
we obtain the following monoid:
	\begin{align*}
	\monoid{M}_{\cl{\rho}} &= \ideal{1,a,p \mid a^2 = 1, p^4 = p^5 = ap^4} \\
	\Next &= \{1, ap, ap^2, ap^3 \} \\
	\Prev &= \{a, p^2\} \\
	\Left &= \{p\} \\
	\Right &= \{p^3, p^4\}
	\end{align*}
with the additive notation in $\cl{\rho}$ becoming a multiplicative notation in $\monoid{M}_{\cl{\rho}}$.  

It is worthwhile justifying how we know that the relations we give in the monoid are the only relations which exist.  The relations which are there each correspond to one of the indistinguishability relations given in Corollary \ref{cor-cl-rho-indistinguishability}:
	\begin{enumerate}
		\item $a^2 = 1$ corresponds to $* + * \equiv 0 \imod{\cl{\rho}}$,
		\item $p^4 = p^5$ corresponds to $4  \rho \equiv u  \rho \imod{\cl{\rho}}$ for any $u \in \mathbb{Z}^{\ge 4}$,
		\item $p^4 = ap^4$ corresponds to  $4  \rho \equiv * + u  \rho \imod{\cl{\rho}}$ for any $u \in \mathbb{Z}^{\ge 4}$.
	\end{enumerate}
Using these relations to reduce elements in the monoid, we obtain the following elements in the monoid:
\begin{center}
\begin{tabular}{c}
\textbf{Positions}\\
\rowcolor[gray]{.8} 1\\
$a$\\
\rowcolor[gray]{.8} $p$\\
$p^2$\\
\rowcolor[gray]{.8} $p^3$\\
$p^4$\\
\rowcolor[gray]{.8} $ap$\\
$ap^2$\\
\rowcolor[gray]{.8} $ap^3$\\
\end{tabular}
\end{center}
But Proposition \ref{prop-cl-rho-pairwise} tells us that all these positions are pairwise distinguishable, and so we have completely determined the monoid elements, and the relations in the monoid are the only relations which exist.

In impartial games, every finite \mis monoid has either cardinality one or is of even cardinality \cite{MISQUOTIENT}.  Contrast this with the cardinality of $\monoid{M}_{\cl{\rho}}$, which is nine.  This is our first partizan \mis monoid result which differs from that of impartial play and is important enough to place into a theorem for safe-keeping.

\begin{theorem}
	In impartial games, every \mis monoid has either cardinality one or is of even cardinality.  This is \textbf{not} true for partizan games.
\end{theorem}

We will now determine the partial order of the \mis monoid $\monoid{M}_{\cl{\rho}}$.

Recall that, under monoid multiplicative notation,
	\[ x \ge y \text{ if } o^-(xz) \ge o^-(yz) \text{ for all elements } z\]
and that, in terms of outcomes, the outcome lattice is given in Figure \ref{fig-outcome-poset}.

Notice that if $o^-(x) \not \ge o^-(y)$, then $x \not \ge y$ since $o^-(x1) \not \ge o^-(y1)$.  

\begin{proposition}\label{prop-A-poset}
	Figure \ref{fig-A-poset} gives the partially ordered set of positions $\monoid{M}_{\cl{\rho}}$ up to indistinguishability.  Thus, the partially ordered set is down-directed but not up-directed.  
\end{proposition}

\begin{figure}[htb]
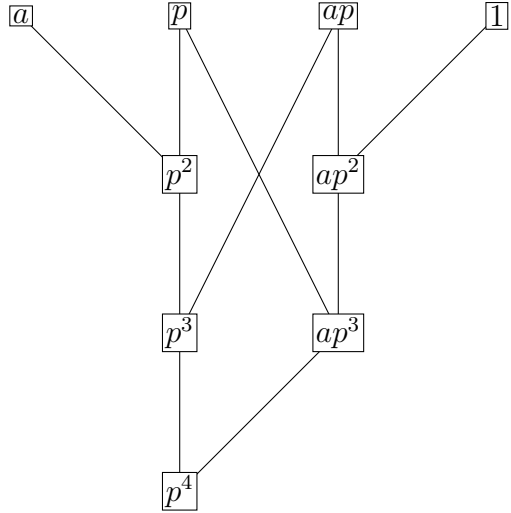

\begin{center}
		\unitlength 20pt
		\begin{center}
		\begin{graph}(9,9)(0,0)
		\graphnodesize{0.4}
		\fillednodestrue

		\textnode{p4}(3,0){$p^4$}
		\textnode{p3}(3,3){$p^3$}
		\textnode{p2}(3,6){$p^2$}
		\textnode{p}(3,9){$p$}
		
		\textnode{a}(0,9){$a$}
		
		\textnode{ap}(6,9){$ap$}
		\textnode{ap2}(6,6){$ap^2$}
		\textnode{ap3}(6,3){$ap^3$}
		\textnode{1}(9,9){1}
		
		\edge{a}{p2}
		\edge{p}{p2}
		\edge{p}{ap3}
		\edge{ap}{p3}
		\edge{ap}{ap2}
		\edge{1}{ap2}
		\edge{p2}{p3}
		\edge{ap2}{ap3}
		\edge{p3}{p4}
		\edge{ap3}{p4}	
		\end{graph}
		\end{center}

\end{center}
\caption{$\cl{\rho}$'s Partially Ordered Set.}
\label{fig-A-poset}
\end{figure}

\begin{proof}
	We will show that all relations in Figure \ref{fig-A-poset} exist and then show that no other relations exist.
	
	\begin{enumerate}
		\item $p^3 \ge p^4$ and $ap^3 \ge p^4$:  Since $p^4x = p^4$ for any $x \in \monoid{M}_{\cl{\rho}}$, we have $o^-(p^4x) = \Right$, so $o^-(p^3x) \ge o^-(p^4x)$ and $o^-(ap^3x) \ge o^-(p^4x)$, giving the desired inequalities.
		
		\item $p^2 \ge p^3$ and $ap \ge p^3$:  We know that $o^-(p^3x) = \Right$ unless $x = a$, in which case $o^-(pa) = \Next$.  Thus, for every $x \not = a$, $o^-(p^2x) \ge o^-(p^3x)$ and $o^-(apx) \ge o^-(p^3x)$.  Suppose $x=a$.  Then $o^-(p^2a) = \Next$ and $o^-(apa) = \Left$, so $o^-(p^2a) \ge o^-(p^3a)$ and $o^-(apa) \ge o^-(p^3a)$.  Thus the desired inequalities are obtained. 

		\item $p \ge ap^3$, and $ap^2 \ge ap^3$:  We know that $o^-(ap^3x) = \Right$ unless $x=1$, in which case $o^-(ap^3) = \Next$.  Thus, for every $x \not = 1$, $o^-(px) \ge o^-(ap^3x)$, and $o^-(ap^2x) \ge o^-(ap^3x)$.  Now examine $x=1$.  Then $o^-(p1) = \Left$, and $o^-(ap^2) = \Next$, so the outcome inequalities also hold for $x=1$.  Thus the desired inequalities are obtained.

		\item $a \ge p^2$ and $p \ge p^2$:  We know that $o^-(p^2x) = \Right$ unless $x = 1$, $a$, or $ap$.  So, for $x \not = 1, a,$ or $ap$, $o^-(ax) \ge o^-(p^2x)$ and $o^-(px) \ge o^-(p^2x)$.  Remains to check what happens when $x = 1, a$, or $ap$, the results of which are given in Table \ref{table-A-a>p2}. 

			\begin{table}[htb]
			\begin{center}
			\begin{tabular}{rrcrcrc}
			$\boldsymbol{x}$&$\boldsymbol{ax}$&$\boldsymbol{o^-(ax)}$&$\boldsymbol{px}$&$\boldsymbol{o^-(px)}$&$\boldsymbol{p^2x}$&$\boldsymbol{o^-(p^2x)}$\\
			\rowcolor[gray]{.8}$1$&$a$&$\Prev$&$p$&$\Left$&$p^2$&$\Prev$\\
			$a$&$1$&$\Next$&$ap$&$\Next$&$ap^2$&$\Next$\\
			\rowcolor[gray]{.8}$ap$&$p$&$\Left$&$ap^2$&$\Next$&$ap^3$&$\Next$\\
			\end{tabular}
			\caption{Showing $o^-(ax) \ge o^-(p^2x)$ and $o^-(px) \ge o^-(p^2x)$ for $x \in \{1$, $a$, $ap\}$ in $\monoid{M}_{\cl{\rho}}$.}
			\label{table-A-a>p2}
			\end{center}
			\end{table}
			
		Thus, when $x = 1,a$, or $ap$, also have $o^-(ax) \ge o^-(p^2x)$ and $o^-(px) \ge o^-(p^2x)$.  Thus the desired inequalities are obtained.
		
		\item $ap \ge ap^2$ and $1 \ge ap^2$:  We know $o^-(ap^2x) = \Right$ unless $x = 1, a$, or $p$.  So, for $x \not = 1,a$, or $p$, $o^-(apx) \ge o^-(ap^2x)$ and $o^-(1x) \ge o^-(ap^2x)$.  Remains to check what happens when $x = 1,a$ or $p$, the results of which are given in Table \ref{table-A-ap>ap2}.
			
			\begin{table}[htb]
			\begin{center}
			\begin{tabular}{rrcrcrc}
			$\boldsymbol{x}$&$\boldsymbol{apx}$&$\boldsymbol{o^-(apx)}$&$\boldsymbol{1x}$&$\boldsymbol{o^-(1x)}$&$\boldsymbol{ap^2x}$&$\boldsymbol{o^-(ap^2x)}$\\
			\rowcolor[gray]{.8}$1$&$ap$&$\Next$&$1$&$\Next$&$ap^2$&$\Next$\\
			$a$&$p$&$\Left$&$a$&$\Prev$&$p^2$&$\Prev$\\
			\rowcolor[gray]{.8}$p$&$ap^2$&$\Next$&$p$&$\Left$&$ap^3$&$\Next$\\
			\end{tabular}
			\caption{Showing $o^-(apx) \ge o^-(ap^2x)$ and $o^-(1x) \ge o^-(ap^2x)$ for $x\in\{1,a,p\}$ in $\monoid{M}_{\cl{\rho}}$.}
			\label{table-A-ap>ap2}
			\end{center}
			\end{table}		
		
		Thus, when $x=1,a,$ or $p$, also have $o^-(apx) \ge o^-(ap^2x)$ and $o^-(1x) \ge o^-(ap^2x)$.  Therefore the desired inequalities are obtained.
	\end{enumerate}
		
	By Proposition \ref{prop-cl-rho-pairwise}, all the elements in Figure \ref{fig-A-poset} are distinguishable.  Thus all the inequalities determined above are strict.  
	
	It remains to show that there are no other inequalities, i.e.\ if there is no downward path between two elements, then they are incomparable.  For there to be no downward path between elements $x$ and $y$, either:
		\begin{enumerate}
			\item there exists $z$ such that $o^-(xz) = \Next$ while $o^-(yz) = \Prev$; or
			
			\item there exist $z_1$ and $z_2$ such that $o^-(xz_1) > o^-(yz_1)$ while $o^-(xz_2) < o^-(yz_2)$. 
		\end{enumerate}
	There are sixteen pairs of elements in Figure \ref{fig-A-poset}	which do not have a downward path between them.  Fifteen pairs $(p_1,p_2)$ fall into the first case, with the element rendering them incomparable ($e$) and outcomes given in Table \ref{table-A-incomparable}.
			
	\begin{table}[htb]
	\begin{center}
	\begin{tabular}{rrrrcrc}
	$\boldsymbol{p_1}$&$\boldsymbol{p_2}$&$\boldsymbol{e}$&$\boldsymbol{p_1e}$&$\boldsymbol{o^-(p_1e)}$&$\boldsymbol{p_2e}$&$\boldsymbol{o^-(p_2e)}$\\

	\rowcolor[gray]{.8}1&
	$a$&
	$1$
	&$1$&
	$\Next$&
	$a$
	&$\Prev$\\
	
	1&
	$p$&
	$a$
	&$a$&
	$\Prev$&
	$ap$
	&$\Next$\\
	
	\rowcolor[gray]{.8}1&
	$p^2$&
	$1$
	&$1$&
	$\Next$&
	$p^2$
	&$\Prev$\\

	1&
	$p^3$&
	$a$
	&$a$&
	$\Prev$&
	$ap^3$
	&$\Next$\\

	\rowcolor[gray]{.8}1&
	$ap$&
	$p^2$
	&$p^2$&
	$\Prev$&
	$ap^3$
	&$\Next$\\

	$a$&
	$p$&
	$p$
	&$ap$&
	$\Next$&
	$p^2$
	&$\Prev$\\

	\rowcolor[gray]{.8}$a$&
	$ap$&
	$1$
	&$a$&
	$\Prev$&
	$ap$
	&$\Next$\\

	$a$&
	$ap^2$&
	$1$
	&$a$&
	$\Prev$&
	$ap^2$
	&$\Next$\\

	\rowcolor[gray]{.8}$a$&
	$ap^3$&
	$1$
	&$a$&
	$\Prev$&
	$ap^3$
	&$\Next$\\

	$p$&
	$ap$&
	$p$
	&$p^2$&
	$\Prev$&
	$ap^2$
	&$\Next$\\

	\rowcolor[gray]{.8}$p$&
	$ap^2$&
	$p$
	&$p^2$&
	$\Prev$&
	$ap^3$
	&$\Next$\\

	$p^2$&
	$ap$&
	$1$
	&$p^2$&
	$\Prev$&
	$ap$
	&$\Next$\\

	\rowcolor[gray]{.8}$p^2$&
	$ap^2$&
	$1$
	&$p^2$&
	$\Prev$&
	$ap^2$
	&$\Next$\\

	$p^2$&
	$ap^3$&
	$1$
	&$p^2$&
	$\Prev$&
	$ap^3$
	&$\Next$\\

	\rowcolor[gray]{.8}$p^3$&
	$ap^2$&
	$a$
	&$ap^3$&
	$\Next$&
	$p^2$
	&$\Prev$\\
	\end{tabular}
	\caption{Determining which elements are incomparable in $\monoid{M}_{\cl{\rho}}$.}
	\label{table-A-incomparable}
	\end{center}
	\end{table}	
	
	The only pair yet to be examined is $p^3$ and $ap^3$.  We have $o^-(p^3) = \Right$ while $o^-(ap^3) = \Next$.  Therefore $o^-(p^31) < o^-(ap^31)$.  But $o^-(p^3a) = \Next$ while $o^-(ap^3a) = \Right$ so $o^-(p^3a) > o^-(ap^3a)$.  Thus these two elements are incomparable.  

	Thus Figure \ref{fig-A-poset} completely determines the poset of the elements of $\monoid{M}_{\cl{\rho}}$. 
\end{proof}

As we can see, in the poset of $\monoid{M}_{\cl{\rho}}$, there is no top element, no meets, and no joins.  However, there is a bottom element, $p^4$.

\section{The \Mis Monoid of $\cl{\rho,\rhob}$}\label{sec-rho-rhob}

This section gives an example of a \mis monoid with the following properties:
	\begin{itemize}
		\item infinite cardinality,
		\item partial order is both down-directed and up-directed, but it is not a lattice.
	\end{itemize}

This example, $\cl{\rho,\rhob}$, has an infinite \mis monoid like that of the \mis monoid of $\cl{1, \bar{1}}$.  However, unlike the \mis monoid of $\cl{1, \bar{1}}$, the \mis monoid of $\cl{\rho,\rhob}$ is not a lattice.  The definition of $\rho$ can be found in Definition \ref{def-rho}.

As in the other examples, we begin by determining the outcome class of an arbitrary position.  

\begin{theorem}\label{theorem-B-oc}
	Given an arbitrary element of $\cl{\rho,\rhob}$,
		\[ n  * + m  \rho + \l  \rhob,\]
	the outcome class of this element is given in Table \ref{table-B-oc}.
\end{theorem}

\begin{table}[htb]
\begin{center}
\begin{tabular}{rcccccc}
&$m \le \l+4$&\multicolumn{4}{c}{$\l-4 < m < \l+4$}&$m \ge \l+4$ \\
&&$m \equiv \l$&$m \equiv \l+1$&$m \equiv \l+2$&$m \equiv \l+3$&\\
$n \equiv 0 \imod 2$&\cellcolor[gray]{0.8}$\Left$&\cellcolor[gray]{0.8}$\Next$&\cellcolor[gray]{0.8}$\Left$&\cellcolor[gray]{0.8}$\Prev$&\cellcolor[gray]{0.8}$\Right$&\cellcolor[gray]{0.8}$\Right$\\
$n \equiv 1 \imod 2$&$\Left$&$\Prev$&$\Next$&$\Next$&$\Next$&$\Right$\\
\end{tabular}
\end{center}
\caption{The outcome classes of positions in $n * + m \rho + \l \rhob$ where $m \equiv \l + i \imod 4$.}
\label{table-B-oc}
\end{table}

Tables \ref{table-B-oc-n=0} and \ref{table-B-oc-n=1} place the outcome classes into two tables depending on whether $n \equiv 0 \imod 2$ or $n \equiv 1 \imod 2$.  The reader may find it more illuminating to refer to these tables as well as to Table \ref{table-B-oc} when following the proof of Theorem \ref{theorem-B-oc}.

	\begin{table}[p]
	\begin{center}

	\unitlength 12pt
	\begin{graph}(30,20)(0,0)
	\graphnodesize{0}

	\freetext(2,19){$n \equiv 0 \imod 2$}
		
	\roundnode{a1}(4,18)
	\roundnode{a2}(4,0)
	\roundnode{c1}(2,16)
	\roundnode{c2}(2,0)
	
	\roundnode{b1}(2,16)
	\roundnode{b2}(20,16)
	\roundnode{d1}(4,18)
	\roundnode{d2}(20,18)

	\edge{a1}{a2}
	\edge{b1}{b2}
	\edge{c1}{c2}
	\edge{d1}{d2}

	\roundnode{e1}(4,10)
	\roundnode{e2}(13.5,1)
	
	\edge{e1}{e2}
	
	\roundnode{f1}(9.5,16)
	\roundnode{f2}(19,7.5)
	
	\edge{f1}{f2}
	
	\freetext(7,3.5){\begin{huge}$\Left$\end{huge}}
	\freetext(17,13){\begin{huge}$\Right$\end{huge}}

	\freetext(1,8.5){$\l$}
	\freetext(12,19){$m$}
	
	\freetext(3,15){0}
	\freetext(3,13.5){1}
	\freetext(3,12){2}
	\freetext(3,10.5){3}
	\freetext(3,9){4}
	\freetext(3,7.5){5}
	\freetext(3,6){6}
	\freetext(3,4.5){7}
	\freetext(3,3){8}
	\freetext(3,1.5){9}
	
	\freetext(5,17){0}
	\freetext(6.5,17){1}
	\freetext(8,17){2}
	\freetext(9.5,17){3}
	\freetext(11,17){4}
	\freetext(12.5,17){5}
	\freetext(14,17){6}
	\freetext(15.5,17){7}
	\freetext(17,17){8}
	\freetext(18.5,17){9}
		
	\freetext(5,15){$\Next$}
	\freetext(6.5,15){$\Left$}
	\freetext(8,15){$\Prev$}
	\freetext(9.5,15){$\Right$}

	\freetext(5,13.5){$\Right$}
	\freetext(6.5,13.5){$\Next$}
	\freetext(8,13.5){$\Left$}
	\freetext(9.5,13.5){$\Prev$}
	\freetext(11,13.5){$\Right$}

	\freetext(5,12){$\Prev$}
	\freetext(6.5,12){$\Right$}
	\freetext(8,12){$\Next$}
	\freetext(9.5,12){$\Left$}
	\freetext(11,12){$\Prev$}
	\freetext(12.5,12){$\Right$}

	\freetext(5,10.5){$\Left$}
	\freetext(6.5,10.5){$\Prev$}
	\freetext(8,10.5){$\Right$}
	\freetext(9.5,10.5){$\Next$}
	\freetext(11,10.5){$\Left$}
	\freetext(12.5,10.5){$\Prev$}
	\freetext(14,10.5){$\Right$}

	\freetext(6.5,9){$\Left$}
	\freetext(8,9){$\Prev$}
	\freetext(9.5,9){$\Right$}
	\freetext(11,9){$\Next$}
	\freetext(12.5,9){$\Left$}
	\freetext(14,9){$\Prev$}
	\freetext(15.5,9){$\Right$}

	\freetext(8,7.5){$\Left$}
	\freetext(9.5,7.5){$\Prev$}
	\freetext(11,7.5){$\Right$}
	\freetext(12.5,7.5){$\Next$}
	\freetext(14,7.5){$\Left$}
	\freetext(15.5,7.5){$\Prev$}
	\freetext(17,7.5){$\Right$}

	\freetext(9.5,6){$\Left$}
	\freetext(11,6){$\Prev$}
	\freetext(12.5,6){$\Right$}
	\freetext(14,6){$\Next$}
	\freetext(15.5,6){$\Left$}
	\freetext(17,6){$\Prev$}
	\freetext(18.5,6){$\Right$}

	\freetext(11,4.5){$\Left$}
	\freetext(12.5,4.5){$\Prev$}
	\freetext(14,4.5){$\Right$}
	\freetext(15.5,4.5){$\Next$}
	\freetext(17,4.5){$\Left$}
	\freetext(18.5,4.5){$\Prev$}

	\freetext(12.5,3){$\Left$}
	\freetext(14,3){$\Prev$}
	\freetext(15.5,3){$\Right$}
	\freetext(17,3){$\Next$}
	\freetext(18.5,3){$\Left$}

	\freetext(14,1.5){$\Left$}
	\freetext(15.5,1.5){$\Prev$}
	\freetext(17,1.5){$\Right$}
	\freetext(18.5,1.5){$\Next$}
	
	\end{graph}

	\end{center}
	\caption{The outcome classes of positions in $\cl{\rho,\rhob}$ where $n \equiv 0 \imod 2$.}
	\label{table-B-oc-n=0}
	\end{table}

	\begin{table}[p]
	\begin{center}

	\unitlength 12pt
	\begin{graph}(30,20)(0,0)
	\graphnodesize{0}

	\freetext(2,19){$n \equiv 1 \imod 2$}
		
	\roundnode{a1}(4,18)
	\roundnode{a2}(4,0)
	\roundnode{c1}(2,16)
	\roundnode{c2}(2,0)
	
	\roundnode{b1}(2,16)
	\roundnode{b2}(20,16)
	\roundnode{d1}(4,18)
	\roundnode{d2}(20,18)

	\edge{a1}{a2}
	\edge{b1}{b2}
	\edge{c1}{c2}
	\edge{d1}{d2}

	\roundnode{e1}(4,10)
	\roundnode{e2}(13.5,1)
	
	\edge{e1}{e2}
	
	\roundnode{f1}(9.5,16)
	\roundnode{f2}(19,7.5)
	
	\edge{f1}{f2}
	
	\freetext(7,3.5){\begin{huge}$\Left$\end{huge}}
	\freetext(17,13){\begin{huge}$\Right$\end{huge}}

	\freetext(1,8.5){$\l$}
	\freetext(12,19){$m$}
	
	\freetext(3,15){0}
	\freetext(3,13.5){1}
	\freetext(3,12){2}
	\freetext(3,10.5){3}
	\freetext(3,9){4}
	\freetext(3,7.5){5}
	\freetext(3,6){6}
	\freetext(3,4.5){7}
	\freetext(3,3){8}
	\freetext(3,1.5){9}
	
	\freetext(5,17){0}
	\freetext(6.5,17){1}
	\freetext(8,17){2}
	\freetext(9.5,17){3}
	\freetext(11,17){4}
	\freetext(12.5,17){5}
	\freetext(14,17){6}
	\freetext(15.5,17){7}
	\freetext(17,17){8}
	\freetext(18.5,17){9}
		
	\freetext(5,15){$\Prev$}
	\freetext(6.5,15){$\Next$}
	\freetext(8,15){$\Next$}
	\freetext(9.5,15){$\Next$}

	\freetext(5,13.5){$\Next$}
	\freetext(6.5,13.5){$\Prev$}
	\freetext(8,13.5){$\Next$}
	\freetext(9.5,13.5){$\Next$}
	\freetext(11,13.5){$\Next$}

	\freetext(5,12){$\Next$}
	\freetext(6.5,12){$\Next$}
	\freetext(8,12){$\Prev$}
	\freetext(9.5,12){$\Next$}
	\freetext(11,12){$\Next$}
	\freetext(12.5,12){$\Next$}

	\freetext(5,10.5){$\Next$}
	\freetext(6.5,10.5){$\Next$}
	\freetext(8,10.5){$\Next$}
	\freetext(9.5,10.5){$\Prev$}
	\freetext(11,10.5){$\Next$}
	\freetext(12.5,10.5){$\Next$}
	\freetext(14,10.5){$\Next$}

	\freetext(6.5,9){$\Next$}
	\freetext(8,9){$\Next$}
	\freetext(9.5,9){$\Next$}
	\freetext(11,9){$\Prev$}
	\freetext(12.5,9){$\Next$}
	\freetext(14,9){$\Next$}
	\freetext(15.5,9){$\Next$}

	\freetext(8,7.5){$\Next$}
	\freetext(9.5,7.5){$\Next$}
	\freetext(11,7.5){$\Next$}
	\freetext(12.5,7.5){$\Prev$}
	\freetext(14,7.5){$\Next$}
	\freetext(15.5,7.5){$\Next$}
	\freetext(17,7.5){$\Next$}

	\freetext(9.5,6){$\Next$}
	\freetext(11,6){$\Next$}
	\freetext(12.5,6){$\Next$}
	\freetext(14,6){$\Prev$}
	\freetext(15.5,6){$\Next$}
	\freetext(17,6){$\Next$}
	\freetext(18.5,6){$\Next$}

	\freetext(11,4.5){$\Next$}
	\freetext(12.5,4.5){$\Next$}
	\freetext(14,4.5){$\Next$}
	\freetext(15.5,4.5){$\Prev$}
	\freetext(17,4.5){$\Next$}
	\freetext(18.5,4.5){$\Next$}

	\freetext(12.5,3){$\Next$}
	\freetext(14,3){$\Next$}
	\freetext(15.5,3){$\Next$}
	\freetext(17,3){$\Prev$}
	\freetext(18.5,3){$\Next$}

	\freetext(14,1.5){$\Next$}
	\freetext(15.5,1.5){$\Next$}
	\freetext(17,1.5){$\Next$}
	\freetext(18.5,1.5){$\Prev$}
	
	\end{graph}

	\end{center}
	\caption{The outcome classes of positions in $\cl{\rho,\rhob}$ where $n \equiv 1 \imod 2$.}
	\label{table-B-oc-n=1}
	\end{table}

\begin{notation}
	Let $(a,b,c)$ denote the position $a  * + b  \rho + c  \rhob$.  
\end{notation}

\begin{proof}
	We proceed by induction on the options.  
	
	Consider $(0,0,0)$.  Then $n = m = \l = 0$, so $\l-4 < m < \l+4$, $m \equiv \l \imod 4$, and $n \equiv 0 \imod 2$, so the theorem claims that the outcome class is $\Next$, which indeed it is.
	
	Consider position $(n,m,\l)$ and suppose that the outcome classes of all of its options are as given in Table \ref{table-B-oc}.  
	
	From a position $(n,m,\l)$, both Left and Right have three possible moves to positions which we list below.
		\begin{align*}
			&\text{Left's} & &\text{Right's}\\
			&(n-1,m,\l);   & &(n-1,m,\l)  \\
			&(n+1,m-1,\l); & &(n, m-1,\l)\\
			&(n,m,\l-1).   & &(n+1,m,\l-1)
		\end{align*}
	In the proof, we will refer back to these moves.

	Clearly $(n,m,\l)$ must fall into one of the twelve cases given in Table \ref{table-B-oc}.  We will consider each of these twelve cases separately.  
	
	\begin{enumerate}
		\item $m \le \l-4$, $n \equiv 0 \imod 2$:   Consider Left's three moves.		
			\begin{enumerate}
				\item $(n-1,m,\l)$: $(n-1,m,\l)$ is an $\Left$ position by induction since $m \le \l-4$.
				
				\item $(n+1,m-1,\l)$: Since $m \le \l-4$, we have $m-1 \le \l-4$, so, by induction, $(n+1,m-1,\l)$ is also an $\Left$ position.
				
				\item $(n,m,\l-1)$: If $m < \l-4$, then $m \le \l-5$, and so, by induction, $(n,m,\l-1)$ is an $\Left$ position.  Otherwise $m = \l-4$.  Then $(n,m,\l-1)$ becomes $(n, \l-4, \l-1)$.  Then 
					\[\l-4 \equiv 1+(\l-1) \imod 4,\]
				so, by induction, 
					\[o^-((n,\l-4,\l-1)) = \Left \text{ if }n \equiv 0 \imod 2\]
				or 
					\[o^-((n,\l-4,\l-1)) = \Next \text{ if } n \equiv 1 \imod 2.\]
				But if $n \equiv 1 \imod 2$, then $n \ge 1$, so Left can move to $(n-1,m,\l)$ rather than this position, and so Left can win moving first.  
			\end{enumerate}
			
		Therefore Left can win moving first.
		
		Suppose Right moves first.  Consider Right's three moves.
			\begin{enumerate}
				\item $(n-1,m,\l)$: This is an $\Left$ position by induction since $m \le \l-4$.
				
				\item $(n, m-1,\l)$:  Since $m \le \l-4$, we have $m-1 \le \l-4$, so by induction, this is also an $\Left$ position.
				
				\item $(n+1,m,\l-1)$:  By Left's case (c) above, $o^-((n+1,m,\l-1)) = \Left \cup \Next$.
			\end{enumerate}
		
		This shows that Right cannot win moving first.
		
		Therefore $o^-((n,m,\l)) = \Left$ if $m \le \l-4$.
			
		\item $m \le \l-4$, $n \equiv 1 \imod 2$:  In the preceding argument for $m \le \l-4$, $n \equiv 0 \imod 2$, nowhere did we use the fact that $n \equiv 0 \imod 2$.  Therefore the same argument works for $m \le \l-4$, $n \equiv 1 \imod 2$.
		
		\item $\l-4 < m < \l+4$, $m \equiv \l \imod 4$, $n \equiv 0 \imod 2$:  The conditions on $m$ and $\l$ mean $m$ must equal $\l$.  If $m=\l=0$, then the result follows.  Thus, take $m \ge 1$.  
		
		Left moving first moves to $(n,m,\l-1) = (n,m,m-1)$.  Since
			\begin{align*}
				&m-5 < m < m+3,\\
				&m \equiv 1 + (m-1) \imod 4, \text{ and}\\
				&n \equiv 0 \imod 2,
			\end{align*}
		by induction, we have $o^-((n,m,\l-1)) = \Left$.  
		
		Right moving first will move to $(n, m-1,\l)=(n,m-1,m)$.  Since 
			\begin{align*}
				&m-4<m-1<m+4,\\
				&m-1 \equiv 3 + m \imod 4, \text{ and}\\
				&n \equiv 0 \imod 2,
			\end{align*}
		by induction, we have $o^-((n, m-1,\l)) = \Right$.
		
		Therefore $o^-((n,m,\l)) = \Next$ if $\l-4 < m < \l+4$, $m \equiv \l \imod 4$, and $n \equiv 0 \imod 2$.
		
		\item $\l-4 < m < \l+4$, $m \equiv \l \imod 4$, $n \equiv 1 \imod 2$:  The conditions on $m$ and $\l$ mean $m$ must equal $\l$.  If $m=\l=0$, then the result follows.  Thus, take $m \ge 1$.
		
		We consider Left's three moves:
			\begin{enumerate}
				\item $(n-1,m,\l)$: By induction, this is an $\Next$ position.
				
				\item $(n+1,m-1,\l)$:  Since 
					\begin{align*}
						&m-4 < m-1 < m+4,\\
						&m-1 \equiv m+3 \imod 4, \text{ and}\\
						&n+1 \equiv 0 \imod 2,
					\end{align*}
				induction gives $o^-((n+1,m-1,\l)) = \Right$.  
				
				\item $(n,m,\l-1)$:  Since 
					\begin{align*}
						&m-5<m<m+3,\\
						&m \equiv 1 + (m-1) \imod 4, \text{ and}\\
						&n \equiv 1 \imod 2, 
					\end{align*}
				induction gives $o^-((n,m,\l-1)) = \Next$.  
			\end{enumerate}
		Therefore Left moving first does not have a winning move.
		
		We consider Right's three moves.
			\begin{enumerate}
				\item $(n-1,m,\l)$:  By induction, this is an $\Next$ position.
				
				\item $(n, m-1,\l)$:  Since 
					\begin{align*}
						&m-4 < m-1 < m+4,\\
						&m-1 \equiv 3+m \imod 4, \text{ and}\\
						&n \equiv 1 \imod 2,
					\end{align*}
				induction gives $o^-((n, m-1,\l)) = \Next$.  
				
				\item $(n+1,m,\l-1)$:  Since 
					\begin{align*}
						&m-5 < m < m+3,\\
						&m \equiv 1 + (m-1) \imod 4, \text{ and}\\
						&n+1 \equiv 0 \imod 2,
					\end{align*}
				induction gives $o^-((n+1,m,\l-1)) = \Left$.  
			\end{enumerate}
		Thus Right moving first does not have a good move.
		
		Therefore $o^-((n,m,\l)) = \Prev$ if $\l-4 < m < \l+4$, $m \equiv \l \imod 4$,  and $n \equiv 1 \imod 2$.

		\item  $\l-4 < m < \l+4$, $m \equiv 1 + \l \imod 4$, $n \equiv 0 \imod 2$:  Suppose $m=0$.  Then we claim that Left moving to $(n+1,m-1,\l)$ is a winning move for Left.  Either 
			\[\l-4 < m-1 < \l+4 \text{ or } m = \l-3.\]
		Suppose $\l-4 < m-1 < \l+4$.  Then $m-1 \equiv 0 + \l \imod 4$ and $n+1 \equiv 1 \imod 2$.  Therefore, by induction, $o^-((n+1,m-1,\l)) = \Prev$.  Otherwise $m=\l-3$, our initial position is $(n,\l-3,\l)$ and Left moves to the position $(n+1, \l-4,\l)$.  Then $\l -4 \le \l-4$, so by induction, this position is $\Left$.  
		
		Suppose $m=0$.  This forces $\l =3$ and our initial position becomes $(n,0,3)$.  Left moves to $(n,0,2)$, which is a $\Prev$ position by induction.  Therefore Left wins moving first.  
		
		Suppose Right moves first.  Consider Right's three moves.
			\begin{enumerate}
				\item $(n-1,m,\l)$:   Then the relationship between $m$ and $\l$ remains unchanged, while $n-1 \equiv 1 \imod 2$.  By induction, $o^-((n-1,m,\l)) = \Next$.
				
				\item $(n, m-1,\l)$:  Similarly to the argument for Left's winning by moving to $(n+1,m-1,\l)$, we see that $o^-((n, m-1,\l)) = \Next \cup \Left$.
				
				\item $(n+1,m,\l-1)$:  Either 
					\[\l-5<m<\l-3 \text{ or } m = \l+3.\]
				Suppose $\l-5<m<\l+3$.  Then $m \equiv 2 + (\l-1) \imod 4$ and $n+1 \equiv 1 \imod 2$.  Then, by induction, $o^-((n+1,m,\l-1)) = \Next$.  Otherwise $m = \l+3$, which gives $\l+1 \equiv \l+3 \imod 4$, a contradiction.  So $\l-5 < m < \l+3$.
			\end{enumerate}
		
		We then see that Right cannot win moving first.
		
		Therefore $o^-((n,m,\l)) = \Left$ if  $\l-4 < m < \l+4$, $m \equiv 1 + \l \imod 4$, and $n \equiv 0 \imod 2$.
		
		\item  $\l-4 < m < \l+4$, $m \equiv 1 + \l \imod 4$, $n \equiv 1 \imod 2$:  Since $n \equiv 1 \imod 2$, we have $n \ge 1$.  Since $m \equiv 1+\l \imod 4$ and $\l \ge 0$, we have $m \ge 1$ as well.
		
		Left moving first can move to $(n-1,m,\l)$, which is an $\Left$ position by induction.
		
		We now consider Right's move.  Either 
			\[\l-4 < m-1 < \l+4 \text{ or } m=\l-3.\]
		Suppose $\l-4 < m-1 < \l+4$ and Right moves first to $(n, m-1,\l)$.  Then $m-1 \equiv \l \imod 4$ and $n \equiv 1 \imod 2$, so, by induction, $o^-((n, m-1,\l)) = \Prev$.  Otherwise $m = \l-3$.  Then our initial position $(n,m,\l)$ is $(n,\l-3,\l)$.  If Right moves to $(n+1, \l-3,\l-1)$, then $\l-5 < \l-3 < \l+3$ and $\l-3 \equiv 2 +(\l-1) \imod 4$, so, by induction, $o^-((n+1,\l-3,\l-1)) = \Prev$.    Therefore Right has a winning move moving first.
		
		Therefore $o^-((n,m,\l)) = \Next$ if  $\l-4 < m < \l+4$, $m \equiv 1 + \l \imod 4$, and $n \equiv 1 \imod 2$.
		
		\item  $\l-4 < m < \l+4$, $m \equiv 2 + \l \imod 4$, $n \equiv 0 \imod 2$:  Consider Left's three moves.
			\begin{enumerate}
				\item $(n-1,m,\l)$:  The relationship between $m$ and $\l$ remains unchanged, while $n-1 \equiv 1 \imod 2$.  By induction, $o^-((n-1,m,\l)) = \Next$.
				
				\item $(n+1,m-1,\l)$:  Either 
					\[\l-4 < m-1 < \l+4 \text{ or } m=\l-3.\]
				Suppose $\l-4 < m-1 < \l+4$.  Then $m-1 \equiv 1 + \l \imod 4$ and $n+1 \equiv 1 \imod 4$, so, by induction, $o^-((n+1,m-1,\l)) = \Next$.  Otherwise $m = \l-3$, but $m \equiv 2 + \l \imod 4$, a contradiction.  Therefore $\l-4 < m-1 < \l+4$ and Left loses moving first to $(n+1,m-1,\l)$.
				
				\item $(n,m,\l-1)$:  Either 
					\[\l-5 < m < \l+3 \text{ or } m=\l+3.\]
				Suppose $\l-5 < m < \l+3$.  Then $m \equiv 3 + (\l-1) \imod 4$ and $n \equiv 0 \imod 2$, so, by induction, $o^-((n,m,\l-1)) = \Right$.  Otherwise $m = \l+3$, but $m \equiv 2 + \l \imod 4$, a contradiction.  Therefore $\l-5 < m < \l+3$ and Left loses moving first to $(n,m,\l-1)$.  
			\end{enumerate}
		Therefore Left moving first has no winning move.
		
		Suppose Right moves first.
			\begin{enumerate}
				\item $(n-1,m,\l)$: As when Left moved to this position above, this is an $\Next$ position.
				
				\item $(n, m-1,\l)$:  We showed above that $\l-4 < m-1 < \l+4$, and so since $m-1 \equiv 1 + \l \imod 4$ and $n \equiv 0 \imod 2$, induction gives that $o^-((n, m-1,\l)) = \Left$.  
			
				\item $(n+1,m,\l-1)$:  We showed above that $\l-5<m<\l-3$, and so, since $m \equiv 3 + (\l-1) \imod 4$ and $n+1 \equiv 1 \imod 4$, we have $o^-((n+1,m,\l-1)) = \Next$.
			\end{enumerate}
		Thus Right moving first has no good move.
		
		Therefore $o^-((n,m,\l)) = \Prev$ if  $\l-4 < m < \l+4$, $m \equiv 2 + \l \imod 4$, and $n \equiv 0 \imod 2$.
		
		\item  $\l-4 < m < \l+4$, $m \equiv 2 + \l \imod 4$, $n \equiv 1 \imod 2$:  Since $n \equiv 1 \imod 2$, we have $n \ge 1$.  Then, Left and Right have the option to move to $(n-1,m,\l)$,$(n-1,m,\l)$ respectively, which is a $\Prev$ position by induction.
		
		\item $\l-4 < m < \l+4$, $m \equiv 3 + \l \imod 4$, $n \equiv 0 \imod 2$:  Consider Left's possible moves.
			\begin{enumerate}
				\item $(n-1,m,\l)$:  The relationship between $m$ and $\l$ remains unchanged, while $n-1 \equiv 1 \imod 2$.  By induction, $o^-((n-1,m,\l)) = \Next$.
				
				\item $(n+1,m-1,\l)$:  Either 
					\[\l-4 < m-1 < \l+4 \text{ or } m= \l-3.\]
				Suppose $\l-4 < m-1 < \l+4$.  Since $m-1 \equiv 2 + \l \imod 4$ and $n+1 \equiv 1 \imod 2$, induction gives that $o^-((n+1,m-1,\l)) = \Next$.  Otherwise $m = \l-3$.  But $m \equiv 3 + \l \imod 4$, a contradiction.  Therefore $\l-4 < m-1 < \l+4$ and Left loses moving first to $(n+1,m-1,\l)$.
				
				\item $(n,m,\l-1)$:  Either 
					\[\l-5 < m < \l+3 \text{ or } m=\l-3.\]
				Suppose $\l-5 < m  < \l+3$.  Since $m \equiv \l-1 \imod 4$ and $n \equiv 0 \imod 2$, induction gives that $o^-((n,m,\l-1)) = \Next$.  Otherwise $m= \l+3$ and the position $(n,m,\l-1)$ is $(n,\l+3,\l-1)$.  Since $\l+3 \equiv \l-1 \imod 4$, induction gives that this position is $\Next$.
			\end{enumerate}
		
		Suppose Right moves first.  Since $m \equiv 3 + \l \imod 4$ and $\l \ge 0$, we have that $m \ge 1$.  We claim that Right has a winning move moving to $(n, m-1,\l)$.  Above, we showed that $\l-4 < m-1<\l+4$.  Since $m-1 \equiv 2 + \l \imod 4$ and $n \equiv 0 \imod 2$, induction gives that $o^-((n, m-1,\l)) = \Prev$.  
		
		Therefore $o^-((n,m,\l)) = \Right$ if $\l-4 < m < \l+4$, $m \equiv 3 + \l \imod 4$, and $n \equiv 0 \imod 2$.
		
		\item $\l-4 < m < \l+4$, $m \equiv 3 + \l \imod 4$, $n \equiv 1 \imod 2$:  Since $n \equiv 1 \imod 2$, we have $n \ge 1$.  Since $m \equiv 3 + \l \imod 4$ and $\l \ge 0$, we have $m \ge 1$ as well.
		
		Suppose Left is moving first.  We claim that $(n+1,m-1,\l)$ is a winning move for Left.  Either 
			\[\l-4 < m-1 < \l+4 \text{ or } m = \l+3.\]
		Suppose $\l-4 < m-1 < \l+4$.  Then $m-1 \equiv 2 + \l \imod 4$ and $n+1 \equiv 0 \imod 2$, so, by induction $o^-((n+1,m-1,\l)) = \Prev$.  Otherwise $m = \l-3$.  But $m \equiv 3 + \l \imod 4$, a contradiction.  Therefore $\l-4 < m-1 < \l+4$ and Left wins moving first to $(n+1,m-1,\l)$.  
		
		Right moving first moves to $(n-1,m,\l)$.  Since the relationship between $m$ and $\l$ remains unchanged, induction gives that $o^-((n-1,m,\l)) = \Right$.
		
		Therefore $o^-((n,m,\l)) = \Next$ if $\l-4 < m < \l+4$, $m \equiv 3 + \l \imod 4$, and $n \equiv 1 \imod 2$.
			
		\item $m \ge \l+4$, $n \equiv 0 \imod 2$:  If $o^-((n,m,\l)) = \Left$, then by Theorem \ref{theorem-outcome-class-conj}, we have $o^-(\overline{(n,m,\l)}) = \Right$.  But $\overline{(n,m,\l)} = (n,\l,m)$.  By the very first case we examined in this proof, if $m \le \l-4$, we have $o^-((n,m,\l)) = \Left$, giving $o^-((n,\l,m)) = \Right$.  Rearranging the $\l$ and $m$ gives that if $m \ge \l+4$, then $o^-((n,m,\l)) = \Right$.  
		
		\item $m \ge \l+4$, $n \equiv 1 \imod 2$:  The argument given in the preceding case is also valid for this case.  
	\end{enumerate}
	
	This completes the induction and the proof.
\end{proof}

The search for distinguishability and indistinguishability relations within $\cl{\rho,\rhob}$ now begins.

\begin{corollary}\label{cor-B-*+*=0}
	The following indistinguishability relations exist on $\cl{\rho, \rhob}$:
		\begin{enumerate}
			\item\label{item-200-000} $(2,0,0) \equiv (0,0,0) \imod{\cl{\rho,\rhob}}$,
			\item\label{item-n-m-l-n-m+1-l+1} $(n,m,\l) \equiv (n,m+1,\l+1) \imod{\cl{\rho,\rhob}}$.
		\end{enumerate}
\end{corollary}

\begin{proof}
	Take arbitrary $(a,b,c)$ in $\cl{\rho,\rhob}$.
	
		\begin{enumerate}
			\item By Theorem \ref{theorem-B-oc}, $o^-((a+2,b,c)) = o^-((a,b,c))$.
			\item Consider 
				\[o^-((a,b,c) + (n,m,\l)) = o^-((a+n,b+m, c+\l))\]
			and
				\[o^-((a,b,c) + (n,m+1,\l+1)) = o^-((a+n, b+m+1, c+\l+1)).  \]
			Using Theorem \ref{theorem-B-oc}, we will determine the outcome classes of both of these positions.  There are twelve possibilities for relationships between $n$, $m$, and $\l$ based on the ten cases given in the statement of Theorem \ref{theorem-B-oc}.  Number these cases 1 through 12 as in the proof of Theorem \ref{theorem-B-oc}.
			
			\begin{enumerate}
				\item Cases 1 and 2: We have
					\[ \l+c \ge m+b + 4 \iff \l+c +1 \ge (m+b+1) + 4.\]
				Thus $(n+a, m+b,\l+c)$ is covered by case $i$ of Theorem \ref{theorem-B-oc} if and only if $(n+a, m+b+1, \l+c+1)$ is for $i \in \{1,2\}$.  
	
				\item Cases 3 through 10:  We have
					\[ \l+c-4 < m+b < \l+c+4 \iff (\l+c+1)-4 < m+b+1 < (\l+c+1)+4.\]
				Also, since
					\[ m+b \equiv t + (\l+c) \imod 4 \iff m+b+1 \equiv t + (\l+c+1) \imod 4,\]
				and the coefficient for $*$ is the same for both positions.   Therefore both positions are covered by case $i$ of Theorem \ref{theorem-B-oc} for $i \in \{3,4,\ldots,10\}$, and such, they both have the same outcome.
				
				\item Cases 11 and 12: We have
					\[ m+b \ge \l+c+4 \iff m+b+1 \ge (\l+c+1) + 4.\]
				Thus $(n+a, m+b, \l+c)$ is in case $i$ of Theorem \ref{theorem-B-oc} if and only if $(n+a, m+b+1, \l+c+1)$ is, where $i \in \{11,12\}$.			
			\end{enumerate}
			
			Therefore $(n,m,\l)$ and $(n,m,\l+1)$ are indistinguishable.\qedhere
		\end{enumerate}
\end{proof}

Notice that in both $\cl{\rho}$ and $\cl{\rho, \rhob}$, $2 *$ and 0 are indistinguishable.

We continue by examining further indistinguishability relations.

\begin{theorem}\label{theorem-B-ind-elements}
	Suppose $(a,b,c)$ is an arbitrary element of $\cl{\rho,\rhob}$.  Then there exists $u \in \mathbb{Z}^{\ge 0}$ such that $(a,b,c)$ is indistinguishable from one of the following positions:
		\begin{enumerate}
			\item $(0,0,u)$,
			\item $(0,u,0)$,
			\item $(1,0,u)$, or
			\item $(1,u,0)$.
		\end{enumerate}
	Moreover, these four positions are distinguishable.  That is, when examining $\cl{\rho,\rhob}$ up to indistinguishability, only positions of the above form need to be examined.
\end{theorem}

\begin{proof}
	By Corollary \ref{cor-B-*+*=0}\eqref{item-200-000}, $(a,b,c)$ is indistinguishable from either $(0,b,c)$ or $(1,b,c)$.   Let $a^{\prime} \in \{0,1\}$ such that $(a,b,c)$ and $(a^{\prime},b,c)$ are indistinguishable.  
		
	Consider the relationship between $b$ and $c$.   If $b >c$, then repeated applications of Corollary \ref{cor-B-*+*=0}\eqref{item-n-m-l-n-m+1-l+1} $(a^{\prime},b,c)$ is indistinguishable from $(a^{\prime}, b-c,0)$.  If $b\le c$, then, again by repeated applications of Corollary \ref{cor-B-*+*=0}\eqref{item-n-m-l-n-m+1-l+1}, $(a^{\prime},b,c)$ is indistinguishable from $(a^{\prime},0,c-b)$.  
	
	It remains to show that the four types of positions are distinguishable.  There are ten comparisons which must be checked.  Take $u,v \in \mathbb{Z}^{\ge 0}$.

	\begin{enumerate}
		\item\label{item-00u-00v} $(0,0,u)$ and $(0,0,v)$ distinguishable for $u \not = v$:   Suppose, without loss of generality, that $u<v$.  
		
		If $u \in \{0,1,2\}$, then the positions are distinguished by $(0,0,0)$.  
		
		Suppose $u \ge 3$.  The positions are distinguished by $(0,u,0)$ since $o^-((0,u,u)) = \Next$ while $o^-((0,u,v)) \not = \Next$. 
		
		\item\label{item-00u-0v0} $(0,0,u)$ and $(0,v,0)$ are distinguishable for $u+v > 0$:  Consider position $(0,0,v)$.  Then $o^-((0,v,v)) = \Next$ but $o^-((0,0,u+v)) = \Next$ if and only if $u+v = 0$, contradicting that $u+v > 0$.  Thus $o^-((0,0,u+v)) \not = \Next$, so the two positions are distinguished by $(0,0,v)$.
		
		\item\label{item-00u-10v} $(0,0,u)$ and $(1,0,v)$ are distinguishable:  If $o^-((0,0,u)) \not = o^-((1,0,v))$, then the two positions are distinguished by $(0,0,0)$.  When are the outcomes equal?
		
			\begin{enumerate}
				\item $u=0$, $v=1$:  The positions are $(0,0,0)$ and $(1,0,1)$, which are distinguished by $(1,0,0)$ since $o^-((1,0,0)) = \Prev$ while $o^-((0,0,1)) = \Right$.
				
				\item $u=0$, $v=2$:  The positions are $(0,0,0)$ and $(1,0,2)$, which are distinguished by $(1,0,1)$ since $o^-(1,0,1) = \Next$ while $o^-((0,0,3)) = \Left$.
				
				\item $u=0$, $v=3$:  The positions are $(0,0,0)$ and $(1,0,3)$, which are distinguished by $(1,0,2)$ since $o^-((1,0,2)) = \Next$ while $o^-((0,0,5)) = \Left$.
				
				\item $u=2$, $v = 0$:  The positions are $(0,0,2)$ and $(1,0,0)$ which are distinguished by $(0,0,1)$ since $o^-((0,0,3)) = \Left$ while $o^-((1,0,1)) = \Next$.
				
				\item $u \ge 3$, $ v \ge 4$:  Let us first suppose that $u<v$.  Consider the position $(1,v-3,0)$.  Adding this to $(1,0,v)$ gives $(0,v-3,v)$, which is an $\Left$ position.  Adding this to $(0,0,u)$ gives $(1,v-3,u)$, which is only an $\Left$ position if $v-3 \le u-4$, which implies $v \le u-1$, but $u<v$, contradiction.  Therefore $o^-((1,v-3,u)) \not = \Left$, and the two positions are distinguished.
				
				Suppose $u=v$.  Then the two positions are $(0,0,u)$ and $(1,0,u)$, which are distinguished by $(1,u,0)$ since $o^-((1,u,u)) = \Prev$ while $o^-((0,u,u)) = \Next$.
				
				Suppose $u > v$.  Consider the position $(0,u-3,0)$.  Adding this to $(0,0,u)$ gives $(0,u-3,u)$, which is an $\Left$ position.  Adding this to $(1,0,v)$ gives $(1,u-3,v)$, which is an $\Left$ position only if $u-3 \le v-4$, which implies $u \le v-1$ but $v < u$, contradiction.  Therefore $o^-((1,u-3,v)) \not = \Left$, and the two positions are distinguished.
			\end{enumerate}

		\item\label{item-00u-1v0} $(0,0,u)$ and $(1,v,0)$ are distinguishable: If $v=0$, then the two positions are $(0,0,u)$ and $(1,0,0)$, which were shown to be distinguishable in Case \ref{item-00u-10v}.  
		
		Suppose $v \in \{1,2,3\}$.  Then $o^-((1,v,0)) = \Next$ and $o^-((0,0,u)) = \Next$ only if $u = 0$.  But $(0,1,0)$ distinguishes the two since $o^-((0,1,0)) = \Left$ while $o^-((1,v+1,0)) = \Next \cup \Right$.
		
		Suppose $v \ge 4$.  Then $o^-((1,v,0)) = \Right$ and $o^-((0,0,u)) \not = \Right$ unless $u=1$.  Thus if $u \not = 1$, then $(0,0,0)$ distinguishes the two positions.  Suppose $u=1$ and the two positions are $(0,0,1)$ and $(1,v,0)$ with $v \ge 4$.  Then $(0,1,0)$ distinguishes the two positions since $o^-((0,1,1)) = \Next$ and $o^-((1,v+1,0)) = \Right$.  

		\item\label{item-0u0-0v0} $(0,u,0)$ and $(0,v,0)$ are distinguishable for $u \not = v$:  Suppose, without loss of generality, that $u < v$.  
		
		If $u \in \{0,1,2\}$, then the positions are distinguished by $(0,0,0)$.
		
		Suppose $u \ge 3$.  Consider the position $(0,0,u)$.  Then $o^-((0,u,u)) = \Next$ while $o^-((0,v,u)) \not = \Next$. Therefore the two positions are distinguishable.
 				
		\item\label{item-0u0-10v} $(0,u,0)$ and $(1,0,v)$ are distinguishable:  If $o^-((0,u,0)) \not = o^-((1,0,v))$, then the two positions are distinguished by $(0,0,0)$.  When are the two outcomes equal?
		
			\begin{enumerate}
				\item $u=0$ and $v \in \{1,2,3\}$:  This was already considered in Case \ref{item-00u-10v} and the positions shown to be distinguishable.

				\item $u=1$ and $v \ge 4$:  The two positions are $(0,1,0)$ and $(1,0,v)$ with $v \ge 4$, which are distinguished by $(0,0,1)$ since $o^-((0,1,1)) = \Next$ and $o^-((1,0,v+1)) = \Right$.
				
				\item $u=2$ and $v=0$:  The two positions are $(0,2,0)$ and $(1,0,0)$, which are distinguished by $(0,1,0)$ since $o^-((0,3,0)) = \Right$ while $o^-((1,1,0)) = \Next$.				
			\end{enumerate}
		
		\item $(0,u,0)$ and $(1,v,0)$ are distinguishable: If $v = 0$, then the two positions are $(0,u,0)$ and $(1,0,0)$ where $o^-((0,u,0)) \not = o^-((1,0,0))$ unless $u= 2$.  Then the two positions are $(0,2,0)$ and $(1,0,0)$, which were shown to be distinguishable in Case \ref{item-0u0-10v}.  
		
		If $v \in \{1,2,3\}$, then $o^-((0,u,0)) \not = o^-((1,v,0))$ unless $u = 0$.  These positions were shown to be distinguishable in Case \ref{item-00u-1v0}.
		
		If $v \ge 4$, then $o^-((0,u,0)) \not = o^-((1,v,0))$ unless $u \ge 3$.  Suppose then $u \ge 3$ and $v \ge 4$.  Firstly, take $u < v$.  Consider position $(1,0,v-3)$.  Adding this to $(1,v,0)$ gives position $(0,v,v-3)$, which is an $\Right$ position.  Adding it to $(0,u,0)$ gives $(1,u,v-3)$, which is an $\Right$ position only if $u \ge v-3+4$, but $u < v$, contradiction.  Therefore $o^-((1,u,v-3)) \not = \Right$, and so the two initial positions are distinguishable.
		
		Suppose $u=v$.  Then $(0,0,u)$ distinguishes the positions since $o^-((0,u,u)) = \Next$ while $o^-((1,u,u)) = \Prev$.
		
		Suppose $u > v$.  Consider position $(0,0,u-3)$.  Adding this to $(0,u,0)$ gives $(0,u,u-3)$, whose outcome is $\Right$.  Adding it to $(1,v,0)$ gives $(1,v,u-3)$ which is an $\Right$ position only if $v \ge u-3+4$, but $u > v$, contradiction.  Therefore $o^-((1,v,0)) \not = \Right$, and so the two initial positions are distinguishable.  
	
		\item $(1,0,u)$ and $(1,0,v)$ are distinguishable if $u \not = v$:  Repeat the argument given in Case \ref{item-00u-00v} replacing the distinguishing element $(a,b,c)$ by $(a+1 \imod 2, b, c)$.  
		
		\item $(1,0,u)$ and $(1,v,0)$ are distinguishable for $u+v > 0$:  Repeat the argument given in Case \ref{item-00u-0v0} replacing the distinguishing element $(a,b,c)$ by $(a+1 \imod 2, b, c)$.  
				
		\item\label{item-1u0-1v0} $(1,u,0)$ and $(1,v,0)$ are distinguishable for $u \not = v$:  Repeat the argument given in Case \ref{item-0u0-0v0} replacing the distinguishing element $(a,b,c)$ by $(a+1 \imod 2, b, c)$.  \qedhere
	\end{enumerate}
\end{proof}

In Corollary \ref{cor-B-*+*=0}, we saw that the relationship $(2,0,0) \equiv (0,0,0) \imod{\cl{\rho,\rhob}}$ holds, as it does over $\cl{\rho}$ (Corollary \ref{cor-cl-rho-indistinguishability}).  However, this is one of the few relations which stays constant through the addition of $\rhob$ to the closure.  For example, Corollary \ref{cor-cl-rho-indistinguishability} gives that $4  \rho \equiv * + 4  \rho \imod{\cl{\rho}}$, while in $\cl{\rho,\rhob}$, $\rhob$ distinguishes these two elements.  Moreover, and more importantly, the \mis monoid of $\cl{\rho,\rhob}$ contains an infinite number of distinguishable elements, which follows as a corollary from Theorem \ref{theorem-B-ind-elements}.

\begin{corollary}\label{cor-B-infinite}
	In $\cl{\rho,\rhob}$, there are an infinite number of distinguishable elements.
\end{corollary}

\begin{proof}
	This follows directly from the proof of Theorem \ref{theorem-B-ind-elements}.  For example, $(0,0,a)$ and $(0,0,b)$ are distinguishable for all $a,b \in \mathbb{Z}^{\ge 0}$ with $a \not = b$.
\end{proof}

Our next step is to examine the partial order of elements of $\cl{\rho,\rhob}$.  As in $\cl{\rho}$, we first construct the monoid, as working in the monoid is easier visually than working within the \mis monoid.

With the mappings:
	\begin{align*}
		(0,0,0) &\mapsto 1;\\
		(1,0,0) &\mapsto a;\\
		(0,1,0) &\mapsto p;\\
		(0,0,1) &\mapsto q,
	\end{align*}
we obtain the following monoid:
	\begin{align*}
	\monoid{M}_{\cl{\rho,\rhob}} &= \ideal{1,a,p,q \, \left| \, a^2 = 1, p^mq^n = \begin{cases}
	p^{m-n} &\text{if } m > n;\\
	q^{n-m} &\text{if } m \le n.
	\end{cases}\right.} \\
	\Next &= \{1, ap, ap^2, ap^3,aq, aq^2,aq^3\} \\
	\Prev &= \{a, p^2, q^2\} \\
	\Left &= \{p, q^3, q^4, q^5, \ldots, aq^4, aq^5, aq^6 \ldots\} \\
	\Right &= \{q, p^3, p^4, p^5,\ldots, ap^4, ap^5, ap^6, \ldots\}
	\end{align*}
with the additive notation in $\cl{\rho, \rhob}$ becoming a multiplicative notation in $\monoid{M}_{\cl{\rho, \rhob}}$.

We are now ready to start examining the partial order.  At first, we restrict ourselves to elements that do not have an $a$ coefficient, i.e.\ elements of the form either $p^m$, or $q^n$ for $m,n \in \mathbb{Z}^{\ge 0}$.

Since there are so few, our first result is regarding elements which are incomparable.

\begin{proposition}\label{prop-cl-rho-rhob-incomparable}
	The following elements are incomparable:
		\begin{enumerate}
			\item $p$ and $q$;
			\item $p$ and $q^2$;
			\item $q^n$ and $q^t$ where $n,t \in \mathbb{Z}^{\ge 0}$ with $1 \le t-n \le 3$;
			\item $p^n$ and $p^t$ where $n,t \in \mathbb{Z}^{\ge 0}$ with $1 \le t-n \le 3$.
		\end{enumerate}
\end{proposition}

\begin{proof}\text{}
	\begin{enumerate}
		\item Consider $q$ and $p$ with element $p$.  Then $o^-(qp) = o^-(1) = \Next$ while $o^-(pp) = o^-(p^2) = \Prev$.  Thus $q$ and $p$ are incomparable.
		\item Consider $q^2$ and $p$ with element $aq$.  Then $o^-(aqq^2) = o^-(aq^3) = \Next$ while $o^-(aqp) = o^-(a) = \Prev$.  Thus $q^2$ and $p$ are incomparable.
		\item Consider $q^n$ and $q^t$ and element $ap^n$.  Then $ap^nq^n = a$ with $o^-(a) = \Prev$ while $ap^nq^t = a q^{t-n}$ with $o^-(aq^{t-n}) = \Next$ since $1 \le t-n \le 3$.  Therefore $q^n$ and $q^t$ are incomparable.
		\item This case is similar to the previous one.\qedhere
	\end{enumerate}
\end{proof}

These are the only incomparability relations for elements with no $a$ coefficient.  All other pairs of elements of the form $(p^m, q^n)$ for $m,n \in \mathbb{Z}^{\ge 0}$ are comparable, as the following lemmas show.  

\begin{proposition}\label{prop-qn<qt}
	Take $n,t \in \mathbb{Z}^{\ge 0}$ where $t-n \ge 4$.  Then $q^n < q^t$.  
\end{proposition}

\begin{proof}
	Take $m \in \mathbb{Z}$.  Then, in our monoid, there are six types of elements: $1$, $a$, $p^m$, $q^m$, $ap^m$, and $aq^m$.  It suffices to show that for any one element $e$ of these forms, $o^-(q^ne) \le o^-(q^{t}e)$.  That the inequality on the elements is strict comes from the fact that $q^n$ and $q^{t}$ are distinguishable by Theorem \ref{theorem-B-ind-elements}.  

	We construct Tables \ref{table-qn<qn+4}, \ref{table-qn<qn+5}, \ref{table-qn<qn+6}, and \ref{table-qn<qn+7} for the cases $(q^n, q^{n+4})$, $(q^n, q^{n+5})$, $(q^n, q^{n+6})$, and $(q^n, q^t)$ where $t-n \ge 7$ respectively, where an entry is $\star$ if it does not need to be determined as the other entry is either $\Right$ or $\Left$ as appropriate.

	\begin{table}[p]
	\begin{center}
	    \begin{tabular}{p{2.7cm}llll}
	    \multicolumn{1}{c}{$\boldsymbol{e}$} & $\boldsymbol{q^ne}$ & $\boldsymbol{o^-(q^ne)}$ & $\boldsymbol{q^{n+4}e}$ & $\boldsymbol{o^-(q^{n+4}e)}$ 
	    \\
	    \rowcolor[gray]{.8}1 & $\star$ & $\star$ & $q^{n+4}$ & $\Left$ \\
	    $a$ & $\star$ & $\star$ & $aq^{n+4}$ & $\Left$ \\
	    \rowcolor[gray]{.8}	$p^m$ \,\,\,\,\,\,\,\,\,\,\,\,\,\,\,\,\,\,\,\,\,
	    	$m-n \ge 3$
	    & $p^{m-n}$ & $\Right$ & $\star$ & $\star$\\	    
	    	$p^m$ \,\,\,\,\,\,\,\,\,\,\,\,\,\,\,\,\,\,\,\,\,
	    	$m-n=2$&
	    $p^2$ & $\Prev$ & $q^2$ & $\Prev$\\
	    \rowcolor[gray]{.8}	$p^m$ \,\,\,\,\,\,\,\,\,\,\,\,\,\,\,\,\,\,\,\,\,
	    	$m-n \le 1$
	    & $\star$ & $\star$ & $q^{n-m+4}$ & $\Left$\\
	    $q^m$ & $\star$ & $\star$ & $q^{n+m+4}$ & $\Left$\\
	    \rowcolor[gray]{.8}	$ap^m$ \,\,\,\,\,\,\,\,\,\,\,\,\,\,\,\,\,\,\,\,\,
	    	$m-n \ge 4$
	    & $ap^{m-n}$ & $\Right$ & $\star$ & $\star$\\	    
	    	$ap^m$ \,\,\,\,\,\,\,\,\,\,\,\,\,\,\,\,\,\,\,
		$1 \le m-n \le 3$
	    & $ap^{m-n}$ & $\Next$ & $aq^{m-n}$ & $\Next$\\
	    \rowcolor[gray]{.8}	$ap^m$ \,\,\,\,\,\,\,\,\,\,\,\,\,\,\,\,\,\,\,\,\,
		$m-n \le 0$
	    & $\star$ & $\star$ & $aq^{n-m+4}$ & $\Left$\\
	    $aq^m$&$\star$&$\star$&$aq^{n+m+4}$&$\Left$\\
	    \end{tabular}
	\end{center}
	\caption{Showing $q^n < q^{n+4}$ in $\monoid{M}_{\cl{\rho,\rhob}}$.}
	\label{table-qn<qn+4}
	\end{table}

	\begin{table}[p]
	\begin{center}
	    \begin{tabular}{p{2.7cm}llll}
	    \multicolumn{1}{c}{$\boldsymbol{e}$} & $\boldsymbol{q^ne}$ & $\boldsymbol{o^-(q^ne)}$ & $\boldsymbol{q^{n+5}e}$ & $\boldsymbol{o^-(q^{n+5}e)}$ 
	    \\ 
	    \rowcolor[gray]{.8}1 & $\star$ & $\star$ & $q^{n+5}$ & $\Left$ \\
	    $a$ & $\star$ & $\star$ & $aq^{n+5}$ & $\Left$ \\
	    \rowcolor[gray]{.8}
		$p^m$ \,\,\,\,\,\,\,\,\,\,\,\,\,\,\,\,\,\,\,\,\,
		$m-n \ge 3$
	    & $p^{m-n}$ & $\Right$ & $\star$ & $\star$\\    
		$p^m$ \,\,\,\,\,\,\,\,\,\,\,\,\,\,\,\,\,\,\,\,\,
		$m-n \le 2$
	    & $\star$ & $\star$ & $q^{n-m+5}$ & $\Left$\\
	    \rowcolor[gray]{.8}$q^m$ & $\star$ & $\star$ & $q^{n+m+5}$ & $\Left$\\
	    	$ap^m$ \,\,\,\,\,\,\,\,\,\,\,\,\,\,\,\,\,\,\,\,\,
		$m-n \ge 4$
	    & $ap^{m-n}$ & $\Right$ & $\star$ & $\star$\\	    
	    \rowcolor[gray]{.8}
	    	$ap^m$ \,\,\,\,\,\,\,\,\,\,\,\,\,\,\,\,\,\,\,
		$2 \le m-n \le 3$
	    & $ap^{m-n}$ & $\Next$ & $aq^{n-m+5}$ & $\Next$\\
		$ap^m$ \,\,\,\,\,\,\,\,\,\,\,\,\,\,\,\,\,\,\,
		$m-n \le 1$
	    & $\star$ & $\star$ & $aq^{n-m+5}$ & $\Left$\\
	    \rowcolor[gray]{.8} $aq^m$&$\star$&$\star$&$aq^{n+m+5}$&$\Left$\\
	    \end{tabular}
	\end{center}
	\caption{Showing $q^n < q^{n+5}$ in $\monoid{M}_{\cl{\rho,\rhob}}$.}
	\label{table-qn<qn+5}
	\end{table}	

	\begin{table}[p]
	\begin{center}
	    \begin{tabular}{p{2.7cm}llll}
	    \multicolumn{1}{c}{$\boldsymbol{e}$} & $\boldsymbol{q^ne}$ & $\boldsymbol{o^-(q^ne)}$ & $\boldsymbol{q^{n+6}e}$ & $\boldsymbol{o^-(q^{n+6}e)}$ 
	    \\ 
	    \rowcolor[gray]{.8}1 & $\star$ & $\star$ & $q^{n+6}$ & $\Left$ \\
	    $a$ & $\star$ & $\star$ & $aq^{n+6}$ & $\Left$ \\
	    \rowcolor[gray]{.8}
		$p^m$  \,\,\,\,\,\,\,\,\,\,\,\,\,\,\,\,\,\,\,\,\,
		$m-n \ge 3$
	    & $p^{m-n}$ & $\Right$ & $\star$ & $\star$\\	    
		$p^m$  \,\,\,\,\,\,\,\,\,\,\,\,\,\,\,\,\,\,\,\,\,
		$m-n \le 2$
	    & $\star$ & $\star$ & $q^{n-m+6}$ & $\Left$\\
	    \rowcolor[gray]{.8}$q^m$ & $\star$ & $\star$ & $q^{n+m+6}$ & $\Left$\\
		$ap^m$  \,\,\,\,\,\,\,\,\,\,\,\,\,\,\,\,\,\,\,\,\,
		$m-n \ge 4$
	    & $ap^{m-n}$ & $\Right$ & $\star$ & $\star$\\	    
	    \rowcolor[gray]{.8}
		$ap^m$  \,\,\,\,\,\,\,\,\,\,\,\,\,\,\,\,\,\,\,\,\,
		$m-n=3$
	    & $ap^3$ & $\Next$ & $aq^3$ & $\Next$\\	
		$ap^m$  \,\,\,\,\,\,\,\,\,\,\,\,\,\,\,\,\,\,\,\,\,
		$m-n \le 2$
	    & $\star$ & $\star$ & $aq^{n-m+6}$ & $\Left$\\
	    \rowcolor[gray]{.8}$aq^m$&$\star$&$\star$&$aq^{n+m+6}$&$\Left$\\
	    \end{tabular}
	\end{center}
	\caption{Showing $q^n < q^{n+6}$ in $\monoid{M}_{\cl{\rho,\rhob}}$.}
	\label{table-qn<qn+6}
	\end{table}	

	\begin{table}[p]
	\begin{center}
	    \begin{tabular}{p{2.7cm}llp{3cm}l}
	    \multicolumn{1}{c}{$\boldsymbol{e}$} & $\boldsymbol{q^ne}$ & $\boldsymbol{o^-(q^ne)}$ & $\boldsymbol{q^te}$ & $\boldsymbol{o^-(q^te)}$ 
	    \\ 
	    \rowcolor[gray]{.8}1 & $\star$ & $\star$ & $q^t$ & $\Left$ \\
	    $a$ & $\star$ & $\star$ & $aq^t$ & $\Left$ \\
	    \rowcolor[gray]{.8}
		$p^m$ \,\,\,\,\,\,\,\,\,\,\,\,\,\,\,\,\,\,\,\,\,
		$m-n \ge 3$
	    & $p^{m-n}$ & $\Right$ & $\star$ & $\star$\\
		$p^m$ \,\,\,\,\,\,\,\,\,\,\,\,\,\,\,\,\,\,\,\,\,
		$m-n \le 2$
	    & $\star$ & $\star$ & 
	    $q^u$ \,\,\,\,\,\,\,\,\,\,\,\,\,\,\,\,\,\,\,\,\,\,\,\,\,
	    such that $u \ge 5$
	    & $\Left$\\
	    \rowcolor[gray]{.8}
	    $q^m$ & $\star$ & $\star$ & $q^{m+t}$ & $\Left$\\
		$ap^m$ \,\,\,\,\,\,\,\,\,\,\,\,\,\,\,\,\,\,\,\,\,
		$m-n \ge 4$
	    & $ap^{m-n}$ & $\Right$ & $\star$ & $\star$\\
	    \rowcolor[gray]{.8}
		$ap^m$ \,\,\,\,\,\,\,\,\,\,\,\,\,\,\,\,\,\,\,\,\, 
		$m-n \le 3$
	    & $\star$ & $\star$ & 
	    $aq^u$ \,\,\,\,\,\,\,\,\,\,\,\,\,\,\,\,\,\,\,\,\,
	    such that $u \ge 4$
	    & $\Left$\\
	    $aq^m$&$\star$&$\star$&$aq^{m+t}$&$\Left$\\
	    \end{tabular}
	\end{center}
	\caption{Showing $q^n < q^t$ where $t-n\ge 7$ in $\monoid{M}_{\cl{\rho,\rhob}}$.}
	\label{table-qn<qn+7}
	\end{table}	
	
	These tables show that $q^n \le q^t$ if $t-n \ge 4$.  The inequality is strict as $q^n$, $q^t$ are distinguishable.  
\end{proof}

This shows us the results for elements $q^t$.   Elements $p^t$ have similar results.  

\begin{proposition}\label{prop-pn>pt-t-n=4}
	Take $n,t \in \mathbb{Z}^{\ge 0}$ where $t-n \ge 4$.  Then $p^n > p^t$.
\end{proposition}

\begin{proof}
	Proposition \ref{prop-qn<qt} gives that $q^n < q^t$ if $t-n \ge 4$.  That is, in the \mis monoid, $n \rhob < t \rhob$ for $t-n \ge 4$.  Since $\rho$ and $\rhob$ are conjugates, Proposition \ref{prop->-<-conjugate}\eqref{prop->-<-conjugate-<=} means that $n \rho > t \rho$ for $t-n \ge 4$.  Translating this back into monoid notation, this gives $p^n > p^t$ for $t-n \ge 4$.  
\end{proof}

There are still a few more comparability relations which must be noted.

\begin{proposition}\label{prop-monod-<-p-q}
	The following are comparability relations on $\monoid{M}_{\cl{\rho,\rhob}}$:
		\begin{align*}
		p^6 &<q & p^5 &<q   & p^4 &<q & p^3 &< q & p^2 &< q^2 & p &< q^3\\
		    &   & p^5 &<q^2 & p^4 &<q^2 & p^3 &< q^2 & p^2 &< q^3 & p &< q^4\\
		    &   &     &     &p^4 &<q^3 & p^3 &< q^3 & p^2 &< q^4 & p &< q^5\\
	            &   &     &     &    &     & p^3 &< q^4 & p^2 &< q^5 & p &< q^6.
		\end{align*}
\end{proposition}

\begin{proof}
	This proof is similar to that for Proposition \ref{prop-qn<qt}, where constructing a table and comparing the outcomes with the product of the six types of elements of $\monoid{M}_{\cl{\rho,\rhob}}$ will yield the result.
\end{proof}

We place all the information about incomparability and comparability into Figure \ref{fig-B-poset-1}.   

\begin{figure}[p]
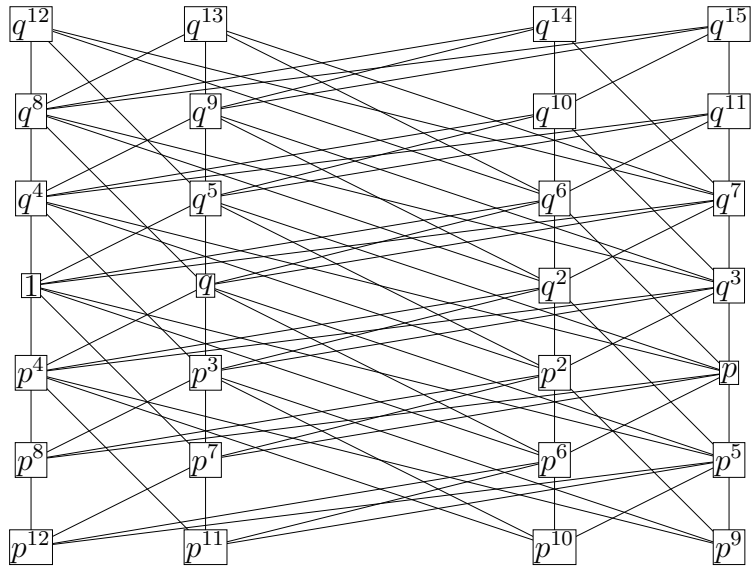

\begin{center}
		\unitlength 11pt
		\begin{center}
		\begin{graph}(27.5,21.5)(0,-1.5)
		\graphnodesize{0}
		\fillednodestrue

		\textnode{p12}(0,0){$p^{12}$}
		\textnode{p8}(0,3){$p^8$}
		\textnode{p4}(0,6){$p^4$}
		\textnode{1}(0,9){$1$}
		\textnode{q4}(0,12){$q^4$}
		\textnode{q8}(0,15){$q^8$}
		\textnode{q12}(0,18){$q^{12}$}

		\textnode{p11}(6,0){$p^{11}$}
		\textnode{p7}(6,3){$p^7$}
		\textnode{p3}(6,6){$p^3$}
		\textnode{q}(6,9){$q$}
		\textnode{q5}(6,12){$q^5$}
		\textnode{q9}(6,15){$q^9$}
		\textnode{q13}(6,18){$q^{13}$}
		
		\textnode{p10}(18,0){$p^{10}$}
		\textnode{p6}(18,3){$p^6$}
		\textnode{p2}(18,6){$p^2$}
		\textnode{q2}(18,9){$q^2$}
		\textnode{q6}(18,12){$q^6$}
		\textnode{q10}(18,15){$q^{10}$}
		\textnode{q14}(18,18){$q^{14}$}

		\textnode{p9}(24,0){$p^9$}
		\textnode{p5}(24,3){$p^5$}
		\textnode{p}(24,6){$p$}
		\roundnode{d2}(27.5,18)
		\textnode{q3}(24,9){$q^3$}
		\textnode{q7}(24,12){$q^7$}
		\textnode{q11}(24,15){$q^{11}$}
		\textnode{q15}(24,18){$q^{15}$}

		\edge{1}{q4}
		\edge{1}{q5}
		\edge{1}{q6}
		\edge{1}{q7}
		
		\edge{q}{p6}
		\edge{q}{q5}
		\edge{q}{q6}
		\edge{q}{q7}
		\edge{q}{q8}
		\edge{q}{p5}
		
		\edge{q2}{q6}
		\edge{q2}{q7}
		\edge{q2}{q8}
		\edge{q2}{q9}
		\edge{q2}{p5}

		\edge{q3}{q7}
		\edge{q3}{q8}
		\edge{q3}{q9}
		\edge{q3}{q10}

		\edge{q4}{q8}
		\edge{q4}{q9}
		\edge{q4}{q10}
		\edge{q4}{q11}
		
		\edge{q5}{q9}
		\edge{q5}{q10}
		\edge{q5}{q11}
		\edge{q5}{q12}

		\edge{q6}{q10}
		\edge{q6}{q11}
		\edge{q6}{q12}
		\edge{q6}{q13}		

		\edge{q7}{q11}
		\edge{q7}{q12}
		\edge{q7}{q13}
		\edge{q7}{q14}		

		\edge{q8}{q12}
		\edge{q8}{q13}
		\edge{q8}{q14}
		\edge{q8}{q15}		

		\edge{q9}{q13}
		\edge{q9}{q14}
		\edge{q9}{q15}

		\edge{q10}{q14}
		\edge{q10}{q15}

		\edge{q11}{q15}

		\edge{p}{p5}
		\edge{p}{p6}
		\edge{p}{p7}
		\edge{p}{p8}

		\edge{p2}{p6}
		\edge{p2}{p7}
		\edge{p2}{p8}
		\edge{p2}{p9}
		
		\edge{p3}{p7}
		\edge{p3}{p8}
		\edge{p3}{p9}
		\edge{p3}{p10}		

		\edge{p4}{p8}
		\edge{p4}{p9}
		\edge{p4}{p10}
		\edge{p4}{p11}

		\edge{p5}{p9}
		\edge{p5}{p10}
		\edge{p5}{p11}
		\edge{p5}{p12}
		\edge{p5}{1}

		\edge{p6}{p10}
		\edge{p6}{p11}
		\edge{p6}{p12}
		\edge{p6}{1}

		\edge{p7}{p11}
		\edge{p7}{p12}
		\edge{p7}{1}

		\edge{p8}{p12}

		\edge{p4}{1}
		\edge{p4}{q}
		\edge{p4}{q2}
		\edge{p4}{q3}
		
		\edge{p3}{q}
		\edge{p3}{q2}
		\edge{p3}{q3}
		\edge{p3}{q4}
		
		\edge{p2}{q2}
		\edge{p2}{q3}
		\edge{p2}{q4}
		\edge{p2}{q5}
		
		\edge{p}{q3}
		\edge{p}{q4}
		\edge{p}{q5}
		\edge{p}{q6}
	
		\end{graph}
		\end{center}

\end{center}
\caption{A snippet of the partially ordered set of elements of the form $p^m$ and $q^n$ in $\monoid{M}_{\cl{\rho,\rhob}}$.}
\label{fig-B-poset-1}
\end{figure}

\begin{figure}[p]
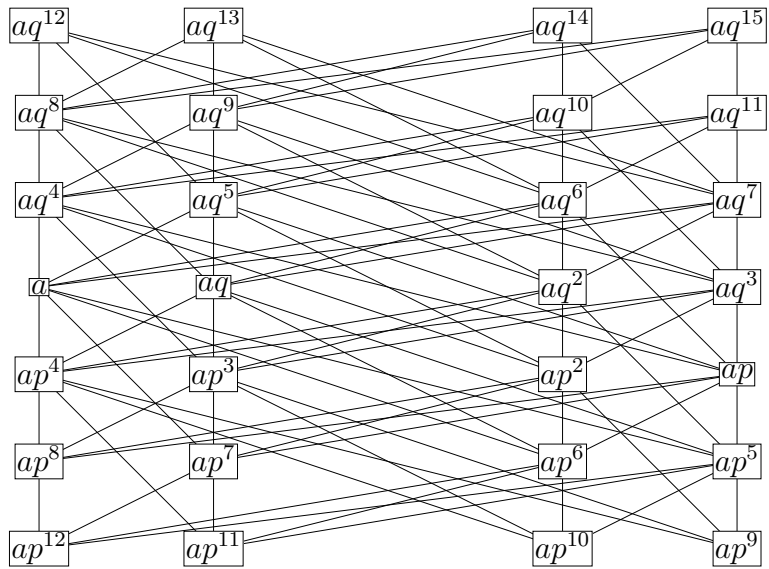

\begin{center}
		\unitlength 11pt
		\begin{center}
		\begin{graph}(27.5,21.5)(0,-1.5)
		\graphnodesize{0}
		\fillednodestrue

		\textnode{p12}(0,0){$ap^{12}$}
		\textnode{p8}(0,3){$ap^8$}
		\textnode{p4}(0,6){$ap^4$}
		\textnode{1}(0,9){$a$}
		\textnode{q4}(0,12){$aq^4$}
		\textnode{q8}(0,15){$aq^8$}
		\textnode{q12}(0,18){$aq^{12}$}

		\textnode{p11}(6,0){$ap^{11}$}
		\textnode{p7}(6,3){$ap^7$}
		\textnode{p3}(6,6){$ap^3$}
		\textnode{q}(6,9){$aq$}
		\textnode{q5}(6,12){$aq^5$}
		\textnode{q9}(6,15){$aq^9$}
		\textnode{q13}(6,18){$aq^{13}$}
		
		\textnode{p10}(18,0){$ap^{10}$}
		\textnode{p6}(18,3){$ap^6$}
		\textnode{p2}(18,6){$ap^2$}
		\textnode{q2}(18,9){$aq^2$}
		\textnode{q6}(18,12){$aq^6$}
		\textnode{q10}(18,15){$aq^{10}$}
		\textnode{q14}(18,18){$aq^{14}$}

		\textnode{p9}(24,0){$ap^9$}
		\textnode{p5}(24,3){$ap^5$}
		\textnode{p}(24,6){$ap$}
		\textnode{q3}(24,9){$aq^3$}
		\textnode{q7}(24,12){$aq^7$}
		\textnode{q11}(24,15){$aq^{11}$}
		\textnode{q15}(24,18){$aq^{15}$}
		\edge{1}{q4}
		\edge{1}{q5}
		\edge{1}{q6}
		\edge{1}{q7}
		
		\edge{q}{p6}
		\edge{q}{q5}
		\edge{q}{q6}
		\edge{q}{q7}
		\edge{q}{q8}
		\edge{q}{p5}
		
		\edge{q2}{q6}
		\edge{q2}{q7}
		\edge{q2}{q8}
		\edge{q2}{q9}
		\edge{q2}{p5}

		\edge{q3}{q7}
		\edge{q3}{q8}
		\edge{q3}{q9}
		\edge{q3}{q10}

		\edge{q4}{q8}
		\edge{q4}{q9}
		\edge{q4}{q10}
		\edge{q4}{q11}
		
		\edge{q5}{q9}
		\edge{q5}{q10}
		\edge{q5}{q11}
		\edge{q5}{q12}

		\edge{q6}{q10}
		\edge{q6}{q11}
		\edge{q6}{q12}
		\edge{q6}{q13}		

		\edge{q7}{q11}
		\edge{q7}{q12}
		\edge{q7}{q13}
		\edge{q7}{q14}		

		\edge{q8}{q12}
		\edge{q8}{q13}
		\edge{q8}{q14}
		\edge{q8}{q15}		

		\edge{q9}{q13}
		\edge{q9}{q14}
		\edge{q9}{q15}

		\edge{q10}{q14}
		\edge{q10}{q15}

		\edge{q11}{q15}

		\edge{p}{p5}
		\edge{p}{p6}
		\edge{p}{p7}
		\edge{p}{p8}

		\edge{p2}{p6}
		\edge{p2}{p7}
		\edge{p2}{p8}
		\edge{p2}{p9}
		
		\edge{p3}{p7}
		\edge{p3}{p8}
		\edge{p3}{p9}
		\edge{p3}{p10}		

		\edge{p4}{p8}
		\edge{p4}{p9}
		\edge{p4}{p10}
		\edge{p4}{p11}

		\edge{p5}{p9}
		\edge{p5}{p10}
		\edge{p5}{p11}
		\edge{p5}{p12}
		\edge{p5}{1}

		\edge{p6}{p10}
		\edge{p6}{p11}
		\edge{p6}{p12}
		\edge{p6}{1}

		\edge{p7}{p11}
		\edge{p7}{p12}
		\edge{p7}{1}

		\edge{p8}{p12}

		\edge{p4}{1}
		\edge{p4}{q}
		\edge{p4}{q2}
		\edge{p4}{q3}
		
		\edge{p3}{q}
		\edge{p3}{q2}
		\edge{p3}{q3}
		\edge{p3}{q4}
		
		\edge{p2}{q2}
		\edge{p2}{q3}
		\edge{p2}{q4}
		\edge{p2}{q5}
		
		\edge{p}{q3}
		\edge{p}{q4}
		\edge{p}{q5}
		\edge{p}{q6}
	
		\end{graph}
		\end{center}

\end{center}
\caption{A snippet of the partially ordered set of elements of the form $ap^m$ and $aq^n$ in $\monoid{M}_{\cl{\rho,\rhob}}$.}
\label{fig-B-poset-2}
\end{figure}

Compared to the partial order obtained for $\monoid{M}_{\cl{\rho}}$ (Figure \ref{fig-A-poset}), Figure \ref{fig-B-poset-1} looks a real mess.  To make matters worse, we haven't even begun to examine elements of the form $ap^m$ or $aq^n$!   Luckily, the structure of $\monoid{M}_{\cl{\rho,\rhob}}$ means we get some results virtually without effort.

\begin{proposition}\label{prop-x<y-ax<ay}
	Suppose $x < y$ for $x,y \in \{p^m, q^n \mid m,n \in \mathbb{Z}^{\ge 0}\}$.  Then $ax < ay$.  Similarly, if $x$ and $y$ are incomparable for $x,y \in \{p^m, q^n \mid m,n \in \mathbb{Z}^{\ge 0}\}$, then so are $ax$ and $ay$.  
\end{proposition}

\begin{proof}
	Suppose $x < y$ but $ax \not < ay$.  Therefore, either
		\begin{enumerate}
			\item $ax \ge ay$: Thus, for every element $z \in \monoid{M}_{\cl{\rho,\rhob}}$, $o^-(axz) \ge o^-(ayz)$.  In particular, $o^-(axa) \ge o^-(aya)$, or $o^-(x) \ge o^-(y)$, contradicting that $x<y$; or
			
			\item $ax$ and $ay$ are incomparable:  But $o^-(ax) \le o^-(ay)$ since $x<y$.  Thus, there must exist $z \in \monoid{M}_{\cl{\rho,\rhob}}$ such that $o^-(axz) > o^-(ayz)$.  That is, there exists $az \in \monoid{M}_{\cl{\rho,\rhob}}$ such that $o^-(x(az)) > o^-(y(az))$, which contradicts $x < y$.
		\end{enumerate}
	Thus $ax < ay$.  
	
	Now suppose $x$ and $y$ are incomparable.  Proposition \ref{prop-cl-rho-rhob-incomparable} gave all pairs $(x,y)$ which were incomparable, as well as element $e_{x,y}$ which showed the incomparability.  Then $ax$ and $ay$ are incomparable via the element $ae_{x,y}$.
\end{proof}

Proposition \ref{prop-x<y-ax<ay} means that we can take all the comparability and incomparability relations we calculated for $p^m$ and $q^n$ and use them to determine the comparability and incomparability relations for $ap^m$ and $aq^n$.  Figure \ref{fig-B-poset-2} draws some of these relationships in the same style as Figure \ref{fig-B-poset-1}.

How are Figure \ref{fig-B-poset-1} and Figure \ref{fig-B-poset-2} related?   By inspection, each element of $\monoid{M}_{\cl{\rho,\rhob}}$ is contained in precisely one of the two figures.  Is there an element $x$ of Figure \ref{fig-B-poset-1} and element $y$ of Figure \ref{fig-B-poset-2} such that $x$ and $y$ are comparable.  Unfortunately, yes.

\begin{proposition}\label{prop-ax-y-relations-B}
	Let $t,n \in \mathbb{Z}^{\ge 0}$.  The following comparability relations exist on $\monoid{M}_{\cl{\rho,\rhob}}$:
			
			\begin{center}
			\begin{tabular}{p{5cm}p{5cm}}
			\begin{enumerate}
				\item If $t-n \ge 6$, we have
				\begin{enumerate}
					\item $aq^n < q^t$;
					\item $q^n < aq^t$;
					\item $ap^n > p^t$;
					\item $p^n > ap^t$;
				\end{enumerate}
			\end{enumerate}
			\begin{enumerate}
			\setcounter{enumi}{2}
			\item If $n \ge 4$, we have
				\begin{enumerate}
					\item $ap^2 < q^n$;
					\item $p^2 < aq^n$;
					\item $aq^2 > p^n$;
					\item $q^2 > ap^n$;%
				\end{enumerate}%
			\end{enumerate}%
			&
			\begin{enumerate}
			\setcounter{enumi}{1}
				\item If $n \ge 5$, we have
				\begin{enumerate}
					\item $ap < q^n$;
					\item $p < aq^n$;
					\item $aq > p^n$;
					\item $q > ap^n$;
				\end{enumerate}
			\end{enumerate}
			\begin{enumerate}
			\setcounter{enumi}{3}
			\item If $3 \le n \le t$, we have
					\begin{enumerate}
						\item $ap^n < q^t$;
						\item $p^n < aq^t$;
						\item $aq^n > p^t$;
						\item $q^n > ap^t$.%
					\end{enumerate}%
			\end{enumerate}%
			\end{tabular}			
			\end{center}
	Otherwise, if $(ax,y)$ is a pair with
		\begin{align*}
			x, y &\in \{p^m, q^n \mid m,n \in \mathbb{Z}^{\ge 0}\},
		\end{align*}
	and $(ax,y)$ not listed in the four cases above, then $x$ and $y$ are incomparable.
\end{proposition}

\begin{proof}
	In each of the four cases listed above, similar to the proof of Proposition \ref{prop-pn>pt-t-n=4}, if (a) and (b) are true, Proposition \ref{prop->-<-conjugate}\eqref{prop->-<-conjugate-<=} forces (c) and (d) to be true since (c) and (d) are the respective statements of (a) and (b) under conjugation.  
	
	Suppose now that for each of the four cases listed above, (a) is true.  Similar to the proof that if $x<y$ then $ax < ay$ given in Proposition \ref{prop-x<y-ax<ay}, we have that if $ax <y$, then $x < ay$.  Thus, if (a) is true, this forces (b) to be true as well.
	
	Therefore, all that must be shown for comparability is case (a) in each of the four cases.  We construct four tables (Tables \ref{table-aqn<qn+6}, \ref{table-ap<qn-n>5}, \ref{table-ap2<qn-n>4}, and \ref{table-apn<qn-3<n<t}) to show these results.  As always, the strictness of the inequalities comes from the distinguishability of the elements.

	\begin{table}[p]
	\begin{center}
	    \begin{tabular}{p{2.5cm}llll}
	    \multicolumn{1}{c}{$\boldsymbol{e}$} & $\boldsymbol{aq^ne}$ & $\boldsymbol{o^-(aq^ne)}$ & $\boldsymbol{q^te}$ & $\boldsymbol{o^-(q^te)}$ 
	    \\ 
	    \rowcolor[gray]{.8}1 & $\star$ & $\star$ & $q^t$ & $\Left$ \\
	    $a$ & $\star$ & $\star$ & $aq^t$ & $\Left$ \\
	    \rowcolor[gray]{.8}
		$p^m$ \,\,\,\,\,\,\,\,\,\,\,\,\,\,\,\,\,\,\,\,\,
		$m \le t- 3$
	   & $\star$ & $\star$ & $q^{t-m}$ & $\Left$\\
	  	$p^m$ \,\,\,\,\,\,\,\,\,\,\,\,\,\,\,\,\,\,
		$m \ge t-2$ \,\,\,\,\,\,\,\,\,\,
		($m \ge n+4$)
	    & $ap^{m-n}$ & $\Right$ & 
	    $\star$
	    & $\star$\\
	    \rowcolor[gray]{.8}
	    $q^m$ & $\star$ & $\star$ & $q^{t+m}$ & $\Left$\\
		$ap^m$ \,\,\,\,\,\,\,\,\,\,\,\,\,\,\,\,\,\,
		$m \le t-4$
	    & $\star$ & $\star$ & $aq^{t-m}$ & $\Left$\\
	    \rowcolor[gray]{.8}
		$ap^m$\,\,\,\,\,\,\,\,\,\,\,\,\,\,\,\,\,\,
		$m \ge  t-3$ \,\,\,\,\,\,\,\,\,\,
		($m \ge n+3$)
	    & $p^{m-n}$ & $\Right$ & 
	    $\star$
	    & $\star$\\
	    $aq^m$&$\star$&$\star$&$aq^{t+m}$&$\Left$\\
	    
	    \end{tabular}
	\end{center}
	\caption{Showing $aq^n < q^t$ for $t-n \ge 6$.}
	\label{table-aqn<qn+6}
	\end{table}		

	\begin{table}[p]
	\begin{center}
	    \begin{tabular}{p{2.5cm}llll}
	    \multicolumn{1}{c}{$\boldsymbol{e}$} & $\boldsymbol{ape}$ & $\boldsymbol{o^-(ape)}$ & $\boldsymbol{q^ne}$ & $\boldsymbol{o^-(q^ne)}$ 
	    \\ 
	    \rowcolor[gray]{.8}1 & $\star$ & $\star$ & $q^n$ & $\Left$ \\
	    $a$ & $\star$ & $\star$ & $aq^n$ & $\Left$ \\
	    \rowcolor[gray]{.8}	$p^m$ \,\,\,\,\,\,\,\,\,\,\,\,\,\,\,\,\,\,\,\,\,
		$m \le n-3$ 
	     & $\star$ & $\star$ & $q^{n-m}$ & $\Left$\\	    
	    	$p^m$\,\,\,\,\,\,\,\,\,\,\,\,\,\,\,\,\,\,
		$m \ge n-2$\,\,\,\,\,\,\,\,\,\, 
		($m \ge 3$)
	    & $ap^{m+1}$ & $\Right$ & 
	    $\star$
	    & $\star$\\
	    \rowcolor[gray]{.8}$q^m$ & $\star$ & $\star$ & $q^{n+m}$ & $\Left$\\
	 	$ap^m$ \,\,\,\,\,\,\,\,\,\,\,\,\,\,\,\,\,\,
		$m \le n-4$
	    & $\star$ & $\star$ & $aq^{n-m}$ & $\Left$\\	    
	    \rowcolor[gray]{.8}	$ap^m$\,\,\,\,\,\,\,\,\,\,\,\,\,\,\,\,\,\,
		$m \ge n-3$\,\,\,\,\,\,\,\,\,\,
		($m \ge 2$)
	    & $p^{m+1}$ & $\Right$ & 
	    $\star$
	    & $\star$\\
	    $aq^m$&$\star$&$\star$&$aq^{n+m}$&$\Left$\\
	    \end{tabular}
	\end{center}
	\caption{Showing $ap < q^n$ for $n \ge 5$.}
	\label{table-ap<qn-n>5}
	\end{table}	

	\begin{table}[p]
	\begin{center}
	    \begin{tabular}{p{2.5cm}llll}
	    \multicolumn{1}{c}{$\boldsymbol{e}$}&$\boldsymbol{ap^2e}$&$\boldsymbol{o^-(ap^2e)}$&$\boldsymbol{q^ne}$&$\boldsymbol{o^-(q^ne)}$ 
	    \\
	    \rowcolor[gray]{.8}1 & $\star$ & $\star$ & $q^n$ & $\Left$ \\
	    $a$ & $\star$ & $\star$ & $aq^n$ & $\Left$ \\
		\rowcolor[gray]{.8}$p^m$\,\,\,\,\,\,\,\,\,\,\,\,\,\,\,\,\,\,\,\,\,
		$m \le n-3$
	    & $\star$ & $\star$ & $q^{n-m}$ & $\Left$\\
		$p^m$\,\,\,\,\,\,\,\,\,\,\,\,\,\,\,\,\,\,
		$m \ge n-2$\,\,\,\,\,\,\,\,\,\, 
		($m \ge 2$)
	    & $ap^{m+2}$ & $\Right$ & 
	    $\star$
	    & $\star$\\
	    \rowcolor[gray]{.8}$q^m$ & $\star$ & $\star$ & $q^{n+m}$ & $\Left$\\
		$ap^m$\,\,\,\,\,\,\,\,\,\,\,\,\,\,\,\,\,\,
		$m \le n-4$
	    & $\star$ & $\star$ & $aq^{n-m}$ & $\Left$\\
		\rowcolor[gray]{.8}$ap^m$\,\,\,\,\,\,\,\,\,\,\,\,\,\,\,\,\,\,
		$m \ge n-3$\,\,\,\,\,\,\,\,\,\,
		($m \ge 1$)
	    & $p^{m+2}$ & $\Right$ & 
	    $\star$
	    & $\star$\\
	    $aq^m$&$\star$&$\star$&$aq^{n+m}$&$\Left$\\
	    \end{tabular}
	\end{center}
	\caption{\,\,Showing $ap^2 < q^n$ for $n \ge 4$.}
	\label{table-ap2<qn-n>4}
	\end{table}
	
	\begin{table}[p]
	\begin{center}
	    \begin{tabular}{p{2.5cm}llll}
	    \multicolumn{1}{c}{$\boldsymbol{e}$}&$\boldsymbol{ap^ne}$&$\boldsymbol{o^-(ap^ne)}$&$\boldsymbol{q^te}$&$\boldsymbol{o^-(q^te)}$ 
	    \\ 
	    \rowcolor[gray]{.8}1 & $\star$ & $\star$ & $q^t$ & $\Left$ \\
	    $a$ & $p^n$ & $\Right$ & $\star$ & $\star$ \\
		\rowcolor[gray]{.8}$p^m$
	    & $ap^{n+m}$ & $\Right$ & $\star$ & $\star$\\    
		$q^m$
	    & $\star$ & $\star$ & 
	    $q^{t+m}$
	    & $\Left$\\
	    \rowcolor[gray]{.8}$ap^m$
	    & $p^{n+m}$ & $\Right$ & $\star$ & $\star$\\	    
	    $aq^m$&$\star$&$\star$&$aq^{t+m}$&$\Left$\\
	    \end{tabular}
	\end{center}
	\caption{\,\,Showing $ap^n < q^t$ for $3 \le n \le t$.}
	\label{table-apn<qn-3<n<t}
	\end{table}		

	What remains is to show that there are no other comparability relations for 
	$(ax,y)$ with
		\begin{align*}
			x, y &\in \{p^m, q^n \mid m,n \in \mathbb{Z}^{\ge 0}\}.
		\end{align*}
	
	Firstly, suppose that $ax$ and $y$ are incomparable.  We claim then that $x$ and $ay$ are incomparable.  Since $ax$ and $y$ are incomparable, either	
		\begin{enumerate}
			\item there exists $z \in \monoid{M}_{\cl{\rho,\rhob}}$ such that, without loss of generality, $o^-(axz) = \Next$ and $o^-(yz) = \Prev$.  Then $o^-(x(az)) = \Next$ while $o^-((ay)(az)) = \Prev$, and so $x$ and $ay$ are incomparable; or
			\item there exist $z_1, z_2 \in \monoid{M}_{\cl{\rho,\rhob}}$ such that, without loss of generality, $o^-(axz_1) \ge o^-(yz_1)$ while $o^-(axz_2) < o^-(yz_2)$.  Then $o^-(x(az_1)) \ge o^-((ay)(az_1))$ while $o^-(x(az_2)) < o^-((ay)(az_2))$, and so $x$ and $ay$ are incomparable.
		\end{enumerate}
	Thus $x$ and $ay$ are incomparable if $ax$ and $y$ are.
	
	Now suppose $(aq^n, p^m)$ are incomparable.  Then $(ap^n, q^m)$ are incomparable since $p$ and $q$ arise from conjugates in the indistinguishability quotient.    Similarly, if $(q^m, q^n)$ are incomparable, then so are $(p^m, p^n)$.
	
	Using these three facts, to check that there are no other comparability relations, it suffices to check that the following pairs of elements are incomparable:
		\begin{align*}
			(aq^n, q^n) && (aq^n, q^{n+1}) && (aq^n, q^{n+2}) && (aq^n, q^{n+3}) && (aq^n, q^{n+4}) && (aq^n, q^{n+5}).\\
			 && (ap,q) && (ap, q^2) && (ap, q^3) && (ap, q^4)\\
			 && (ap^2, q) && (ap^2,q^2) && (ap^2, q^3) \\
			 && (ap^3, q) && (ap^3, q^2) &&\\
			&& (ap^4,q).
		\end{align*}
	Table \ref{table-(ax,y)-incomparable-1} lists elements $(ax,y)$ such that there exists $e \in \monoid{M}_{\cl{\rho,\rhob}}$ such that $o^-(eax) = \Next$ while $o^-(ey) = \Prev$ (or vice versa).  Table \ref{table-(ax,y)-incomparable-2} lists elements $(ax,y)$ such that there exist $e_1, e_2 \in \monoid{M}_{\cl{\rho,\rhob}}$ such that $o^-(e_1ax) \ge o^-(e_2y)$ and $o^-(e_2ax) < o^-(e_2y)$ (or vice versa).
	
	\begin{table}[htb]
	\begin{center}
	\begin{tabular}{rrrrcrc}
	$\boldsymbol{ax}$&$\boldsymbol{y}$&$\boldsymbol{e}$&$\boldsymbol{axe}$&$\boldsymbol{o^-(axe)}$&$\boldsymbol{ye}$&$\boldsymbol{o^-(ye)}$\\
	\rowcolor[gray]{.8}
	$aq^n$&
	$q^n$&
	$p^n$
	&$a$&
	$\Prev$&
	$1$
	&$\Next$\\

	$aq^n$&
	$q^{n+1}$&
	$ap^{n+2}$
	&$p^2$&
	$\Prev$&
	$ap$
	&$\Next$\\
	\rowcolor[gray]{.8}
	$aq^n$&
	$q^{n+3}$&
	$p^{n+1}$
	&$ap$&
	$\Next$&
	$q^2$
	&$\Prev$\\

	$aq^n$&
	$q^{n+4}$&
	$p^{n+2}$
	&$ap^2$&
	$\Next$&
	$q^2$
	&$\Prev$\\
	\rowcolor[gray]{.8}
	$aq^n$&
	$q^{n+5}$&
	$p^{n+3}$
	&$ap^3$&
	$\Next$&
	$q^2$
	&$\Prev$\\
	
	$ap$&
	$q^2$&
	$ap$
	&$p^2$&
	$\Prev$&
	$aq$
	&$\Next$\\
	\rowcolor[gray]{.8}
	$ap$&
	$q^3$&
	$ap$
	&$p^2$&
	$\Prev$&
	$aq^2$
	&$\Next$\\

	$ap$&
	$q^4$&
	$ap$
	&$p^2$&
	$\Prev$&
	$aq^3$
	&$\Next$\\
	\rowcolor[gray]{.8}
	$ap^2$&
	$q$&
	$a$
	&$p^2$&
	$\Prev$&
	$aq$
	&$\Next$\\

	$ap^2$&
	$q^2$&
	$a$
	&$p^2$&
	$\Prev$&
	$aq^2$
	&$\Next$\\
	\rowcolor[gray]{.8}
	$ap^2$&
	$q^3$&
	$a$
	&$p^2$&
	$\Prev$&
	$aq^3$
	&$\Next$\\

	$ap^3$&
	$q$&
	$aq$
	&$p^2$&
	$\Prev$&
	$aq^2$
	&$\Next$\\
	\rowcolor[gray]{.8}
	$ap^3$&
	$q^2$&
	$aq$
	&$p^2$&
	$\Prev$&
	$aq^3$
	&$\Next$\\
	
	$ap^4$&
	$q$&
	$q$
	&$ap^3$&
	$\Next$&
	$q^2$
	&$\Prev$\\	
	\end{tabular}
	\caption{\, The incomparability of elements of the form $(ax,y)$ in $\monoid{M}_{\cl{\rho,\rhob}}$.}
	\label{table-(ax,y)-incomparable-1}
	\end{center}
	\end{table}

	\begin{table}[htb]
	\begin{center}
	\begin{tabular}{rrrrcrc}
	$\boldsymbol{ax}$&$\boldsymbol{y}$&$\boldsymbol{e_1}$&$\boldsymbol{axe_1}$&$\boldsymbol{o^-(axe_1)}$&$\boldsymbol{ye_1}$&$\boldsymbol{o^-(ye_1)}$\\
	\rowcolor[gray]{.8}
	$aq^n$&
	$q^{n+2}$&
	$p^{n+4}$
	&$ap^4$&
	$\Right$&
	$p^2$
	&$\Prev$
	\\

	$ap$&
	$q$&
	$aq^3$
	&$p^2$&
	$\Prev$&
	$aq^4$
	&$\Left$\\
	\end{tabular}
	
	\begin{tabular}{rrrrcrclllllllll}
	$\boldsymbol{ax}$&$\boldsymbol{y}$&$\boldsymbol{e_2}$&$\boldsymbol{axe_2}$&$\boldsymbol{o^-(axe_2)}$&$\boldsymbol{ye_2}$&$\boldsymbol{o^-(ye_2)}$\\
	\rowcolor[gray]{.8}
	$aq^n$&
	$q^{n+2}$&
	$ap^{n+1}$&
	$p$&
	$\Left$&
	$aq$&
	$\Next$
	\\

	$ap$&
	$q$&
	$a$&
	$p$&
	$\Left$&
	$aq$&
	$\Next$
	\\
	\end{tabular}
	\caption{\, The incomparability of elements of the form $(ax,y)$ in $\monoid{M}_{\cl{\rho,\rhob}}$ - Continued.}
	\label{table-(ax,y)-incomparable-2}
	\end{center}
	\end{table}	
	
	This completes the proof.
\end{proof}

This thesis will not attempt to draw the partial order of the elements of $\monoid{M}_{\cl{\rho,\rhob}}$.  However, even without a diagram, we are able to show the following Proposition.

\begin{proposition}
	Given any two elements $u,v \in \monoid{M}_{\cl{\rho,\rhob}}$, there exists $w,z \in \monoid{M}_{\cl{\rho,\rhob}}$ such that $w \le u,v$ and $z \ge u,v$.  That is, the set is both down-directed and up-directed.  However, the partial order of $\monoid{M}_{\cl{\rho,\rhob}}$ does not form a lattice.  
\end{proposition}

\begin{proof}
	We will show that given two elements of $\monoid{M}_{\cl{\rho,\rhob}}$, that we can find an element greater than both of these; the proof for finding an element less than both is similar.
	
	Take $x, y \in \{p^m, q^n \mid m,n \in \mathbb{Z}^{\ge 0}\}$.  If, without loss of generality, $x \le y$, then $y \ge x,y$.  Otherwise, suppose $x$ and $y$ are incomparable.  Then by Proposition \ref{prop-cl-rho-rhob-incomparable}, we have one of the following:
		\begin{enumerate}
			\item $(x,y) = (p,q)$: By Propositions \ref{prop-qn<qt} and \ref{prop-monod-<-p-q}, $p < q^n$ for all $n \ge 3$ while $q < q^t$ for all $t \ge 5$.  Therefore $p, q \le q^5$. 
			\item $(x,y) = (p,q^2)$: By inspection of Figure \ref{fig-B-poset-1}, we can see that $p, q^2 < q^7$.  
			\item\label{item-3-lattice} $(x,y) = (q^n, q^t)$ where $1 \le t-n \le 3$:  By Proposition \ref{prop-qn<qt}, $q^t < q^{t+8}$.   Moreover, $t+8 - n \ge 4$, so Proposition \ref{prop-qn<qt} also gives $q^n < q^{t+8}$.    
			\item $(x,y) = (p^n, p^t)$ where $1 \le t-n \le 3$:  By inspection of Figure \ref{fig-B-poset-1}, $p^m < q^4$ for any $m \in \mathbb{Z}^{\ge 0}$.  This gives the result.
		\end{enumerate}
		
	Similarly, if for pairs $(ax, ay)$, there exists $w \in \monoid{M}_{\cl{\rho,\rhob}}$ such that $ax, ay \le w$.  
	
	Now consider pairs of the form $(ax, y)$.   By Proposition \ref{prop-ax-y-relations-B}, there exists $w \in \{p^m, q^n \mid m,n \in \mathbb{Z}^{\ge 0}\}$ such that $ax < w$.   Consider $(w,y)$.  We have just shown that there exists $z \in \monoid{M}_{\cl{\rho,\rhob}}$ such that $w,y \le z$.  Thus $ax, y < z$, as required.
	
	To show that $\monoid{M}_{\cl{\rho,\rhob}}$ does not form a lattice, consider elements $p^9$, $p^{10}$.  By Proposition \ref{prop-pn>pt-t-n=4}, 
		\begin{align*}
			p^3 &> p^9, p^{10} \\
			p^4 &> p^9, p^{10},
		\end{align*}
	but, by Proposition \ref{prop-cl-rho-rhob-incomparable}, $p^3$ and $p^4$ are incomparable.  Therefore there is no least upper bound for $p^9$ and $p^{10}$.  
\end{proof}

Thus, we see that while the partially ordered set of $\monoid{M}_{\cl{\rho,\rhob}}$ does have some structure, it is not as nice as one might like it to be.  

\section{The \Mis Monoid of $\cl{\tau}$}

In this section, we calculate the \mis monoid of an impartial position, which we will denote as $\tau$.  This position is used extensively in Chapter \ref{chapter-cardinality-left}.  

\begin{definition}\label{def-tau}
	We define the position $\tau$ as $\tau = \combgame{\{*\mid*\}}$.  That is, $\tau$ is the position where both Left and Right have a move to $*$.  
\end{definition}

The game tree of $\tau$ is given in Figure \ref{fig-gt-tau}.
		
		\begin{figure}[htb]
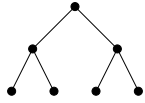

		\unitlength 8pt
		\begin{center}
		\begin{graph}(4,4)(-1,0)
		\graphnodesize{0.4}
		\fillednodestrue

		\roundnode{A}(1,4)
		\roundnode{B}(-1,2)
		\roundnode{C}(3,2)
		\roundnode{D}(-2,0)
		\roundnode{E}(0,0)
		\roundnode{F}(2,0)
		\roundnode{G}(4,0)

		\edge{A}{C}
		\edge{A}{B}
		\edge{B}{D}
		\edge{B}{E}
		\edge{C}{F}
		\edge{C}{G}
		
		\end{graph}
		\caption{The game tree of $\tau$.}
		\label{fig-gt-tau}
		\end{center}	
		\end{figure}

We can see that $\tau = \taub$.  

As with all our other examples, we begin by calculating the outcome classes of arbitrary positions in $\cl{\tau}$.

\begin{proposition}\label{prop-n*+m.tau}
	For $n,m \in \mathbb{Z}^{\ge 0}$, we have
		\[ o^-(n  * + m  \tau) = \begin{cases}
		\Next &\text{if } n \equiv 0 \imod 2;\\
		\Prev &\text{if } n \equiv 1 \imod 2.
		\end{cases}\]
\end{proposition}

\begin{proof}
	We proceed by induction on the options of a position.
	
	If $n=m=0$, then position is 0, which is an $\Next$ position, agreeing with the statement in the proposition.
	
	Consider position $n * + m \tau$ and suppose that the outcomes of all its options are as given in the statement of this proposition.  
	
	Suppose $n=0$.  Left and Right moving first can move to $* + (m-1) \tau$, which is a $\Prev$ position by induction.  Thus $o^-(m \tau) = \Next$.
	
	Suppose $n \equiv 0 \imod 2$ and $n > 0$.  Left and Right moving first can move to $(n-1) * + m \tau$, which is a $\Prev$ position by induction.
	
	Therefore $o^-(n * + m \tau) = \Next$ if $n \equiv 0 \imod 2$.
	
	Suppose $n \equiv 1 \imod 2$.  Left and Right moving first have two options:
		\begin{enumerate}
			\item $(n-1) * + m \tau$, which is an $\Next$ position by induction; or
			\item $(n+1) * + (m-1) \tau$, which is an $\Next$ position by induction.
		\end{enumerate}
	Therefore $o^-(n * + m \tau) = \Prev$ if $n \equiv 1 \imod 2$.
\end{proof}

By examining the results of Proposition \ref{prop-n*+m.tau}, the indistinguishability and distinguishability relationships can be determined.

\begin{corollary}
	The following indistinguishability relationships exist on $\cl{\tau}$:
		\begin{enumerate}
			\item $* + * \equiv 0 \imod{{\cl{\tau}}}$,
			\item $a  \tau \equiv 0 \imod{{\cl{\tau}}}$ for any $a \in \mathbb{Z}^{\ge 0}$.  
			\item $* + a  \tau \equiv * \imod{{\cl{\tau}}}$ for any $a \in \mathbb{Z}^{\ge 0}$.
		\end{enumerate}
\end{corollary}

\begin{proof}
	Take an arbitrary $n * + m \tau$.  
	\begin{enumerate}
		\item\label{item-*+*=0-1-2} By Proposition \ref{prop-n*+m.tau}, we have
			\[o^-((n+2) * + m \tau) = o^-(n * + m \tau),\]
		which gives the result.
		
		\item Since we just showed
			\[ * + * \equiv 0 \imod{{\cl{\tau}}},\]		
		we need only to consider $n=0$ or $n=1$ in our arbitrary position $n * + m \tau$. Then we have
		\begin{align*}
			o^-(a  \tau + (n  * + m  \tau)) &= o^-(n  * + (n+a)  \tau) \\
			&= \begin{cases}
				\Next &\text{if } n = 0;\\
				\Prev &\text{if } n = 1.
			\end{cases}
		\end{align*}
		But
		\[ o^-(0 + (n  * + m  \tau)) = \begin{cases}\
		\Next &\text{if } n \equiv 0 \imod 2;\\
		\Prev &\text{if } n \equiv 1 \imod 2.
		\end{cases}\]
		Therefore the two positions are indistinguishable.
		
		\item This follows from the previous case.\qedhere
	\end{enumerate}
\end{proof}

\begin{corollary}\label{cor-tau-finite}
	Given a position $n * + m \tau$, it is indistinguishable from exactly one of either $0$ or $*$.
\end{corollary}

We now explicitly write the \mis monoid. 
With the mappings:
	\begin{align*}
		0 &\mapsto 1;\\
		* &\mapsto a;\\
		\tau &\mapsto 1;\\
	\end{align*}
we obtain the following monoid:
	\begin{align*}
	\monoid{M}_{\cl{\tau}} &= \ideal{1,a \mid a^2=1} \\
	\Next &= \{1\} \\
	\Prev &= \{a\} \\
	\Left &= \emptyset \\
	\Right &= \emptyset
	\end{align*}
with the additive notation in $\cl{\tau}$ becoming a multiplicative notation in $\monoid{M}_{\cl{\tau}}$.  Again, as with $\monoid{M}_{\cl{\sigma, \sigmab}}$, this monoid is the same as $\monoid{M}_{\cl{*}}$.  Its inclusion is due to its extensive use in Chapter \ref{chapter-cardinality-left}.

Since the two elements of $\monoid{M}_{\cl{\tau}}$ have outcome classes $\Next$ and $\Prev$, they are incomparable.  Therefore, the partial order of $\monoid{M}_{\cl{\tau}}$ contains two incomparable elements.

\section{Conclusion}

After perusing these five examples, the reader should now be comfortable in the calculations required to determine \mis monoids. These initial examples were chosen as they demonstrate the wide range of structure that arise with partizan \mis monoids.  Now, confident in our abilities to calculate partizan \mis monoids, we proceed to more theoretical matters.

\chapter{The Cardinality of Partizan \Mis Monoids}\label{chapter-cardinality-left}

\section{Introduction}

Via the efforts of Plambeck and Siegel, much is known about the \mis monoids of impartial games, especially for octal and other heap based games \cite{ADVANCES, TAMING, MISQUOTIENT, MCF, NOTES, STRUCTURE}.  In examining the \mis monoids of impartial game positions, some quotients are finite while others are infinite.  It is not well-understood why certain positions yield finite quotients while others yield infinite ones and no investigation has yet occurred regarding partizan \mis monoids.  This chapter examines certain partizan positions and whether their \mis monoids are finite or infinite.   It succeeds in the following:
	\begin{itemize}
		\item finding a set of positions such that the inclusion of these positions yields an infinite monoid; and
		\item finding an infinite set of positions whose monoids are always finite.
	\end{itemize}

\section{Closures of Binary Positions of Birthday Two or Less and their Conjugates Under Indistinguishability}

\begin{definition}
	A position is a \textbf{binary position}\index{binary position}, if at any point, a player has at most one move available.   Equivalently, a position is a \textbf{binary position} if the game tree representing that position is a binary tree. 
\end{definition}

\begin{example}
	Some examples of binary positions are 0, $*$, $\sigma$, $\rho$, $\rhob$, and $\tau$ (see Chapter \ref{chapter-examples} or Appendix \ref{appendixA} for definitions of these positions).
\end{example}

Let $\chi$ be a binary position of birthday two or less.  The following theorem completely characterizes when $\monoid{M}_{\cl{\chi, \overline{\chi}}}$ is infinite.

\begin{theorem}\label{theorem-4-elements-infinite}
	Let $\chi$ be a  binary position of birthday two or less.  Then $\monoid{M}_{\cl{\chi, \overline{\chi}}}$ is infinite if and only if at least one of the following positions is an element of $\monoid{M}_{\cl{\chi, \overline{\chi}}}$: $1$, $\bar{1}$, $\rho$, or $\rhob$.
\end{theorem}

\begin{proof}
	Suppose $\chi = 0$.  Then $|\monoid{M}_{\cl{0}}| = 1$.  

	Suppose $\chi$ is a binary position with birthday one.  Then $\chi$ is either $1$, $\bar{1}$, or $*$.  Of these, $\monoid{M}_{\cl{1, \bar{1}}}$ is infinite by Corollary \ref{cor-1,bar1-infinite}, while $\monoid{M}_{\cl{*}}$ has cardinality 2 by Example \ref{example-*}.  
	
	Suppose $\chi$ is a binary position with birthday two such that $1, \bar{1} \not \in \cl{\chi}$.  Then, up to symmetry, $\chi$ is one of the following:
		\unitlength 8pt
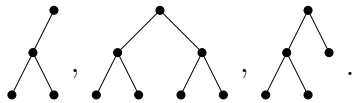
		\begin{center}
		\begin{graph}(16,4)(0,0)
		\graphnodesize{0.4}
		\fillednodestrue

		\roundnode{e1}(2,4)
		\roundnode{e2}(1,2)
		\roundnode{e3}(0,0)
		\roundnode{e4}(2,0)
		
		\edge{e1}{e2}
		\edge{e2}{e3}
		\edge{e2}{e4}
		
		\freetext(3,1){,}
		
		\roundnode{f1}(7,4)
		\roundnode{f2}(9,2)
		\roundnode{f3}(5,2)
		\roundnode{f4}(4,0)
		\roundnode{f5}(6,0)
		\roundnode{f6}(8,0)
		\roundnode{f7}(10,0)
		
		\edge{f1}{f3}
		\edge{f1}{f2}
		\edge{f3}{f4}
		\edge{f3}{f5}
		\edge{f2}{f6}
		\edge{f2}{f7}
		
		\freetext(11,1){,}
		
		\roundnode{g1}(14,4)
		\roundnode{g2}(13,2)
		\roundnode{g3}(12,0)
		\roundnode{g4}(14,0)
		\roundnode{g5}(15,2)
		
		\edge{g1}{g2}
		\edge{g2}{g3}
		\edge{g2}{g4}
		\edge{g1}{g5}
		
		\freetext(16,1){.}
		\end{graph}
		\end{center}		
	The first position is $\sigma$.  By Corollary \ref{cor-sigma-finite}, $\monoid{M}_{\cl{\sigma, \sigmab}}$ has only two elements, so is finite.
	
	The second position is $\tau$.  By Corollary \ref{cor-tau-finite}, $\monoid{M}_{\cl{\tau,\taub}}$ has only two elements, so is finite.

	The third position is $\rho$.  By Corollary \ref{cor-B-infinite}, $\monoid{M}_{\cl{\rho,\rhob}}$ is infinite.
\end{proof}

\begin{corollary}\label{cor-1-rho-infinite}
	Let $\xi$ be any position (not necessarily binary).  If $1, \bar{1}, \rho$, or $\rhob \in \cl{\xi}$, then $\cl{\xi, \xib}$ under indistinguishability is infinite.
\end{corollary}

\begin{proof}
	Suppose, for example, $1 \in \cl{\xi}$.  Then $1, \bar{1} \in \cl{\xi, \xib}$, and $\cl{1, \bar{1}} \subseteq \cl{\xi, \xib}$.  By Corollary \ref{cor-a1-b1-distinguishable}, $a1$ and $b1$ are distinguishable in $\cl{1, \bar{1}}$ for all $a \not = b$.  By Proposition \ref{proposition-distinguishable-subseteq}, $a1$ and $b1$ remain distinguishable in $\cl{\xi, \xib}$.   Thus, $\cl{\xi,\xib}$ contains an infinite number of distinguishable elements, and so the \mis monoid of $\cl{\xi,\xib}$ is infinite.
	
	The same argument works for $\rho$ or $\rhob$ being elements of $\cl{\xi}$. 
\end{proof}

Thus, given certain knowledge about the closure of a position, one can begin to determine whether the closure of that position with its conjugate is infinite.  However, are there elements other than $1, \bar{1}$, $\rho$, and $\rhob$ which give the result?   Unfortunately, as will be shown shortly, the answer is yes, meaning that while Corollary \ref{cor-1-rho-infinite} is a sufficient condition for an infinite \mis monoid, it is not necessary;  there are other binary positions $\alpha$ such that $\alpha \in \cl{\xi}$ implies the \mis monoid of $\cl{\xi}$ is infinite.  One such $\alpha$ will be shown in Section \ref{sec-closure-of-xi_L}.  

\section{The Closure of $\L(\xi)$}\label{sec-closure-of-xi_L}

We first define the following:

\begin{definition}\label{def-L(xi)}
	Suppose $\xi$ is a position.  Then $\L(\xi) = \combgame{\{\xi\mid \cdot\}}$\index{$\L(\xi)$}.  That is, $\L(\xi)$ is the position where Left's only move is to $\xi$ while Right has no move.  Similarly, $\R(\xi) = \{ \cdot \mid \xi\}$.  
\end{definition}

We examine $\L(\xi)$ because it is a position which is relatively similar to $\xi$ and a position which is fairly easy to analyse if $\monoid{M}_{\cl{\xi}}$ has already been determined.  

\begin{example}\label{example-L*=sigma}
	$\L(*) = \sigma$ (Definition \ref{def-sigma}).  
	
		\begin{figure}[htb]
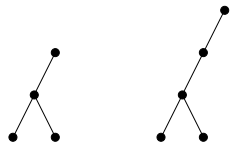

		\unitlength 8pt
		\begin{center}
		\begin{graph}(12,6)(-1,0)
		\graphnodesize{0.4}
		\fillednodestrue

		\roundnode{A}(1,4)
		\roundnode{B}(0,2)
		\roundnode{D}(-1,0)
		\roundnode{E}(1,0)

		\edge{A}{B}
		\edge{B}{D}
		\edge{B}{E}

		\roundnode{A1}(8,4)
		\roundnode{B1}(7,2)
		\roundnode{C1}(9,6)
		\roundnode{D1}(6,0)
		\roundnode{E1}(8,0)

		\edge{A1}{B1}
		\edge{B1}{D1}
		\edge{B1}{E1}
		\edge{A1}{C1}

		\end{graph}
		\caption{The game trees of $\sigma$ and $\L(\sigma)$.}
		\label{fig-gt-sigma-L}
		\end{center}	
		\end{figure}

		Modifying Proposition \ref{prop-n*+m-sigma+m-sigma_L} gives
			\[ o^-(n* + m \sigma) = \begin{cases}
			\Next &\text{if } n \equiv 0 \imod{2};\\
			\Prev &\text{if } n \equiv 1 \imod 2.
			\end{cases}\]
		Moreover, $\sigma \equiv 0 \imod{\cl{\sigma}}$, since, for arbitrary $n * + m \sigma$,
			\begin{align*}
				o^-((n * + m \sigma) + \sigma) &= o^-(n * 	+ (m+1) \sigma) \\&=\begin{cases}
					\Next &\text{if } n \equiv 0 \imod{2};\\
					\Prev &\text{if } n \equiv 1 \imod 2.
					\end{cases}\\
				&= o^-(n * + m \sigma).		
			\end{align*}
		Thus, via the mapping
			\begin{align*}
				0 &\mapsto 1;\\
				* &\mapsto a;\\
				\sigma &\mapsto 1;
			\end{align*}
		the following monoid is obtained
			\begin{align*}
			\monoid{M}_{\cl{\sigma}} &= \ideal{1,a \mid a^2=1} \\
			\Next &= \{1\} \\
			\Prev &= \{a\} \\
			\Left &= \emptyset \\
			\Right &= \emptyset,
			\end{align*}
		which is isomorphic to $\monoid{M}_{\cl{*}}$.   That is, $\monoid{M}_{\cl{*}} \cong \monoid{M}_{\cl{\L(*)}}$, and both are finite. 
\end{example}

Example \ref{example-L*=sigma} showed that with $\sigma$,  $\monoid{M}_{\cl{\sigma}}$ and $\monoid{M}_{\cl{\L(\sigma)}}$ are the same.  However, this may not always be the case.  In fact, the following Proposition shows that even if $\monoid{M}_{\cl{\xi}}$ is finite, $\monoid{M}_{\cl{\L(\xi)}}$ may be infinite.

\begin{proposition}\label{prop-sigma_L-infinite}
	 $\monoid{M}_{\cl{\L(\sigma)}}$ is infinite.   
\end{proposition}

\begin{proof}
	We will show the following: For $c \ge 1$,
			\[ o^-(a  * + b  \sigma + c  \L(\sigma)) = \begin{cases}
				\Right &\text{if } a < c; \\
				\Next &\text{if } a \ge c,\text{ } a \equiv c \imod 2;\\
				\Prev &\text{if } a \ge c,\text{ } a \not \equiv c \imod 2.
			\end{cases}
			\]
	Assuming the outcome statement is true and we are given $c  \L(\sigma)$, $c^{\prime}  \L(\sigma)$ for $c, c^{\prime} \in \nat$ with $c \not = c^{\prime}$.  Moreover, without loss of generality, we may assume that $c < c^{\prime}$.  Then $o^-(c  * + c  \L(\sigma)) = \Next$ while $o^-(c  * + c^{\prime}  \L(\sigma)) = \Right$.  Therefore the two positions are distinguished by $c  *$ and so $\monoid{M}_{\cl{\L(\sigma)}}$ is infinite. 

	It remains to show the outcome statement.  We proceed by induction on the options.

	\begin{notation}
		Similar to Theorem \ref{theorem-B-oc}, let $(a,b,c)$ denote the position $a  * + b  \sigma + c  \L(\sigma)$.  
	\end{notation}
			
	Consider the position 0.  Then $a=b=c=0$, so the outcome statement gives $o^-((0,0,0)) = \Next$, which is true.
	
	Consider the position $(a,b,c)$ and suppose the outcome statement is true for all options.  Then $(a,b,c)$ must fall into one of the three cases given in the outcome statement above.
	\begin{enumerate}
		\item $a<c$:  Suppose Left is moving first.  She has three possible moves:
			\begin{enumerate}
				\item $(a-1, b, c)$:  Since $a<c$, we have $a-1 < c$.  By induction, $o^-((a-1, b, c)) = \Right$.
				
				\item $(a+1, b-1, c)$:  If $a+1 < c$, by induction we have $o^-((a+1,b-1,c)) = \Right$.  Otherwise $a+1 \ge c$.  Since $a < c$, this implies that $c= a+1$, and so, Left has moved to position $(c, b-1, c)$.  By induction $o^-((c,b-1,c)) = \Next$.
				
				\item $(a, b+1, c-1)$:  If $a < c-1$, by induction we have $o^-(a,b+1,c-1)) = \Right$.  Otherwise $a \ge c-1$.  Since $a< c$, this implies that $a = c-1$, and so, Left has moved to position $o^-((a,b+1,a))$, which is an $\Next$ position by induction.
			\end{enumerate}
		Therefore Left loses moving first.
		
		Suppose Right moves first.  Since $a<c$, this means that $c \ge 1$.  If $a=0$, $c \ge 1$, then Right has no moves available and so wins moving first.  Otherwise $a \ge 1$ also.  Then Right moving to $(a-1,b,c)$ is a winning move as $a<c$ implies $a-1 < c$.  Therefore Right wins moving first.
		
		Thus if $a<c$, $o^-((a,b,c)) = \Right$.
		
		\item $a \ge c$, $a \equiv c \imod 2$:  Suppose Left is moving first.  If $a=c=0$, either $b=0$, and so Left wins moving next, or $b \ge 1$.  Then Left moves to $(1,b-1,0)$, which is a $\Prev$ position by induction.  Otherwise $a \ge 1$.  Suppose Left moves to $(a-1,b,c)$.  Then, by induction, $o^-((a-1,b,c)) = \Prev$ unless $a-1<c$.  Suppose $a-1<c$.  Since $a \ge c$, this means $a = c$ and the initial position is $(a,b,a)$.  Since $a \ge 1$, Left can move to $(a, b+1, a-1)$, which is a $\Prev$ position by induction.  Therefore Left can win $(a,b,c)$ moving first.
		
		Suppose Right is moving first.  If $a=c=0$, then Right wins as he has no moves available.  Thus, we may assume $a \ge 1$.  Again, if Right moves to $(a-1,b,c)$, then this is a $\Prev$ position by induction unless $a-1 <c$, in which case it is an $\Right$ position.  Therefore Right can win $(a,b,c)$ moving first.
		
		Thus is $a \ge c$ and $a \equiv c \imod 2$, then $o^-((a,b,c)) = \Next$.
		
		\item $a \ge c$, $a \not \equiv c \imod 2$:  Suppose Left is moving first.  She has three possible moves:
			\begin{enumerate}
				\item $(a-1,b,c)$:  Since $a \not \equiv c \imod 2$, we have $a-1 \equiv c \imod 2$, and so, by induction, $o^-((a-1,b,c)) = \Next$ unless $a=c$, in which case $o^-((a-1,b,c)) = \Right$ by induction.
				
				\item $(a+1,b-1,c)$:  Since $a \ge c$, we have $a+1 \ge c$.  Since $a \not \equiv c \imod 2$, we have $a+1 \equiv c \imod 2$.  Therefore, by induction, $o^-((a+1,b,c)) = \Next$.
				
				\item $(a, b+1, c-1)$:  Since $a \ge c$, we have $a \ge c-1$.  Since $a \not \equiv c \imod 2$, we have $a \equiv c-1 \imod 2$.  Therefore, by induction, $o^-((a, b+1,c-1)) = \Next$.
			\end{enumerate}
		Therefore Left loses moving first.
		
		Suppose Right is moving first.  Unless $a=0$, Right has only one move available, to $(a-1,b,c)$.  Then, by induction, $o^-((a-1,b,c)) = \Next$ unless $a-1 < c$.  However, since $a \ge c$, this implies $a=c$, contradicting $a \not \equiv c \imod 2$.  Therefore $o^-((a-1,b,c)) = \Next$.  
		
		If $a=0$, then $c=0$, contradicting $a \not \equiv c \imod 2$.  Therefore Right loses moving first.
		
		Thus if $a \ge c$, $a \not \equiv c \imod 2$, we have $o^-((a,b,c)) = \Prev$.
	\end{enumerate}	
	
	This completes the induction.
\end{proof}

\begin{corollary}
	If $\L(\sigma)$ or  $\R(\sigma)$ is an element of $\cl{\xi}$ for some position $\xi$, then $\monoid{M}_{\cl{\xi}}$ is infinite.  
\end{corollary}

We can now extend our list of positions which give an infinite monoid.

\begin{corollary}\label{cor-cl-xi-infinite}
	If the following are elements of $\cl{\xi}$ for some $\xi$, then $\monoid{M}_{\cl{\xi}}$ is infinite:
		\begin{itemize}
			\item $1$ and $\overline{1}$,
			\item $\rho$ and $\overline{\rho}$,
			\item $\L(\sigma)$,
			\item $\R(\sigma)$.
		\end{itemize}
\end{corollary}

We have the following open problem:

\begin{openproblem}\label{op-chi-finite}
	Classify those positions $\chi$ where $\monoid{M}_{\cl{\chi}}$ is a finite/infinite monoid.
\end{openproblem}

\section{The Closure of $\L(\tau^n)$}\label{section-cl{tau_L}}

Proposition \ref{prop-sigma_L-infinite} showed that there exist positions $\xi$ where $\monoid{M}_{\cl{\xi}}$ is finite while $\monoid{M}_{\cl{\L(\xi)}}$ is infinite.  This section gives a set of positions $\chi$ such that $\monoid{M}_{\cl{\L(\chi)}}$ is finite (note that if $\monoid{M}_{\cl{\L(\chi)}}$ is finite then $\monoid{M}_{\cl{\chi}}$ must also be finite).  

Recall from earlier that $\tau$ is the position in which both Right and Left have one move to $*$.  

\begin{definition}\label{def-tau^n}
	$\tau^n$ is the position in which both Left and Right have one move each to the position $\tau^{n-1}$ with $\tau^0 = *$. 
\end{definition}

\begin{example}\text{}
	\begin{align*}
		\tau^0 &= *,\\
		\tau^1 &= \combgame{\{*\mid*\}},\\
		\tau^2 &= \left\{\combgame{\{*\mid*\}}\mid\combgame{\{*\mid*\}}\right\}.
	\end{align*}
\end{example}

We can easily determine the outcome classes of arbitrary sums of $\tau^n$:

\begin{proposition}\label{prop-tau^i-oc}
	\[ o^-\left(\sum_{i=0}^n a_i  \tau^i\right) = \begin{cases}
	\Next &\text{if } \displaystyle	\left(\sum_{\substack{i=0,\\i \equiv 0 \imod 2}}^n a_i \right) \equiv 0 \imod 2;\vspace{0.2in}\\
	\Prev &\text{if } \displaystyle	\left( \sum_{\substack{i=0,\\i \equiv 0 \imod 2}}^n a_i \right) \equiv 1 \imod 2.
	\end{cases}\]
	That is, a position is a $\Prev$ position if the sum of the co-ordinates with even indices is odd.  Otherwise it is a $\Next$ position.
\end{proposition}

\begin{proof}
	We proceed by induction on the options.
	
	Suppose the position is 0, i.e.\ $a_i = 0$ for every $i$.  Hence
		\[ \sum_{\substack{i=0,\\i \equiv 0 \imod 2}}^n a_i \equiv 0 \imod 2,\]
	and so the proposition claims that $o^-(0) = \Next$.  This is indeed the case, and so the base case holds.
	
	Suppose now we have the position 
		\[\sum_{i=0}^n a_i \tau^i\]
	where there exists $k \in \{1,2,\ldots,n\}$ such that $a_k \not =0$ and that the outcome of all options of this position is as given in the statement of the proposition.   Take $a_k \not = 0$ and move  $a_k \tau^k$ to $\tau^{k-1} + (a_k-1) \tau^k$.   Making this move means that the sum of the even indexed indices changes parity, i.e.\ if it was 1 before, it is now 0, and vice versa.  Thus, by induction, the first player moves to a $\Prev$ position if the initial sum of the even numbered indices was 0 or to an $\Next$ position if the initial sum of the even numbered indices was 1.  This gives the desired result.
\end{proof}

Note that $\tau^n$ is an impartial position for any $n \in \mathbb{Z}^{\ge 0}$.  Because of this, we could have also shown the result for Proposition \ref{prop-tau^i-oc} using the traditional method for examining impartial \mis play games, namely Genus Theory (for a full explanation of Genus Theory see \cite{MTHESIS, WW2, ONAG}).  

We will now begin examining elements found in $\cl{\L(\tau^k)}$ for some $k \in \nat$.  

\begin{theorem}\label{theorem-2n+1-tauL}
	\[o^-\left(\sum_{i=0}^{2n+1} a_i  \tau^i + b  (\L(\tau^{2n+1}))\right) = \begin{cases}
	\Right &\text{if } a_i = 0 \text{ } \forall \, i, \text{ } b \equiv 1 \imod 2;\\
	\Next &\text{if } \displaystyle \left[ \sum_{\substack{i=0,\\i \equiv 0 \imod 2}}^{2n+1} a_i\right] + b \equiv 0 \imod 2;\vspace{0.2in}\\
	\Prev &\text{if } \displaystyle  \left[\sum_{\substack{i=0,\\i \equiv 0 \imod 2}}^{2n+1} a_i\right] + b \equiv 1 \imod 2 \text{ and } \exists \, a_i \not = 0;
	\end{cases}\]
\end{theorem}

\begin{proof}
	Proceed by induction on the options of a position.  
	
	Consider 0.  Then $a_i = 0$ for all $i$ and $b=0$.  Therefore
		 \[ \left[\sum_{\substack{i=0,\\i \equiv 0 \imod 2}}^{2n+1} a_i\right] + b \equiv 0 \imod 2,\]
	and so the outcome statement gives $o^-(0) = \Next$, which is true.  
	
	Consider a position 
		\[\sum_{i=1}^{2n+1} a_i \tau^i + b (\L(\tau^{2n+1}))\]
	such that 
		\[\sum_{i=1}^{2n+1} a_i + b \not = 0\] 
	and the induction hypothesis holds for all options.  Then this position must fall into one of the three cases given in the theorem.
		\begin{enumerate}
			\item $a_i = 0$ for all $i$, $b \equiv 1 \imod 2$:  Suppose Left moves first.  She only has one available move, to $\tau^{2n+1} + (b-1)(\L(\tau^{2n+1}))$.  In this new position, the sum of the even numbered indices is 0.  Since $b \equiv 1 \imod 2$, we have $b-1 \equiv 0 \imod 2$.  Thus, by induction, 
				\[o^-(\tau^{2n+1} + (b-1)(\L(\tau^{2n+1})) = \Next.\]
			Therefore Left loses moving first.
			
			If Right moves first in $b (\L(\tau^{2n+1}))$, then Right has no moves available, so Right wins.
			
			Therefore $o^-(b(\L(\tau^{2n+1}))) = \Right$.
			
			\item $\displaystyle \left[ \sum_{\substack{i=0,\\i \equiv 0 \imod 2}}^{2n+1} a_i\right] + b \equiv 0 \imod 2$:  Firstly, suppose there exists $k \in \{1,2,\ldots,n\}$ such that $a_k \not = 0$.  Then the next player to move can move $a_k \tau^k$ to $\tau^{k-1} + (a_k-1) \tau^k$.  This changes the parity of the sum of $b$ and the even numbered indices from 0 to 1.  Thus, by induction, the next player has moved to a $\Prev$ position.  Therefore, the initial position is  $\Next$.
			
			Suppose now that $a_i = 0$ for all $i \in \{1,2,\ldots,n\}$.  Since $\displaystyle \left[ \sum_{\substack{i=0,\\i \equiv 0 \imod 2}}^{2n+1} a_i\right] + b \equiv 0 \imod 2$, this means that $b \equiv 0 \imod 2$.  Suppose Left moves first.  Left's only available move is to $\tau^{2n+1} + (b-1)(\L(\tau^{2n+1}))$.  In this new position, the sum of $b-1$ and the even numbered indices is equivalent to $1 \imod 2$.  By induction, this new position is a $\Prev$ position, so Left wins moving first.  If Right moves first in the initial position, Right has no available moves, so he wins.  Therefore the initial position is $\Next$.
			
			Thus, in both cases, we get that the initial position is an $\Next$ position.
			
			\item $\displaystyle \left[\sum_{\substack{i=0,\\i \equiv 0 \imod 2}}^{2n+1} a_i\right] + b \equiv 1 \imod 2$ and there exists an $a_k \not = 0$:   Take $k \in \{1,2,\ldots,n\}$ such that $a_k \not = 0$.  If the next player to move moves $a_k \tau^k$ to $\tau^{k-1} + (a_k-1) \tau^k$, then the parity of the sum of $b$ and the even numbered indices changes from 1 to 0, so, by induction, the new position is an $\Next$ position, and so the first player to move loses.
			
			Suppose instead that the first player to move plays in $b (\L(\tau^{2n+1}))$.  Right has no moves available, so we need only concern ourselves with Left.  Suppose Left moves to $\tau^{2n+1} + (b-1) (\L(\tau^{2n+1}))$.   Then the sum of $b-1$ and the even numbered indices in this new position is equivalent to $0 \imod 2$, and so, by induction, Left has moved to an $\Next$ position.  
			
			Thus, any initial move is bad for the first player, so the initial position is a $\Prev$ position.
		\end{enumerate}
	This completes the induction and the proof.
\end{proof}

We have a comparable theorem for when there is an even number of $\tau^i$'s:

\begin{theorem}\label{theorem-tau-2n}
		\[ o^-\left(\sum_{i=0}^{2n} a_i  \tau^i  + b  (\L(\tau^{2n}))\right) = \begin{cases}
		\Next &\text{if }\displaystyle  \left(\sum_{\substack{i=0,\\i \equiv 0 \imod 2}}^{2n} a_i \right) \equiv 0 \imod 2;\vspace{0.2in} \\
		\Prev &\text{if }\displaystyle \left( \sum_{\substack{i=0,\\i \equiv 0 \imod 2}}^{2n} a_i \right) \equiv 1 \imod 2.
		\end{cases}\]
\end{theorem}

\begin{proof}
	The proof of this theorem is virtually identical to that of Theorem \ref{theorem-2n+1-tauL}.  	
\end{proof}

It is interesting to note that in Theorem \ref{theorem-tau-2n}, the number of $\L(\tau^{2n})$'s does not in any way affect the outcome.  

Now that the outcomes of all positions in $\cl{\L(\tau^m)}$ have been determined, it is possible to examine the indistinguishability relations.

\begin{theorem}\label{theorem-indistinguishability-2n+1-tauL}
	The \mis monoid of $\cl{\L(\tau^{2n+1})}$ has four elements, which correspond to the equivalence classes of $0$, $*$, $* + *$, and $\L(\tau^{2n+1})$.
\end{theorem}

\begin{proof}
	Firstly, we will show that all elements in $\cl{\L(\tau^{2n+1})}$ whose outcome classes are $\Right$ are all indistinguishable.  That is, the first thing we will show is 	
		\[\L(\tau^{2n+1}) \equiv m  \L(\tau^{2n+1}) \imod{{\cl{\L(\tau^{2n+1})}}} \text{ for } m \equiv 1 \imod 2.  \]
	Take an arbitrary element of the closure,
		\[ \sum_{i=0}^{2n+1} a_i  \tau^i + b  \L(\tau^{2n+1}) \]
	and consider
		\begin{align*}
			&o^-\left(\left[\sum_{i=0}^{2n+1} a_i  \tau^i + b  \L(\tau^{2n+1}) \right] + \L(\tau^{2n+1})\right)\\
			&= o^-\left(\sum_{i=0}^{2n+1} a_i  \tau^i + (b+1)  \L(\tau^{2n+1})\right) \\
			&= \begin{cases}
				\Right &\text{if } a_i = 0, \text{ } (b+1) \equiv 1 \imod 2;\\
				\Next &\text{if } \displaystyle \left[\sum_{\substack{i=0,\\i \equiv 0 \imod 2}}^{2n+1} a_i\right] + (b+1) \equiv 0 \imod 2;\vspace{0.2in}\\
				\Prev &\text{if } \displaystyle \left[\sum_{\substack{i=0,\\i \equiv 0 \imod 2}}^{2n+1} a_i\right] + (b+1) \equiv 1 \imod 2 \text{ and } \exists \, a_i \not = 0;
			\end{cases}\\
			&= \begin{cases}
				\Right &\text{if } a_i = 0, \text{ } (b+m) \equiv 1 \imod 2;\\
				\Next &\text{if } \displaystyle \left[\sum_{\substack{i=0,\\i \equiv 0 \imod 2}}^{2n+1} a_i\right] + (b+m) \equiv 0 \imod 2;\vspace{0.2in}\\
				\Prev &\text{if } \displaystyle \left[\sum_{\substack{i=0,\\i \equiv 0 \imod 2}}^{2n+1} a_i\right] + (b+m) \equiv 1 \imod 2 \text{ and } \exists \, a_i \not = 0;
			\end{cases}\\
			&= o^-\left(\sum_{i=0}^{2n+1} a_i  \tau^i + (b+m)  \L(\tau^{2n+1})\right) \\
			&= o^-\left(\left[\sum_{i=0}^{2n+1} a_i  \tau^i + b  \L(\tau^{2n+1}) \right] + m  \L(\tau^{2n+1})\right)
		\end{align*}
	by Theorem \ref{theorem-2n+1-tauL}.  
	
	We will now show that all non-zero $\Next$ positions are equivalent.  That is, 
		\[ * + * \equiv \sum_{i=0}^{2n+1} c_i  \tau^i + d  \L(\tau^{2n+1}) \imod{{\cl{\L(\tau^{2n+1})}}}\]
	where
		\[\left[\sum_{\substack{i=0,\\i \equiv 0 \imod 2}}^{2n+1} c_i \right] + d \equiv 0 \imod 2.  \]
	Take an arbitrary element of the closure,
		\[ \sum_{i=0}^{2n+1} a_i  \tau^i + b  \L(\tau^{2n+1})\]
	and consider 
		\[ o^-\left(\left[\sum_{i=0}^{2n+1} c_i  \tau^i + d  \L(\tau^{2n+1}) \right]+  \left[ \sum_{i=0}^{2n+1} a_i  \tau^i + b  \L(\tau^{2n+1}) \right] \right). \]
	By Theorem \ref{theorem-2n+1-tauL}, the outcome of the above position is determined by the parity of
		\[ \left[\sum_{\substack{i=0,\\i \equiv 0 \imod 2}}^{2n+1} (a_i + c_i) \right]+ (b+d), \]
	but
		\[\left[\sum_{\substack{i=0,\\i \equiv 0 \imod 2}}^{2n+1} c_i \right] + d \equiv 0 \imod 2\]
	so the outcome of the above position is determined by the parity 2 of
		\[ \left[\sum_{\substack{i=0,\\i \equiv 0 \imod 2}}^{2n+1} a_i\right] + b\]
	which is the same as the parity module 2 of
		\[ 2 + \left[\sum_{\substack{i=0,\\i \equiv 0 \imod 2}}^{2n+1} a_i\right] + b,\]
	which is the sum arising from the position
		\[ * + * + \left[\sum_{i=0}^{2n+1} a_i  \tau^i + b  \L(\tau^{2n+1})\right].\]
	Therefore the two positions are indistinguishable in $\cl{\L(\tau^{2n+1})}$.  
	
	We will now show that all $\Prev$ positions are equivalent.  That is,
		\[ * \equiv \sum_{i=0}^{2n+1} c_i  \tau^i + d  \L(\tau^{2n+1}) \imod{{\cl{\L(\tau^{2n+1})}}}\]
	where at least one $c_i \not = 0$ and
		\[ \left[\sum_{\substack{i=0,\\i \equiv 0 \imod 2}}^{2n+1} c_i\right] + d \equiv 1 \imod 2. \]
	Take an arbitrary element of the closure,
		\[ \sum_{i=0}^{2n+1} a_i  \tau^i + b  \L(\tau^{2n+1})\]
	and consider 
		\[ o^-\left(\left[\sum_{i=0}^{2n+1} c_i  \tau^i + d  \L(\tau^{2n+1}) \right]+  \left[ \sum_{i=0}^{2n+1} a_i  \tau^i + b  \L(\tau^{2n+1}) \right] \right). \]
	By Theorem \ref{theorem-2n+1-tauL}, the outcome of the above position is determined by the parity 2 of
		\[ \left[\sum_{\substack{i=0,\\i \equiv 0 \imod 2}}^{2n+1} (a_i + c_i) \right]+ (b+d), \]
	but
		\[\left[\sum_{\substack{i=0,\\i \equiv 0 \imod 2}}^{2n+1} c_i \right] + d \equiv 1 \imod 2\]
	so the outcome of the above position is determined by the parity 2 of
		\[ \left[\sum_{\substack{i=0,\\i \equiv 0 \imod 2}}^{2n+1} a_i\right] + b\]
	which is the same as the parity module 2 of
		\[ 1 + \left[\sum_{\substack{i=0,\\i \equiv 0 \imod 2}}^{2n+1} a_i\right] + b,\]
	which is the sum arising from the position
		\[ * + \left[\sum_{i=0}^{2n+1} a_i  \tau^i + b  \L(\tau^{2n+1})\right].\]
	Therefore the two positions are indistinguishable in $\cl{\L(\tau^{2n+1})}$.  
	
	Thus, so far up to indistinguishability, we have four elements $0$, $\L(\tau^{2n+1})$, $*$, and $* + *$.  It remains to show that these four elements are pairwise distinguishable.  By Theorem \ref{theorem-2n+1-tauL}, $o^-(\L(\tau^{2n+1})) = \Right$, while $o^-(0) = \Next$, $o^-(*) = \Prev$, and $o^-(* + *) = \Next$.  Therefore, the pairwise distinguishability follows trivially for all pairs except $0$ and $* + *$, which both have the same outcome class.  However, these two positions are distinguished by $\L(\tau^{2n+1})$ since $o^-(0 + \L(\tau^{2n+1})) = \Right$ while $o^-(2  * + \L(\tau^{2n+1})) = \Prev$ by Theorem \ref{theorem-2n+1-tauL}.  
	
	Therefore the \mis monoid of $\cl{\L(\tau^{2n+1})}$ has four elements which correspond to the equivalence classes of $0$, $\L(\tau^{2n+1})$, $*$, and $* + *$.
\end{proof}

We now explicitly write $\monoid{M}_{\cl{\L(\tau^{2n+1})}}$.  Via the mapping
	\begin{align*}
		0 &\mapsto 1;\\
		* &\mapsto a;\\
		\L(\tau^{2n+1}) &\mapsto t;
	\end{align*}
the following monoid is obtained
	\begin{align*}
	\monoid{M}_{\cl{\L(\tau^{2n+1})}} &= \ideal{1,a, a^2, t \mid t^{2n+1} = t, a^3=a, at=a^2, a^2t = a} \\
	\Next &= \{1, a^2\} \\
	\Prev &= \{a\} \\
	\Left &= \emptyset \\
	\Right &= \{t\},
	\end{align*}
with the additive notation in $\cl{\L(\tau^{2n+1})}$ becoming a multiplicative notation in $\monoid{M}_{\cl{\L(\tau^{2n+1})}}$.  

The \mis monoid for $\cl{\L(\tau^{2n})}$ is even simpler.

\begin{theorem}\label{theorem-tau^2n_L-monoid}
	The \mis monoid of $\cl{\L(\tau^{2n})}$ has two elements, which correspond to the equivalence classes of $0$ and $*$.
\end{theorem}

\begin{proof}
	The proofs are similar to those given in Theorem \ref{theorem-indistinguishability-2n+1-tauL} for $* + *$ and the $\Next$ positions and $*$ and the $\Prev$ positions respectively.
\end{proof}

We now explicitly write $\monoid{M}_{\cl{\L(\tau^{2n})}}$.  Via the mapping
	\begin{align*}
		0 &\mapsto 1;\\
		* &\mapsto a;\\
		\L(\tau^{2n}) &\mapsto 1;
	\end{align*}
the following monoid is obtained
	\begin{align*}
	\monoid{M}_{\cl{\sigma}} &= \ideal{1,a \mid a^2=1} \\
	\Next &= \{1\} \\
	\Prev &= \{a\} \\
	\Left &= \emptyset \\
	\Right &= \emptyset,
	\end{align*}
with the additive notation in $\cl{\L(\tau^{2n})}$ becoming a multiplicative notation in $\monoid{M}_{\cl{\L(\tau^{2n})}}$.  Again, this is the same monoid as $\monoid{M}_{\cl{*}}$, a tetrapartite monoid whose underlying monoid structure is the same as $(\mathbb{Z}_2, \oplus)$.

Therefore, we have found an infinite set of positions ($\tau^m$) and operation ($\L(\tau^m)$) which always ensure that the \mis monoid of $\cl{\L(\tau^m)}$ is finite. 

\section{Conclusion}
While this chapter did find a list of positions which cause their indistinguishability monoid to be infinite (Corollary \ref{cor-cl-xi-infinite}), by no means did we manage to completely characterise when this occurs for partizan positions.  As the result is still lacking for impartial positions, this is hardly surprising.  Further investigation is required to more fully understand the relationship between a position and the cardinality of its \mis monoid.  

\chapter{Stars}\label{chapter-stars-0}

\section{Introduction}
One of the first results proven about \mis play games is that if $\xi$ is an impartial position , then $* + * \equiv 0 \imod{\cl{\xi}}$ \cite{NOTES}.  In Chapter \ref{chapter-examples}, we saw four examples where $\xi$ was a partizan position and this result still held.  However, this result is not true in general for partizan positions, as the following proposition shows.

\begin{proposition}\label{prop-1-*+*=0}
	Suppose $1, * \in \cl{\zeta}$.  Then $* + * \not \equiv 0 \imod{\cl{\zeta}}$.  Similarly, if $\bar{1}, * \in \cl{\zeta}$, the equivalence also does not hold.
\end{proposition}

\begin{proof}
	Suppose $1, * \in \cl{\zeta}$.  Then $o^-(1) = \Right$.  We will show that $o^-(1+* + *) \not = \Right$, and so the positions $0$ and $* + *$ are not equivalent.
	
	Consider $1+* + *$ with Right moving first.  Right's only move is to $1 + *$.  Left responds by moving to $*$, which is a $\Prev$ position.  Therefore $o^-(1 + * + *) \not = \Right$.  Thus $* + * \not \equiv 0 \imod{{\cl{\zeta}}}$.  
\end{proof}

What about the converse, i.e.\ if $1, \bar{1} \not \in \cl{\xi}$ but $* \in \cl{\xi}$, does this mean that $* + * \equiv 0 \imod{\cl{\xi}}$?  Unfortunately, no.

\begin{example}\label{example-sigma_L-*+*=0}
	Recall the position $\sigma$ from Section \ref{sec-sigma}.  Consider $\L(\sigma)$, whose game tree is given in Figure \ref{fig-sigma_L}. 

		\begin{figure}[htb]
		\unitlength 8pt
		\begin{center}
		\begin{graph}(4,6)(0,0)
		\graphnodesize{0.4}
		\fillednodestrue

		\roundnode{A}(2,4)
		\roundnode{B}(1,2)
		\roundnode{D}(0,0)
		\roundnode{E}(2,0)
		\roundnode{F}(3,6)

		\edge{A}{B}
		\edge{B}{D}
		\edge{B}{E}
		\edge{A}{F}
		
		\end{graph}
		\caption{The game tree of $\L(\sigma)$.}
		\label{fig-sigma_L}
		\end{center}	
		\end{figure}
		
	Then $* \in \cl{\L(\sigma)}$ and $1, \bar{1} \not \in \cl{\L(\sigma)}$.  
	
	Clearly, $o^-(\L(\sigma)) = \Right$.  Consider $\L(\sigma) + * + *$.  Suppose Right moves first.  He makes his only move, namely taking a $*$ and moving to $\L(\sigma) + *$.  Left responds by moving to $\sigma + *$, which, by Proposition \ref{prop-n*+m-sigma+m-sigma_L}, is a $\Prev$ position.  Thus Right moving first loses $\L(\sigma) + * + *$.  Therefore $* + * \not \equiv 0 \imod{\cl{\L(\sigma)}}$.  
\end{example}

We would like to find a position $\xi$ with $* + * \equiv 0 \imod{\cl{\xi}}$.  Doing such would give us a set of partizan positions which have similarities to impartial ones.  As the \mis monoid theory is more developed for impartial games, this may give us a start in applying some of the impartial theory to partizan positions. 

\section{All-Small}

While Proposition \ref{prop-1-*+*=0} gave a large set of positions which did not have our desired property, we will show that we do have the property for all-small games.  Recall our definition of all-small from Chapter \ref{chapter-intro}:

\begin{definition}
	A game is \textbf{all-small} if for every position in the game, Left can move if and only if Right can.  A position $\xi$ is all-small if Left can move in $\xi$ if and only if Right can, and every option of $\xi$ is all-small.
\end{definition}

We have the following theorem for all-small positions.

\begin{theorem}\label{theorem-as-*+*=0}
	Let $\xi$ be all-small and not equal to 0.  Then $* + * \equiv 0 \imod{\cl{\xi}}$.
\end{theorem}

Before we begin the proof, we note that this theorem follows the proof given for impartial positions in \cite{NOTES}.  

\begin{proof}\footnote{David Wolfe \cite{DAVID} suggests the following, shorter, proof for this theorem: If Alice has a winning strategy on $\nu$, Alice can use the
same strategy to win on $\nu+*+*$, delaying any move on $*+*$ until her opponent
moves on it.  The only time this fails is if she is out of moves on $\nu$, but
since the game is all-small, all that remains is $*+*$, from which she wins.
Since the winner is preserved, the outcome class of $\nu$ is the same as that of
$\nu+*+*$.}
	Take $\nu \in \cl{\xi}$.  This proof proceeds by showing that 
		\begin{align*}
			o^-(\nu) = \Right &\implies o^-(\nu+*+ *) = 
			\Right;\\
			o^-(\nu) = \Left &\implies o^-(\nu+*+ *) = \Left;\\
			o^-(\nu) = \Next &\implies o^-(\nu+*+ *) = \Next;\\
			o^-(\nu) = \Prev &\implies o^-(\nu+*+ *) = \Prev,
		\end{align*}
	and therefore $o^-(\nu) \in \mathcal{X} \iff o^-(\nu+*+*) \in \mathcal{X}$.  Thus $0 \equiv * + * \imod{\cl{\xi}}$ .
	
	Suppose $o^-(\nu) = \Right$ and consider $\nu + * + *$.   It will be shown that Right has winning strategy in $\nu + * + *$.   
	
	Suppose Right moves first.  A winning move for Right in $\nu + * + *$ is to play his winning move in $\nu$, to which there must be a Left response, since $o^-(\nu) = \Right$.  Left either plays one of her responses in $\nu^R$ or takes a $*$.  If Left plays a response, then Right plays again in $\nu^{RL}$, if such a Right moves exists.  If one does not exist, then the position has become $* + *$, as $\nu$ is an all-small position and if Right has no moves in $\nu^{RL}$, neither does Left.  Thus Right is playing next in an $\Next$ position, so Right wins.  If Left does not play a response in $\nu^R$ and, rather, takes a $*$, then Right takes the other $*$, leaving Left to go next in $\nu^R$, which is a winning position for Right moving second.  
	
	Now suppose Left moves first.  If she takes a $*$, then Right takes the other, leaving $\nu$, a Right win.  If Left plays in $\nu$, then Right responds with his winning move, unless no such move exists, i.e.\ Right is unable to play in $\nu^L$.  Then since $\nu$ is an all-small position, Left is also unable to play in $\nu^L$, and the position has become $* + *$ with Right moving next, so Right wins.  Thus, assume Right has a winning move in $\nu^L$, to which Left must have a response since $o^-(\nu) = \Right$.  Thus if Left plays in the $\nu$ component, Right responds unless there is no Right, and hence, no Left response available.
	
	Therefore, $o^-(\nu) = \Right$ implies $o^-(\nu + * + *) = \Right$.
	
	Similarly, if $o^-(\nu) = \Left$, then $o^-(\nu + * + *) = \Left$.
	
	Suppose $o^-(\nu) = \Next$.  Consider $o^-(\nu + * + *)$.  If $\nu = 0$, then the position $\nu + * + *$ becomes $* + *$, which is an $\Next$ position, so $o^-(0) = \Next$ and $o^-(0 + * + *) = \Next$, as required.  Otherwise, suppose $\nu$ is a non-zero position.  In $\nu + * + *$, Player One (P1) makes her/his winning move in $\nu$, to which Player Two (P2) must have a response.  P1 plays and responds in the $\nu$ component with her/his winning moves unless
		\begin{enumerate}
			\item\label{item-p2-cgstar} P2 takes a $*$; or
			\item\label{item-p1-no-nu} P1 has no moves in the $\nu$ component
		\end{enumerate}
	If \eqref{item-p2-cgstar}, then P1 takes the other $*$ and P2's next move must be in the $\nu$ component.  If P2 cannot move in the $\nu$ component, since $\nu$ is an all-small position, P1 also has no moves available in the $\nu$ component and would have made the last move within it, a contradiction since P1 makes her/his winning moves in the $\nu$ component.  
	
	If \eqref{item-p1-no-nu}, since $\nu$ is an all-small position, P2 also has no moves in the $\nu$ component and the position has become either 0 or $* + *$ with P1 making the next move.  Thus P1 wins.
	
	Therefore $o^-(\nu) = \Next$ implies $o^-(\nu + * + *) = \Next$.
	
	Similarly $o^-(\nu) = \Prev$ implies $o^-(\nu + * + *) = \Prev$.
\end{proof}

We can also extend Theorem \ref{theorem-as-*+*=0} to the following corollary.

\begin{corollary}\label{cor-upsilon-as-*+*=0}
	Let $\Upsilon$ be a set of positions closed under addition such that for any $\xi \in \Upsilon$, $\xi$ is an all-small position, and $* \in \Upsilon$.   Then $*+* \equiv 0 \imod{\cl{\Upsilon}}$.
\end{corollary}

\begin{proof}
	The proof is identical to that of Theorem \ref{theorem-as-*+*=0}.
\end{proof}

Since the sum of all-small positions is an all-small positions \cite{LIP}, we can take $\Upsilon$ in Corollary \ref{cor-upsilon-as-*+*=0} to be any set of all-small positions, including the set of all all-small positions.   Since every impartial position is all-small, we can now just say that the result is true for all-small positions, encompassing the previous result for impartial positions.

However, while being all-small is a necessary condition for this equivalence, it is not sufficient.  We saw this in Example \ref{example-L*=sigma}, which showed $* + * \equiv 0 \imod{\sigma, \sigmab}$ where $\sigma$ is not an all-small position. 

\section{Extending the $* + * \equiv 0$ Results to Non-All-Small Games}

Theorem \ref{theorem-as-*+*=0} and Corollary \ref{cor-upsilon-as-*+*=0} showed us that all all-small positions have the property that $* + * \equiv 0$.  We now try to examine which non-all-small positions also have this property.  We know such non-all-small positions exist; for example $* + * \equiv 0 \imod{\cl{\sigma}}$ by Example \ref{example-L*=sigma}.

This thesis' only attempt at determining what types of non-all-small positions have the desired property is to look at the $\L$ operation, first introduced in Section \ref{sec-closure-of-xi_L}.  Using $\sigma$, Example \ref{example-sigma_L-*+*=0} showed that if $\xi$ is a non-all-small position with the property that $* + * \equiv 0 \imod{\cl{\xi}}$, this does not imply that $* + * \equiv 0 \imod{\cl{\L(\xi)}}$.  This is also true if $\xi$ is  an all-small position, as demonstrated by the following example.

\begin{example}\label{example-eta}
	Let $\eta = \combgame{\{\combgame{\{0\mid\combgame{\{*\mid0\}}\}}\mid*\}}$.  The game tree of $\eta$ is the left-most position in Figure \ref{fig-gt-eta}.

		\begin{figure}[htb]
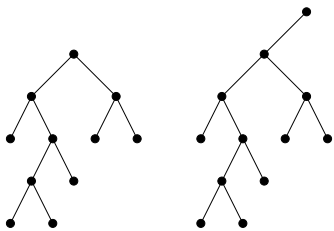

		\unitlength 8pt
		\begin{center}
		\begin{graph}(15,10)(0,0)
		\graphnodesize{0.4}
		\fillednodestrue

		\roundnode{a1}(0,0)
		\roundnode{a2}(2,0)
		
		\roundnode{b1}(1,2)
		\roundnode{b2}(3,2)
		
		\roundnode{c1}(0,4)
		\roundnode{c2}(2,4)
		\roundnode{c3}(4,4)
		\roundnode{c4}(6,4)
		
		\roundnode{d1}(1,6)
		\roundnode{d2}(5,6)
		
		\roundnode{e1}(3,8)
		
		\edge{a1}{b1}
		\edge{a2}{b1}
		\edge{b1}{c2}
		\edge{b2}{c2}
		\edge{c1}{d1}
		\edge{c2}{d1}
		\edge{c3}{d2}
		\edge{c4}{d2}
		\edge{d1}{e1}
		\edge{d2}{e1}

		\roundnode{A1}(9,0)
		\roundnode{A2}(11,0)
		
		\roundnode{B1}(10,2)
		\roundnode{B2}(12,2)
		
		\roundnode{C1}(9,4)
		\roundnode{C2}(11,4)
		\roundnode{C3}(13,4)
		\roundnode{C4}(15,4)
		
		\roundnode{D1}(10,6)
		\roundnode{D2}(14,6)
		
		\roundnode{E1}(12,8)
		
		\roundnode{F1}(14,10)
		
		\edge{A1}{B1}
		\edge{A2}{B1}
		\edge{B1}{C2}
		\edge{B2}{C2}
		\edge{C1}{D1}
		\edge{C2}{D1}
		\edge{C3}{D2}
		\edge{C4}{D2}
		\edge{D1}{E1}
		\edge{D2}{E1}
		\edge{E1}{F1}
	
		\end{graph}
		\caption{The game trees of $\eta$ and $\L(\eta)$, respectively.}
		\label{fig-gt-eta}
		\end{center}	
		\end{figure}	
	
		Clearly $\eta$ is an all-small position, $* \in \cl{\eta}$, and $o^-(\eta) = \Next$.   By Theorem \ref{theorem-as-*+*=0}, $* + * \equiv 0 \imod{\cl{\eta}}$.
		
		Consider $\L(\eta)$, which is the right-most position in Figure \ref{fig-gt-eta}.  Then $o^-(\L(\eta)) = \Right$, but claim that Right cannot win moving first in $\L(\eta) + * + *$.   Figure \ref{fig-right-loses-1-eta_L+*+*} gives Left's winning response when Right moves first.  
		
		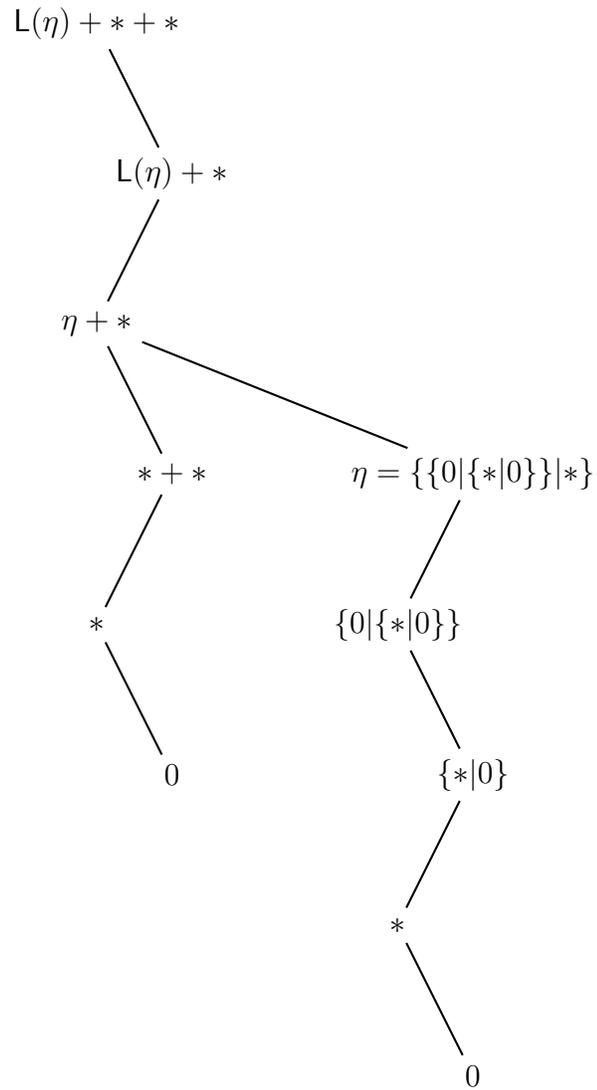
\begin{figure}[p]
		\begin{center}
		
		\begin{pspicture}\unit{0.5cm}(6,14)
		\put(1,14)
		{\cgtree
			{
				{\L(\eta) + * + *}
				(|{\L(\eta) + *}
				({\eta+*}
				(|{*+*}
				(* (|0)|) 
				++ {\eta=\combgame{\{\combgame{\{0|\combgame{\{*|0\}}\}}|*\}}}
				({\combgame{\{0|\combgame{\{*|0\}}\}}}
				(|{\combgame{\{*|0\}}}(*(|0)|))|))|))
			}		
		}
		\end{pspicture}

		\end{center}
		\caption{Right loses moving first in $\L(\eta) + * + *$.}
		\label{fig-right-loses-1-eta_L+*+*}
		\end{figure}		
		
		So $o^-(\L(\eta) + * + *) \not = o^-(\L(\eta))$.  Thus $* + * \not \equiv 0 \imod{\cl{\L(\eta)}}$.
\end{example}

\section{Conclusion}

In this chapter, we extended an important result from impartial positions to all-small positions, namely that if $\xi$ is an all-small position with $* \in \cl{\xi}$, then $* + * \equiv 0 \imod{\cl{\xi}}$.  We would like to find what other positions have this property.  As such, we conclude this chapter with the following opening problem:

\begin{openproblem}\label{op-*+*=0}
	Classify all positions $\xi$ with $* \in \cl{\xi}$ such that $* + * \equiv 0 \imod{\cl{\xi}}.$
\end{openproblem}

\chapter{Zeroes}\label{chapter-0}

\section{Introduction}

Suppose we have a position $\xi$.  Under the normal play convention, $\xi + \xib = 0$.  That is, if we had a disjunctive sum of positions
	\[ \alpha_1 + \alpha_2 + \cdots + \alpha_n + \xi + \xib,\]
we could replace $\xi + \xib$ by the position 0 without affecting the outcome of the sum.  This is partly due to the Tweedledum-Tweedledee strategy defined by Definition \ref{def-tweedle}.  We have seen similar results with the \mis monoids of partizan positions, namely in Corollary \ref{cor-B-*+*=0} with $\rho + \rhob \equiv 0 \imod{\cl{\rho,\rhob}}$.  However, as is common with \mis play games, this is not true in general.  We will show that this result is false for the following set of positions.

\begin{definition}\label{def-cgstarn}
	For $n \in \nat$, the position $*_n$ is defined recursively as follows:
		\begin{align*}
			*_n &= \combgame{\{0, *_1, *_2, \ldots, *_{n-1}\mid0, *_1, *_2, \ldots, *_{n-1}\}}.
		\end{align*}
	Generally, instead of $*_1$, we merely write $*$.
\end{definition}

We already made much use of the positions 0 and $*$, and we saw $*_2$ in Table \ref{table-normal-mis-outcomes} and in the definition of Genus (Definition \ref{def-genus}).  We call these positions \emph{nimbers} as they come from playing the game Nim (Definition \ref{def-genus}).

\begin{proposition}
	Consider $*_n$ for any $n \in \mathbb{Z}^{\ge 2}$.  Then $o^-(*_n + *_n) = \Prev$.  Thus, if $*_n \in \cl{\xi}$ for any $\xi$, $*_n + *_n \not \equiv 0 \imod{\cl{\xi}}$.
\end{proposition}

\begin{proof}
	For those familiar with genus \cite{MTHESIS, WW2, ONAG}, the genus of $*_n + *_n$ for $n \in \mathbb{Z}^{\ge 2}$ is $0^{02}$, and therefore $o^-(*_n + *_n) = \Prev$.
	
	For those unfamiliar with genus, it follows as a simple induction argument.  Suppose $n=2$ and consider $*_2 + *_2$ with Left moving first.  Figure \ref{fig-*2+*2-prev} shows how Right can force a win.  By symmetry, the same argument works for Left when Right plays first.    

		\begin{figure}[htb]
		\begin{center}

		\begin{pspicture}\unit{0.5cm}(6,6)
		\put(4,6)
		{\cgtree
			{
				{*_2 + *_2}({*_2}(|{*}(0|))+ {* + *_2}(|{*}(0|))|)
			}		
		}
		\end{pspicture}

		\end{center}
		\caption{$o^-(*_2 + *_2) = \Prev$.}
		\label{fig-*2+*2-prev}
		\end{figure}

	Now suppose true for all $2 \le k < n$ and consider $*_n + *_n$ with Left moving first.  Left can move to any of the following positions:
		\begin{itemize}
			\item $*_n$:  Right responds by moving to $*$, a $\Prev$ position.
			
			\item $* + *_n$:  Right responds by moving to $*$, a $\Prev$ position.
			
			\item $*_m + *_n$ for some $1 < m < n$:  Right responds by moving to $*_m + *_m$, which is a $\Prev$ position by induction.
		\end{itemize}
	
	Therefore $o^-(*_n + *_n) = \Prev$ for $n \in \mathbb{Z}^{\ge 2}$.  Since $o^-(0) = \Next$, the two cannot possibly be equivalent.
\end{proof}

Even restricting ourselves to binary all-small positions is not enough to guarantee the result, as the following example shows.

\begin{example}\label{example-theta-prev}
	Let $\theta = \combgame{\{\combgame{\{*\mid\rho\}}\mid\combgame{\{\rhob\mid*\}}\}}$.   Note that $\theta$ is symmetric and $\bar{\theta} = \theta$.  The game tree for $\theta$ is given in Figure \ref{fig-gt-theta}.

		\begin{figure}[htb]
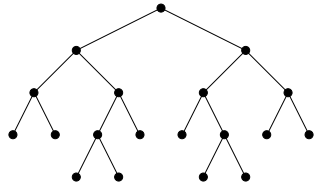

		\unitlength 8pt
		\begin{center}
		\begin{graph}(14,10)(0,0)
		\graphnodesize{0.4}
		\fillednodestrue

		\roundnode{a1}(3,0)
		\roundnode{a2}(5,0)
		\roundnode{a3}(9,0)
		\roundnode{a4}(11,0)
		
		\roundnode{b1}(0,2)
		\roundnode{b2}(2,2)
		\roundnode{b3}(4,2)
		\roundnode{b4}(6,2)
		\roundnode{b5}(8,2)
		\roundnode{b6}(10,2)
		\roundnode{b7}(12,2)
		\roundnode{b8}(14,2)

		\roundnode{c1}(1,4)
		\roundnode{c2}(5,4)
		\roundnode{c3}(9,4)
		\roundnode{c4}(13,4)
		
		\roundnode{d1}(3,6)
		\roundnode{d2}(11,6)
		
		\roundnode{e1}(7,8)
		
		\edge{b3}{a1}
		\edge{b3}{a2}
		\edge{c2}{b3}
		\edge{c2}{b4}
		\edge{b1}{c1}
		\edge{b2}{c1}
		\edge{d1}{c1}
		\edge{d1}{c2}
		\edge{c3}{b5}
		\edge{c3}{b6}
		\edge{c4}{b7}
		\edge{c4}{b8}
		\edge{b6}{a3}
		\edge{b6}{a4}
		\edge{d2}{c3}
		\edge{d2}{c4}
		\edge{e1}{d1}
		\edge{e1}{d2}
		
		\end{graph}
		\caption{The game tree of $\theta$.}
		\label{fig-gt-theta}
		\end{center}	
		\end{figure}

		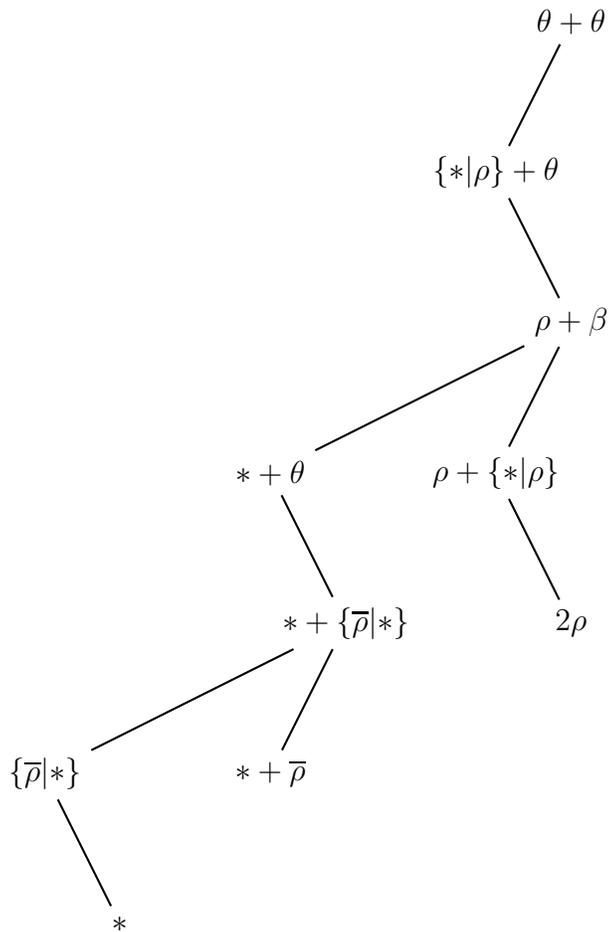
\begin{figure}[htb]
				\begin{center}

				\begin{pspicture}\unit{0.5cm}(6,12)
				\put(4,12)
				{\cgtree
					{
						{\theta + \theta}({\combgame{\{*|\rho\}}+\theta}(|{\rho + \beta}({\rho+\combgame{\{*|\rho\}}}(|{2 \rho})+{*+\theta}(|{*+\combgame{\{\rhob|*\}}}({*+\rhob}+{\combgame{\{\rhob|*\}}}(|*)|))|))|)
					}		
				}
				\end{pspicture}
				
				\end{center}
				\caption{Right can force a win in $\theta + \bar{\theta}$ with Left moving first.}
				\label{fig-beta-betab-prev}
		\end{figure}
		
		Suppose Left is moving first in $\theta + \theta$.  Then Figure \ref{fig-beta-betab-prev} shows how Right can force a win, recalling that $o^-(*) = \Prev$, and, by Theorem \ref{theorem-B-oc}, $o^-(* + \rhob) = \Next$, $o^-(2 \rho) = \Prev$.  
		
		Thus, Left loses moving first.  By symmetry, so does Right.  Thus $o^-(\theta +\theta) = \Prev$, and so $\theta + \theta \not \equiv 0 \imod{\cl{\theta}}$.
\end{example}

We would like to find positions $\xi$ where $\xi + \xib \equiv 0 \imod{\cl{\xi,\xib}}$.  As this is always true for normal play games, this would allow us to find certain positions which share behaviour between normal play and \mis play.  In this chapter, we indeed find a set of positions, which we call $\ab{3}$, where this is the case.  We also show that if $\xi$ is one of these $\ab{3}$ positions, then there is a Tweedledum-Tweedledee (Definition \ref{def-tweedle}) type strategy we can use in $\xi + \xib$.

\section{$\xi+\xib$ in $\ab{3}$ positions}

What can we say about $\xi + \xib$ in general?   In every example we have seen so far, $o^-(\xi + \xib) = \Next \cup \Prev$.  In fact, those are the only outcomes possible.

\begin{proposition}\label{prop-xi-xib-n-p}
	For any position $\xi$, $o^-(\xi + \xib) = \Next \cup \Prev$.
\end{proposition}

\begin{proof}
	Suppose Left moving first in $\xi + \xib$ has a winning move, either to $\xi^L + \xib$ or $\xi + \overline{\xi}{}^L$.  By symmetry, then if Right moves first in $\xi + \xib$, Right has a winning move with either $\xi + \overline{\xi}{}^R$ or $\xi^R + \xib$ respectively.  
	
	Similarly, Left moving second has a winning move.  This shows the result.
\end{proof}

We will now build a set of positions which have the property that $\xi + \xib \equiv 0 \imod{\cl{\xi,\xib}}$.  

\begin{definition}
	A position $\xi$ is called $\ab{n}$ if
		\begin{enumerate}
			\item $\xi$ is an all-small position (hence the $\mathrm{a}$),
			\item $\xi$ is binary (hence the $\mathrm{b}$), 
			\item each alternating path in the game tree of $\xi$ is of length $n$ or less (hence the $\mathit{n}$). 
		\end{enumerate}
\end{definition}

\begin{example}\text{}
	\begin{enumerate}
		\item The position $*$ is $\ab{1}$.
		\item The positions $\tau$ (Figure \ref{fig-gt-tau}), $\rho$ (Figure \ref{fig-gt-rho}), and $\rhob$ are $\ab{2}$.  
		\item The positions $\eta$ (Figure \ref{fig-gt-eta}) and $\theta$ (Figure \ref{fig-gt-theta}) are $\ab{4}$.   
	\end{enumerate}
\end{example}

Note that if $\xi$ is $\ab{n}$, then $\xi$ is $\ab{m}$ for all $m > n$.    

We are interested in $\ab{3}$ positions.  We will show that if $\xi \in$ $\ab{3}$, then $\xi + \xib \equiv 0 \imod{\cl{\xi, \xib}}$.  We will first show that $o^-(\xi + \xib) = \Next$.

\begin{proposition}\label{prop-xi+xib-next}
	Let $\xi$ be $\ab{3}$.  Then $o^-(\xi + \xib) = \Next$. 
\end{proposition}

\begin{proof}
	We proceed by induction on the birthday of $\xi$.  
	
	Suppose $\xi = 0$.  Then $o^-(0+0) = o^-(0) = \Next$, as required.
	
	Suppose true for all $\mu$ which are $\ab{3}$ and which have smaller birthday than $\xi$.  Consider $\xi$.
	
	Suppose $\xi^L = 0$.  Then $\xib^R = 0$.  Suppose Left is moving first in $\xi + \xib$.  Then Figure \ref{fig-xi+xib-xiL=0} shows Left's winning moves.
		
		\begin{figure}[htb]
		\begin{center}
		
		\begin{pspicture}\unit{0.5cm}(6,4)
		\put(2.5,4)
		{\cgtree
			{
				{\xi+\xib}({\xi^L+\xib=\xib}(|{\xib^R = 0})|)
			}		
		}
		\end{pspicture}
		\end{center}
		\caption{Left wins $\xi + \xib$ moving first if $\xi$ is $\ab{3}$ and $\xi^L = 0$.}
		\label{fig-xi+xib-xiL=0}
		\end{figure}
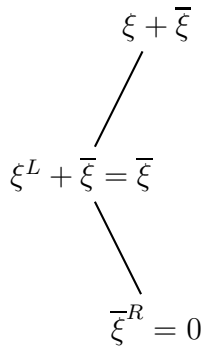

	Now suppose $\xi^{LRL} = 0$.  Then $\xib^{RLR} = 0$.  Figure \ref{fig-xi+xib-xiLRL=0} shows how Left moving first can win $\xi + \xib$, noting that $\overline{\xi^L} = \xib^R$, and that the birthday of  $\xi^L$ is strictly less than the birthday of $\xi$, so $o^-(\xi^L + \xib^R) = \Next$ by induction.  
		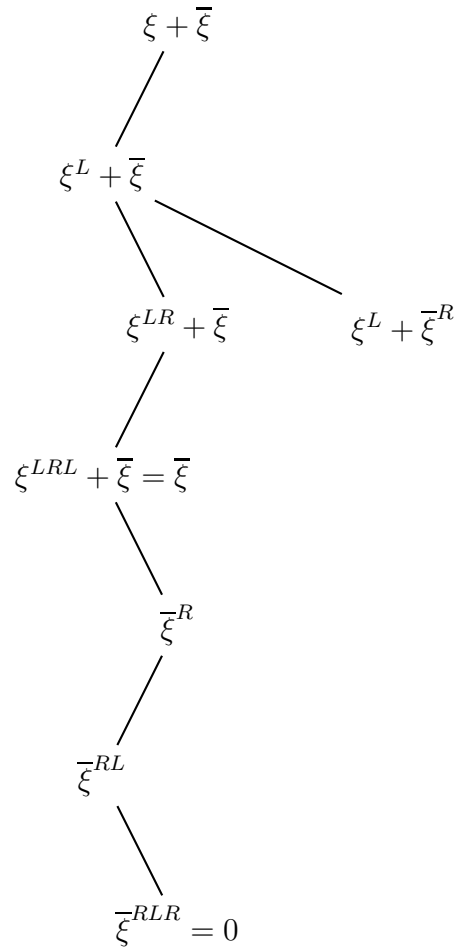
\begin{figure}[htb]
		\begin{center}

		\begin{pspicture}\unit{0.5cm}(6,12)
		\put(2.5,12)
		{\cgtree
			{
				{\xi+\xib}({\xi^L + \xib}(|{\xi^{LR} + \xib}({\xi^{LRL} + \xib=\xib}(|{\xib^R}({\xib^{RL}}(|{\xib^{RLR}=0})|))|)+{\xi^L + \xib^R})|)
			}		
		}
		\end{pspicture}
		\end{center}
		\caption{Left wins $\xi + \xib$ moving first if $\xi$ is $\ab{3}$ and $\xi^{LRL} = 0$.}
		\label{fig-xi+xib-xiLRL=0}
		\end{figure}

	Suppose that $\xi^{LR} = 0$.  Then $\xib^{RL} = 0$.  If $\xi^R = 0$ or $\xi^{RLR} = 0$, then repeat one of the above arguments to get that Left wins moving first in $\xi + \xib$.  Otherwise, suppose that $\xi^{RL} = 0$.  Figure \ref{fig-xi+xib-xiRL=0} shows how Left moving first can win $\xi + \xib$.  
	
		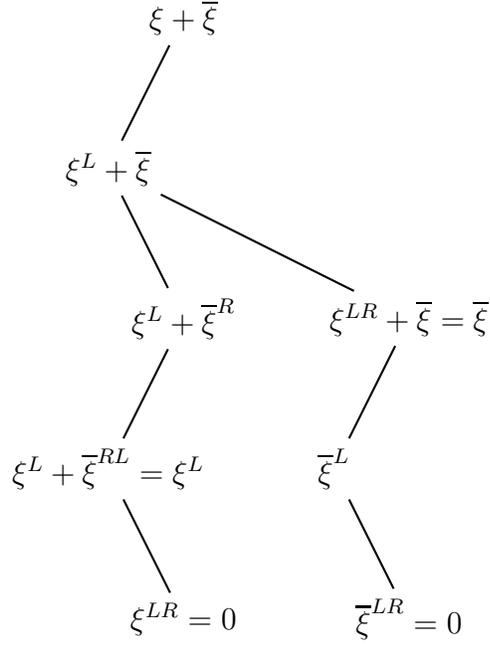
\begin{figure}[htb]
		\begin{center}

		\begin{pspicture}\unit{0.5cm}(6,8)
		\put(2.5,8)
		{\cgtree
			{
				{\xi+\xib}({\xi^L+\xib}(|{\xi^L + \xib^R}({\xi^L + \xib^{RL} = \xi^L}(|{\xi^{LR} = 0})|)+{\xi^{LR} + \xib=\xib}({\xib^L}(|{\xib^{LR}=0})|))|)
			}		
		}
		\end{pspicture}
		\end{center}
		\caption{Left wins $\xi + \xib$ moving first if $\xi$ is $\ab{3}$ and $\xi^{LR} = \xi^{RL} = 0$.}
		\label{fig-xi+xib-xiRL=0}
		\end{figure}
	
	A symmetric argument shows how Right wins moving first in $\xi + \xib$, and so the result holds.
\end{proof}

\begin{corollary}
	If $\xi$ is ab1 or ab2, then $o^-(\xi + \xib) = \Next$.
\end{corollary}

\begin{proof}
	Use the proof of Proposition \ref{prop-xi+xib-next} when $\xi^L=0$ and when $\xi^{LR} =0$ respectively.
\end{proof}

The argument of Proposition \ref{prop-xi+xib-next} seem so nice that it may be tempting to try and extend it to $\ab{n}$ positions for $n \ge 4$.  However, by following the same method as in Example \ref{example-theta-prev}, we see that if $\xi$ is $\ab{4}$ with $\xi = \xib$ and $\xi^L = \overline{\xi^R}$, then $o^-(\xi + \xib) = \Prev$.  Since, $\ab{4} \subseteq \ab{n}$ for all $n > 4$, we see that this counterexample given in $\ab{4}$ means we cannot extend the result of Proposition \ref{prop-xi+xib-next}.


We will use Proposition \ref{prop-xi+xib-next} in the next result, namely:

\begin{theorem}\label{theorem-xi-xib=0}
	Suppose $\xi$ is $\ab{3}$.  Then $\xi + \xib \equiv 0 \imod{\cl{\xi,\xib}}$.
\end{theorem}

\begin{proof}
	We proceed by induction on the birthday of $\xi$.  
	
	Suppose $\xi = 0$.  Then clearly $0 + 0 \equiv 0 \imod{\cl{0}}$.  
	
	Suppose true for all $\mu$ which are $\ab{3}$ and which have smaller birthdays than $\xi$.  Consider $\xi + \xib$ in $\cl{\xi, \xib}$. We  want $o^-(\nu) = o^-(\nu + \xi + \xib)$ for any $\nu \in \cl{\xi, \xib}$.  We proceed by induction on $\nu$.
	
	Suppose $\nu = 0$.  Then $o^-(0) = \Next$, and, by Proposition \ref{prop-xi+xib-next}, $o^-(\xi + \xib) = \Next$.  This shows the base case for the induction on $\nu$.
	
	Now suppose $o^-(\mu) = o^-(\mu + \xi + \xib)$ for all $\mu \in \cl{\xi, \xib}$ with lesser birthday than that of $\nu$. 
	
	Suppose Left wins moving first in $\nu$.  We claim that Left can win moving first in $\nu + \xi + \xib$ with the move $\nu^L + \xi + \xib$.  Since the birthday of $\nu^L$ is strictly less than the birthday of $\nu$, we have $o^-(\nu^L + \xi + \xib) = o^-(\nu^L)$.  Since Left wins moving first in $\nu$, this means $o^-(\nu) = \Prev \cup \Left$, so $o^-(\nu^L + \xi + \xib) = \Next \cup \Left$, so Left wins moving first in $\nu + \xi + \xib$.
	
	Suppose Left wins moving second in $\nu$.  Consider Right's three possible first moves in $\nu + \xi + \xib$.  He can move to either $\nu^R + \xi + \xib$, $\nu + \xi^R + \xib$, or $\nu + \xi + \xib^R$. Suppose Right makes the first move to $\nu^R + \xi + \xib$.  Since Left wins moving second in $\nu$, this gives $o^-(\nu^R) = \Next \cup \Left$.  Since the birthday of $\nu^R$ is strictly less than the birthday of $\nu$, by induction, $o^-(\nu^R + \xi + \xib) = \Next \cup \Left$, so Left wins $\nu + \xi + \xib$ moving second if Right's first move is to $\nu^R + \xi + \xib$.

	Suppose Right makes the first move to $\nu + \xi^R + \xib$.  Left responds by moving to $\nu + \xi^R + \xib^L$.  Since $\overline{\xi^R} = \xib^L$ and $\xi^R$ is $\ab{3}$, by induction, $o^-(\nu + \xi^R + \xib^L) = o^-(\nu)$.  Since Left wins moving second in $\nu$, by induction, Left wins moving second in $\nu + \xi^R + \xi$.   Similarly, if Right's first move is to $\nu+ \xi + \xib^R$, then Left will also win.	
		
	Therefore, if Left moving first (or second) wins $\nu$, then Left moving first (or second) wins $\nu + \xi + \xib$.  A symmetric argument works for Right.  Similarly, if Left (or Right) loses moving first (or second) in $\nu$, Left (or Right) loses moving first (or second) in $\nu + \xi + \xib$.  
	
	Therefore $o^-(\nu) = o^-(\nu + \xi + \xib)$, and so $\xi + \xib \equiv 0 \imod{\cl{\xi,\xib}}$.  
\end{proof}

\begin{definition}
	The set $\cl{\ab{3}}$ is the smallest closed set containing all positions which are $\ab{3}$.
\end{definition}

\begin{corollary}\label{cor-xi-xib-0}
	Suppose $\xi$ is $\ab{3}$.  Then $\xi + \xib \equiv 0 \imod{\cl{\ab{3}}}$.
\end{corollary}

\begin{proof}
	The proof is the same as Theorem \ref{theorem-xi-xib=0} with taking $\mu \in \cl{\ab{3}}$ rather than $\mu \in \cl{\xi, \xib}$.  
\end{proof}

\begin{corollary}
	Suppose $\xi$ is $\ab{3}$ and $\xi \not = 0$.  Then $* \in \cl{\xi}$ and $* + * \equiv 0 \imod{\cl{\xi}}$.
\end{corollary}

\begin{corollary}\label{cor-sum-xi-xib-0}
	Suppose $\xi_i$ is $\ab{3}$ for $i \in \{1,2,\ldots,n\}$.  Then
		\[ \sum_{i=1}^n \xi_i + \overline{\sum_{i=1}^n \xi_i} \equiv 0 \imod{\cl{\ab{3}}}.\]
\end{corollary}
	
\begin{proof}
	Recall that 
		\[ \overline{\sum_{i=1}^n \xi_i}= \sum_{i=1}^n \overline{\xi_i}, \]
	and so
		\begin{align*}
			\sum_{i=1}^n \xi_i + \overline{\sum_{i=1}^n \xi_i}  &= \sum_{i=1}^n \xi_i + \sum_{i=1}^n \overline{\xi_i} \\
			&= \sum_{i=1}^n (\xi_i + \overline{\xi_i}).
		\end{align*}
	By Corollary \ref{cor-xi-xib-0}, $\xi_1 + \overline{\xi_1} \equiv 0 \imod{\cl{\ab{3}}}$.  Thus
		\[  \sum_{i=1}^n (\xi_i + \overline{\xi_i}) \equiv \sum_{i=2}^n (\xi_i + \overline{\xi_i}) \imod{\cl{\ab{3}}}.\]
	Continuing as such, we get that 
		\[ \sum_{i=1}^n (\xi_i + \overline{\xi_i}) \equiv  0 \imod{\cl{\ab{3}}}.\]
\end{proof}

Theorem \ref{theorem-xi-xib=0} and Corollary \ref{cor-sum-xi-xib-0} are important results in bridging the gap between normal play and \mis play.  Played under the normal play convention, $\xi + \xib = 0$ for any $\xi$ \cite{LIP, WW1}.  While Theorem \ref{theorem-xi-xib=0} and Corollary \ref{cor-sum-xi-xib-0} do restrict our examination to $\ab{3}$ positions and to the closure of such positions, it does give an infinite set of positions which share some of the properties associated with normal play games.  Not only that, but these results give a useful tool in the analysis of $\ab{3}$ positions.   One may recall the lengthy calculations required in Chapter \ref{chapter-examples} to calculate $\monoid{M}_{\cl{\rho, \rhob}}$.  With these new results, we can eliminate much of the preliminary work required to determine the outcome classes and indistinguishability relations.  

How much further can Theorem \ref{theorem-xi-xib=0} be extended?  Example \ref{example-theta-prev} showed that there are $\ab{4}$ positions $\xi$ such that $\xi + \xib \not \equiv 0 \imod{\cl{\xi, \xib}}$, thus the result cannot be extended immediately to $\ab{4}$ and above.

We conclude this subsection with one further set of $\ab{3}$ positions which is equivalent to 0.

\begin{proposition}
	Let $\xi$ be $\ab{3}$ such that $\xi^{LR} = \xi^{RL} = 0$.  Then $\xi \equiv 0 \imod{\cl{\ab{3}}}$.  
\end{proposition}

\begin{proof}
	Let $\alpha \in \cl{\ab{3}}$.  We want that $o^-(\xi + \alpha) = o^-(\alpha)$.  We proceed by induction on the birthday of $\alpha$.
	
	If $\alpha = 0$, then $o^-(\xi) = o^-(0 + \xi) = \Next$.  Thus the base case holds.
	
	Now suppose true for all positions in $\cl{\ab{3}}$ which have birthday strictly less than that of $\alpha$.  
	
	We proceed in the proof similarly to that of Theorem \ref{theorem-xi-xib=0}. 
	Suppose Left wins moving first in $\alpha$, where $\alpha^L$ is the winning move for Left.  By induction, $o^-(\alpha^L) = o^-(\alpha^L + \xi)$, so this is a winning move for Left in $\alpha + \xi$.
	
	Suppose Left wins moving second in $\alpha$.  If Right's first move in $\alpha + \xi$ is to $\alpha^R + \xi$, by induction, $o^-(\alpha^R + \xi) = o^-(\alpha^R)$, and Left wins moving next in $\alpha^R$, so Left wins moving next in $\alpha^R + \xi$.  If Right's first move is to $\alpha + \xi^R$, then Left can respond with $\alpha + \xi^{RL} = \alpha$, with Right moving next.  However, Left wins $\alpha$ with Right moving first, so Left will win $\alpha + \xi$ with Right's initial first move to $\alpha + \xi^R$.
	
	Therefore, if Left moving first (or second) wins $\alpha$, then Left moving first (or second) wins $\alpha + \xi$.  A symmetric argument works for Right.  Similarly, if Left (or Right) loses moving first (or second) in $\alpha$, Left (or Right) loses moving first (or second) in $\alpha + \xi$.  Thus $\xi \equiv 0 \imod{\cl{\ab{3}}}$.  
\end{proof}

\section{Tweedledum-Tweedledee on $\ab{3}$}\label{sec-results-ab3}

Recall the Tweedledum-Tweedledee strategy from Definition \ref{def-tweedle}: Suppose Left and Right are playing in the position $\xi + \xib$ with Left moving first.  Either Left has no move available, or, without loss of generality, she can move to $\xi^L + \xib$.  Right can respond by moving to $\xi^L + \xib^R$.  However, $\xib^R = \overline{\xi^L}$, so Right has moved to $\xi^L + \overline{\xi^L}$.  Now either Left has no move available, or Left has a move available, to which Right can respond by making the mirror-image move in the opposite component.  This strategy ensures that the second player makes the final move.  In normal play, it is a winning strategy for the second player in the position $\xi + \xib$.   In \mis play, it is a particularly poor strategy for the second player, as not only will it guarantee his loss,  we cannot even say whether the outcome of $\xi + \overline{\xi}$ is $\Next$ or $\Prev$ for arbitrary $\xi$ (Proposition \ref{prop-xi-xib-n-p}).   However, as Proposition \ref{prop-xi+xib-next} showed, if $\xi \in \cl{\ab{3}}$, then $o^-(\xi + \overline{\xi}) = \Next$.  Using this fact, we are able to construct a strategy for $\xi + \xib$ if $\xi$ is $\ab{3}$ which mimics Tweedledum-Tweedledee strategy for normal play. It is important to note that this strategy gives a win for the next player to move, not the previous player as the Tweedledum-Tweedledee strategy for normal play does.

\begin{construction}[A Tweedledum-Tweedledee type strategy for $\cl{\ab{3}}$]
	Take an arbitrary element of $\cl{\ab{3}}$, say $\displaystyle \sum_{i=1}^n \xi_i$ where each $\xi_i$ is $\ab{3}$.  Then we are looking to construct a winning strategy for Left moving first in
		\[ \sum_{i=1}^n \xi_i + \overline{\sum_{i=1}^n \xi_i} = \sum_{i=1}^n (\xi_i + \overline{\xi_i}).\]
	
	The overview of the algorithm is as such:
		\begin{enumerate}
			\item Consider the sum as follows:
				\[ \xi_1 + \overline{\xi_1} + \sum_{i=2}^n \xi_i + \overline{\sum_{i=2}^n \xi_i}\]
				
			\item Left's first move is in $\xi_1 + \overline{\xi_1}$.
			
			\item If Right responds by playing in $(\xi_1 + \overline{\xi_1})^L$, then so does Left.
			
			\item Otherwise, Right plays in 
				\[\sum_{i=2}^n \xi_i + \overline{\sum_{i=2}^n \xi_i},\]
			to which Left responds by playing the Tweedledum-Tweedledee strategy in the sum.
		\end{enumerate}
	
	The details of how this algorithm works now follows.
	
	We divide up the sum as given in the overview to
		\[ \xi_1 + \overline{\xi_1} + \sum_{i=2}^n \xi_i + \overline{\sum_{i=2}^n \xi_i}.\]
	Since $\xi_i$ is $\ab{3}$ for each $i \in \{1,2,\ldots,n\}$, one of the following must be true: 
		\begin{enumerate}
			\item there exists an $i$ such that $\xi_i^L = 0$; or
			\item there does not exist an $i$ such that $\xi_i^L = 0$ but there does exist an $i$ such that $\xi_i^{LRL} = 0$; or
			\item for every $i \in \{1,2,\ldots,n\}$, $\xi_i^{LR} = \overline{\xi_i}{}^{RL} = 0$.
		\end{enumerate}
	
	We will consider each of these cases separately.
	
	\begin{enumerate}
		\item Suppose there exists a $\xi_i$ such that $\xi_i^L = 0$. Reorder the sum so that $i=1$, i.e.\ $\xi_1^L = 0$.
		
		We claim that Left's moving first to 
		\[\xi_1^L + \overline{\xi_1} + \displaystyle \sum_{i=2}^n (\xi_i + \overline{\xi_i}) = 0 + \overline{\xi_1} + \displaystyle \sum_{i=2}^n (\xi_i + \overline{\xi_i})\]
		is a winning move for Left.  
	
		Suppose $n=1$.  Then our initial sum is merely $\xi_1 + \overline{\xi_1}$.  If Left moves to 
			\[\xi_1^L + \overline{\xi_1} = \overline{\xi_1},\]
		Right has no choice but to respond to $\overline{\xi_1}{}^R = \overline{\xi_1^L} = 0$, and so Right loses $\xi_1 + \overline{\xi_1}$ moving second.
		
		Thus suppose $n > 2$.  If Right moves to 
			\[ \overline{\xi_1}{}^R + \displaystyle \sum_{i=2}^n (\xi_i + \overline{\xi_i}) = 0 + \displaystyle \sum_{i=2}^n (\xi_i + \overline{\xi_i}),\]
		then the algorithm begins again with Left moving first.  Otherwise Right moves in the sum.  Left plays Tweedledum-Tweedledee in the sum until Right moves outside the sum and moves $\overline{\xi_1}$ to $\overline{\xi_1}{}^R = 0$.  Then either the entire position has become 0 or the position is now 
			\[ \sum_{i=1}^m (\xi_i^{\prime} + \overline{\xi_i^{\prime}}),\]
		for some $\xi_i^{\prime}$ which are $\ab{3}$, and the algorithm begins again with Left moving first.
	
		\item Suppose there does not exist $i \in \{1,2,\ldots,n\}$ such that $\xi_i^L = 0$ but there does exist an $i$ such that  $\xi_i^{LRL} = \overline{\xi_i}{}^{RLR} = 0$.  As in the preceding case, reorder our sum so that $i=1$, i.e.\ $\xi_1^{LRL} =0$.  
		
		Then Left's first move to 
			\[\xi_1^L + \overline{\xi_1} + \displaystyle \sum_{i=2}^n (\xi_i + \overline{\xi_i}) = 0 + \overline{\xi_1} + \displaystyle \sum_{i=2}^n (\xi_i + \overline{\xi_i})\]
		is a winning move for Left.   If Right plays in the sum, then Left plays Tweedledum-Tweedledee in the sum until Right moves outside the sum.
		
		Suppose Right has made his first move outside the sum and he moves to 
			\[\xi_1^{LR} + \overline{\xi_1} + \sum_{i=1}^m (\xi_i^{\prime} + \overline{\xi_i^{\prime}}),\]
		then Left responds by playing to
			\[\xi_1^{LRL} + \overline{\xi_1} + \sum_{i=1}^m (\xi_i^{\prime} + \overline{\xi_i^{\prime}}) = 0 + \overline{\xi_1} + \sum_{i=1}^m (\xi_i^{\prime} + \overline{\xi_i^{\prime}}). \]	
		If Right returns to playing in the sum, then Left returns to playing Tweedledum-Tweedledee in the sum until Right moves outside the sum.  Right's moving outside the sum takes the position to
			\[ \overline{\xi_1}{}^R + \sum_{i=1}^u (\xi_i^{\prime\prime} + \overline{\xi_i^{\prime\prime}}).\]
		Left responds by moving to
			\[ \overline{\xi_1}{}^{RL} + \sum_{i=1}^u (\xi_i^{\prime\prime} + \overline{\xi_i^{\prime\prime}}).\]
		Either Right then moves to
			\[ \overline{\xi_1}{}^{RLR} + \sum_{i=1}^u (\xi_i^{\prime\prime} + \overline{\xi_i^{\prime\prime}}) = 0 + \sum_{i=1}^u (\xi_i^{\prime\prime} + \overline{\xi_i^{\prime\prime}}), \]	
		and the algorithm begins again with Left moving first, or Right plays in the sum, with Left responding via Tweedledum-Tweedledee until Right moves in $\overline{\xi_1}{}^{RL}$, taking it to $\overline{\xi_1}{}^{RLR} = 0$.  At this point, either the entire position is 0 with Left to move next, or the position is
			\[ \sum_{i=1}^v (\xi_i^{\prime\prime\prime} + \overline{\xi_i^{\prime\prime\prime}}), \]
		and the algorithm begins again.
	
		Otherwise, we suppose Right's first move outside the sum is to
			\[ \xi_1^L + \overline{\xi_1}{}^R + \sum_{i=1}^m (\xi_i^{\prime} + \overline{\xi_i^{\prime}}).\]
		Since $\overline{\xi_1}{}^R = \overline{\xi_1^L}$, we can start the algorithm over again with Left moving first.   
	
		\item Now suppose $\xi_i^{LR}= \xi_i^{RL} = 0$ for all $i \in \{1,2,\ldots,n\}$.  Again, we claim that Left's first move to 
			\[\xi_1^L + \overline{\xi_1} + \sum_{i=2}^n (\xi_i + \overline{\xi_i}) \]
		is a winning move for Left.   If Right plays in the sum, then Left plays Tweedledum-Tweedledee in the sum until Right moves outside the sum.  
		
		Suppose Right has made his first move outside the sum and he moves to 
			\[ \xi_1^{LR} + \overline{\xi_1} +\sum_{i=1}^m (\xi_i^{\prime} + \overline{\xi_i^{\prime}}) =0 +  \overline{\xi_1} +\sum_{i=1}^m (\xi_i^{\prime} + \overline{\xi_i^{\prime}}). \]
		Left responds by moving to
			\[ \overline{\xi_i}{}^L +\sum_{i=1}^m (\xi_i^{\prime} + \overline{\xi_i^{\prime}}). \]
		If Right returns to playing in the sum, then Left returns to playing Tweedledum-Tweedledee in the sum until Right moves outside the sum.  Right moving outside the sum has one option, namely to move to 
			\[ \overline{\xi_1}{}^{LR} + \sum_{i=1}^u (\xi_i^{\prime\prime} + \overline{\xi_i^{\prime\prime}}) = 0 + \sum_{i=1}^u (\xi_i^{\prime\prime} + \overline{\xi_i^{\prime\prime}}), \]	
		and either this position is 0 with Left moving next or the algorithm begins again with Left moving first.
	
	Suppose Right's first move outside the sum is to
		\[ \xi_1^L + \overline{\xi_1}{}^R + \sum_{i=1}^m (\xi_i^{\prime} + \overline{\xi_i^{\prime}}). \]
	Since $\overline{\xi_1}{}^R = \overline{\xi_1^L}$, we can start the algorithm over again with Left moving first.
	
	\end{enumerate}
	
	Thus, we have developed a strategy for Left moving first
		\[ \sum_{i=1}^n \xi_i + \overline{\sum_{i=1}^n \xi_i} = \sum_{i=1}^n (\xi_i + \overline{\xi_i})\]
	that is similar to the Tweedledum-Tweedledee strategy for normal play positions.  For Right, simply reverse the Lefts, Rights, $^L$'s, and $^R$'s in above argument.
\end{construction}

\section{Conclusion}
The discovery of a set of partizan positions with the property that $\xi + \xib \equiv 0 \imod{\cl{\xi,\xib}}$ is intriguing, as it gives a set of \mis play positions with a strong similarity to its normal play counterparts.  This similarity suggests the following open problem:

\begin{openproblem}\label{op-ab3}
	Investigate whether $\ab{3}$ positions share any other normal play properties.
\end{openproblem}

\chapter{A Brief Categorical Interlude}\label{chapter-cat}

\section{Introduction}
In this chapter, we will discuss the possibility of forming a category of \mis play games. In 1977, Joyal \cite{CAT} formed a category of games for games played under the normal play convention using disjunctive sum.  His category of normal play games used constructions that are well-used and understood in game theory, such as the Tweedledum-Tweedledee strategy.  If a similar category of \mis play games exists, this would allow us to again draw parallels between normal play and \mis play, as well as giving us the entirety of category theory with which to analyse the structure of \mis play games.

\begin{definition}
	A \textbf{category} $\mathscr{C}$ consists of a set of objects $\text{Ob}(\mathscr{C})$ with, for any two objects $A$, $B \in \text{Ob}(\mathscr{C})$, a set of arrows, denoted by $\text{Hom}(A,B)$, together with the following:
		\begin{itemize}
			\item If $f \in \text{Hom}(A,B)$ and $g \in \text{Hom}(B,C)$, then there is an arrow $g \circ f \in \text{Hom}(A,C)$ (this arrow is called the \textbf{composite of $f$ and $g$});
			
			\item For each object $A$, there exists $1_A \in \text{Hom}(A,A)$ (this arrow is called the \textbf{identity arrow on $A$}),
		\end{itemize}
	such that the following are satisfied:
		\begin{itemize}
			\item Composition is associative, i.e.\ for $f \in \text{Hom}(A,B)$, $g \in \text{Hom}(B,C)$, and $h \in \text{Hom}(C,D)$, we have
				\[ f \circ (g \circ h) = (f \circ g) \circ h.\]
			
			\item For $f \in \text{Hom}(A,B)$, 
				\[ f \circ 1_A = f = 1_B \circ f.\]
		\end{itemize}
\end{definition}

For those wishing to We know more about category theory, a classic reference is \cite{WORK};  a more modern reference is \cite{AWODEY}.

We would like to make a category of games, where the objects of the category are positions of games (not necessarily positions from the same game) and the arrows are some way of relating two positions.  There are two easy ways to make a category out of a set of positions.  The arrows are, respectively,
	\begin{enumerate}
		\item $\alpha \to \beta$ if $\alpha = \beta$; or
		\item $\alpha \to \beta$ if $\alpha \le \beta$, where $\le$ is as in Definition \ref{def-<=}.
	\end{enumerate}
However, neither of these categories are particularly interesting, or yield much new information in how to relate positions.  Essentially, all either is doing is replacing $=$ or $\le$ respectively with arrows $\rightarrow$.  Joyal \cite{CAT} formed a much more interesting category of games using the following:
	\begin{itemize}
		\item Objects: positions of games;
		\item Arrows:  An arrow $\alpha \to \beta$ is a winning strategy for Left moving second in the position $\beta + \overline{\alpha}$.  Thus if Left can win moving second in the position $\beta + \overline{\alpha}$, then at least one arrow exists between $\alpha$ and $\beta$ (there could be more than one if more than one strategy for winning exists).  
	\end{itemize}
Under this construction, the identity arrow $\alpha  \stackrel{\text{id}_{\alpha}} 
{\longrightarrow}  \alpha$, i.e.\ the winning strategy in the game $\alpha + \overline{\alpha}$,  is the Tweedledum-Tweedledee strategy, and there exists a manner of composition such that if $\alpha \stackrel{f}{\to} \beta$ and $\beta \stackrel{g}{\to} \gamma$, then an inherited strategy $g \circ f$ exists with $\alpha \stackrel{g \circ f}{\longrightarrow} \gamma$.  This inherited strategy arises from examining $\alpha + \beta + \overline{\beta} +\gamma$, and noting, since we are playing under the normal play convention, $\beta + \overline{\beta} = 0$. Then, we can combine elements of the strategy $f$ and $g$ to a new strategy, which we call $g \circ f$, so that Left moving second wins $\gamma + \overline{\alpha}$.  

We will call a category in which the objects are positions of games and the arrows between positions are based on a strategy for winning some combination of the positions under some sum a \textbf{Joyal-style category}.  This chapter examines considering positions under a variety of sums to see whether Joyal-style categories can be formed.  Unfortunately, for the sums considered, this is not the case and this failure arises from the inability to build an identity arrow.  

Another option which is briefly considered is to change what constitutes an arrow between positions.  Joyal's definition uses that an arrow $\alpha \to \beta$ exists if there is a winning strategy for Left moving second in $\beta + \overline{\alpha}$.  We need not be fixated on this, or slight variations on it (such as Left wins moving first) as our arrow condition.  Section \ref{sec-adjoint-ds} discusses the attempt of trying to build a category using disjunctive sum, and, rather than using $\overline{\alpha}$, we use $\alpha^{\circ}$, the adjoint operation defined by Siegel in \cite{MCF}.  Unfortunately, this also fails.

Finally, we restrict ourselves just to examining positions which are $\ab{3}$, as Section \ref{sec-results-ab3} showed that $\ab{3}$ has a Tweedledum-Tweedledee type strategy and the Tweedledum-Tweedledee strategy is used to form Joyal's normal play category.  Unfortunately, again, we are unable to construct an identity arrow.

This chapter concludes by suggesting another construction, a taxon, which may yield more favourable results, as the conditions for identity are weakened somewhat.

\section{A Variety of Sums}\label{sec-sums}

We begin by defining a variety of sums.

\begin{definition}\label{def-sums}
	Let $\alpha$ and $\beta$ be two positions.  
		\begin{enumerate}
			\item The \textbf{disjunctive sum of $\alpha$ and $\beta$}, denoted by $\alpha + \beta$, is the position in which a player moves in either $\alpha$ or $\beta$.  This is the sum most commonly used in combinatorial game theory.
			
			\item The \textbf{AND of $\alpha$ and $\beta$}, denoted by $\alpha \wedge \beta$, is the position played as such:  a player must play in both $\alpha$ and $\beta$.  If there is no move available for the player in either $\alpha$ or $\beta$ (or both), then play ends, even if the other player still has moves available in $\alpha$ or $\beta$ (or both).
			
			\item The \textbf{DisAND of $\alpha$ and $\beta$}, denoted by $\alpha \dand \beta$, is the position played as such:  a player plays in both $\alpha$ and $\beta$.  If on a player's turn, that player has no move available in, say, $\alpha$, then the position is $\alpha$ is discarded from the sum, even if the other player still has available moves in $\alpha$, and play continues in $\beta$.
			
			The name for this sum comes from the fact that we are \emph{discarding} (Dis) positions, but essentially still playing by the AND rules.  
			
			\item The \textbf{OR of $\alpha$ and $\beta$}, denoted by $\alpha \vee \beta$, is the position played as such:  a player must play in either $\alpha$ or $\beta$, or both.  If there is no move available for the player in either $\alpha$ or $\beta$ (or both), then play ends even if the other player still has moves available in $\alpha$ or $\beta$ (or both).
			
			\item The \textbf{DisOR of $\alpha$ and $\beta$}, denoted by $\alpha \dor \beta$, is the position played as such:  a player plays in either $\alpha$ or $\beta$, or both.  If there is no move available for the player in, say, $\alpha$, then the position is $\alpha$ is discarded from the sum, even if the other player still has available moves in $\alpha$, and play continues in $\beta$.
			
			The name for this sum comes from the fact that we are \emph{discarding} (Dis) positions, but essentially still playing by the OR rules. 

			\item The \textbf{sequential join of $\alpha$ and $\beta$}, denoted by $\alpha \then \beta$, is the position played as such:  all play is in $\alpha$ until the next player has no moves available in $\alpha$.  Then $\alpha$ is discarded, the next player plays in $\beta$ and all further moves are played in $\beta$.
			
			\item The \textbf{ordinal sum of $\alpha$ and $\beta$}, denoted by $\alpha : \beta$, is defined recursively as the position with options 
				\[ \alpha:\beta = \combgame{\{\alpha^L, \alpha:\beta^L\mid\alpha^R, \alpha:\beta^R\}}\]
			with $0:0 =0$.  Essentially, once a player plays in $\alpha$, the $\beta$ component is discarded.
			
		\end{enumerate}
\end{definition}

We will show that it is impossible to construct an identity arrow irrespective of which of the above sums we use.

\begin{proposition}\label{prop-sum-0-N}
	Fix $\square$ to be one of the six sums listed above.  Then $o^-(0 \square 0) = \Next$ for whatever we choose $\square$ to be.  
\end{proposition}

\begin{proof}
	Each number in the list corresponds to the appropriate sum in Definition \ref{def-sums}.
	
	\begin{enumerate}
		\item We have seen many times before that $o^-(0+0) = \Next$.
		
		\item There are no moves available in 0, so there are no moves available in $0 \wedge 0$, so $o^-(0 \wedge 0) = \Next$.
		
		\item Consider $0 \dand 0$.  For the next player to go, there is no move available in 0, so it is discarded and we have the game 0.  Since there is no move available in 0, it is also discarded, leaving us with nothing in which to move.  So the first player has no moves available, meaning $o^-(0 \dand 0) = \Next$.
		
		\item Consider $0 \vee 0$.  There is no move available for the next player in either 0 or 0, so play ends.  That is, $o^-(0 \vee 0) = \Next$.
		
		\item Consider $0 \dor 0$.  Much like $0 \dand 0$, this we have $o^-(0 \dor 0) = \Next$.
		
		\item Consider $0 \then 0$.  The first player has no moves in the first 0, so we move to the second 0, where there are also no moves available.  Hence $o^-(0 \then 0) = \Next$.
		
		\item Consider $0:0$.  However, we define $0:0 = 0$, and we know $o^-(0) = \Next$. \qedhere
	\end{enumerate}
\end{proof}
	
If we build a structure with 
	\begin{itemize}
		\item Objects: positions of games;
		\item Arrows:  An arrow $\alpha \to \beta$ is a winning strategy for Left moving second in the position $\beta \square \overline{\alpha}$,
	\end{itemize}	
then there does not exist an arrow $0 \to 0$, as Proposition \ref{prop-sum-0-N} showed us that $o^-(0 \square 0) = \Next$, so Left cannot win moving second in $0 \square 0$.  As such, we will be unable to build an identity arrow for all positions.  Therefore, instead we must have a structure with the following properties:
	\begin{itemize}
		\item Objects: positions of games;
		\item Arrows:  An arrow $\alpha \to \beta$ is a winning strategy for Left moving first in the position $\beta \square \overline{\alpha}$,
	\end{itemize}	
However, for each $\square$, we can now find an $\alpha$ such that $\alpha \to \alpha$ does not exist where each number in the list corresponds to the appropriate sum in Definition \ref{def-sums}:
	\begin{enumerate}
		\item Let $\alpha = *_2 + *_2$.  We know that $o^-(*_2 + *_2) = \Prev$, so Left moving first cannot win $*_2 + *_2$.  
		
		\item Let $\alpha = *_2 + *_2$.  Then Left moving first cannot win the position $(*_2 + *_2) \wedge (*_2 + *_2)$.  As a first move, Left has four possible first moves:
			\begin{enumerate}
				\item $*_2 \wedge *_2$,
				\item $(* + *_2) \wedge (*_2)$,
				\item $*_2 \wedge (* + *_2)$.
				\item $(* + *_2) \wedge (* + *_2)$.
			\end{enumerate}
		However, in each of these cases, Right can move to $* \wedge *$.  Left's only available moves are to $0 \wedge 0$, $* \wedge 0$, or $0 \wedge *$, all of which have no moves available to Right.  Therefore, Left moving first cannot win $(*_2 + *_2) \wedge (*_2 + *_2)$.

		\item Let $\alpha = *_2$.  Then Left cannot win moving first in $*_2 \dand *_2$.  As a first move, Left has four possible moves:
			\begin{enumerate}
				\item $0 \dand *_2$,
				\item $*_2 \dand 0$,
				\item $* \dand *_2$,
				\item $*_2 \dand *$.
			\end{enumerate}
		In the first two cases, Right discards the 0 and is playing next in $*_2$, an $\Next$ position.  In the third and fourth cases, Right can move to $0 \dand *$ and $* \dand 0$ respectively.  Left will discard the 0 in both of these positions, and is playing next in $*$, a $\Prev$ position.  Therefore, Left moving first cannot win $*_2 \dand *_2$.

		\item Let $\alpha = *$.  Then Left moving first cannot win in $* \vee *$.  Left moving first can move to one of three positions:
			\begin{enumerate}
				\item $0 \vee *$,
				\item $* \vee 0$, 
				\item $0 \vee 0$.
			\end{enumerate}
		In all three cases, Right has no available moves.  Therefore Left cannot win $* \vee *$ moving first. 

		\item Let $\alpha = (* + *)$.  Then Left moving first cannot win $(*+ *) \dor (* + *)$.  Left moving first can move to one of three possible positions:
			\begin{itemize}
				\item $* \dor *$,
				\item $* \dor (* + *)$,
				\item $(* + *) \dor *$.
			\end{itemize}
		In all three cases, Right can then move to either $0 \dor *$ or $* \dor 0$.  Left discards the 0, and is playing next in $*$, a $\Prev$ position.  Therefore Left moving first cannot win $(* + *) \dor (* + *)$.

		\item  Let $\alpha = *_2 + *_2$.  Then Left moving first cannot win the position $(*_2 + *_2) \then (*_2 + *_2)$, as shown in Figure \ref{fig-*2+*2-then-*2+*2} recalling $o^-(*_2 + *_2) = \Prev$.  Therefore Left moving first cannot win $(*_2 +*_2) \then (*_2 + *_2)$.  

		\item Let $\alpha = *_2 + *_2$.  Then, similarly to the previous case, Left moving first cannot win $(*_2 +*_2) : (*_2 + *_2)$, as shown in Figure \ref{fig-*2+*2-:-*2+*2}.

		\begin{figure}[p]
		\begin{center}
		
		\begin{pspicture}\unit{0.5cm}(6,8)
		\put(4,8)
		{\cgtree
			{
				{(*_2 + *_2) \then (*_2 + *_2)}({(* + *_2)\then(*_2 + *_2)}(|{(*+*)\then(*_2+*_2)}({*\then(*_2 + *_2)}(|{*_2 + *_2})|))++{*_2\then(*_2+*_2)}(|{*_2 + *_2})|)
			}		
		}
		\end{pspicture}
		
		\end{center}		
		\caption{Left loses moving first in $(*_2 + *_2) \then (*_2 + *_2)$, recalling $o^-(*_2 + *_2) = \Prev$.}
		\label{fig-*2+*2-then-*2+*2}
		\end{figure}
		
		\begin{figure}[p]
		\begin{center}
		
		\begin{pspicture}\unit{0.5cm}(9,8)
		\put(9,8)
				{\cgtree
					{
						{(*_2 + *_2) : (*_2 + *_2)}({(*_2 + *_2):(* + *_2)}(|{(*_2+*_2):(*+*)}({(*_2+*_2):*}(|{*_2 + *_2})|))++{(*_2 + *_2):*_2}(|{*_2 + *_2})++{*_2}+{* + *_2}|)
					}		
				}
		\end{pspicture}

		\end{center}
		\caption{Left loses moving first in $(*_2 + *_2) : (*_2 + *_2)$, recalling $o^-(*_2 + *_2) = \Prev$ and $o^-(*_2) = o^-(* + *_2) = \Next$.}
		\label{fig-*2+*2-:-*2+*2}
		\end{figure}
	\end{enumerate}
Thus, for every sum, we have found an $\alpha$ such that $\alpha \to \alpha$ does not exist.  Since all our counterexamples were impartial positions, we cannot modify our requirements to having Right winning moving first or second, or the positions being  $\Next$ or $\Prev$,  or to restrict ourselves to impartial games, or to change our arrow existence requirements from $\beta \square \overline{\alpha}$ to $\beta \square \alpha$.  Therefore, using these sums, we cannot possibly construct an identity arrow using the same method as Joyal.

\section{Adjoints and Disjunctive Sum}\label{sec-adjoint-ds}
In Section \ref{sec-sums}, we attempted to construct a Joyal-style category using a variety of sums, but keeping the idea of negative, namely $\overline{\alpha}$, constant.  In this section, we choose our sum to remain constant, using disjunctive sum.  However, we will use a different form of negative, the \textbf{adjoint} of a position, as first defined by Siegel in \cite{MCF}.

Recall:

\begin{definition}
	The \textbf{adjoint of $\alpha$}, denoted by $\alpha^{\circ}$, is given by:
		\[ \alpha^{\circ} = \begin{cases}
		* & \text{if } \alpha = 0;\\
		\combgame{\{(\alpha^R)^{\circ}\mid 0\}} &\text{if } \alpha \not = 0 \text{ and $\alpha$ is a Left end};\\
		\combgame{\{0\mid(\alpha^L)^{\circ}\}} &\text{if } \alpha \not = 0 \text{ and $\alpha$ is a Right end};\\
		\combgame{\{(\alpha^R)^{\circ}\mid(\alpha^L)^{\circ}\}} &\text{else},
		\end{cases}\]
	where $\alpha$ is a \textbf{Left (Right) end} if $\alpha$ has no Left (Right) option.
\end{definition}

In \cite{MCF}, Siegel uses the adjoints in determining the canonical forms of partizan \mis play positions.  To familiarize ourselves with adjoints, we will now calculate some examples.

\begin{example}\label{example-adjoints}
	Using our examples from Chapter \ref{chapter-examples}, we calculate their adjoints:
		\begin{align*}
			0^{\circ} &= *, \\
			*^{\circ} &= \combgame{\{*\mid*\}} = \tau,\\
			1^{\circ} &= \combgame{\{0\mid*\}} = \rhob,\\
			\sigma^{\circ} &= \combgame{\{0\mid\combgame{\{*\mid*\}}\}},\\
			\tau^{\circ} &= \combgame{\{\combgame{\{*\mid*\}}\mid\combgame{\{*\mid*\}}\}},\\
			\rho^{\circ}&= \combgame{\{*\mid\combgame{\{*\mid*\}}\}}.
		\end{align*}
\end{example}

In \cite{MCF}, Siegel shows that $o^-(\alpha + \alpha^{\circ}) = \Prev$.  Thus, supposing we build our structure as:
	\begin{itemize}
		\item Objects: positions of games;
		\item Arrows:  An arrow $\alpha \to \beta$ is a winning strategy for Left moving second in the position $\beta + \alpha^{\circ}$,
	\end{itemize}
we see that for any position $\alpha$, there exists an arrow $\alpha \to \alpha$.  Therefore, the possibility of building an identity exists.  However, we do not have composition, as the following example shows:

\begin{example}
	We claim there exists $* \to 1$ and $1 \to \rhob$, but there does not exist $* \to \rhob$.  
	
	Firstly, look at $* \to 1$.  Such an arrow exists if Left wins moving second in $1 + *^{\circ}$.  Figure \ref{fig-1+*^circ} shows that $o^-(1 + *^{\circ}) = \Prev$, so indeed, such an arrow must exist.   
		\begin{figure}[htb]
		\begin{center}

		\begin{pspicture}\unit{0.5cm}(6,6)
		\put(3,6)
		{\cgtree
			{
				{1 + *^{\circ}=\combgame{\{0|\cdot\}}+\combgame{\{*|*\}}}({\combgame{\{0|\cdot\}}+*}(|{\combgame{\{0|\cdot\}}+0}(0|))+{0+\combgame{\{*|*\}}}(|{*}(0|))|+++{\combgame{\{0|\cdot\}}+*}({0+*}(|0)|))
			}		
		}
		\end{pspicture}
		\end{center}
		\caption{$o^-(1 + *^{\circ}) = \Prev$.}
		\label{fig-1+*^circ}
		\end{figure}
	
	Now examine $1 \to \rhob$.  Such an arrow exists if Left wins moving second in $\rhob + 1^{\circ}$.  Example \ref{example-adjoints} gives that $1^{\circ} = \rhob$, and Theorem \ref{theorem-B-oc} gives $o^-(2 \rhob) = \Prev$.  Therefore Left does indeed win moving second in $\rhob + 1^{\circ}$, and so $1 \to \rhob$ exists.
	
	However, there does not exist an arrow $* \to \rhob$.  Consider $\rhob + *^{\circ} = \rhob + \tau$.  If Right's first move is in the $\rhob$ component, moving the position to $* + \tau$, then Proposition \ref{prop-n*+m.tau} gives $o^-(* + \tau) = \Prev$, meaning Right can move to a previous position.  Therefore Right can win $\rhob + *^{\circ}$ moving first (in fact, $o^-(\rhob + *^{\circ}) = \Next$), and so there is no arrow between $*$ and $\rhob$, so composition fails.    
\end{example}

The above example also shows that even if we change our arrow existence condition to either:
	\begin{itemize}
		\item an arrow $\alpha \to \beta$ is a winning strategy for Right moving second in the position $\beta + \alpha^{\circ}$, or
		\item an arrow $\alpha \to \beta$ is a winning strategy for the second player to move in $\beta + \alpha^{\circ}$,
	\end{itemize}
composition still fails.  Therefore, we cannot hope to construct a category in the style of Joyal using disjunctive sum as our sum and the adjoint as our ``negative".

\section{Restricting ourselves to $\ab{3}$}
Since $\cl{\ab{3}}$ shares some similarities with normal play games, most notably that $\xi + \xib \equiv 0 \imod{\cl{\ab{3}}}$, can we form a Joyal-style game category?  As stated at the beginning of the chapter, we cannot.  The following counterexamples show how no possibility of a Joyal-style category can exist.

\begin{example}\label{example-ab3-no-cat-1}
	Firstly, take that an arrow $\alpha \to \beta$ is a winning strategy for Left moving first in the position $\beta + \overline{\alpha}$.  Then, the possibility of an identity arrow exists, since $o^-(\alpha + \overline{\alpha}) = \Next$, so there exist arrows $\alpha \to \alpha$.  However, we are not guaranteed composition.  
	
	Suppose $\alpha = *$, $\beta = \rhob$, and $\gamma = 0$.  By Theorem \ref{theorem-B-oc}, $o^-(\beta + \overline{\alpha}) = \Next$ and $o^-(\gamma + \overline{\beta}) = \Left$, so there exist arrows $\alpha \stackrel{f}{\to} \beta$ and $\beta \stackrel{g}{\to} \gamma$.  However, since $o^-(\gamma + \overline{\alpha}) = \Prev$, there are no arrows between $\alpha$ and $\gamma$, so no composition $g \circ f$ can exist.
\end{example}

\begin{example}\label{example-ab3-no-cat-2}
	Take that an arrow $\alpha \to \beta$ is a winning strategy for Right moving first in the position $\beta + \overline{\alpha}$.  Then the conjugate of the counterexample given in Example \ref{example-ab3-no-cat-1} shows that composition may not exist.
\end{example}

\begin{example}
	Suppose we assume that an arrow $\alpha \to \beta$ is a winning strategy for the next player to win in the position $\beta + \overline{\alpha}$

	Let $\xi_1 = *$, $\xi_2 = \combgame{\{\cdot\mid\rho\}}$, and $\xi_3 = \combgame{\{* \mid \rho\}}$. The game trees of $\xi_1$, $\xi_2$, and $\xi_3$ are given in Figure \ref{fig-gt-a-b-g}.  Note that $o^-(\xi_1) = o^-(\xi_2) = \Prev$, and $o^-(\xi_3) = \Left$.  
	
		\begin{figure}[p]
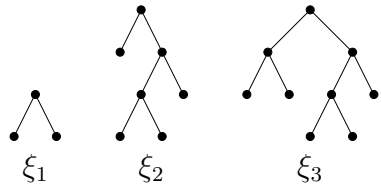

		\unitlength 8pt
		\begin{center}
		\begin{graph}(17,7.5)(0,-1.5)
		\graphnodesize{0.4}
		\fillednodestrue

		\roundnode{a1}(0,0)
		\roundnode{a2}(1,2)
		\roundnode{a3}(2,0)
		
		\edge{a1}{a2}
		\edge{a2}{a3}

		\freetext(1,-1.5){$\xi_1$}

		\roundnode{b1}(5,4)
		\roundnode{b2}(6,6)
		\roundnode{b3}(7,4)
		\roundnode{b4}(6,2)
		\roundnode{b5}(8,2)
		\roundnode{b6}(5,0)
		\roundnode{b7}(7,0)
		
		\edge{b1}{b2}
		\edge{b2}{b3}
		\edge{b3}{b4}
		\edge{b3}{b5}
		\edge{b4}{b6}
		\edge{b4}{b7}
		
		\freetext(6.5,-1.5){$\xi_2$}

		\roundnode{g1}(11,2)
		\roundnode{g2}(12,4)
		\roundnode{g3}(13,2)
		\roundnode{g4}(14,6)
		\roundnode{g5}(16,4)
		\roundnode{g6}(15,2)
		\roundnode{g7}(17,2)
		\roundnode{g8}(14,0)
		\roundnode{g9}(16,0)
		
		\edge{g1}{g2}
		\edge{g2}{g3}
		\edge{g2}{g4}
		\edge{g4}{g5}
		\edge{g5}{g6}
		\edge{g5}{g7}
		\edge{g6}{g8}
		\edge{g6}{g9}
		
		\freetext(14, -1.5){$\xi_3$}
		
		\end{graph}
		\caption{The game trees of $\xi_1$, $\xi_2$, and $\xi_3$.}
		\label{fig-gt-a-b-g}
		\end{center}	
		\end{figure}
		
		We have that $o^-(\xi_2 + \overline{\xi_1}) = \Next$;  both Left and Right's initially moving to $\xi_2$, a $\Prev$ position, gives the result.  Therefore, there exists an arrow $\xi_1 \stackrel{f}{\to} \xi_2$.
		
		Figure \ref{fig-gamma+betab-in-N} shows that $\xi_3 + \overline{\xi_2} = \combgame{\{* \mid \rho\}} + \combgame{\{\rhob\mid\cdot\}}$ is also an $\Next$ position, recalling that $o^-(\rho + \rhob) = \Next$, $o^-(\rho) = \Left$, and $o^-(\rhob) = \Right$.  Therefore, there exists an arrow $\xi_2 \to \xi_3$.  
		
		\begin{figure}[p]
		\begin{center}
		\begin{pspicture}\unit{0.5cm}(6,8)
		\put(3,8)
		{\cgtree
			{
				{\xi_3 + \overline{\xi_2} = \combgame{\{* | \rho\}} + \combgame{\{\rhob|\cdot\}}}
				(+++{\combgame{\{* | \rho\}}+\rhob}(|{\combgame{\{*|\rho\}}+*}({\combgame{\{*|\rho\}}}(|\rho)|){\rho+\rhob})|+++{\rho+\combgame{\{\rhob|\cdot\}}}({*+\combgame{\{\rhob|\cdot\}}}(|{\combgame{\{\rhob|\cdot\}}}(\rhob|)){\rho+\rhob}|))
			}		
		}
		\end{pspicture}
		\end{center}
		\caption{$o^-(\xi_3 + \overline{\xi_2}) = \Next$.}
		\label{fig-gamma+betab-in-N}
		\end{figure}

		We now examine the position $\xi_3 + \overline{\xi_1} = \combgame{\{*\mid \rhob\}} + *$.  Figure \ref{fig-gamma-alphab-Rl} shows that Right loses moving first, recalling that $o^-(\rhob + *) = \Next$ and $o^-(\xi_3) = \Left$.  Therefore, there exist no arrows between $\xi_1$ and $\xi_3$.  Thus, while there exist $\xi_1 \stackrel{f}{\to} \xi_2$ and $\xi_2 \stackrel{g}{\to} \xi_3$, there can be no arrows between $\xi_1$ and $\xi_3$ and so $g \circ f$ cannot exist.  

		\begin{figure}[p]
		\begin{center}
		\begin{pspicture}\unit{0.5cm}(6,2)
		\put(1,2)
		{\cgtree
			{
				{\xi_3 + \overline{\xi_1} = \combgame{\{*| \rhob\}} + *}(|{\rhob + *}+{\combgame{\{*|\rhob\}}})
			}		
		}
		\end{pspicture}
		\end{center}
		\caption{Right moving first loses $o^-(\xi_3 + \overline{\xi_1})$.}
		\label{fig-gamma-alphab-Rl}
		\end{figure}		
\end{example}

Note that since an identity arrow must exist if we wish to form a category, these three examples give the only possibilities for what arrows between positions can be since Proposition \ref{prop-xi+xib-next} gives $o^-(\alpha + \overline{\alpha}) = \Next$ for any $\alpha \in \cl{\ab{3}}$.

\section{Conclusion}
The quest to find an appropriate sum and negative which would yield a Joyal-style category is worth pursuing.  Much as the result that normal play games played under the disjunctive sum form a group, the discovery as to what sum and negative would yield a Joyal-style category would be extraordinarily beneficial.  It would allow us to take structure and theory already well-understood on categories and apply them to \mis play combinatorial games.   However, it may be that either the sum or the negative which would yield a Joyal-style category is so convoluted that it becomes relatively useless in practise.

Interestingly, in Section \ref{sec-adjoint-ds}, in trying to find an example to show how composition fails, the author was unable to find an example in which $\alpha \to \beta$ and $\beta \to \gamma$ exist, while $\alpha \to \gamma$ does not exist with $\alpha$, $\beta$, and $\gamma$ all being impartial positions.  Thus, a possible conjecture is as follows:
	\begin{conjecture}\label{conjecture-arrow-impartial}
		If $\alpha$, $\beta$, and $\gamma$ are impartial positions and $\alpha \to \beta$ and $\beta \to \gamma$ exist, then there exists an arrow $\alpha \to \gamma$ where an arrow $\delta \to \varepsilon$ exists if Left moving second can win $\varepsilon + \delta^{\circ}$ (or something similar).  
	\end{conjecture}
Of course, this is still far from the final desired result.  The existence of an arrow does not mean that we have composition;  for a Joyal-style category, we would need that the arrow, i.e.\ the strategy, from $\alpha$ to $\gamma$ be influenced in an ``obvious" way, whatever that may be, by the arrows $\alpha \to \beta$ and $\beta \to \gamma$.  An identity arrow must be established, as well as all the other rules which apply to arrows in a category.  However, it may be a starting point to restrict ourselves solely to impartial games and then extend our results to partizan positions, much as was done in normal play games, with the results on {\sc nim} and the Sprague-Grundy theory for impartial games arising before the results on partizan games were fully explored.  

Another possibility is to try and construct a \emph{taxon} rather than a category, the idea of which originally appeared in \cite{TAXON}.  The idea of a taxon is to have a category without a formal identity, but which still retains some identity-like properties.  The definition is as follows:

\begin{definition}
	A \textbf{taxon} $\mathscr{T}$ consists of a set of objects $\text{Ob}(\mathscr{T})$ with, for any two objects $A$, $B \in \text{Ob}(\mathscr{T})$, a set of arrows $\text{Hom}(A,B)$, together with the following:
		\begin{itemize}
			\item If $f \in \text{Hom}(A,B)$ and $g \in \text{Hom}(B,C)$, then there is an arrow $g \circ f \in \text{Hom}(A,C)$ (this arrow is called the \textbf{composite of $f$ and $g$});
			
			\item  For each $f \in \text{Hom}(A,B)$, there exists a non-empty set of factorisations $(E, a, b)$ where $E \in \text{Ob}(\mathscr{T})$, $a \in \text{Hom}(A,E)$, and $b \in \text{Hom}(E,B)$ such that the following diagram commutes:
			
			\unitlength 12pt
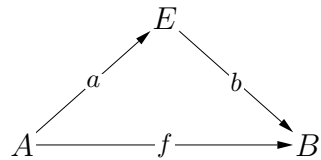
			\begin{center}
			\begin{graph}(9,4)(0,0)
			\graphlinecolour{1}
			\graphlinewidth{.01}
			\grapharrowlength{.5}
			\newcommand{\arr}[3]{%
			\edge{#1}{#2}[\graphlinewidth{.1}]
			\diredge{#1}{#2}[\graphlinecolour{0}]
			\edgetext{#1}{#2}{\footnotesize #3}}

			\textnode{A}(0,0){$A$}
			\textnode{B}(9,0){$B$}
			\textnode{E}(4.5,4){$E$}
			
			\arr{A}{B}{$f$}
			\arr{A}{E}{$a$}
			\arr{E}{B}{$b$}

			\end{graph}			
			\end{center}
		\end{itemize}
	such that the following are satisfied:
		\begin{itemize}
			\item Composition is associative, i.e.\ for $f \in \text{Hom}(A,B)$, $g \in \text{Hom}(B,C)$, and $h \in \text{Hom}(C,D)$, we have
				\[ f \circ (g \circ h) = (f \circ g) \circ h;\]
			
			\item If you have two factorisations of $f \in \text{Hom}(A,B)$, say

			\unitlength 12pt
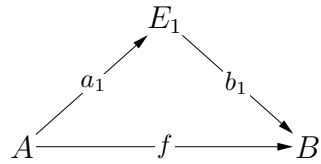
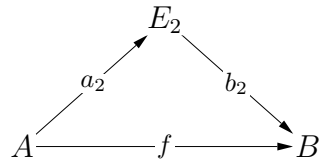
			\begin{center}
			\begin{graph}(9,4)(0,0)
			\graphlinecolour{1}
			\graphlinewidth{.01}
			\grapharrowlength{.5}
			\newcommand{\arr}[3]{%
			\edge{#1}{#2}[\graphlinewidth{.1}]
			\diredge{#1}{#2}[\graphlinecolour{0}]
			\edgetext{#1}{#2}{\footnotesize #3}}

			\textnode{A}(0,0){$A$}
			\textnode{B}(9,0){$B$}
			\textnode{E}(4.5,4){$E_1$}
			
			\arr{A}{B}{$f$}
			\arr{A}{E}{$a_1$}
			\arr{E}{B}{$b_1$}

			\end{graph}			
			\end{center}
			
			and

			\unitlength 12pt
			\begin{center}
			\begin{graph}(9,4)(0,0)
			\graphlinecolour{1}
			\graphlinewidth{.01}
			\grapharrowlength{.5}
			\newcommand{\arr}[3]{%
			\edge{#1}{#2}[\graphlinewidth{.1}]
			\diredge{#1}{#2}[\graphlinecolour{0}]
			\edgetext{#1}{#2}{\footnotesize #3}}

			\textnode{A}(0,0){$A$}
			\textnode{B}(9,0){$B$}
			\textnode{E}(4.5,4){$E_2$}
			
			\arr{A}{B}{$f$}
			\arr{A}{E}{$a_2$}
			\arr{E}{B}{$b_2$}

			\end{graph}			
			\end{center}			
			
			then there exists a unique arrow $g \in \text{Hom}(E_1, E_2)$ such that the following diagram commutes.

			\unitlength 12pt
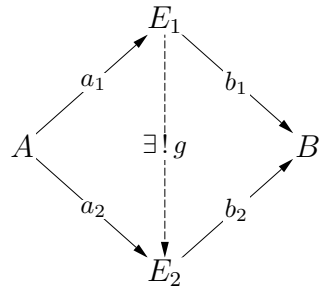
			\begin{center}
			\begin{graph}(9,8)(0,-4)
			\graphlinecolour{1}
			\graphlinewidth{.01}
			\grapharrowlength{.5}
			
			\newcommand{\arr}[3]{%
			\edge{#1}{#2}[\graphlinewidth{.1}]
			\diredge{#1}{#2}[\graphlinecolour{0}]
			\edgetext{#1}{#2}{\footnotesize #3}}
			
			\newcommand{\Uarr}[3]{%
			\edge{#1}{#2}[\graphlinewidth{.1}]
			\diredge{#1}{#2}[\graphlinecolour{0} \graphlinedash{3 1}]
			\edgetext{#1}{#2}{\footnotesize #3}}

			\textnode{A}(0,0){$A$}
			\textnode{B}(9,0){$B$}
			\textnode{E1}(4.5,4){$E_1$}
			\textnode{E2}(4.5,-4){$E_2$}
			
			\arr{A}{E1}{$a_1$}
			\arr{E1}{B}{$b_1$}
			\arr{A}{E2}{$a_2$}
			\arr{E2}{B}{$b_2$}
			
			\Uarr{E1}{E2}{$\exists \, ! \, g$}

			\end{graph}			
			\end{center}

		\end{itemize}
\end{definition}

One can easily build a taxon out of a category by constructing factorisations from the identity arrows.  Thus Joyal's normal play category is also a taxon.  However, in using a taxon rather than a category for \mis play positions, we may be able to avoid the problem we had in which needing an identity arrow from the position $0$ to $0$ forced what our arrow/strategy definition had to be. 

\chapter{Isomorphic Monoids}\label{chapter-congruent}

\section{Introduction}

Having come through our categorical interlude unscathed, we return our attention to \mis monoids.  We now examine the idea of isomorphic \mis monoids and results which arise from this.

\begin{definition}\label{def-iso-monoids}
	Two \mis monoids $\monoid{M}_{1}$ and $\monoid{M}_{2}$ are \textbf{isomorphic}, written as $\monoid{M}_{1} \cong \monoid{M}_{2}$, if there exists function $\varphi: \monoid{M}_{1} \to \monoid{M}_{2}$ such that
		\begin{enumerate}
			\item $\varphi$ is a \emph{tetrapartite monoid homomorphism}.  That is, 
				\begin{itemize}
					\item $\varphi(1) = 1$, 
					\item for $a,b \in \monoid{M}_{1}$, $\varphi(ab) = \varphi(a) \varphi(b)$;
					\item the outcome tetrapartitions (Definition \ref{def-outcome-tetra}) of $\monoid{M}_{1}$ and $\monoid{M}_{2}$ agree.  That is, for any $m \in \monoid{M}_{1}$, $o^-( m) = \mathcal{X} \iff o^-(\varphi(m)) = \mathcal{X}$.  
				\end{itemize}
			\item $\varphi^{-1}: \monoid{M}_{2} \to \monoid{M}_{1}$ exists and is also a tetrapartite monoid homomorphism.
		\end{enumerate}
\end{definition}

When determining whether two \mis monoids are isomorphic, it is tempting to assume that if we determine that they are isomorphic as monoids, i.e.\ satisfy the all the conditions in Definition \ref{def-iso-monoids} except for the outcome tetrapartitions agreeing, then this must force the outcome tetrapartitions to agree.  This is not the case, as the following example shows.

\begin{example}
	Consider $\monoid{M}_{\cl{1}}$ and $\monoid{M}_{\cl{\overline{1}}}$.  We saw in Example \ref{example-M-1} that
		\begin{align*}
		\monoid{M}_{\cl{1}} &= \ideal{1,a \mid a^2 = a} \\
		\Next &= \{1\} \\
		\Prev &= \emptyset \\
		\Left &= \emptyset \\
		\Right &= \{a\}.
	\end{align*}
	
	Similarly,
		\begin{align*}
		\monoid{M}_{\cl{\overline{1}}} &= \ideal{1,b \mid b^2 = b} \\
		\Next &= \{1\} \\
		\Prev &= \emptyset \\
		\Left &= \{b\} \\
		\Right &= \emptyset.
	\end{align*}	
	
	Under the map
		\begin{center}
		\begin{tabular}{cccc}
		$\varphi:$&$\monoid{M}_{\cl{1}}$&$\rightarrow$&$\monoid{M}_{\cl{\overline{1}}}$\\
		      &	$1$&	$\mapsto$&	$1$ \\
		      &	$a$&	$\mapsto$&	$b$ \\
		\end{tabular}
		\end{center}
	we see that the two monoids are isomorphic as monoids.  However, Left would be loathe to play in $\cl{1}$ rather than $\cl{\overline{1}}$, as the outcome tetrapartitions do not agree.  Thus, $\monoid{M}_{\cl{1}}$ and $\monoid{M}_{\cl{\overline{1}}}$ are not isomorphic as \mis monoids.
\end{example}

In Chapter \ref{chapter-examples}, we saw that $\monoid{M}_{\cl{*}} \cong \monoid{M}_{\cl{\sigma, \sigmab}}$.  In Chapter \ref{chapter-cardinality-left}, we extended this result to show $\monoid{M}_{\cl{*}} \cong \monoid{M}_{\cl{\L(\tau^{2n})}}$ (Theorem \ref{theorem-tau^2n_L-monoid}), showing that there are partizan positions whose \mis monoid is the same as that of $*$'s.  This chapter examines the phenomenon, and, in doing such, gives the two most important results of this thesis, namely necessary and sufficient conditions on a set of positions $\Upsilon$ such that $\monoid{M}_{\cl{\Upsilon}} \cong \monoid{M}_{\cl{*}}$ (Theorem \ref{theorem-P-iso-*}) and a construction theorem which builds all positions $\xi$ such that $\monoid{M}_{\cl{\xi}} \cong \monoid{M}_{\cl{*}}$ (Theorem \ref{theorem-xi-^-L-*} and Theorem \ref{theorem-*-built}).  The proofs for these two results are quite detailed.  As such, for the reader who simply wishes to make use of them, we place their statements here.

\noindent \textbf{Theorem 7.3.6.} \emph{Let $\Upsilon$ be a set of positions.  Then $\monoid{M}_{\cl{\Upsilon}} \cong \monoid{M}_{\cl{*}}$ if and only if the following are all satisfied:
		\begin{enumerate}
			\item  $\Upsilon$ contains a position other than 0;
			\item for $\xi \in \cl{\Upsilon}$, $o^-(\xi) = \Next \cup \Prev$;
			\item for $\xi \in \cl{\Upsilon}$, if $o^-(\xi) = \Next$, then $o^-(\xi^L) = \Prev$ for every Left option of $\xi$ and $o^-(\xi^R) = \Prev$ for every Right option of $\xi$.  That is, from an $\Next$ position, a player can never move to another $\Next$ position.		
		\end{enumerate}}

\noindent \textbf{Theorem 7.3.9} 
	\emph{
	Let $\xi_1, \xi_2, \ldots, \xi_n$, $\kappa_1, \kappa_2, \ldots, \kappa_n$ be positions such that $\monoid{M}_{\cl{\xi_i}} \cong \monoid{M}_{\cl{\kappa_j}} \cong \monoid{M}_{\cl{*}}$ for all $i \in \{1,2,\ldots,n\}$, $j \in \{1,2,\ldots,m\}$.  Then
	\begin{enumerate}
		\item If $o^-(\xi_i) = o^-(\kappa_j)$ for all $i \in \{1,2,\ldots,n\}$, $j \in \{1,2,\ldots,m\}$, then 
				\[\monoid{M}_{\cl{\combgame{\{\xi_1, \xi_2, \ldots, \xi_n\mid\kappa_1, \kappa_2, \ldots, \kappa_m\}}}} \cong \monoid{M}_{\cl{*}}.\]
		\item If $o^-(\xi_i) = \Prev$ for all $i \in \{1,2,\ldots,n\}$, then $
			\monoid{M}_{\cl{\combgame{\{\xi_1, \xi_2, \ldots, \xi_n\mid\cdot\}}}} \cong \monoid{M}_{\cl{*}}$. 
		\item If $o^-(\xi_i) =o^-(\kappa_j) = \Next$ for all $i \in \{1,2,\ldots,n\}$, $j \in \{1,2,\ldots,m\}$, then the following hold:
				\begin{enumerate}
					\item $\monoid{M}_{\cl{\combgame{\{\xi_1, \xi_2, \ldots, \xi_n \mid0\}}}} \cong  \monoid{M}_{\cl{*}}$;
					\item $\monoid{M}_{\cl{\combgame{\{\xi_1, \xi_2, \ldots, \xi_n\mid\kappa_1, \kappa_2, \ldots, \kappa_m, 0\}}}} \cong  \monoid{M}_{\cl{*}}$;
					\item $\monoid{M}_{\cl{\combgame{\{\xi_1, \xi_2, \ldots, \xi_n, 0\mid \kappa_1, \kappa_2, \ldots, \kappa_m, 0\}}}} \cong  \monoid{M}_{\cl{*}}$.
				\end{enumerate}
		\item For each $\xi_i$, $\kappa_j$, $\monoid{M}_{\cl{\overline{\xi_i}}} \cong \monoid{M}_{\cl{\overline{\kappa_j}}} \cong \monoid{M}_{\cl{*}}$.
	\end{enumerate}}	

Theorem \ref{theorem-*-built} gives that if $\xi$ is a position with $\monoid{M}_{\cl{\xi}} \cong \monoid{M}_{\cl{*}}$, then $\xi$ was formed by iteratively applying Theorem \ref{theorem-xi-^-L-*} to $*$.


\section{To the Left: $\L(\xi)$}

Before we begin our investigation into $\monoid{M}_{\cl{*}}$, we have the following result.  Recall from Chapter \ref{chapter-cardinality-left}, $\L(\xi)$ is the position $\combgame{\{\xi\mid\cdot\}}$.  Proposition \ref{prop-xi-as-xiL-cong-0}  gives a condition on $\xi$ such that $\monoid{M}_{\cl{\xi}} \cong \monoid{M}_{\cl{\L(\xi)}}$.

\begin{proposition}\label{prop-xi-as-xiL-cong-0}
	Suppose $\xi$ is an all-small position and Right's only available move from $\xi$ is to $0$.  Then $\L(\xi) \equiv 0 \imod{\cl{\L(\xi)}}$. 
\end{proposition}

\begin{proof}
	To begin, we will take $\mu \in \cl{\xi}$ and show that $o^-(\mu) = o^-(\mu + \L(\xi))$.  Since $\xi$ is an all-small position, this means that any option of $\xi$ is an all-small position.  As well, since the sum of all-small positions is itself an all-small position, this means that $\mu$ is an all-small position.
		
	Let us first suppose that Left moving first (or second) can win $\mu$.  Suppose Left is moving first (or second) in $\mu + \L(\xi)$.  Right cannot make any moves in $\L(\xi)$ until Left does, and so Left restricts all her moves to the $\mu$ component, playing her winning strategy there.  Thus Right makes the last move in the $\mu$ component, and, since $\mu$ is an all-small position, neither Right nor Left have any further moves available in this component.  Then Left is moving first in $\L(\xi)$, which she moves to $\xi$.  Right responds with his only move to 0, and so Left wins.  
	
	Now suppose Right moving first (or second) can win $\mu$.  Right plays his winning strategy in $\mu$.  If Left never plays in $\L(\xi)$, then since Right can win moving first (or second) in $\mu$, this means that Left made the last move in $\mu$, leaving Right to make the first move in $\L(\xi)$.  But Right has no move in $\L(\xi)$, so Right wins.  Otherwise, Left plays in $\L(\xi)$ before $\mu$ is finished.  So assume that Left has just made her move in $\L(\xi)$ to $\xi$.  Right responds in $\xi$ immediately by moving $\xi$ to 0, and then play resumes in the $\mu$ component, with Left moving next.  Since Right made the previous move in the $\mu$ component and is playing with his winning strategy, this ensures that Right can win the $\mu$ component.
	
	Thus, when we add $\L(\xi)$ to the set $\cl{\xi}$, its inclusion does not affect the outcomes of any of the elements of $\cl{\xi}$.
	
	Now consider the position $\mu + \L(\xi) + \L(\xi)$.  We can apply the same arguments as above to show that $o^-(\mu + \L(\xi) + \L(\xi)) = o^-(\mu)$.  An inductive argument gives us that $o^-(\mu + k\L(\xi)) = o^-(\mu)$ for all $k \in \mathbb{Z}^{\ge 0}$.  
	
	Take an arbitrary element of $\cl{\L(\xi)}$, $\mu + k\L(\xi)$ for $\mu \in \cl{\xi}$ and $k \in \mathbb{Z}^{\ge 0}$.  Then,
		\[ o^-(\mu + k\L(\xi) + \L(\xi)) = o^-(\mu + \L(\xi)).\]
	Therefore $\L(\xi) \equiv 0 \imod{\cl{\L(\xi)}}$.
\end{proof}

The true strength of Proposition \ref{prop-xi-as-xiL-cong-0} lies in the following corollary.

\begin{corollary}
	Suppose $\xi$ is an all-small position and Right's only available move from $\xi$ is to $0$.  Then $\monoid{M}_{\cl{\xi}} \cong \monoid{M}_{\cl{\L(\xi)}}$.  
\end{corollary}

\begin{proof}
	By Proposition \ref{prop-xi-as-xiL-cong-0}, $\L(\xi) \equiv 0 \imod{\cl{\L(\xi)}}$.  Thus, adjoining $\L(\xi)$ to $\cl{\xi}$ and taking the closure of the new set does not yield any new distinguishability relations.  Therefore the two monoids are isomorphic.  
\end{proof}

We have already seen this result in practise, namely with $*$ and $\sigma = \L(*)$ and the result that $\monoid{M}_{\cl{*}} \cong \monoid{M}_{\cl{\sigma}}$ (Example \ref{example-L*=sigma}).  However, we know that this is not always true; for example $\monoid{M}_{\cl{\sigma}} \not \equiv \monoid{M}_{\cl{\L(\sigma)}}$ (Proposition \ref{prop-sigma_L-infinite}).

The following open problem arises:

\begin{openproblem}\label{op-L-monoid}
	Classify all positions $\xi$ such that $\monoid{M}_{\cl{\xi}} \cong \monoid{M}_{\cl{\L(\xi)}}$. 
\end{openproblem}

\section{Isomorphic to $\monoid{M}_{\cl{*}}$}\label{sec-iso-*}

We now concern ourselves with the question of being isomorphic to $\monoid{M}_{\cl{*}}$.  The reasoning for this is simple; $*$ behaves nicely under the \mis play convention.  If a position has identical monoid properties to $*$, it will also behave nicely.

For this section, we will label the elements $\monoid{M}_{\cl{*}}$ as follows:

\begin{notation}\label{notation-*}
	\begin{align*}
	\monoid{M}_{\cl{*}} &= \ideal{1,a \mid a^2=1} \\
	\Next &= \{1\} \\
	\Prev &= \{a\} \\
	\Left &= \emptyset \\
	\Right &= \emptyset.
	\end{align*}
\end{notation}

Our first goal is not to try to enumerate all sets of positions $\Upsilon$ with $\monoid{M}_{\cl{\Upsilon}} \cong \monoid{M}_{\cl{*}}$; rather we will first clarify what was meant in the preceding paragraph by \emph{behaving nicely}.

\begin{proposition}\label{proposition-P-*-monoids}
	Suppose $\Upsilon$ is a set of positions such that $\monoid{M}_{\cl{\Upsilon}} \cong \monoid{M}_{\cl{*}}$.  Then the following are true:
		\begin{enumerate}
			\item\label{item-xi-np} If $\xi \in \cl{\Upsilon}$, then $o^-(\xi) = \Next \cup \Prev$.
			
			\item $1, \overline{1} \not \in \cl{\Upsilon}$ while $* \in \cl{\Upsilon}$.  
			
			\item\label{item-proposition-P-*-monoids-3} If $\xi \in \cl{\Upsilon}$ with $o^-(\xi) = \Next$, then $o^-(\xi^L) = \Prev$ for every Left option of $\xi$ and $o^-(\xi^R) = \Prev$ for every Right option of $\xi$. That is, from an $\Next$ position, a player can never move to another $\Next$ position.
		\end{enumerate}
\end{proposition}

\begin{proof}\text{}
	\begin{enumerate}
		\item The outcome tetrapartition of $\monoid{M}_{\cl{*}}$ has $\Left = \emptyset$ and $\Right = \emptyset$.  Since $\monoid{M}_{\cl{\Upsilon}} \cong \monoid{M}_{\cl{*}}$, their outcome tetrapartitions must agree.  That is, in $\monoid{M}_{\cl{\Upsilon}}$, $\Left = \emptyset$ and $\Right = \emptyset$.  Therefore, all elements $\xi \in \cl{\Upsilon}$ are either $\Next$ or $\Prev$ positions.  
		
		\item We have just shown that no elements in $\cl{\Upsilon}$ are  $\Left$ or $\Right$.  Since $o^-(1) = \Right$ and $o^-(\overline{1}) = \Left$, this means $1, \overline{1} \not \in \cl{\Upsilon}$.  
		
		Suppose $\Upsilon = 0$.  Then there are no $\Prev$ positions in $\monoid{M}_{\cl{\Upsilon}}$, and so the outcome tetrapartitions of $\monoid{M}_{\cl{\Upsilon}}$ and $\monoid{M}_{\cl{*}}$ do not agree.  Since $\Upsilon$ is closed, $0 \in \Upsilon$, but we have just shown that $\{0\} \subset \Upsilon$.  Therefore $\Upsilon$ must contain an element of birthday 1.  Since $\Upsilon$ contains neither $1$ nor $\overline{1}$, $*$ must be an element of $\Upsilon$.
		
		\item Suppose that there exists a position $\xi \in \Upsilon$ such that $o^-(\xi) = \Next$ and there exists a Left option such that $o^-(\xi^L) = \Next$.  Moreover, suppose $\xi$ is the least element (in terms of birthday) with this property (for either Left or Right).  We will show that $*$ distinguishes $\xi$ and 0, so the number of elements in the $\Next$ portion of $\monoid{M}_{\cl{\Upsilon}}$ is at least two, while the number of elements in the $\Next$ portion of $\monoid{M}_{\cl{*}}$ is one.  Thus the two monoids cannot be isomorphic.
		
		We know that $o^-(*) = \Prev$, but we claim that Left moving first can win $o^-(* + \xi)$.  Below shows how Left can win moving first, recalling that $o^-(\xi^L) = \Next$ and, since $\xi$ is the least element with the desired property, $o^-(\xi^{LR}) = \Prev$ for all Right options of $\xi^L$.  
			
			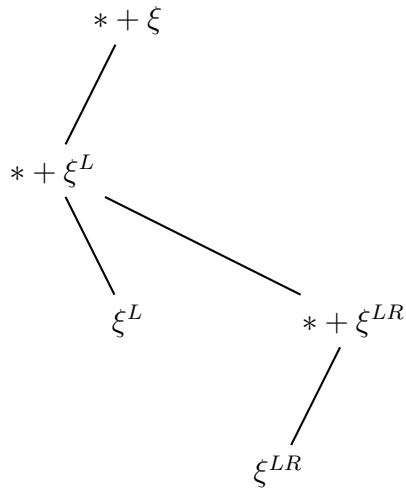
\begin{figure}[htb]
			\begin{center}
			\begin{pspicture}\unit{0.5cm}(6,6)
			\put(1,6)
			{\cgtree
				{
					{*+\xi}({* + \xi^L}(|{\xi^L}+{*+\xi^{LR}}({\xi^{LR}}|))|)
				}		
			}
			\end{pspicture}
			\end{center}
			\caption{Left can win $* + \xi$ moving first.}
			\label{fig-*+xi-Left-win}
			\end{figure}	
			
		Therefore $o^-(* + \xi) \not = \Prev$, and so $0 \not \equiv \xi \imod{\cl{\Upsilon}}$.\qedhere
	\end{enumerate}
\end{proof}

Our regular complaint for \mis play games is that the sum of any two positions may be any outcome irrespective of the outcomes of the positions themselves \cite{MESDAL}.  However, when we have sets of positions $\Upsilon_1$ and $\Upsilon_2$ with 
	\[ \monoid{M}_{\cl{\Upsilon_1}} \cong \monoid{M}_{\cl{\Upsilon_2}} \cong \monoid{M}_{\cl{*}},\]
we are able to determine the outcome of sums of positions where the summands come from both $\Upsilon_1$ and $\Upsilon_2$. 

\begin{theorem}\label{theorem-p1-p2-*}
	Suppose $\Upsilon_1$, $\Upsilon_2$ are sets of positions such that 
		\[ \monoid{M}_{\cl{\Upsilon_1}} \cong \monoid{M}_{\cl{\Upsilon_2}} \cong \monoid{M}_{\cl{*}},\]
	where 
		\begin{align*}
		&\mathscr{Q}_1: \cl{\Upsilon_1} \to \monoid{M}_{\cl{\Upsilon_1}},\\ 
		&\mathscr{Q}_2: \cl{\Upsilon_2} \to \monoid{M}_{\cl{\Upsilon_2}}
		\end{align*}
	are the canonical quotient maps from the closures to their respective monoids and
		\begin{align*}
			&\varphi_1:\monoid{M}_{\cl{\Upsilon_1}} \to \monoid{M}_{\cl{*}}, \\
			&\varphi_2:\monoid{M}_{\cl{\Upsilon_2}} \to \monoid{M}_{\cl{*}}
		\end{align*}
	are the \mis monoid isomorphisms.  Then, for $\xi_1 \in \cl{\Upsilon_1}$, $\xi_2 \in \cl{\Upsilon_2}$, 
		\begin{enumerate}
			\item $o^-(\xi_1 + \xi_2) = \Next \cup \Prev$;
			\item $o^-(\xi_1 + \xi_2) = o^-([\varphi_1 \circ \mathscr{Q}_1(\xi_1)] \cdot [\varphi_2 \circ \mathscr{Q}_2(\xi_2)])$,
		\end{enumerate}
	where we are explicitly using $\cdot$ to denote the binary operation in $\monoid{M}_{\cl{*}}$.
		
\end{theorem}

\begin{proof}
	We proceed by induction on the options of $\xi_1 + \xi_2$.  If $\xi_1 + \xi_2 = 0$, then both $\xi_1$ and $\xi_2$ are 0, and the result holds.
	
	Now suppose the results of Theorem \ref{theorem-p1-p2-*} are true for all options of $\xi_1 + \xi_2$ and consider $\xi_1 + \xi_2$.  We will firstly show that $o^-(\xi_1 + \xi_2) = \Next \cup \Prev$.
	
	Suppose, contrary to what we must show, that $o^-(\xi_1 + \xi_2) = \Left$.  Then there exists $(\xi_1 + \xi_2)^L$ such that $o^-((\xi_1 + \xi_2)^L) = \Prev \cup \Left$, while for every $(\xi_1 + \xi_2)^R$, $o^-((\xi_1 + \xi_2)^R) = \Next \cup \Left$.
	Since these are both options of $\xi_1 + \xi_2$, induction gives that these cannot be elements of $\Left$.  Thus, there exists $(\xi_1 + \xi_2)^L$ such that $o^-((\xi_1 + \xi_2)^L) = \Prev$, while for all $(\xi_1 + \xi_2)^R$, $o^-((\xi_1 + \xi_2)^R) = \Next$.
	
	Take our position $(\xi_1 + \xi_2)^L$ such that $o^-((\xi_1 + \xi_2)^L) = \Prev$.   Then, either $o^-(\xi_1^L + \xi_2) = \Prev$ or $o^-(\xi_1 + \xi_2^L) = \Prev$.  Suppose, without loss of generality, that $o^-(\xi_1^L + \xi_2) = \Prev$, and fix one such $\xi_1^L$. 
	
	We will now assume that $\xi_1^R$ exists.  Since $o^-((\xi_1 + \xi_2)^R) = \Next$ for every Right option, we have $o^-(\xi_1^R + \xi_2) = \Next$.  Fix one such $\xi_1^R+\xi_2$.  By induction
		\begin{align*}
			o^-(\xi_1^L + \xi_2)& = o^-([\varphi_1 \circ \mathscr{Q}_1(\xi_1^L)] \cdot [\varphi_2 \circ \mathscr{Q}_2(\xi_2)]) \\
			o^-(\xi_1^R + \xi_2) &= o^-([\varphi_1 \circ \mathscr{Q}_1(\xi_1^R)] \cdot [\varphi_2 \circ \mathscr{Q}_2(\xi_2)]). 
		\end{align*}
	Therefore
		\begin{align*}
			o^-([\varphi_1 \circ \mathscr{Q}_1(\xi_1^L)] \cdot [\varphi_2 \circ \mathscr{Q}_2(\xi_2)]) &= \Prev\\
			o^-([\varphi_1 \circ \mathscr{Q}_1(\xi_1^R)] \cdot [\varphi_2 \circ \mathscr{Q}_2(\xi_2)]) &=\Next.
		\end{align*}
	Both $[\varphi_1 \circ \mathscr{Q}_1(\xi_1^L)] \cdot [\varphi_2 \circ \mathscr{Q}_2(\xi_2)]$ and $[\varphi_1 \circ \mathscr{Q}_1(\xi_1^R)] \cdot [\varphi_2 \circ \mathscr{Q}_2(\xi_2)]$ are elements of $\monoid{M}_{\cl{*}}$, which only has two elements, $a$ with $o^-(a) = \Prev$ and $1$ with $o^-(1) = \Next$.  Since 
		\[ \monoid{M}_{\cl{\Upsilon_1}} \cong \monoid{M}_{\cl{\Upsilon_2}} \cong \monoid{M}_{\cl{*}},\]
	combining all these facts gives		
		\begin{align*}
			a &= [\varphi_1 \circ \mathscr{Q}_1(\xi_1^L)] \cdot [\varphi_2 \circ \mathscr{Q}_2(\xi_2)] \\
			1 &= [\varphi_1 \circ \mathscr{Q}_1(\xi_1^R)] \cdot [\varphi_2 \circ \mathscr{Q}_2(\xi_2)].
		\end{align*}
	Similarly, since $\varphi_2 \circ \mathscr{Q}_2(\xi_2)$ is an element of $\monoid{M}_{\cl{*}}$ either $\varphi_2 \circ \mathscr{Q}_2(\xi_2) = 1$ or $\varphi_2 \circ \mathscr{Q}_2(\xi_2) = a$.  Suppose the first.  Then
		\[ a = [\varphi_1 \circ \mathscr{Q}_1(\xi_1^L)] \cdot 1 \implies a = \varphi_1 \circ \mathscr{Q}_1(\xi_1^L) \implies o^-(\xi_1^L) = \Prev,\]
	and
		\[1 = [\varphi_1 \circ \mathscr{Q}_1(\xi_1^R)] \cdot 1 \implies 1 = \varphi_1 \circ \mathscr{Q}_1(\xi_1^R) \implies o^-(\xi_1^R) = \Next. \]
	Since $\xi_1^R$ was an arbitrary Right option of $\xi_1$, the above means that $o^-(\xi_1) = \Left$, which contradicts Proposition \ref{proposition-P-*-monoids}\eqref{item-xi-np}.  Therefore, we must assume that $\varphi_2 \circ \mathscr{Q}_2(\xi_2) = a$.  Then
		\[ a = [\varphi_1 \circ \mathscr{Q}_1(\xi_1^L)] \cdot a \implies 1 = \varphi_1 \circ \mathscr{Q}_1(\xi_1^L) \implies o^-(\xi_1^L) = \Next,\]
	and
		\[1 = [\varphi_1 \circ \mathscr{Q}_1(\xi_1^R)] \cdot a \implies a = \varphi_1 \circ \mathscr{Q}_1(\xi_1^R) \implies o^-(\xi_1^R) = \Prev. \]
	Since $o^-(\xi_1^L) = \Next$, Theorem \ref{proposition-P-*-monoids}\eqref{item-proposition-P-*-monoids-3} gives that $o^-(\xi) = \Prev$, otherwise we would have an $\Next$ position moving to an $\Next$ position.  But $o^-(\xi_1^R) = \Prev$, so we have a $\Prev$ position moving to a $\Prev$ position, which is a contradiction.  
	
	Therefore neither $\varphi_2 \circ \mathscr{Q}_2(\xi_2) = 1$ nor $\varphi_2 \circ \mathscr{Q}_2(\xi_2) = a$ holds, but one of the two statements must be true, giving us a contradiction.  We arrived at this contradiction by assuming that $\xi_1^R$ exists.  Therefore, Right must not have any moves available in $\xi_1$.  Thus, if Right has a move available in $\xi_1 + \xi_2$ it must be to $\xi_1 + \xi_2^R$.   
	
	Fix a $\xi_2^R$ and consider $\xi_1^L + \xi_2$ and $\xi_1 + \xi_2^R$.  By induction,
		\begin{align*}
			o^-(\xi_1^L + \xi_2) &= o^-([\varphi_1 \circ \mathscr{Q}_1(\xi_1^L)] \cdot [\varphi_2 \circ \mathscr{Q}_2(\xi_2)]) \\
			o^-(\xi_1 + \xi_2^R) &= o^-([\varphi_1 \circ \mathscr{Q}_1(\xi_1)] \cdot [\varphi_2 \circ \mathscr{Q}_2(\xi_2^R)]).
		\end{align*}
	Similarly to when we assumed that $\xi_1^R$ existed, we have 		\begin{align*}
			a &= [\varphi_1 \circ \mathscr{Q}_1(\xi_1^L)] \cdot [\varphi_2 \circ \mathscr{Q}_2(\xi_2)] \\
			1 &= [\varphi_1 \circ \mathscr{Q}_1(\xi_1)] \cdot [\varphi_2 \circ \mathscr{Q}_2(\xi_2^R)].
		\end{align*}
	
	Since Right has no moves available in $\xi_1$, Right wins $\xi_1$ moving first.  By Proposition \ref{proposition-P-*-monoids}\eqref{item-xi-np}, this means $o^-(\xi_1) = \Next$.  By Proposition \ref{proposition-P-*-monoids}\eqref{item-proposition-P-*-monoids-3}, $o^-(\xi_1^L) = \Prev$.  Therefore
		\begin{align*}
			\varphi_1 \circ \mathscr{Q}_1(\xi_1) &= 1,\\
			\varphi_1 \circ \mathscr{Q}_1(\xi_1^L) &= a.
		\end{align*}
	Then
		\[ a = a \cdot [\varphi_2 \circ \mathscr{Q}_2(\xi_2)] \implies 1 = \varphi_2 \circ \mathscr{Q}_2(\xi_2) \implies o^-(\xi_2) = \Next,\]
	and
		\[1 = 1 \cdot [\varphi_2 \circ \mathscr{Q}_2(\xi_2^R)] \implies 1 = \varphi_2 \circ \mathscr{Q}_2(\xi_2^R) \implies o^-(\xi_2^R) = \Next. \]
	Thus we have an $\Next$ position moving to an $\Next$ position, which contradicts Proposition \ref{proposition-P-*-monoids}\eqref{item-proposition-P-*-monoids-3}.  
	
	Thus Right must not have a move available in either $\xi_1$ or $\xi_2$, and so Right has no move available in their sum.  Therefore Right can win moving first in $\xi_1 + \xi_2$, which contradicts our assumption that $o^-(\xi_1 + \xi_2) = \Left$.  A similar argument shows that $o^-(\xi_1 + \xi_2) \not = \Right$, leaving the result that $o^-(\xi_1 + \xi_2) = \Next \cup \Prev$.  
  	
	We will now show 
		\[ o^-(\xi_1 + \xi_2) = o^-([\varphi_1 \circ \mathscr{Q}_1(\xi_1)] \cdot [\varphi_2 \circ \mathscr{Q}_2(\xi_2)]).\]
	We have that
		\[ \varphi_1 \circ \mathscr{Q}_1(\xi_1), \varphi_2 \circ \mathscr{Q}_2(\xi_2)) \in \{1,a\}\]
	and so there are four possibilities for 
		\[ (\varphi_1 \circ \mathscr{Q}_1(\xi_1), \varphi_2 \circ \mathscr{Q}_2(\xi_2))),\]
	namely 
		\[ (1,1), \text{ } (1,a), \text{ } (a,1), \text{ and } (a,a).\]
	We will examine each of these four possibilities separately.  
	
	\begin{enumerate}
	\item
	Suppose
		\begin{align*}
			\varphi_1 \circ \mathscr{Q}_1(\xi_1) &= 1\\
			\varphi_2 \circ \mathscr{Q}_2(\xi_2) &= 1.
		\end{align*}
	Then 
		\[ [\varphi_1 \circ \mathscr{Q}_1(\xi_1)] \cdot [\varphi_2 \circ \mathscr{Q}_2(\xi_2)] = 1 \cdot 1 = 1, \]
	and so
		\[o^-([\varphi_1 \circ \mathscr{Q}_1(\xi_1)] \cdot [\varphi_2 \circ \mathscr{Q}_2(\xi_2)]) = \Next.\]  
	We want $o^-(\xi_1 + \xi_2) = \Next$.  Suppose, contrary to what we must show, that $o^-(\xi_1 + \xi_2) = \Prev$.  Suppose Left has no move in $\xi_1 + \xi_2$.  Then Left wins $\xi_1 + \xi_2$ moving first, contradicting that $o^-(\xi_1 + \xi_2) = \Prev$.  Thus Left must have an available move in $\xi_1 + \xi_2$, and, without loss of generality, we will assume that Left can move to $\xi_1^L + \xi_2$.  Since this is a losing move for Left, $o^-(\xi_1^L + \xi_2) = \Next \cup \Right$.  By induction, $o^-(\xi_1^L + \xi_2) = \Next$ and 
		\[o^-(\xi_1^L + \xi_2) = o^-([\varphi_1 \circ \mathscr{Q}_1(\xi_1^L)] \cdot [\varphi_2 \circ \mathscr{Q}_2(\varphi_2)]).\]
	So
		\[ o^-([\varphi_1 \circ \mathscr{Q}_1(\xi_1^L)] \cdot [\varphi_2 \circ \mathscr{Q}_2(\varphi_2)]) = \Next,\]
	giving,
		\[ 1 = [\varphi_1 \circ \mathscr{Q}_1(\xi_1^L)] \cdot [\varphi_2 \circ \mathscr{Q}_2(\xi_2)].\]
	We assumed that $\varphi_2 \circ \mathscr{Q}_2(\xi_2) = 1$, giving us
		\[ 1 = [\varphi_1 \circ \mathscr{Q}_1(\xi_1^L)] \cdot 1 \implies 1 = \varphi_1 \circ \mathscr{Q}_1(\xi_1^L)\implies o^-(\xi_1^L) = \Next.\]
	Since $\xi_1^L$ was arbitrary Left move, this is true for all $\xi_1^L$.  Therefore $o^-(\xi_1) = \Right \cup \Prev$.  By Theorem \ref{proposition-P-*-monoids}\eqref{item-xi-np}, $o^-(\xi_1) = \Prev$, but $\varphi_1 \circ \mathscr{Q}_1(\xi_1) = 1$, which implies $o^-(\xi) = \Next$, a contradiction.  Therefore $o^-(\xi_1 + \xi_2) = \Next$.
	
	\item 
	Suppose
		\begin{align*}
			\varphi_1 \circ \mathscr{Q}_1(\xi_1) &=1 \\
			\varphi_2 \circ \mathscr{Q}_2(\xi_2) &=a.
		\end{align*}
	Then
		\[[\varphi_1 \circ \mathscr{Q}_1(\xi_1)] \cdot [\varphi_2 \circ \mathscr{Q}_2(\xi_2)] = 1 \cdot a = a,\]
	and so
		\[ o^-([\varphi_1 \circ \mathscr{Q}_1(\xi_1)] \cdot [\varphi_2 \circ \mathscr{Q}_2(\xi_2)]) = \Prev.\]
	We want $o^-(\xi_1 + \xi_2) = \Prev$.  Suppose, contrary to what we must show, that $o^-(\xi_1 + \xi_2) = \Next$.  Then  Left can win $\xi_1 + \xi_2$ moving first.  That is, there exists $\xi_1^L$, $\xi_2^L$ such that either $\xi_1^L + \xi_2$ or $\xi_1 + \xi_2^L$ is a winning move for Left.  That is, by induction, either $o^-(\xi_1^L + \xi_2) = \Prev$ or $o^-(\xi_1 + \xi_2^L) = \Prev$.
	
	Suppose the former.  By induction
		\[ o^-([\xi_1^L + \xi_2) = o^-(\varphi_1 \circ \mathscr{Q}_1(\xi_1^L)] \cdot [\varphi_2 \circ \mathscr{Q}_2(\xi_2)]).\]
	So
		\[ o^-([\varphi_1 \circ \mathscr{Q}_1(\xi_1^L)] \cdot [\varphi_2 \circ \mathscr{Q}_2(\xi_2)]) = \Prev\]
	which implies 
		\[ a = [\varphi_1 \circ \mathscr{Q}_1(\xi_1^L)] \cdot [\varphi_2 \circ \mathscr{Q}_2(\xi_2)].\]
	We assumed $\varphi_2 \circ \mathscr{Q}_2(\xi_2) = a$, giving us
		\[ a = [\varphi_1 \circ \mathscr{Q}_1(\xi_1^L)] \cdot a \implies 1 = \varphi_1 \circ \mathscr{Q}_1(\xi_1^L) \implies o^-(\xi_1^L) = \Next.\]
	But $\varphi_1 \circ \mathscr{Q}_1(\xi_1) = 1$, so $o^-(\xi_1) = \Next$, contradicting Proposition \ref{proposition-P-*-monoids}\eqref{item-proposition-P-*-monoids-3}.  
	
	Therefore Left's winning move must be to $\xi_1 + \xi_2^L$ with $o^-(\xi_1 + \xi_2^L) = \Prev$. By induction,
		\[ o^-(\xi_1 + \xi_2^L) = o^-([\varphi_1 \circ \mathscr{Q}_1(\xi_1)] \cdot [\varphi_2 \circ \mathscr{Q}_2(\xi_2^L)]).\]
	So
		\[ o^-([\varphi_1 \circ \mathscr{Q}_1(\xi_1)] \cdot [\varphi_2 \circ \mathscr{Q}_2(\xi_2^L)]) = \Prev,\]
	giving,
		\[ a = [\varphi_1 \circ \mathscr{Q}_1(\xi_1)] \cdot [\varphi_2 \circ \mathscr{Q}_2(\xi_2^L)].\]
	We assumed $\varphi_1 \circ \mathscr{Q}_1(\xi_1) =1$, giving us
		\[ a = 1 \cdot  [\varphi_2 \circ \mathscr{Q}_2(\xi_2^L)] \implies a =  \varphi_2 \circ \mathscr{Q}_2(\xi_2^L) \implies o^-(\xi_2^L) = \Prev.\]
	But $\varphi_2 \circ \mathscr{Q}_2(\xi_2) = a$, giving $o^-(\xi_2) = \Prev$.  Therefore, we have a $\Prev$ position moving to a $\Prev$ position, a contradiction.
	
	Thus, there does not exist either $\xi_1^L$ or $\xi_2^L$ such that $\xi_1^L + \xi_2$ or $\xi_1 + \xi_2^L$ is a winning move for Left.  Thus $o^-(\xi_1 + \xi_2) = \Right \cup \Prev$.  However, we already showed that $o^-(\xi_1 + \xi_2) = \Next \cup \Prev$, which means $o^-(\xi_1 + \xi_2) = \Prev$, as required.

	\item	Suppose
		\begin{align*}
			\varphi_1 \circ \mathscr{Q}_1(\xi_1) &=a \\
			\varphi_2 \circ \mathscr{Q}_2(\xi_2) &=1.
		\end{align*}
	An argument similar to the previous case works to show the desired result.

	\item	
	Suppose
		\begin{align*}
			\varphi_1 \circ \mathscr{Q}_1(\xi_1) &= a\\
			\varphi_2 \circ \mathscr{Q}_2(\xi_2) &= a.
		\end{align*}
	Then
		\[ [\varphi_1 \circ \mathscr{Q}_1(\xi_1)] \cdot [\varphi_2 \circ \mathscr{Q}_2(\xi_2)] = a \cdot a = 1, \]
	and so
		\[ o^-([\varphi_1 \circ \mathscr{Q}_1(\xi_1)] \cdot [\varphi_2 \circ \mathscr{Q}_2(\xi_2)]) = \Next.\]
	We want $o^-(\xi_1 + \xi_2) = \Next$.  Suppose, contrary to what we must show, that $o^-(\xi_1 + \xi_2) = \Prev$.    Then Left cannot win $\xi_1 + \xi_2$ moving first.  Left must have a move in $\xi_1 + \xi_2$, otherwise she would win moving first.  Assume, without loss of generality, that Left can move to $\xi_1^L + \xi_2$.  By induction, $o^-(\xi_1^L + \xi_2) = \Next$, and
		\[ o^-(\xi_1^L + \xi_2) = o^-([\varphi_1 \circ \mathscr{Q}_1(\xi_1^L)] \cdot [\varphi_2 \circ \mathscr{Q}_2(\xi_2)]).\]
	That is,
		\[  o^-([\varphi_1 \circ \mathscr{Q}_1(\xi_1^L)] \cdot [\varphi_2 \circ \mathscr{Q}_2(\xi_2)]) = \Next, \]
	and so,
		\[ 1 = [\varphi_1 \circ \mathscr{Q}_1(\xi_1^L)] \cdot [\varphi_2 \circ \mathscr{Q}_2(\xi_2)].\]
	We assumed $\varphi_2 \circ \mathscr{Q}_2(\xi_2) = a$, giving us
		\[ 1 = [\varphi_1 \circ \mathscr{Q}_1(\xi_1^L)] \cdot a \implies a = \varphi_1 \circ \mathscr{Q}_1(\xi_1^L) \implies o^-(\xi_1^L) = \Prev.\]
	But $\varphi_1 \circ \mathscr{Q}_1(\xi_1) = a$, so $o^-(\xi_1) = \Prev$, giving us a $\Prev$ position which moves to a $\Prev$ position, a contradiction.  Therefore $o^-(\xi_1 + \xi_2) = \Next$, as required.
	\end{enumerate}

	This shows 
		\[ o^-(\xi_1 + \xi_2) = o^-([\varphi_1 \circ \mathscr{Q}_1(\xi_1)] \cdot [\varphi_2 \circ \mathscr{Q}_2(\xi_2)])\]
	and completes the proof.
\end{proof}

Theorem \ref{theorem-p1-p2-*} says that to find the outcome of two positions from differing sets with monoids isomorphic to $\monoid{M}_{\cl{*}}$, all we need do is determine to what each position is equivalent in $\monoid{M}_{\cl{*}}$, multiply these elements, reduce in $\monoid{M}_{\cl{*}}$, and take the outcome result.  

The importance of this result can not be stressed enough.  For the first time, we are able to take elements from two different sets whose behaviour is well-understood and determine their outcome without having to compute the monoid of the closure of these positions explicitly.  Moreover, even while this result currently is only proven for monoids which are isomorphic to $\monoid{M}_{\cl{*}}$, we have found partizan positions whose monoids are isomorphic to $\monoid{M}_{\cl{*}}$, such as $\sigma$ (Example \ref{example-L*=sigma}) and $\L(\tau^{2n})$ (Theorem \ref{theorem-tau^2n_L-monoid}).  Thus, we have partizan \mis play positions which behave in exactly the same way as the simplest, non-trivial impartial position.  

An example of using Theorem \ref{theorem-p1-p2-*} now follows:

\begin{example}
	Consider the positions $8 \sigma$ and $3 \L(\tau^{4})$.  By Example \ref{example-L*=sigma} and Theorem \ref{theorem-tau^2n_L-monoid}, we know 
		\[ \monoid{M}_{\cl{\sigma}} \cong \monoid{M}_{\cl{\L(\tau^{2n})}} \cong \monoid{M}_{\cl{*}},\]
	and so we can apply Theorem \ref{theorem-p1-p2-*} to determine $o^-(8 \sigma + 3 \L(\tau^{4}))$.    In $\monoid{M}_{\cl{*}}$, both $\sigma$ and $\L(\tau^4)$ are mapped to 1, and so 
		\begin{align*}
			o^-(8 \sigma + 3 \L(\tau^4)) 
			&= o^-(1^8 \cdot 1^3) \\
			&= o^-(1) \\
			&= \Next
		\end{align*}
	Therefore $8 \sigma + 3 \L(\tau^4)$ is an $\Next$ position.
\end{example}

Our next concern is whether we can determine for a set of positions $\Upsilon$ if $\monoid{M}_{\cl{\Upsilon}} \cong \monoid{M}_{\cl{*}}$ without explicitly calculating the monoid.  We can, as shown by the following proposition:

\begin{proposition}\label{prop-p-is-*}
	Suppose $\Upsilon$ is a set of positions such that
		\begin{enumerate}
			\item $\Upsilon$ contains a position other than 0;
			\item\label{item-n-cup-prev} for $\xi \in \cl{\Upsilon}$, $o^-(\xi) = \Next \cup \Prev$;
			\item\label{item-n-to-n-no} for $\xi \in \cl{\Upsilon}$, if $o^-(\xi) = \Next$, then $o^-(\xi^L) = \Prev$ for every Left option of $\xi$ and $o^-(\xi^R) = \Prev$ for every Right option of $\xi$.  That is, from an $\Next$ position, a player can never move to another $\Next$ position.
		\end{enumerate}
	Then $\monoid{M}_{\cl{\Upsilon}} \cong \monoid{M}_{\cl{*}}$.
\end{proposition}

\begin{proof}
	Since $\Upsilon \not = \{0\}$, this means $|\cl{\Upsilon}| \ge 2$.  Thus, there must be an element of birthday 1 in $\cl{\Upsilon}$.  Since $o^-(1), o^-(\overline{1}) \not = \Next \cup \Prev$, neither of these positions are in $\cl{\Upsilon}$.  The only remaining birthday one position is $*$, so $* \in \cl{\Upsilon}$.  
	
	We claim that for $\xi \in \cl{\Upsilon}$,
		\begin{align*}
			0 &\equiv \xi \imod{\cl{\Upsilon}} \text{ if } o^-(\xi) = \Next,\\
			* &\equiv \xi \imod{\cl{\Upsilon}} \text{ if } o^-(\xi) = \Prev.
		\end{align*}
	We proceed by induction on the options of $\xi$.  If $\xi =0$ or $*$, then the result is clearly true. 
	
	Consider $\xi$ and suppose the claim is true for all options of $\xi$.  Fix $\nu \in \cl{\Upsilon}$.  We want
		\begin{align*}
			o^-(\nu) &= o^-(\nu + \xi) \text{ if } o^-(\xi) = \Next,\\
			o^-(\nu + *) &= o^-(\nu + \xi) \text{ if } o^-(\xi) = \Prev.
		\end{align*}
	We proceed by induction on the options of $\nu$.  If $\nu = 0$, then the result holds and shows the base case.  Now consider $\nu$ and suppose the above is true for all the options of $\nu$.  We break our proof into two cases depending on whether $o^-(\xi) = \Next$ or $o^-(\xi) = \Prev$.
	
	\begin{enumerate}
	\item 
	Suppose $o^-(\xi) = \Next$.
		\begin{enumerate}
			\item Suppose $o^-(\nu) = \Next$.  We want $o^-(\nu + \xi) = \Next$.  Consider Left moving first in $\nu + \xi$.  Since $o^-(\nu) = \Next$, either $\nu = 0$, which is dealt with in the base case, or there exists $\nu^L$ such that $o^-(\nu^L) = \Prev$.  Then, by induction on $\nu$, $o^-(\nu^L) = o^-(\nu^L + \xi)$.  Therefore $o^-(\nu^L + \xi) = \Prev$, so Left can win $\nu + \xi$ moving first.  Similarly, Right can win $\nu + \xi$ moving first.  Therefore $o^-(\nu + \xi) = \Next$.  
			
			\item Suppose $o^-(\nu) = \Prev$.  We want $o^-(\nu + \xi) = \Prev$.  
			
			Suppose Left has no moves in $\nu + \xi$.  Then Left has no moves in $\nu$, and so Left wins moving first in $\nu$.  But, by Condition \eqref{item-n-cup-prev} in the statement of the proposition, $o^-(\nu) = \Next \cup \Prev$.  This means $o^-(\nu) = \Next$, contradicting our assumption that $o^-(\nu) = \Prev$.  Therefore Left must have an opening move in $\nu + \xi$.  
			
			Consider Left moving first in $\nu + \xi$.  She has two possibilities:
				\begin{enumerate}
					\item $\nu^L + \xi$:  By induction on $\nu$, $o^-(\nu^L + \xi) = o^-(\nu^L)$.  Since $o^-(\nu) = \Prev$, we have $o^-(\nu^L) = \Next$, so $o^-(\nu^L + \xi) = \Next$, meaning this is a bad opening move for Left.
					
					\item $\nu + \xi^L$:  Since $o^-(\xi) = \Next$, Condition \eqref{item-n-to-n-no} in the statement of the Proposition says that $o^-(\xi^L) = \Prev$.  By induction on $\xi$, $o^-(\nu + \xi^L) = o^-(\nu + *)$.  Since $o^-(\nu) = \Prev$, Left or Right moving first in $\nu + *$ can move to $\nu$, so $o^-(\nu + *) = \Next$, meaning $o^-(\nu + \xi^L) = \Next$, and, again, this is a bad opening move for Left.
				\end{enumerate}
			Therefore Left has no good opening moves in $\nu + \xi$.  Similarly, neither does Right.  Therefore $o^-(\nu + \xi) = \Prev$.
		\end{enumerate}
	
	\item Suppose $o^-(\xi) = \Prev$.  
		\begin{enumerate}
			\item Suppose $o^-(\nu + *) = \Next$.  We want $o^-(\nu + \xi) = \Next$.  Since $o^-(\xi) = \Prev$, we have that $\xi^L$ exists.  We will show that Left moving in $\nu + \xi$ to $\nu + \xi^L$ is a winning move.  Since $o^-(\xi) = \Prev$, we have $o^-(\xi^L) = \Next$.  By induction on $\xi$, $o^-(\nu) = o^-(\nu + \xi^L)$.  Therefore $o^-(\nu + \xi^L) = \Prev$, and so Left wins moving first in $\nu + \xi$.  Similarly, Right can win moving first in $\nu + \xi$.  Therefore $o^-(\nu + \xi) = \Next$.
			
			\item Suppose $o^-(\nu + *) = \Prev$.  We want $o^-(\nu + \xi) = \Prev$.  Since $o^-(\nu + *) = \Prev$, we have $o^-(\nu) = \Next$, otherwise Left or Right moving first in $\nu + *$ would move to $\nu$ and win, a contradiction.  
			
			Consider Left moving first in $\nu + \xi$.  Since Left has a move in $\xi$, Left has a move in $\nu + \xi$.  Either Left moves to
				\begin{enumerate}
					\item $\nu^L + \xi$:  Since $o^-(\nu) = \Next$, condition \eqref{item-n-to-n-no} in the Proposition gives that $o^-(\nu^L) = \Prev$.  By induction on $\nu$, we have $o^-(\nu^L + \xi) = o^-(\nu^L + *)$.  Then $o^-(\nu^L + *) = \Next$, as both Left and Right can move to $\nu^L$.  Therefore $o^-(\nu^L + \xi) = \Next$ and is a bad opening move for Left.
					
					\item $\nu + \xi^L$:  Since $o^-(\xi) = \Prev$, we have that $o^-(\xi^L) = \Next$.  By induction on $\xi$, $o^-(\nu) = o^-(\nu + \xi^L)$.  But $o^-(\nu) = \Next$, so $o^-(\nu + \xi^L) = \Next$, making it a bad opening move for Left.
				\end{enumerate}
			Therefore Left has no good opening move in $\nu + \xi$.  Similarly, neither does Right.  Therefore $o^-(\nu + \xi) = \Prev$.
		\end{enumerate}
	\end{enumerate}
	
	Therefore 
		\begin{align*}
			o^-(\nu) &= o^-(\nu + \xi) \text{ if } o^-(\xi) = \Next,\\
			o^-(\nu + *) &= o^-(\nu + \xi) \text{ if } o^-(\xi) = \Prev,
		\end{align*}	
	and so
		\begin{align*}
			0 &\equiv \xi \imod{\cl{\Upsilon}} \text{ if } o^-(\xi) = \Next,\\
			* &\equiv \xi \imod{\cl{\Upsilon}} \text{ if } o^-(\xi) = \Prev.
		\end{align*}
	Since $o^-(* + *) = \Next$, we also have
		\[ * + * \equiv 0 \imod{\cl{\Upsilon}}.\]
	Taking these results, we calculate the following \mis monoid:  via the map
		\begin{align*}
			\xi &\mapsto \begin{cases}
			1 &\text{if } o^-(\xi) = \Next,\\
			a &\text{if } o^-(\xi) = \Prev, 
			\end{cases}
		\end{align*}
	the following monoid is obtained
		\begin{align*}
		\monoid{M}_{\cl{\Upsilon}} &= \ideal{1,a \mid a^2 = 1} \\
		\Next &= \{1\} \\
		\Prev &= \{a\} \\
		\Left &= \emptyset \\
		\Right &= \emptyset.
		\end{align*}		
		
	It is clear to see that $\monoid{M}_{\cl{\Upsilon}} \cong \monoid{M}_{\cl{*}}$, as required.
\end{proof}

Combining Propositions \ref{proposition-P-*-monoids} and \ref{prop-p-is-*} we obtain the following corollary, which is important enough to be reclassified as a theorem in its own right.  

\begin{theorem}\label{theorem-P-iso-*}
	Let $\Upsilon$ be a set of positions.  Then $\monoid{M}_{\cl{\Upsilon}} \cong \monoid{M}_{\cl{*}}$ if and only if the following are all satisfied:
		\begin{enumerate}
			\item\label{item-P-iso-*-1}  $\Upsilon$ contains a position other than 0;
			\item\label{item-P-iso-*-2} for $\xi \in \cl{\Upsilon}$, $o^-(\xi) = \Next \cup \Prev$;
			\item\label{item-P-iso-*-3}  for $\xi \in \cl{\Upsilon}$, if $o^-(\xi) = \Next$, then $o^-(\xi^L) = \Prev$ for every Left option of $\xi$ and $o^-(\xi^R) = \Prev$ for every Right option of $\xi$.  That is, from an $\Next$ position, a player can never move to another $\Next$ position.		
		\end{enumerate}
\end{theorem}

Using Theorem \ref{theorem-P-iso-*}, we are able to determine the monoids of certain sets of positions without any explicit monoid calculations.  Then, once we have found two sets of positions whose monoids are isomorphic to $\monoid{M}_{\cl{*}}$, we can apply Theorem \ref{theorem-p1-p2-*} to determine the outcome of arbitrary sums of positions from both sets.  This is exceptionally good news.  However, the checking the conditions involved to show $\monoid{M}_{\cl{\Upsilon}} \cong \monoid{M}_{\cl{*}}$, especially \eqref{item-P-iso-*-2} and \eqref{item-P-iso-*-3}, are almost as overwhelming as directly calculating the $\monoid{M}_{\cl{\Upsilon}}$ itself.  One of the uses of Theorem \ref{theorem-P-iso-*} is showing when monoids are not isomorphic to $\monoid{M}_{\cl{*}}$, usually by showing that condition \eqref{item-P-iso-*-3} does not hold.  Another use of Theorem \ref{theorem-P-iso-*} is that we use it to extend positions $\xi$ whose monoids are isomorphic to $\monoid{M}_{\cl{*}}$ without changing the \mis monoid, as the next few results will demonstrate.

Before we continue, we need to define the following:

\begin{definition}
	Given the following:
		\begin{itemize}
			\item $\xi$ is a position such that $\monoid{M}_{\cl{\xi}} \cong \monoid{M}_{\cl{*}}$, with  $\varphi : \monoid{M}_{\cl{\xi}} \to \monoid{M}_{\cl{*}}$ the \mis monoid isomorphism between the two \mis monoids;
			\item $\mathscr{Q} : \cl{\xi} \to \monoid{M}_{\cl{\xi}}$ is the canonical quotient map from $\cl{\xi}$ to $\monoid{M}_{\cl{\xi}}$;
			\item $\psi \in \cl{\xi}$.
		\end{itemize}
	Then we say $\psi \sim *$ if $\varphi \circ \mathscr{Q}(\psi) = a$.  We say $\psi \sim 0$ if $\varphi \circ \mathscr{Q}(\psi) = 1$, where $a$ and $1$ are the elements of $\monoid{M}_{\cl{*}}$ as given in Notation \ref{notation-*}.  
\end{definition}

We will now use Theorem \ref{theorem-P-iso-*} to build monoids which are isomorphic to $\monoid{M}_{\cl{*}}$.  We start with the following lemma, which will become the base case in a larger result we wish to prove.

\begin{lemma}\label{lemma-xi-^-L-*}
	Let $\xi$, $\kappa$ be positions such that $\monoid{M}_{\cl{\xi}} \cong \monoid{M}_{\cl{\kappa}} \cong \monoid{M}_{\cl{*}}$.  Then
		\begin{enumerate}
			\item\label{item-xi-^-L-1} If $o^-(\xi) = o^-(\kappa)$, then $\monoid{M}_{\cl{\combgame{\{\xi\mid\kappa\}}}} \cong \monoid{M}_{\cl{*}}$.
			
			\item\label{item-xi-^-L-2} If $o^-(\xi) = \Prev$, then $\monoid{M}_{\cl{\L(\xi)}} \cong \monoid{M}_{\cl{*}}$.  
			
			\item\label{item-xi-^-L-zero} If $o^-(\xi) =o^-(\kappa) = \Next$, then the following hold:
				\begin{enumerate}
					\item $\monoid{M}_{\cl{\combgame{\{\xi\mid0\}}}} \cong  \monoid{M}_{\cl{*}}$;
					\item $\monoid{M}_{\cl{\combgame{\{\xi\mid\kappa, 0\}}}} \cong  \monoid{M}_{\cl{*}}$;
					\item $\monoid{M}_{\cl{\combgame{\{\xi, 0\mid \kappa, 0\}}}} \cong  \monoid{M}_{\cl{*}}$.
				\end{enumerate}
		\end{enumerate}
\end{lemma}

\begin{proof}\text{}
	\begin{enumerate}
		\item 	Consider an arbitrary position of $\cl{\combgame{\{\xi\mid\kappa\}}}$,
		\[t \combgame{\{\xi\mid\kappa\}} + \sum_{i=1}^n a_i \chi_i + \sum_{j=1}^m b_j \lambda_j,\]
	where $\chi_i \in \cl{\xi}$ for all $i$ and $\lambda_j \in \cl{\kappa}$ for all $j$.  Since $\cl{\xi} \subset \cl{\combgame{\{\xi\mid\kappa\}}}$, we know that $0 \subset \cl{\xi} \subset \cl{\combgame{\{\xi\mid\kappa\}}}$.  Thus \eqref{item-P-iso-*-1} of Theorem \ref{theorem-P-iso-*} is satisfied.  It remains to show  \eqref{item-P-iso-*-2} and  \eqref{item-P-iso-*-3} of Theorem \ref{theorem-P-iso-*}.  
	
	We proceed by induction on $t$.  When $t=0$, the position becomes 
		\[ \sum_{i=1}^n a_i \chi_i + \sum_{j=1}^m b_j \lambda_j.\]
	Since $\monoid{M}_{\cl{\xi}} \cong \monoid{M}_{\cl{\kappa}} \cong \monoid{M}_{\cl{*}}$, we have
		\begin{align*}
			\sum_{i=1}^n a_i \chi_i &\sim t_1 \\
			\sum_{j=1}^m b_j \lambda_j &\sim t_2,
		\end{align*}
	for $t_1, t_2 \in \cl{*}$.  By Theorem \ref{theorem-p1-p2-*},
		\[ \sum_{i=1}^n a_i \chi_i + \sum_{j=1}^m b_j \lambda_j \sim t_1 + t_2,\]
	and $o^-(t_1 + t_2) = \Next \cup \Prev$.  We have thus satisfied condition \eqref{item-P-iso-*-2} of Theorem \ref{theorem-P-iso-*}.  
	
	We will now show that condition \eqref{item-P-iso-*-3} of Theorem \ref{theorem-P-iso-*} is satisfied.  Suppose
		\[ o^-\left(\sum_{i=1}^n a_i \chi_i + \sum_{j=1}^m b_j \lambda_j\right) = \Next.\]
	If
		\[ \sum_{i=1}^n a_i \chi_i + \sum_{j=1}^m b_j \lambda_j = 0,\]
	then neither Left nor Right has a move, so we cannot move from this $\Next$ position to another $\Next$ position.  Otherwise, one of Left or Right has a move.  Suppose, without loss of generality, that Left can move to
		\[ \left(\sum_{i=1}^n a_i \chi_i\right)^L + \sum_{j=1}^m b_j \lambda_j.\]
	We break our analysis into two cases:
		\begin{enumerate}
			\item $\displaystyle o^-\left(\sum_{i=1}^n a_i \chi_i\right) = \Next$:  Since 
				\[ \sum_{i=1}^n a_i \chi_i \in \cl{\xi}\]
			and $\monoid{M}_{\cl{\xi}} \cong \monoid{M}_{\cl{*}}$, Theorem \ref{theorem-P-iso-*} gives	
				\[ o^-\left(\left(\sum_{i=1}^n a_i \chi_i\right)^L\right) = \Prev,\]
			with
				\[ \left(\sum_{i=1}^n a_i \chi_i\right)^L \sim *.\]
			Since
				\[ o^-\left(\sum_{i=1}^n a_i \chi_i + \sum_{j=1}^m b_j \lambda_j\right) = \Next,\]
			we have
				\[ \sum_{i=1}^n a_i \chi_i + \sum_{j=1}^m b_j \lambda_j \sim 0,\]
			and since
				\[ o^-\left(\sum_{i=1}^n a_i \chi_i\right) = \Next, \]
			we have
				\[  \sum_{i=1}^n a_i \chi_i \sim 0.\]
			Therefore
				\[ \sum_{j=1}^m b_j \lambda_j \sim 0.\]
			
			Combining these two facts gives
				\[ \left(\sum_{i=1}^n a_i \chi_i\right)^L + \sum_{j=1}^m b_j \lambda_j \sim * + 0\]
			so
				\[ o^-\left(\left(\sum_{i=1}^n a_i \chi_i\right)^L + \sum_{j=1}^m b_j \lambda_j\right) = \Prev.\]
			
			\item $\displaystyle o^-\left(\sum_{i=1}^n a_i \chi_i\right) = \Prev$:  Since 
				\[ \sum_{i=1}^n a_i \chi_i \in \cl{\xi}\]
			and $\monoid{M}_{\cl{\xi}} \cong \monoid{M}_{\cl{*}}$, Theorem \ref{theorem-P-iso-*} gives
				\[ o^-\left(\left(\sum_{i=1}^n a_i \chi_i\right)^L\right) = \Next.\]
			Similarly to the arguments given in the previous case, we have
				\[ \sum_{j=1}^m b_j \lambda_j \sim *,\]
			and so
				\[ \left(\sum_{i=1}^n a_i \chi_i \right)^L + \sum_{j=1}^m b_j \lambda_j \sim 0 + *,\]
			so
				\[ o^-\left(\left(\sum_{i=1}^n a_i \chi_i \right)^L + \sum_{j=1}^m b_j \lambda_j\right) = \Prev.\]
		\end{enumerate}
	Therefore \eqref{item-P-iso-*-3} of Theorem \ref{theorem-P-iso-*} is satisfied, and the base case is shown.
	
	Suppose true for all positions with $t < u$ and consider 
		\[ u \combgame{\{\xi\mid\kappa\}} + \sum_{i=1}^n a_i \chi_i + \sum_{j=1}^m b_j \lambda_j.\]
	We proceed by induction on the options of 
		\[ \sum_{i=1}^n a_i \chi_i + \sum_{j=1}^m b_j \lambda_j.\]
	When 
		\[\sum_{i=1}^n a_i \chi_i + \sum_{j=1}^m b_j \lambda_j = 0,\] 
	we have $u \combgame{\{\xi\mid\kappa\}}$.  Firstly, we want $o^-(u \combgame{\{\xi\mid\kappa\}}) = \Next \cup \Prev$.  Left moving first has only one move, to the position $(u-1) \combgame{\{\xi\mid\kappa\}} + \xi$. Similarly Right moving first has only one move, to the position $(u-1) \combgame{\{\xi\mid\kappa\}}+\kappa$.   By induction, 
		\[o^-((u-1) \combgame{\{\xi\mid\kappa\}} + \xi) = o^-((u-1) \combgame{\{\xi\mid\kappa\}}+\kappa) = \Next \cup \Prev.\]
	We want that both positions have the same outcome.  We consider two cases. 
		\begin{enumerate}
			\item $\displaystyle o^-((u-1) \combgame{\{\xi\mid\kappa\}}+\kappa) = \Next$: Since this falls under the induction hypothesis, this means that 
				\[o^-((u-2) \combgame{\{\xi\mid\kappa\}} + \xi + \kappa) = \Prev,\]
			since Left can move from 
				\[(u-1) \combgame{\{\xi\mid\kappa\}}+\kappa\]
			to this position.  But Right moving from 
				\[(u-1) \combgame{\{\xi\mid\kappa\}} + \xi\]
			can move to 
				\[(u-2) \combgame{\{\xi\mid\kappa\}} + \xi + \kappa,\]
			so Right can win 
				\[(u-1) \combgame{\{\xi\mid\kappa\}} + \xi\]
			moving first.  Since 
				\[o^-((u-1) \combgame{\{\xi\mid\kappa\}} + \xi) = \Next \cup \Prev,\]
			this gives that 
				\[o^-((u-1) \combgame{\{\xi\mid\kappa\}} + \xi) = \Next,\]
			and so the outcomes agree.
	
			\item $\displaystyle o^-((u-1) \combgame{\{\xi\mid\kappa\}}+\kappa) = \Prev$: Then 
				\[o^-((u-2) \combgame{\{\xi\mid\kappa\}} + \xi + \kappa) 	= \Next,\]
			since Left can move to this position.  But Right can move to this position from
				\[(u-1) \combgame{\{\xi\mid\kappa\}} + \xi,\]
			and both these positions falling under the induction hypothesis means that 
				\[o^-((u-1) \combgame{\{\xi\mid\kappa\}} + \xi) = \Prev,\]
			otherwise we would have an $\Next$ position where Right could move to another $\Next$ position, contradicting \eqref{item-P-iso-*-3} of Theorem \ref{theorem-P-iso-*}.  
		\end{enumerate}	
	
	Therefore from $u \combgame{\{\xi\mid\kappa\}}$, either Left and Right both move to an $\Next$ or to a $\Prev$ position.  Thus 
		\[o^-(u\combgame{\{\xi\mid\kappa\}}) = \Next \cup \Prev.\]
	If $o^-(u \combgame{\{\xi\mid\kappa\}}) = \Next$, then as Left and Right each have exactly one possible move, neither can move from this $\Next$ position to another $\Next$ position.  Therefore \eqref{item-P-iso-*-2} and \eqref{item-P-iso-*-3} of Theorem \ref{theorem-P-iso-*} are satisfied.  
	
	Now suppose that \eqref{item-P-iso-*-2} and \eqref{item-P-iso-*-3} of Theorem \ref{theorem-P-iso-*} are true for all 
		\[u \combgame{\{\xi\mid\kappa\}} + \Omega\]
	where $\Omega$ is an option of 
		\[\sum_{i=1}^n a_i \chi_i + \sum_{j=1}^m b_j \lambda_j\]
	and consider the position 
		\[ u \combgame{\{\xi\mid\kappa\}} + \sum_{i=1}^n a_i \chi_i + \sum_{j=1}^m b_j \lambda_j.\]
	
	Consider the set 
		\[ \mathscr{S} = \{ m \combgame{\{\xi\mid\kappa\}} + \Omega \mid m \in \{0,1,\ldots,u-1\}, \Omega \in \cl{\xi} + \cl{\kappa}\}.\]
	By the induction hypothesis, $\mathscr{S}$ satisfies all the conditions of Theorem \ref{theorem-P-iso-*}.  Moreover $\mathscr{S}$ is option closed, although not closed under addition (see Definition \ref{definition-closed}).  But, we can build the monoid of this restricted set, obtaining  $\monoid{M}_{\mathscr{S}} \cong \monoid{M}_{\cl{*}}$.  That is, $\monoid{M}_{\mathscr{S}}$ has two elements $s_1$ and $s_2$ such that $s_1 \sim 0$ and $s_2 \sim *$.  
	
	We will first show that  \eqref{item-P-iso-*-2} of Theorem \ref{theorem-P-iso-*} is satisfied.  That is, we will show
		\[ o^-\left(u \combgame{\{\xi\mid\kappa\}} + \sum_{i=1}^n a_i \chi_i + \sum_{j=1}^m b_j \lambda_j\right) = \Next \cup \Prev. \]
	To do this, we examine the position 
		\[(u-1) \combgame{\{\xi\mid\kappa\}} + \xi+ \sum_{i=1}^n a_i \chi_i + \sum_{j=1}^m b_j \lambda_j.\]
	This position falls under the induction hypothesis, and so 			
		\[o^-\left((u-1) \combgame{\{\xi\mid\kappa\}} + \xi+ \sum_{i=1}^n a_i \chi_i + \sum_{j=1}^m b_j \lambda_j\right) = \Next \cup \Prev.\]
	We consider these two cases separately.   
	
		\begin{enumerate}
			\item $o^-\left(\displaystyle (u-1) \combgame{\{\xi\mid\kappa\}} + \xi+ \sum_{i=1}^n a_i \chi_i + \sum_{j=1}^m b_j \lambda_j\right) = \Next$:  Suppose Left is moving first in 
				\[u \combgame{\{\xi\mid\kappa\}} + \sum_{i=1}^n a_i \chi_i + \sum_{j=1}^m b_j \lambda_j.\]
			Left moves to either
				\begin{enumerate}
					\item $\displaystyle (u-1) \combgame{\{\xi\mid\kappa\}} + \xi+ \sum_{i=1}^n a_i \chi_i + \sum_{j=1}^m b_j \lambda_j$, or
					\item $\displaystyle u \combgame{\{\xi\mid\kappa\}} + \left(\sum_{i=1}^n a_i \chi_i + \sum_{j=1}^m b_j \lambda_j\right)^L$.
				\end{enumerate}
			Clearly (i) is a bad move, as Left moves to an $\Next$ position.  Thus, suppose Left moves to 
				\[u \combgame{\{\xi\mid\kappa\}} + \left(\sum_{i=1}^n a_i \chi_i + \sum_{j=1}^m b_j \lambda_j\right)^L.\]
			This position falls under the induction hypothesis, and so 		
				\[o^-\left(u \combgame{\{\xi\mid\kappa\}} + \left(\sum_{i=1}^n a_i \chi_i + \sum_{j=1}^m b_j \lambda_j\right)^L\right) = \Next \cup \Prev.\]
			Suppose 
				\[o^-\left(u \combgame{\{\xi\mid\kappa\}} + \left(\sum_{i=1}^n a_i \chi_i + \sum_{j=1}^m b_j \lambda_j\right)^L\right) = \Prev.\]
			Right moving first in 
				\[u \combgame{\{\xi\mid\kappa\}} + \left(\sum_{i=1}^n a_i \chi_i + \sum_{j=1}^m b_j \lambda_j\right)^L\]
			can move to 
				\[ (u-1) \combgame{\{\xi\mid\kappa\}} + \kappa+ \left(\sum_{i=1}^n a_i \chi_i + \sum_{j=1}^m b_j \lambda_j\right)^L,\]
			whose outcome must be $\Next$.  Left can move to this position from the position
				\[ (u-1) \combgame{\{\xi\mid\kappa\}} + \kappa+ \sum_{i=1}^n a_i \chi_i + \sum_{j=1}^m b_j \lambda_j.\]
			By the induction hypothesis, we then have
				\[ o^-\left((u-1) \combgame{\{\xi\mid\kappa\}} + \kappa+ \sum_{i=1}^n a_i \chi_i + \sum_{j=1}^m b_j \lambda_j\right) = \Prev,\]
			otherwise we would have an $\Next$ position moving to an $\Next$ position.  Therefore, we have
				\begin{align*}
					o^-\left((u-1) \combgame{\{\xi\mid\kappa\}} + \kappa+ \sum_{i=1}^n a_i \chi_i + \sum_{j=1}^m b_j \lambda_j\right) &= \Prev\\
					o^-\left(\displaystyle (u-1) \combgame{\{\xi\mid\kappa\}} + \xi+ \sum_{i=1}^n a_i \chi_i + \sum_{j=1}^m b_j \lambda_j\right) &= \Next,
				\end{align*}
			and both of these positions are in $\monoid{M}_{\mathscr{S}}$.  Thus, we have the following
				\begin{align*}
					(u-1) \combgame{\{\xi\mid\kappa\}} + \kappa+ \sum_{i=1}^n a_i \chi_i + \sum_{j=1}^m b_j \lambda_j &\sim *, \\
					(u-1) \combgame{\{\xi\mid\kappa\}} + \xi+ \sum_{i=1}^n a_i \chi_i + \sum_{j=1}^m b_j \lambda_j&\sim 0
				\end{align*}
			and since $\monoid{M}_{\mathscr{S}} \cong \monoid{M}_{\cl{*}}$, this means $\xi \not \sim \kappa$, so $o^-(\xi) \not = o^-(\kappa)$, a contradiction.  Therefore
				\[ \displaystyle o^-\left(u \combgame{\{\xi\mid\kappa\}} + \left(\sum_{i=1}^n a_i \chi_i + \sum_{j=1}^m b_j \lambda_j\right)^L\right) = \Next.\]
			Thus, Left can only move to $\Next$ positions, and therefore
				\[ o^-\left(u \combgame{\{\xi\mid\kappa\}} + \sum_{i=1}^n a_i \chi_i + \sum_{j=1}^m b_j \lambda_j\right) = \Right \cup \Prev.\]
						
			Consider Right moving first in 
				\[ u \combgame{\{\xi\mid\kappa\}} + \sum_{i=1}^n a_i \chi_i +\sum_{j=1}^m b_j \lambda_j.\]
			Right has two possible moves:
				\begin{enumerate}
					\item $\displaystyle (u-1) \combgame{\{\xi\mid\kappa\}} + \kappa + \sum_{i=1}^n a_i \chi_i + \sum_{j=1}^m b_j \lambda_j$, or
					\item $\displaystyle u \combgame{\{\xi\mid\kappa\}} + \left(\sum_{i=1}^n a_i \chi_i + \sum_{j=1}^m b_j \lambda_j \right)^R$.
				\end{enumerate}
			Suppose Right moves to the position given in (i).   In $\monoid{M}_{\mathscr{S}}$, we have
				\begin{align*}
					(u-1) \combgame{\{\xi\mid\kappa\}} + \xi + \sum_{i=1}^n a_i \chi_i +\sum_{j=1}^m b_j \lambda_j &\sim 0.
				\end{align*}
			In our statement of this part of the theorem, we assumed
				\[ \kappa \sim \xi,\]
			this implies
				\[ (u-1) \combgame{\{\xi\mid\kappa\}} + \kappa +\sum_{i=1}^n a_i \chi_i +\sum_{j=1}^m b_j \lambda_j \sim 0, \]
			so
				\[ o^-\left((u-1) \combgame{\{\xi\mid\kappa\}} + \kappa +\sum_{i=1}^n a_i \chi_i +\sum_{j=1}^m b_j \lambda_j\right) = \Next.\]
			Thus, this would be a bad opening move for Right.
			
			Suppose Right moves to the position given in (ii).  By induction,
				\[ o^-\left(u \combgame{\{\xi\mid\kappa\}} + \left(\sum_{i=1}^n a_i \chi_i + \sum_{j=1}^m b_j \lambda_j \right)^R\right) = \Next \cup \Prev.\]
			Suppose	
				\[ o^-\left(u \combgame{\{\xi\mid\kappa\}} + \left(\sum_{i=1}^n a_i \chi_i + \sum_{j=1}^m b_j \lambda_j \right)^R\right) = \Prev.\]
			From this position, Left can move to
				\[ (u-1) \combgame{\{\xi\mid\kappa\}} + \xi + \left(\sum_{i=1}^n a_i \chi _i  + \sum_{j=1}^m b_j \lambda_j\right)^R,\]
			which must be an $\Next$ position.  But Right can also move to this position from
				\[ (u-1) \combgame{\{\xi\mid\kappa\}} + \xi + \sum_{i=1}^n a_i \chi_i + \sum_{j=1}^m b_j \lambda_j,\]
			which we have assumed to be an $\Next$ position.  Since all these positions fall under the induction hypothesis and we cannot have an $\Next$ position moving to an $\Next$ position, we have a contradiction.  Therefore
				\[ o^-\left(u \combgame{\{\xi\mid\kappa\}} + \left(\sum_{i=1}^n a_i \chi_i + \sum_{j=1}^m b_j \lambda_j \right)^R\right) = \Next.\]
			Thus Right has no winning first move from the position
				\[ u \combgame{\{\xi\mid\kappa\}} + \sum_{i=1}^n a_i \chi_i + \sum_{j=1}^m b_j \lambda_j.\]
			
			Therefore
				\[ o^-\left(u \combgame{\{\xi\mid\kappa\}} + \sum_{i=1}^n a_i \chi_i + \sum_{j=1}^m b_j \lambda_j\right) = \Prev.\]

			\item $o^-\left(\displaystyle (u-1) \combgame{\{\xi\mid\kappa\}} + \xi+ \sum_{i=1}^n a_i \chi_i + \sum_{j=1}^m b_j \lambda_j \right) = \Prev$:  Suppose Left moves first in
				\[ u\combgame{\{\xi\mid\kappa\}} + \sum_{i=1}^n a_i \chi_i + \sum_{j=1}^m b_j \lambda_j.\]
			Left can move to
				\[ (u-1) \combgame{\{\xi\mid\kappa\}}+\xi + \sum_{i=1}^n a_i \chi_i + \sum_{j=1}^m b_j \lambda_j,\]
			a $\Prev$ position by assumption.  Therefore
				\[ o^-\left(u\combgame{\{\xi\mid\kappa\}} + \sum_{i=1}^n a_i \chi_i + \sum_{j=1}^m b_j \lambda_j\right) = \Left \cup \Next.\]
				
			Now consider Right moving first in
				\[u\combgame{\{\xi\mid\kappa\}} + \sum_{i=1}^n a_i \chi_i + \sum_{j=1}^m b_j \lambda_j.\]
			Right can move to
				\[ (u-1) \combgame{\{\xi\mid\kappa\}} +\kappa + \sum_{i=1}^n a_i \chi_i + \sum_{j=1}^m b_j \lambda_j,\]
			which is an element of $\mathscr{S}$.  We have
				\begin{align*}
					(u-1) \combgame{\{\xi\mid\kappa\}} + \xi+ \sum_{i=1}^n a_i \chi_i + \sum_{j=1}^m b_j \lambda_j&\sim *.
				\end{align*}
			Since 
				\[\xi \sim \kappa,\]
			we then have
				\[ (u-1) \combgame{\{\xi\mid\kappa\}} + \kappa + \sum_{i=1}^n a_i \chi_i + \sum_{j=1}^m b_j \lambda_j \sim *. \] 
			Thus
				\[ o^-\left(((u-1) \combgame{\{\xi\mid\kappa\}} + \kappa + \sum_{i=1}^n a_i \chi_i + \sum_{j=1}^m b_j \lambda_j\right) = \Prev.\]
			
			Therefore 
				\[ o^-\left(u \combgame{\{\xi\mid\kappa\}} + \sum_{i=1}^n a_i \chi_i + \sum_{j=1}^m b_j \lambda_j\right) = \Next.\]
		\end{enumerate}

		So, we have shown
			\[ o^-\left(u \combgame{\{\xi\mid\kappa\}} + \sum_{i=1}^n a_i \chi_i + \sum_{j=1}^m b_j \lambda_j\right) = \Next \cup \Prev,\]
		and
			\begin{eqnarray}
				  \label{eqn-7.1}\lefteqn{o^-\left(u \combgame{\{\xi\mid\kappa\}} + \sum_{i=1}^n a_i \chi_i + \sum_{j=1}^m b_j \lambda_j \right) = \Next \iff } \\
				  \nonumber & & o^-\left((u-1) \combgame{\{\xi\mid\kappa\}} + \xi + \sum_{i=1}^n a_i \chi_i + \sum_{j=1}^m b_j \lambda_j \right) = \Prev, \\
		 		 \lefteqn{o^-\left(u \combgame{\{\xi\mid\kappa\}} + \sum_{i=1}^n a_i \chi_i + \sum_{j=1}^m b_j \lambda_j \right) = \Prev \iff} \\
				 \nonumber & & o^-\left((u-1) \combgame{\{\xi\mid\kappa\}} + \xi + \sum_{i=1}^n a_i \chi_i + \sum_{j=1}^m b_j \lambda_j \right) = \Next.
			\end{eqnarray}
		
		It remains to show that (3) of Theorem \ref{theorem-P-iso-*} is satisfied, i.e.\ if
			\[ o^-\left(u \combgame{\{\xi\mid\kappa\}} + \sum_{i=1}^n a_i \chi_i + \sum_{j=1}^m b_j \lambda_j\right) = \Next,\]
		then the positions
			\[ \left(u \combgame{\{\xi\mid\kappa\}} + \sum_{i=1}^n a_i \chi_i + \sum_{j=1}^m b_j \lambda_j\right)^L, \left(u \combgame{\{\xi\mid\kappa\}} + \sum_{i=1}^n a_i \chi_i + \sum_{j=1}^m b_j \lambda_j\right)^R\]
		are $\Prev$ positions.
		
		Thus, suppose
			\[ o^-\left(u \combgame{\{\xi\mid\kappa\}} + \sum_{i=1}^n a_i \chi_i + \sum_{j=1}^m b_j \lambda_j \right) = \Next.\]
		If Left moves first in
			\[ u \combgame{\{\xi\mid\kappa\}} + \sum_{i=1}^n a_i \chi_i + \sum_{j=1}^m b_j \lambda_j.\]
		she has two possible first moves:
			\begin{enumerate}
				\item $\displaystyle (u-1) \combgame{\{\xi\mid\kappa\}} + \xi + \sum_{i=1}^n a_i \chi_i + \sum_{j=1}^m b_j \lambda_j$, or
				\item $\displaystyle u \combgame{\{\xi\mid\kappa\}} + \left(\sum_{i=1}^n a_i \chi_i + \sum_{j=1}^m b_j \lambda_j \right)^L$.
			\end{enumerate}
		By Equation \eqref{eqn-7.1} above, since
			\[ o^-\left( u \combgame{\{\xi\mid\kappa\}} + \sum_{i=1}^n a_i \chi_i + \sum_{j=1}^m b_j \lambda_j\right) = \Next,\]
		we have
			\[ o^-\left((u-1) \combgame{\{\xi\mid\kappa\}} + \xi + \sum_{i=1}^n a_i \chi_i + \sum_{j=1}^m b_j \lambda_j \right) = \Prev.\]
		Thus, consider Left's other move, and suppose
			\[ o^-\left(u \combgame{\{\xi\mid\kappa\}} + \left(\sum_{i=1}^n a_i \chi_i + \sum_{j=1}^m b_j \lambda_j \right)^L\right) = \Next. \]
		This position falls under the induction hypothesis, and so
			\[ o^-\left((u-1) \combgame{\{\xi\mid\kappa\}} + \kappa + \left(\sum_{i=1}^n a_i \chi_i + \sum_{j=1}^m b_j \lambda_j\right)^L\right) = \Prev. \]
		Left can move to this position from the position
			\[ (u-1) \combgame{\{\xi\mid\kappa\}} + \kappa  + \sum_{i=1}^n a_i \chi_i + \sum_{j=1}^m b_j \lambda_j,\]
		and so
			\begin{align}
				\label{eqn-7.3} o^-\left((u-1) \combgame{\{\xi\mid\kappa\}} + \kappa  + \sum_{i=1}^n a_i \chi_i + \sum_{j=1}^m b_j \lambda_j\right) = \Next,
			\end{align}
		which gives us that
			\[ (u-1) \combgame{\{\xi\mid\kappa\}} + \kappa  + \sum_{i=1}^n a_i \chi_i + \sum_{j=1}^m b_j \lambda_j \sim 0.\]
		However, by assumption, we have
			\[ \xi \sim \kappa,\]
		and Equation \eqref{eqn-7.1} gives us
			\[ (u-1) \combgame{\{\xi\mid\kappa\}} + \xi + \sum_{i=1}^n a_i \chi_i + \sum_{j=1}^m b_j \lambda_j \sim *.\]
		Combining these three equations gives us a contradiction.  Therefore
			\[ o^-\left(u \combgame{\{\xi\mid\kappa\}} + \left(\sum_{i=1}^n a_i \chi_i + \sum_{j=1}^m b_j \lambda_j \right)^L\right) = \Prev. \]
		
		If Right moves first in
			\[ u \combgame{\{\xi\mid\kappa\}} + \sum_{i=1}^n a_i \chi_i + \sum_{j=1}^m b_j \lambda_j,\]
		he has two possible first moves:
			\begin{enumerate}
				\item $\displaystyle (u-1) \combgame{\{\xi\mid\kappa\}} + \kappa + \sum_{i=1}^n a_i \chi_i + \sum_{j=1}^m b_j \lambda_j$, or
				\item $\displaystyle u \combgame{\{\xi\mid\kappa\}} + \left(\sum_{i=1}^n a_i \chi_i + \sum_{j=1}^m b_j \lambda_j \right)^R$.
			\end{enumerate}
		Suppose Right moves to
			\[ (u-1) \combgame{\{\xi\mid\kappa\}} + \kappa + \sum_{i=1}^n a_i \chi_i + \sum_{j=1}^m b_j \lambda_j,\]
		and, moreover, suppose
			\[ o^-\left((u-1) \combgame{\{\xi\mid\kappa\}} + \kappa + \sum_{i=1}^n a_i \chi_i + \sum_{j=1}^m b_j \lambda_j\right) = \Next.\]
		However, this is the same as Equation \eqref{eqn-7.3}, which was  shown to result in a contradiction if we assume Equation \eqref{eqn-7.3} to be true.  Therefore
			\[ o^-\left((u-1) \combgame{\{\xi\mid\kappa\}} + \kappa + \sum_{i=1}^n a_i \chi_i + \sum_{j=1}^m b_j \lambda_j\right) = \Prev.\]
		
		Now suppose Right moves to
			\[ u \combgame{\{\xi\mid\kappa\}} + \left(\sum_{i=1}^n a_i \chi_i + \sum_{j=1}^m b_j \lambda_j \right)^R,\]
		and suppose
			\[ o^-\left(u \combgame{\{\xi\mid\kappa\}} + \left(\sum_{i=1}^n a_i \chi_i + \sum_{j=1}^m b_j \lambda_j \right)^R\right) = \Next.\]
		Since this position falls under the induction hypothesis, this means
			\[ o^-\left((u-1) \combgame{\{\xi\mid\kappa\}} +\kappa + \left(\sum_{i=1}^n a_i \chi_i + \sum_{j=1}^m b_j \lambda_j\right)^R\right) = \Prev.\]
		But Right can move to this position from 
			\[ (u-1) \combgame{\{\xi\mid\kappa\}} + \kappa + \sum_{i=1}^n a_i\chi_i + \sum_{j=1}^m b_j \lambda_j,\]
		and so
			\[ o^-\left((u-1) \combgame{\{\xi\mid\kappa\}} + \kappa + \sum_{i=1}^n a_i\chi_i + \sum_{j=1}^m b_j \lambda_j\right) = \Next,\]
		However, this is the same as Equation \eqref{eqn-7.3}, which was  shown to result in a contradiction if we assume Equation \eqref{eqn-7.3} to be true.  Therefore
			\[ o^-\left(u \combgame{\{\xi\mid\kappa\}} + \left(\sum_{i=1}^n a_i \chi_i + \sum_{j=1}^m b_j \lambda_j \right)^R\right) = \Prev.\]
		
		Therefore if 
			\[ o^-\left(u \combgame{\{\xi\mid\kappa\}} + \sum_{i=1}^n a_i \chi_i + \sum_{j=1}^m b_j \lambda_j \right) = \Next,\]
		then the positions
			\[ \left(u \combgame{\{\xi\mid\kappa\}} + \sum_{i=1}^n a_i \chi_i + \sum_{j=1}^m b_j \lambda_j\right)^L, \left(u \combgame{\{\xi\mid\kappa\}} + \sum_{i=1}^n a_i \chi_i + \sum_{j=1}^m b_j \lambda_j\right)^R\]
		are $\Prev$ positions.	
		
			\item
		While this proof is similar to the previous case, there are a few subtle differences.  As such, it is presented here for completeness.
		
		Consider an arbitrary position of $\cl{\L(\xi)}$,
			\[ t \L(\xi) + \sum_{i=1}^n a_i \chi_i,\]
		where 
			\[ \chi_i \in \cl{\xi}.\]
		We want that this position satisfies conditions \eqref{item-P-iso-*-1}, \eqref{item-P-iso-*-2} and  \eqref{item-P-iso-*-3} of Theorem \ref{theorem-P-iso-*}.   We know that $0 \subset \cl{\xi} \subset \cl{\L(\xi)}$.  Thus \eqref{item-P-iso-*-1} of Theorem \ref{theorem-P-iso-*} is satisfied.  It remains to show the other two.  We proceed by induction on $t$.
		
		When $t=0$, the position becomes
			\[ \sum_{i=1}^n a_i \chi_i,\]
		which satisfies conditions  \eqref{item-P-iso-*-2} and  \eqref{item-P-iso-*-3} of Theorem \ref{theorem-P-iso-*} as this position is an element of $\cl{\xi}$.  
		
		Suppose true for all positions with $t < u$ and consider the position
			\[ u \L(\xi) + \sum_{i=1}^n a_i \chi_i.\]
		
		Define the following set:
			\[ \mathscr{S} = \{ m \L(\xi) + \Omega \mid m \in \{0,1,\ldots,u-1\}, \Omega \in \cl{\xi}\}.\]
		As in the proof of the first part of this Lemma (Lemma \ref{lemma-xi-^-L-*} \eqref{item-xi-^-L-1}), we can see that
			\[ \monoid{M}_{\mathscr{S}} \cong \monoid{M}_{\cl{*}}.\]
		
		We proceed by induction on the options of 
			\[ \sum_{i=1}^n a_i \chi_i.\]
		When 
			\[\sum_{i=1}^n a_i \chi_i = 0,\]
		we have $u \L(\xi)$.  We claim that $o^-(u \L(\xi)) = \Next$.  
		We proceed by induction on the number of copies of $\L(\xi)$.  If there are no copies, we have the position $0$, whose outcome is $\Next$.  Now assume true for all $w < u$ and consider the position $u \L(\xi)$.  Right moving first has no available moves, so $o^-(u \L(\xi)) = \Next \cup \Right$.  Left moving first has one available move, the move to
			\[ (u-1) \L(\xi) + \xi,\]
		which is an element of $\mathscr{S}$.  By induction,
			\[ o^-(u-1) \L(\xi)) = \Next,\]
		so
			\[ (u-1) \L(\xi) \sim 0,\]
		and, by assumption,
			\[ o^-(\xi) = \Prev,\]
		so
			\[ \xi \sim *.\]
		Therefore
			\[ (u-1) \L(\xi) + \xi \sim *,\]
		so	
			\[ o^-(u-1) \L(\xi) + \xi) = \Prev.\]
		Therefore
			\begin{align}
			\label{eqn-7.4}	o^-(u \L(\xi)) = \Next,
			\end{align}
		as required, and we have satisfied \eqref{item-P-iso-*-2} of Theorem \ref{theorem-P-iso-*}.  We have also shown that \eqref{item-P-iso-*-3} of Theorem \ref{theorem-P-iso-*} is also satisfied as Left has only one move from $u \L(\xi)$ and it was just shown to be a move to a $\Prev$ position.
		
		Now suppose \eqref{item-P-iso-*-2} and  \eqref{item-P-iso-*-3} of Theorem \ref{theorem-P-iso-*} are true for all positions
			\[ u \L(\xi) + \Omega\]
		where $\Omega$ is an option of 
			\[ \sum_{i=1}^n a_i \chi_i\]
		and consider the position
			\[ u \L(\xi) + \sum_{i=1}^n a_i \chi_i.\]
		We want to show \eqref{item-P-iso-*-2} of Theorem \ref{theorem-P-iso-*}, i.e.\
			\[ o^-\left(u \L(\xi) + \sum_{i=1}^n a_i \chi_i\right) = \Next \cup \Prev.\]
		We do this by first examining
			\[ (u-1) \L(\xi) + \xi + \sum_{i=1}^n a_i \chi_i.\]
		This position falls under the induction hypothesis, so
			\[ o^-\left((u-1) \L(\xi) + \xi + \sum_{i=1}^n a_i \chi_i\right) = \Next \cup \Prev.  \]
		We will consider the two outcomes separately.
		
		\begin{enumerate}
			\item $\displaystyle o^-\left((u-1) \L(\xi) + \xi + \sum_{i=1}^n a_i \chi_i\right) = \Next$:  Suppose Left is moving first in
				\[ u \L(\xi) + \sum_{i=1}^n a_i \chi_i.\]
			She will move to either
				\begin{enumerate}
					\item $\displaystyle (u-1) \L(\xi) + \xi + \sum_{i=1}^n a_i \chi_i$, or
					
					\item $\displaystyle u\L(\xi) + \left(\sum_{i=1}^n a_i \chi_i\right)^L$.
				\end{enumerate}
			However, the outcome of (i) is assumed to be $\Next$.  Thus, suppose Left moves to
				\[ u\L(\xi) + \left(\sum_{i=1}^n a_i \chi_i\right)^L.\]
			As this position falls under the induction hypothesis, we have
				\[ o^-\left(u\L(\xi) + \left(\sum_{i=1}^n a_i \chi_i\right)^L\right) = \Next \cup \Prev.\]
			Suppose
				\[  o^-\left(u\L(\xi) + \left(\sum_{i=1}^n a_i \chi_i\right)^L\right) = \Prev.\]
			If Right had no move in 
				\[ u\L(\xi) + \left(\sum_{i=1}^n a_i \chi_i\right)^L,\]
			then this would be an $\Next$ position, so Right must have a move.  Since Right has no move in $u \L(\xi)$, the position
				\[ u\L(\xi) + \left(\sum_{i=1}^n a_i \chi_i\right)^{LR}\]
			must exist, and, since
				\[ o^-\left(u\L(\xi) + \left(\sum_{i=1}^n a_i \chi_i\right)^L\right) = \Prev,\]
			we have
				\[ o^-\left(u\L(\xi) + \left(\sum_{i=1}^n a_i \chi_i\right)^{LR}\right) = \Next.\]
			By the induction hypothesis,
				\[ o^-\left((u-1)\L(\xi) + \xi + \left(\sum_{i=1}^n a_i \chi_i\right)^{LR}\right) = \Prev,\]
			so
				\[ o^-\left((u-1)\L(\xi) + \xi + \left(\sum_{i=1}^n a_i \chi_i\right)^L\right) = \Next\]
			and
				\[ o^-\left((u-1)\L(\xi) + \xi +  \sum_{i=1}^n a_i \chi_i\right) = \Prev,\]
			a contradiction.  Therefore
				\[ o^-\left(u\L(\xi) + \left(\sum_{i=1}^n a_i \chi_i\right)^L\right) = \Next\]
			and
				\[ o^-\left(u \L(\xi) + \sum_{i=1}^n a_i \chi_i \right) = \Right \cup \Prev.\]
			
			Suppose now that Right moves first in
				\[ u \L(\xi) + \sum_{i=1}^n a_i \chi_i.\]
			Suppose Right has no move available.   This implies that Right cannot move in 
				\[ \sum_{i=1}^n a_i \chi_i.\]
			By induction,
				\[ o^-\left(\sum_{i=1}^n a_i \chi_i\right) = \Next \cup \Prev,\]
			so, combining these two results. we get
				\[ o^-\left(\sum_{i=1}^n a_i \chi_i\right) = \Next.\]
			Earlier (Equation \eqref{eqn-7.4}), we showed
				\[ o^-((u-1) \L(\xi)) = \Next.\]
			Then we have the following three $\sim$s:
				\begin{align*}
					\sum_{i=1}^n a_i \chi_i &\sim 0,\\
					(u-1) \L(\xi) &\sim 0,\\
					\xi &\sim *.
				\end{align*}
			Combining these gives 
				\[ (u-1) \L(\xi) + \xi + \sum_{i=1}^n a_i \chi_i \sim *,\]
			implying
				\[ o^-\left((u-1) \L(\xi) + \xi + \sum_{i=1}^n a_i \chi_i\right) = \Prev,\]
			a contradiction.  Therefore Right must have a move available in
				\[ u \L(\xi) + \sum_{i=1}^n a_i \chi_i,\]
			i.e.\
				\[\left(\sum_{i=1}^n a_i \chi_i\right)^R\]
			exists and Right moving first in
				\[ u \L(\xi) + \sum_{i=1}^n a_i \chi_i\]
			moves to
				\[ u \L(\xi) + \left(\sum_{i=1}^n a_i \chi_i\right)^R.\]
			This position falls under the induction hypothesis, so
				\[ o^-\left(u \L(\xi) + \left(\sum_{i=1}^n a_i \chi_i\right)^R\right) = \Next \cup \Prev.\]
			Suppose
				\[ o^-\left(u \L(\xi) + \left(\sum_{i=1}^n a_i \chi_i\right)^R\right) = \Prev.\]
			Then
				\[ o^-\left((u-1) \L(\xi) + \xi + \left(\sum_{i=1}^n a_i \chi_i\right)^R \right) = \Next, \]
			but Right can move to this position from
				\[ (u-1) \L(\xi) + \xi + \sum_{i=1}^n a_i \chi_i,\]
			which we assumed to be an $\Next$ position.  Since both these positions fall under the induction hypothesis, we have an $\Next$ position moving to an $\Next$ position, contradicting \eqref{item-P-iso-*-3} of Theorem \ref{theorem-P-iso-*}.  Therefore
				\[ o^-\left(u\L(\xi) + \left(\sum_{i=1}^n a_i \chi_i\right)^R\right) = \Next.\]
			
			Thus, from our initial position of
				\[ u \L(\xi) + \sum_{i=1}^n a_i \chi_i, \]
			both Left and Right lose moving first.  Therefore
				\[ o^-\left(u \L(\xi) + \sum_{i=1}^n a_i \chi_i\right) = \Prev.\]
			
			\item $\displaystyle o^-\left((u-1) \L(\xi) + \xi + \sum_{i=1}^n a_i \chi_i\right) = \Prev$:  If Left is moving first in
				\[ u\L(\xi) + \sum_{i=1}^n a_i \chi_i,\]
			she can move to
				\[ (u-1) \L(\xi) + \xi + \sum_{i=1}^n a_i \chi_i,\]
			which we assumed to be a $\Prev$ position.  Therefore
				\[ o^-\left(u\L(\xi) + \sum_{i=1}^n a_i \chi_i\right) = \Left \cup \Next.\]
				
			Consider Right moving first in
				\[ u\L(\xi) + \sum_{i=1}^n a_i \chi_i.\]
			If Right has no move available, then Right wins moving first.  Otherwise, Right has a move available.  Since he is unable to move in $u\L(\xi)$, Right moves to
				\[ u\L(\xi) + \left(\sum_{i=1}^n a_i \chi_i\right)^R,\]
			which falls under the induction hypothesis.  Thus
				\[ o^-\left( u\L(\xi) + \left(\sum_{i=1}^n a_i \chi_i\right)^R\right) = \Next \cup \Prev.\]
			Suppose
				\[ o^-\left( u\L(\xi) + \left(\sum_{i=1}^n a_i \chi_i\right)^R\right) = \Next.\]
			Then, by induction
				\[ o^-\left( (u-1)\L(\xi)+ \xi + \left(\sum_{i=1}^n a_i \chi_i\right)^R\right) = \Prev.\]
			But, by assumption
				\[ o^-\left((u-1) \L(\xi) + \xi + \sum_{i=1}^n a_i \chi_i\right) = \Prev, \]
			so we have a $\Prev$ position with a move to a $\Prev$ position, contradiction.  Therefore
				\[ o^-\left( u\L(\xi) + \left(\sum_{i=1}^n a_i \chi_i\right)^R\right) = \Prev,\]
			and so
				\[ o^-\left(u \L(\xi) + \sum_{i=1}^n a_i \chi_i \right) = \Next.\]
		\end{enumerate}
		
		Thus \eqref{item-P-iso-*-2} of Theorem \ref{theorem-P-iso-*} is satisfied with
			\begin{align}
				  \label{eqn-7.5}o^-\left(u \L(\xi) + \sum_{i=1}^n a_i \chi_i  \right) = \Next &\iff  o^-\left((u-1) \L(\xi) + \xi + \sum_{i=1}^n a_i \chi_i\right) = \Prev, \\
		 		 \label{eqn-7.6}o^-\left(u \L(\xi) + \sum_{i=1}^n a_i \chi_i \right) = \Prev &\iff 
				 o^-\left((u-1) \L(\xi) + \xi + \sum_{i=1}^n a_i \chi_i \right) = \Next.
			\end{align}		
		
		It remains to show \eqref{item-P-iso-*-3} of Theorem \ref{theorem-P-iso-*}.
		
		Suppose
			\[ o^-\left(u \L(\xi) + \sum_{i=1}^n a_i \chi_i \right) = \Next.\]
		We want that all the options are $\Prev$ positions.
		
		Suppose Left is moving first.  She has two possible moves:
			\begin{enumerate}
				\item $\displaystyle (u-1) \L(\xi) + \xi + \sum_{i=1}^n a_i \chi_i$:  By Equation \eqref{eqn-7.5} above, this position is a $\Prev$ position.
				
				\item $\displaystyle u \L(\xi) + \left(\sum_{i=1}^n a_i \chi_i \right)^L$:   This falls under the induction hypothesis, so
					\[ o^-\left(u \L(\xi) + \left(\sum_{i=1}^n a_i \chi_i \right)^L\right) = \Next \cup \Prev.\]
				Suppose
					\[ o^-\left(u \L(\xi) + \left(\sum_{i=1}^n a_i \chi_i \right)^L\right) = \Next.\]
				Then, by induction
					\[ o^-\left((u-1) \L(\xi) + \xi + \left(\sum_{i=1}^n a_i \chi_i\right)^L\right) = \Prev,\]
				and, by Equation \eqref{eqn-7.5},
					\[ o^-\left((u-1) \L(\xi) + \xi + \sum_{i=1}^n a_i \chi_i\right) = \Prev,\]
				giving a $\Prev$ position moving to a $\Prev$ position, a contradiction.  Therefore
					\[ o^-\left(u \L(\xi) + \left(\sum_{i=1}^n a_i \chi_i \right)^L\right) = \Prev.\]
			\end{enumerate}
		
		Suppose Right is moving first in
			\[ u\L(\xi) + \sum_{i=1}^n a_i \chi_i.\]
		Right has only one possible move, to
			\[ u\L(\xi) + \left(\sum_{i=1}^n a_i \chi_i\right)^R.\]
		By induction
			\[ o^-\left(u\L(\xi) + \left(\sum_{i=1}^n a_i \chi_i\right)^R\right) = \Next \cup \Prev. \]
		Suppose
			\[ o^-\left(u\L(\xi) + \left(\sum_{i=1}^n a_i \chi_i\right)^R\right) = \Next. \]
		Then, by induction,
			\[ o^-\left((u-1) \L(\xi) + \xi + \left(\sum_{i=1}^n a_i \chi_i \right)^R \right) = \Prev,\]
		and, by Equation \eqref{eqn-7.5},
			\[ o^-\left((u-1) \L(\xi) + \xi + \sum_{i=1}^n a_i \chi_i\right) = \Prev,\]
		giving a $\Prev$ position moving to a $\Prev$ position, a contradiction.  Therefore
			\[ o^-\left(u\L(\xi) + \left(\sum_{i=1}^n a_i \chi_i\right)^R\right) = \Prev. \]
			
		Therefore, we do not have an $\Next$ position moving to an $\Next$ position, so \eqref{item-P-iso-*-3} of Theorem \ref{theorem-P-iso-*} is satisfied.
		
		\item The proofs for these three assertions follow an almost identical pattern to the proof of \eqref{item-xi-^-L-1}. \qedhere
	\end{enumerate}
\end{proof}

\begin{theorem}\label{theorem-xi-^-L-*}
	Let $\xi_1, \xi_2, \ldots, \xi_n$, $\kappa_1, \kappa_2, \ldots, \kappa_n$ be positions such that $\monoid{M}_{\cl{\xi_i}} \cong \monoid{M}_{\cl{\kappa_j}} \cong \monoid{M}_{\cl{*}}$ for all $i \in \{1,2,\ldots,n\}$, $j \in \{1,2,\ldots,m\}$.  Then
		\begin{enumerate}
			\item\label{item-xi-^-L-1-cor} If $o^-(\xi_i) = o^-(\kappa_j)$ for all $i \in \{1,2,\ldots,n\}$, $j \in \{1,2,\ldots,m\}$, then 
				\[\monoid{M}_{\cl{\combgame{\{\xi_1, \xi_2, \ldots, \xi_n\mid\kappa_1, \kappa_2, \ldots, \kappa_m\}}}} \cong \monoid{M}_{\cl{*}}.\]
			
			\item\label{item-xi-^-L-2-cor} If $o^-(\xi_i) = \Prev$ for all $i \in \{1,2,\ldots,n\}$, then $
			\monoid{M}_{\cl{\combgame{\{\xi_1, \xi_2, \ldots, \xi_n\mid\cdot\}}}} \cong \monoid{M}_{\cl{*}}$.  

			\item\label{item-xi-^-L-zero-cor} If $o^-(\xi_i) =o^-(\kappa_j) = \Next$ for all $i \in \{1,2,\ldots,n\}$, $j \in \{1,2,\ldots,m\}$, then the following hold:
				\begin{enumerate}
					\item $\monoid{M}_{\cl{\combgame{\{\xi_1, \xi_2, \ldots, \xi_n \mid0\}}}} \cong  \monoid{M}_{\cl{*}}$;
					\item $\monoid{M}_{\cl{\combgame{\{\xi_1, \xi_2, \ldots, \xi_n\mid\kappa_1, \kappa_2, \ldots, \kappa_m, 0\}}}} \cong  \monoid{M}_{\cl{*}}$;
					\item $\monoid{M}_{\cl{\combgame{\{\xi_1, \xi_2, \ldots, \xi_n, 0\mid \kappa_1, \kappa_2, \ldots, \kappa_m, 0\}}}} \cong  \monoid{M}_{\cl{*}}$.
				\end{enumerate}

			\item\label{item-xi-^-L-cong} For each $\xi_i$, $\kappa_j$, $\monoid{M}_{\cl{\overline{\xi_i}}} \cong \monoid{M}_{\cl{\overline{\kappa_j}}} \cong \monoid{M}_{\cl{*}}$.
		\end{enumerate}
\end{theorem}

\begin{proof}
	To prove the first three assertions of this corollary, use induction on $n$ and $m$ using Lemma \ref{lemma-xi-^-L-*} as the base case. 
	
	The fourth follows from Theorem \ref{theorem-P-iso-*} as every position $\chi \in \cl{\overline{\xi_i}}$ is $\Next$ or $\Prev$ and since there are no moves from $\Next$ positions to $\Next$ positions in $\cl{\xi_i}$, the same holds for those positions in $\cl{\overline{\xi_i}}$.
\end{proof}

Theorem \ref{theorem-xi-^-L-*} is an exceeding useful result.  Using it, we can now easily construct positions whose monoids are isomorphic to that of $\monoid{M}_{\cl{*}}$, as the following example demonstrates.

\begin{example}\label{example-*-iso}
	In this example, we will use Theorem \ref{theorem-xi-^-L-*} to construct some positions of birthday three or less whose monoids are isomorphic to that of $*$'s.  
	
	Four positions are born by day 1 \cite{MCF}.  These are $0$, $1$, $\overline{1}$, and $*$.  Clearly, $*$ is the only position from this list whose monoid is isomorphic to that of $*$.  Thus, by day 1, we have only one position, $*$.
	
	There are 256 positions born by day 2 \cite{MCF}.  We now wish to construct positions born on day 2 with monoids isomorphic to $\monoid{M}_{\cl{*}}$ from the position $*$.  Since $o^-(*) = \Prev$, Theorem \ref{theorem-xi-^-L-*}(\ref{item-xi-^-L-2-cor}) gives that 
		\[\monoid{M}_{\cl{\L(*)}} \cong \monoid{M}_{\cl{*}}.\]
	In Chapter \ref{chapter-examples}, we gave $\L(*)$ a name, $\sigma$.  Thus
		\[ \monoid{M}_{\cl{\sigma}} \cong \monoid{M}_{\cl{*}}.\]
	Using Theorem \ref{theorem-xi-^-L-*}(\ref{item-xi-^-L-cong}), we get
		\[ \monoid{M}_{\cl{\sigmab}} \cong \monoid{M}_{\cl{*}}.\]
	Using Theorem \ref{theorem-xi-^-L-*}(\ref{item-xi-^-L-1-cor}), we get
		\[ \monoid{M}_{\cl{\combgame{\{*\mid*\}}}} \cong \monoid{M}_{\cl{*}}.\]
	Again, in Chapter \ref{chapter-examples}, we gave $\combgame{\{*\mid*\}}$ a name, $\tau$.  Thus
		\[ \monoid{M}_{\cl{\tau}} \cong \monoid{M}_{\cl{*}}.\]
	There are no other operations we can perform from Theorem \ref{theorem-xi-^-L-*} on $*$.  Thus, using Theorem \ref{theorem-xi-^-L-*}, we have found three positions born on day 2 whose monoids are isomorphic to $\monoid{M}_{\cl{*}}$: $\sigma$, $\sigmab$, and $\tau$.  
	
	We now want to take $*$, $\sigma$, $\sigmab$, and $\tau$ and build new positions with monoids isomorphic to $\monoid{M}_{\cl{*}}$.  Since $o^-(*) = \Prev$ while $o^-(\sigma) = o^-(\sigmab) = o^-(\tau) = \Next$, we cannot build a position with monoid isomorphic to $\monoid{M}_{\cl{*}}$ which has $*$ and at least one element of $\{\sigma, \sigmab, \tau\}$ as it would either be a position with outcome neither $\Next$ nor $\Prev$, or have an $\Next$ position moving to an $\Next$ position.  As we just built all positions which come from applying Theorem \ref{theorem-xi-^-L-*} directly upon $*$, this means that the positions we construct born on day 3 will be made from $\sigma$, $\sigmab$, and $\tau$.  Using Theorem \ref{theorem-xi-^-L-*}, there are 224 positions built directly from $\sigma$, $\sigmab$, $\tau$, and 0 which are listed in Table \ref{table-day-3-*}.  Since $\overline{0} = 0$, $\taub = \tau$ and we are using both $\sigma$ and $\sigmab$, for a position in the list of 224 positions, its conjugate is also in the list, meaning we needn't worry about positions arising from applying Theorem \ref{theorem-xi-^-L-*}(\ref{item-xi-^-L-cong}).

	\begin{table}[htb]
	\begin{center}
	\begin{scriptsize}
		\begin{tabular}{llllll}

$\combgame{\{0\mid\tau\}}$&
$\combgame{\{0\mid\sigmab\}}$&
$\combgame{\{0\mid\sigma\}}$&
$\combgame{\{0\mid\tau,0\}}$&
$\combgame{\{0\mid\sigmab,0\}}$&
$\combgame{\{0\mid\sigma,0\}}$
\\

\rowcolor[gray]{.8} 
$\combgame{\{0\mid\sigmab,\tau\}}$&
$\combgame{\{0\mid\sigma,\tau\}}$&
$\combgame{\{0\mid\sigma,\sigmab\}}$&
$\combgame{\{0\mid\sigmab,\tau,0\}}$&
$\combgame{\{0\mid\sigma,\tau,0\}}$&
$\combgame{\{0\mid\sigma,\sigmab,0\}}$
\\

$\combgame{\{0\mid\sigma,\sigmab,\tau\}}$&
$\combgame{\{0\mid\sigma,\sigmab,\tau,0\}}$&
$\combgame{\{\tau\mid0\}}$&
$\combgame{\{\tau\mid\tau\}}$&
$\combgame{\{\tau\mid\sigmab\}}$&
$\combgame{\{\tau\mid\sigma\}}$
\\

\rowcolor[gray]{.8}
$\combgame{\{\tau\mid\tau,0\}}$&
$\combgame{\{\tau\mid\sigmab,0\}}$&
$\combgame{\{\tau\mid\sigma,0\}}$&
$\combgame{\{\tau\mid\sigmab,\tau\}}$&
$\combgame{\{\tau\mid\sigma,\tau\}}$&
$\combgame{\{\tau\mid\sigma,\sigmab\}}$
\\

$\combgame{\{\tau\mid\sigmab,\tau,0\}}$&
$\combgame{\{\tau\mid\sigma,\tau,0\}}$&
$\combgame{\{\tau\mid\sigma,\sigmab,0\}}$&
$\combgame{\{\tau\mid\sigma,\sigmab,\tau\}}$&
$\combgame{\{\tau\mid\sigma,\sigmab,\tau,0\}}$&
$\combgame{\{\sigmab\mid0\}}$
\\

\rowcolor[gray]{.8}
$\combgame{\{\sigmab\mid\tau\}}$&
$\combgame{\{\sigmab\mid\sigmab\}}$&
$\combgame{\{\sigmab\mid\sigma\}}$&
$\combgame{\{\sigmab\mid\tau,0\}}$&
$\combgame{\{\sigmab\mid\sigmab,0\}}$&
$\combgame{\{\sigmab\mid\sigma,0\}}$
\\

$\combgame{\{\sigmab\mid\sigmab,\tau\}}$&
$\combgame{\{\sigmab\mid\sigma,\tau\}}$&
$\combgame{\{\sigmab\mid\sigma,\sigmab\}}$&
$\combgame{\{\sigmab\mid\sigmab,\tau,0\}}$&
$\combgame{\{\sigmab\mid\sigma,\tau,0\}}$&
$\combgame{\{\sigmab\mid\sigma,\sigmab,0\}}$
\\

\rowcolor[gray]{.8}
$\combgame{\{\sigmab\mid\sigma,\sigmab,\tau\}}$&
$\combgame{\{\sigmab\mid\sigma,\sigmab,\tau,0\}}$&
$\combgame{\{\sigma\mid0\}}$&
$\combgame{\{\sigma\mid\tau\}}$&
$\combgame{\{\sigma\mid\sigmab\}}$&
$\combgame{\{\sigma\mid\sigma\}}$
\\

$\combgame{\{\sigma\mid\tau,0\}}$&
$\combgame{\{\sigma\mid\sigmab,0\}}$&
$\combgame{\{\sigma\mid\sigma,0\}}$&
$\combgame{\{\sigma\mid\sigmab,\tau\}}$&
$\combgame{\{\sigma\mid\sigma,\tau\}}$&
$\combgame{\{\sigma\mid\sigma,\sigmab\}}$
\\

\rowcolor[gray]{.8}
$\combgame{\{\sigma\mid\sigmab,\tau,0\}}$&
$\combgame{\{\sigma\mid\sigma,\tau,0\}}$&
$\combgame{\{\sigma\mid\sigma,\sigmab,0\}}$&
$\combgame{\{\sigma\mid\sigma,\sigmab,\tau\}}$&
$\combgame{\{\sigma\mid\sigma,\sigmab,\tau,0\}}$&
$\combgame{\{\tau,0\mid0\}}$
\\

$\combgame{\{\tau,0\mid\tau\}}$&
$\combgame{\{\tau,0\mid\sigmab\}}$&
$\combgame{\{\tau,0\mid\sigma\}}$&
$\combgame{\{\tau,0\mid\tau,0\}}$&
$\combgame{\{\tau,0\mid\sigmab,0\}}$&
$\combgame{\{\tau,0\mid\sigma,0\}}$
\\

\rowcolor[gray]{.8}
$\combgame{\{\tau,0\mid\sigmab,\tau\}}$&
$\combgame{\{\tau,0\mid\sigma,\tau\}}$&
$\combgame{\{\tau,0\mid\sigma,\sigmab\}}$&
$\combgame{\{\tau,0\mid\sigmab,\tau,0\}}$&
$\combgame{\{\tau,0\mid\sigma,\tau,0\}}$&
$\combgame{\{\tau,0\mid\sigma,\sigmab,0\}}$
\\

$\combgame{\{\tau,0\mid\sigma,\sigmab,\tau\}}$&
$\combgame{\{\tau,0\mid\sigma,\sigmab,\tau,0\}}$&
$\combgame{\{\sigmab,0\mid0\}}$&
$\combgame{\{\sigmab,0\mid\tau\}}$&
$\combgame{\{\sigmab,0\mid\sigmab\}}$&
$\combgame{\{\sigmab,0\mid\sigma\}}$
\\

\rowcolor[gray]{.8}
$\combgame{\{\sigmab,0\mid\tau,0\}}$&
$\combgame{\{\sigmab,0\mid\sigmab,0\}}$&
$\combgame{\{\sigmab,0\mid\sigma,0\}}$&
$\combgame{\{\sigmab,0\mid\sigmab,\tau\}}$&
$\combgame{\{\sigmab,0\mid\sigma,\tau\}}$&
$\combgame{\{\sigmab,0\mid\sigma,\sigmab\}}$
\\

$\combgame{\{\sigmab,0\mid\sigmab,\tau,0\}}$&
$\combgame{\{\sigmab,0\mid\sigma,\tau,0\}}$&
$\combgame{\{\sigmab,0\mid\sigma,\sigmab,0\}}$&
$\combgame{\{\sigmab,0\mid\sigma,\sigmab,\tau\}}$&
$\combgame{\{\sigmab,0\mid\sigma,\sigmab,\tau,0\}}$&
$\combgame{\{\sigma,0\mid0\}}$
\\

\rowcolor[gray]{.8}
$\combgame{\{\sigma,0\mid\tau\}}$&
$\combgame{\{\sigma,0\mid\sigmab\}}$&
$\combgame{\{\sigma,0\mid\sigma\}}$&
$\combgame{\{\sigma,0\mid\tau,0\}}$&
$\combgame{\{\sigma,0\mid\sigmab,0\}}$&
$\combgame{\{\sigma,0\mid\sigma,0\}}$
\\

$\combgame{\{\sigma,0\mid\sigmab,\tau\}}$&
$\combgame{\{\sigma,0\mid\sigma,\tau\}}$&
$\combgame{\{\sigma,0\mid\sigma,\sigmab\}}$&
$\combgame{\{\sigma,0\mid\sigmab,\tau,0\}}$&
$\combgame{\{\sigma,0\mid\sigma,\tau,0\}}$&
$\combgame{\{\sigma,0\mid\sigma,\sigmab,0\}}$
\\

\rowcolor[gray]{.8}
$\combgame{\{\sigma,0\mid\sigma,\sigmab,\tau\}}$&
$\combgame{\{\sigma,0\mid\sigma,\sigmab,\tau,0\}}$&
$\combgame{\{\sigmab,\tau\mid0\}}$&
$\combgame{\{\sigmab,\tau\mid\tau\}}$&
$\combgame{\{\sigmab,\tau\mid\sigmab\}}$&
$\combgame{\{\sigmab,\tau\mid\sigma\}}$
\\

$\combgame{\{\sigmab,\tau\mid\tau,0\}}$&
$\combgame{\{\sigmab,\tau\mid\sigmab,0\}}$&
$\combgame{\{\sigmab,\tau\mid\sigma,0\}}$&
$\combgame{\{\sigmab,\tau\mid\sigmab,\tau\}}$&
$\combgame{\{\sigmab,\tau\mid\sigma,\tau\}}$&
$\combgame{\{\sigmab,\tau\mid\sigma,\sigmab\}}$
\\

\rowcolor[gray]{.8}
$\combgame{\{\sigmab,\tau\mid\sigmab,\tau,0\}}$&
$\combgame{\{\sigmab,\tau\mid\sigma,\tau,0\}}$&
$\combgame{\{\sigmab,\tau\mid\sigma,\sigmab,0\}}$&
$\combgame{\{\sigmab,\tau\mid\sigma,\sigmab,\tau\}}$&
$\combgame{\{\sigmab,\tau\mid\sigma,\sigmab,\tau,0\}}$&
$\combgame{\{\sigma,\tau\mid0\}}$
\\

$\combgame{\{\sigma,\tau\mid\tau\}}$&
$\combgame{\{\sigma,\tau\mid\sigmab\}}$&
$\combgame{\{\sigma,\tau\mid\sigma\}}$&
$\combgame{\{\sigma,\tau\mid\tau,0\}}$&
$\combgame{\{\sigma,\tau\mid\sigmab,0\}}$&
$\combgame{\{\sigma,\tau\mid\sigma,0\}}$
\\

\rowcolor[gray]{.8}
$\combgame{\{\sigma,\tau\mid\sigmab,\tau\}}$&
$\combgame{\{\sigma,\tau\mid\sigma,\tau\}}$&
$\combgame{\{\sigma,\tau\mid\sigma,\sigmab\}}$&
$\combgame{\{\sigma,\tau\mid\sigmab,\tau,0\}}$&
$\combgame{\{\sigma,\tau\mid\sigma,\tau,0\}}$&
$\combgame{\{\sigma,\tau\mid\sigma,\sigmab,0\}}$
\\

$\combgame{\{\sigma,\tau\mid\sigma,\sigmab,\tau\}}$&
$\combgame{\{\sigma,\tau\mid\sigma,\sigmab,\tau,0\}}$&
$\combgame{\{\sigma,\sigmab\mid0\}}$&
$\combgame{\{\sigma,\sigmab\mid\tau\}}$&
$\combgame{\{\sigma,\sigmab\mid\sigmab\}}$&
$\combgame{\{\sigma,\sigmab\mid\sigma\}}$
\\

\rowcolor[gray]{.8}
$\combgame{\{\sigma,\sigmab\mid\tau,0\}}$&
$\combgame{\{\sigma,\sigmab\mid\sigmab,0\}}$&
$\combgame{\{\sigma,\sigmab\mid\sigma,0\}}$&
$\combgame{\{\sigma,\sigmab\mid\sigmab,\tau\}}$&
$\combgame{\{\sigma,\sigmab\mid\sigma,\tau\}}$&
$\combgame{\{\sigma,\sigmab\mid\sigma,\sigmab\}}$
\\

$\combgame{\{\sigma,\sigmab\mid\sigmab,\tau,0\}}$&
$\combgame{\{\sigma,\sigmab\mid\sigma,\tau,0\}}$&
$\combgame{\{\sigma,\sigmab\mid\sigma,\sigmab,0\}}$&
$\combgame{\{\sigma,\sigmab\mid\sigma,\sigmab,\tau\}}$&
$\combgame{\{\sigma,\sigmab\mid\sigma,\sigmab,\tau,0\}}$&
$\combgame{\{\sigmab,\tau,0\mid0\}}$
\\

\rowcolor[gray]{.8}
$\combgame{\{\sigmab,\tau,0\mid\tau\}}$&
$\combgame{\{\sigmab,\tau,0\mid\sigmab\}}$&
$\combgame{\{\sigmab,\tau,0\mid\sigma\}}$&
$\combgame{\{\sigmab,\tau,0\mid\tau,0\}}$&
$\combgame{\{\sigmab,\tau,0\mid\sigmab,0\}}$&
$\combgame{\{\sigmab,\tau,0\mid\sigma,0\}}$
\\

$\combgame{\{\sigmab,\tau,0\mid\sigmab,\tau\}}$&
$\combgame{\{\sigmab,\tau,0\mid\sigma,\tau\}}$&
$\combgame{\{\sigmab,\tau,0\mid\sigma,\sigmab\}}$&
$\combgame{\{\sigmab,\tau,0\mid\sigmab,\tau,0\}}$&
$\combgame{\{\sigmab,\tau,0\mid\sigma,\tau,0\}}$&
$\combgame{\{\sigmab,\tau,0\mid\sigma,\sigmab,0\}}$
\\

\rowcolor[gray]{.8}
$\combgame{\{\sigmab,\tau,0\mid\sigma,\sigmab,\tau\}}$&
$\combgame{\{\sigmab,\tau,0\mid\sigma,\sigmab,\tau,0\}}$&
$\combgame{\{\sigma,\tau,0\mid0\}}$&
$\combgame{\{\sigma,\tau,0\mid\tau\}}$&
$\combgame{\{\sigma,\tau,0\mid\sigmab\}}$&
$\combgame{\{\sigma,\tau,0\mid\sigma\}}$
\\

$\combgame{\{\sigma,\tau,0\mid\tau,0\}}$&
$\combgame{\{\sigma,\tau,0\mid\sigmab,0\}}$&
$\combgame{\{\sigma,\tau,0\mid\sigma,0\}}$&
$\combgame{\{\sigma,\tau,0\mid\sigmab,\tau\}}$&
$\combgame{\{\sigma,\tau,0\mid\sigma,\tau\}}$&
$\combgame{\{\sigma,\tau,0\mid\sigma,\sigmab\}}$
\\

\rowcolor[gray]{.8}
$\combgame{\{\sigma,\tau,0\mid\sigmab,\tau,0\}}$&
$\combgame{\{\sigma,\tau,0\mid\sigma,\tau,0\}}$&
$\combgame{\{\sigma,\tau,0\mid\sigma,\sigmab,0\}}$&
$\combgame{\{\sigma,\tau,0\mid\sigma,\sigmab,\tau\}}$&
$\combgame{\{\sigma,\tau,0\mid\sigma,\sigmab,\tau,0\}}$&
$\combgame{\{\sigma,\sigmab,0\mid0\}}$
\\

$\combgame{\{\sigma,\sigmab,0\mid\tau\}}$&
$\combgame{\{\sigma,\sigmab,0\mid\sigmab\}}$&
$\combgame{\{\sigma,\sigmab,0\mid\sigma\}}$&
$\combgame{\{\sigma,\sigmab,0\mid\tau,0\}}$&
$\combgame{\{\sigma,\sigmab,0\mid\sigmab,0\}}$&
$\combgame{\{\sigma,\sigmab,0\mid\sigma,0\}}$
\\

\rowcolor[gray]{.8}
$\combgame{\{\sigma,\sigmab,0\mid\sigmab,\tau\}}$&
$\combgame{\{\sigma,\sigmab,0\mid\sigma,\tau\}}$&
$\combgame{\{\sigma,\sigmab,0\mid\sigma,\sigmab\}}$&
$\combgame{\{\sigma,\sigmab,0\mid\sigmab,\tau,0\}}$&
$\combgame{\{\sigma,\sigmab,0\mid\sigma,\tau,0\}}$&
$\combgame{\{\sigma,\sigmab,0\mid\sigma,\sigmab,0\}}$
\\

$\combgame{\{\sigma,\sigmab,0\mid\sigma,\sigmab,\tau\}}$&
$\combgame{\{\sigma,\sigmab,0\mid\sigma,\sigmab,\tau,0\}}$&
$\combgame{\{\sigma,\sigmab,\tau\mid0\}}$&
$\combgame{\{\sigma,\sigmab,\tau\mid\tau\}}$&
$\combgame{\{\sigma,\sigmab,\tau\mid\sigmab\}}$&
$\combgame{\{\sigma,\sigmab,\tau\mid\sigma\}}$
\\

\rowcolor[gray]{.8}
$\combgame{\{\sigma,\sigmab,\tau\mid\tau,0\}}$&
$\combgame{\{\sigma,\sigmab,\tau\mid\sigmab,0\}}$&
$\combgame{\{\sigma,\sigmab,\tau\mid\sigma,0\}}$&
$\combgame{\{\sigma,\sigmab,\tau\mid\sigmab,\tau\}}$&
$\combgame{\{\sigma,\sigmab,\tau\mid\sigma,\tau\}}$&
$\combgame{\{\sigma,\sigmab,\tau\mid\sigma,\sigmab\}}$
\\

$\combgame{\{\sigma,\sigmab,\tau\mid\sigmab,\tau,0\}}$&
$\combgame{\{\sigma,\sigmab,\tau\mid\sigma,\tau,0\}}$&
$\combgame{\{\sigma,\sigmab,\tau\mid\sigma,\sigmab,0\}}$&
$\combgame{\{\sigma,\sigmab,\tau\mid\sigma,\sigmab,\tau\}}$&
$\combgame{\{\sigma,\sigmab,\tau\mid\sigma,\sigmab,\tau,0\}}$&
$\combgame{\{\sigma,\sigmab,\tau,0\mid0\}}$
\\

\rowcolor[gray]{.8}
$\combgame{\{\sigma,\sigmab,\tau,0\mid\tau\}}$&
$\combgame{\{\sigma,\sigmab,\tau,0\mid\sigmab\}}$&
$\combgame{\{\sigma,\sigmab,\tau,0\mid\sigma\}}$&
$\combgame{\{\sigma,\sigmab,\tau,0\mid\tau,0\}}$&
$\combgame{\{\sigma,\sigmab,\tau,0\mid\sigmab,0\}}$&
$\combgame{\{\sigma,\sigmab,\tau,0\mid\sigma,0\}}$
\\

$\combgame{\{\sigma,\sigmab,\tau,0\mid\sigmab,\tau\}}$&
$\combgame{\{\sigma,\sigmab,\tau,0\mid\sigma,\tau\}}$&
$\combgame{\{\sigma,\sigmab,\tau,0\mid\sigma,\sigmab\}}$&
$\combgame{\{\sigma,\sigmab,\tau,0\mid\sigmab,\tau,0\}}$&
$\combgame{\{\sigma,\sigmab,\tau,0\mid\sigma,\tau,0\}}$&
$\combgame{\{\sigma,\sigmab,\tau,0\mid\sigma,\sigmab,0\}}$
\\

\rowcolor[gray]{.8}
$\combgame{\{\sigma,\sigmab,\tau,0\mid\sigma,\sigmab,\tau\}}$&
$\combgame{\{\sigma,\sigmab,\tau,0\mid\sigma,\sigmab,\tau,0\}}$

		\end{tabular}	
	\caption{224 positions born on day 3 with monoids isomorphic to $\monoid{M}_{\cl{*}}$.}
	\label{table-day-3-*}
	\end{scriptsize}
	\end{center}
	\end{table}

		\begin{figure}[htb]	
		\begin{center}
		\unitlength 8pt
		\begin{graph}(40,7.5)(-20,-1.5)
		\graphnodesize{0.4}
		\fillednodestrue

		\roundnode{A}(-19,2)
		\roundnode{B}(-20,0)
		\roundnode{C}(-18,0)
		
		\edge{A}{C}
		\edge{A}{B}
		\freetext(-19,-1.5){$*$}

		\roundnode{A3}(-9,4)
		\roundnode{B3}(-10,2)
		\roundnode{D3}(-11,0)
		\roundnode{E3}(-9,0)

		\edge{A3}{B3}
		\edge{B3}{D3}
		\edge{B3}{E3}
		\freetext(-10,-1.5){$\sigma$}

	\roundnode{49a}(10,6)
		\roundnode{49At}(1,4)
		\roundnode{49Bt}(0,2)
		\roundnode{49Ct}(2,2)
		\roundnode{49Dt}(-.5,0)
		\roundnode{49Et}(.5,0)
		\roundnode{49Ft}(1.5,0)
		\roundnode{49Gt}(2.5,0)

		\edge{49At}{49Ct}
		\edge{49At}{49Bt}
		\edge{49Bt}{49Dt}
		\edge{49Bt}{49Et}
		\edge{49Ct}{49Ft}
		\edge{49Ct}{49Gt}					

		\roundnode{49As}(5,4)
		\roundnode{49Bs}(4,2)
		\roundnode{49Ds}(3.5,0)
		\roundnode{49Es}(4.5,0)

		\edge{49As}{49Bs}
		\edge{49Bs}{49Ds}
		\edge{49Bs}{49Es}

		\roundnode{49Ab}(7,4)
		\roundnode{49Cb}(8,2)
		\roundnode{49Fb}(7.5,0)
		\roundnode{49Gb}(8.5,0)

		\edge{49Ab}{49Cb}
		\edge{49Cb}{49Fb}
		\edge{49Cb}{49Gb}	

		\roundnode{49AT}(19,4)
		\roundnode{49BT}(20,2)
		\roundnode{49CT}(18,2)
		\roundnode{49DT}(20.5,0)
		\roundnode{49ET}(19.5,0)
		\roundnode{49FT}(18.5,0)
		\roundnode{49GT}(17.5,0)

		\edge{49AT}{49CT}
		\edge{49AT}{49BT}
		\edge{49BT}{49DT}
		\edge{49BT}{49ET}
		\edge{49CT}{49FT}
		\edge{49CT}{49GT}					

		\roundnode{49AS}(15,4)
		\roundnode{49BS}(16,2)
		\roundnode{49DS}(16.5,0)
		\roundnode{49ES}(15.5,0)

		\edge{49AS}{49BS}
		\edge{49BS}{49DS}
		\edge{49BS}{49ES}

		\roundnode{49AB}(13,4)


	\edge{49a}{49At}
	\edge{49a}{49As}
	\edge{49a}{49Ab}
	\edge{49a}{49AT}
	\edge{49a}{49AS}
	\edge{49a}{49AB}
	\freetext(10,-1.5){$\combgame{\{\tau, \sigma, \sigmab \mid 0, \sigmab, \tau\}}$}
		
		\end{graph}
		\end{center}
		\caption{Some positions with monoids isomorphic to $\monoid{M}_{\cl{*}}$.}
		\label{fig-day-3-*}
		\end{figure}	
	
That is, for example, the three positions in Figure \ref{fig-day-3-*} all have monoids isomorphic to $\monoid{M}_{\cl{*}}$.
\end{example}

As shown in Example \ref{example-*-iso}, by iteratively applying Theorem \ref{theorem-xi-^-L-*} to $*$, we obtain a large variety of positions with monoids isomorphic to $\monoid{M}_{\cl{*}}$.  We wish to differentiate between positions which are built via iteratively applying Theorem \ref{theorem-xi-^-L-*} to $*$ versus those which were not.

\begin{definition}
	For a position $\xi$ with $\monoid{M}_{\cl{\xi}} \cong \monoid{M}_{\cl{*}}$, we say that a $\xi$ is \textbf{$*$-built} if either
		\begin{itemize}
			\item $\xi = *$, or
			\item $\xi$ was obtained by iteratively applying Theorem \ref{theorem-xi-^-L-*} to $*$.
		\end{itemize}
\end{definition}

From this definition, an obvious question arises: are there positions $\xi$ with $\monoid{M}_{\cl{\xi}} \cong \monoid{M}_{\cl{*}}$ which are not $*$-built?  The answer is no.

\begin{theorem}\label{theorem-*-built}
	All positions $\xi$ with $\monoid{M}_{\cl{\xi}} \cong \monoid{M}_{\cl{*}}$ are $*$-built.
\end{theorem}

\begin{proof}
	Suppose, contrary to what we must show, that there exist positions $\xi$ with $\monoid{M}_{\cl{\xi}} \cong \monoid{M}_{\cl{*}}$, but $\xi$ is not $*$-built.  Take such a $\xi$ of minimal birthday.
	
	Suppose that all the options of $\xi$ are 0, that is $\xi = \combgame{\{0\mid0\}} = *$, a contradiction.  Thus, suppose that $\xi$ has a non-zero option, say $\xi^L$.  Then $\monoid{M}_{\cl{\xi^L}} \cong \monoid{M}_{\cl{*}}$ since Theorem \ref{theorem-P-iso-*} is satisfied for $\cl{\xi^L}$ because it is satisfied for $\cl{\xi}$.  That is, all non-zero options $\xi^L$ and $\xi^R$ of $\xi$ have monoids which are isomorphic to $\monoid{M}_{\cl{*}}$.  Of these non-zero options, all must be $*$-built, as they all have birthday strictly less than that of $\xi$.
	
	Suppose $\xi$ has a move to 0.  Since $o^-(0) = \Next$, Theorem \ref{theorem-P-iso-*} gives that $o^-(\xi) = \Prev$.  Thus both Left and Right must each have at least one move available, all moves are to options whose outcomes are $\Next$, and all these options are $*$-built.  Thus $\xi$ was constructed by applying Theorem \ref{theorem-xi-^-L-*}\eqref{item-xi-^-L-zero-cor} to its options, contradicting our assumption that $\xi$ is not $*$-built.
	
	Suppose now that $\xi$ has no move to 0.  Suppose also that $o^-(\xi) = \Prev$.  Thus both Left and Right must each have at least one move available, all moves are to options whose outcomes are $\Next$, and all these options are $*$-built.  Thus $\xi$ was constructed by applying Theorem \ref{theorem-xi-^-L-*}\eqref{item-xi-^-L-1-cor} to its options, contradicting our assumption that $\xi$ is not $*$-built.
	
	Suppose now that $o^-(\xi) = \Next$.   Moreover, suppose that Right has no move available.  Then Left must have a move available otherwise $\xi = 0$, and $\monoid{M}_{\cl{0}} \not \cong \monoid{M}_{\cl{*}}$.  By Theorem \ref{theorem-P-iso-*}, all of Left's options must be $\Prev$, and all of these options are $*$-built.  Thus $\xi$ was constructed by applying Theorem \ref{theorem-xi-^-L-*}\eqref{item-xi-^-L-2-cor} to its options, contradicting our assumption that $\xi$ is not $*$-built.  Similarly, if Left has no move available, but Right does, the $\xi$ was constructed by applying Theorem \ref{theorem-xi-^-L-*}\eqref{item-xi-^-L-2-cor} and Theorem \ref{theorem-xi-^-L-*}\eqref{item-xi-^-L-cong} to its options.
	
	Finally, suppose that $o^-(\xi) = \Next$ and that both Right and Left have moves available.  By Theorem \ref{theorem-P-iso-*}, all these options are  $\Prev$, and all of these options are $*$-built.  Thus $\xi$ was constructed by applying Theorem \ref{theorem-xi-^-L-*}\eqref{item-xi-^-L-1-cor} to its options, contradicting our assumption that $\xi$ is not $*$-built.
	
	Therefore no such $\xi$ can exist.  That is, all $\xi$ with $\monoid{M}_{\cl{\xi}} \cong \monoid{M}_{\cl{*}}$ are $*$-built.
\end{proof}

We have now fully classified positions with monoid isomorphic to $\monoid{M}_{\cl{*}}$, the first such classification result in all partizan \mis play game theory.

Now that we have positions $\xi$ with $\monoid{M}_{\cl{\xi}} \cong \monoid{M}_{\cl{*}}$, how can we use this result?  Ideally, much as in normal play, we would like $\xi$ and $*$ to be interchangeable, i.e.\ if $*$ is an option of some position, then we can replace $*$ by $\xi$ if $\monoid{M}_{\cl{\xi}} \cong \monoid{M}_{\cl{*}}$.  However, this is not true, as the following example shows.

\begin{example}\label{example-replace-*}
	Recall the position $\rho = \combgame{\{*\mid0\}}$ introduced in Section \ref{section-rho}.  In Section \ref{section-rho}, we showed that $\monoid{M}_{\cl{\rho}}$ is as follows:
		\begin{align*}
			\monoid{M}_{\cl{\rho}} &= \ideal{1,a,p \mid a^2 = 1, p^4 = p^5 = ap^4} \\
			\Next &= \{1, ap, ap^2, ap^3 \} \\
			\Prev &= \{a, p^2\} \\
			\Left &= \{p\} \\
			\Right &= \{p^3, p^4\}.
		\end{align*}
	Clearly, $\monoid{M}_{\cl{\rho}} \not \cong \monoid{M}_{\cl{*}}$.
	
	We have seen that $\monoid{M}_{\cl{\sigma}} \cong \monoid{M}_{\cl{*}}$ (Section \ref{sec-sigma} and Example \ref{example-*-iso}).  However, if we replace $*$ by $\sigma$ in $\rho = \combgame{\{*\mid0\}}$, we get the position $\combgame{\{\sigma\mid0\}}$.  However, Theorem \ref{theorem-xi-^-L-*} gives that
	$\monoid{M}_{\cl{\combgame{\{\sigma\mid0\}}}} \cong \monoid{M}_{\cl{*}}$. 
\end{example}

Therefore, we cannot exchange $*$ and $\xi$ and guarantee that our resultant monoid remain the same.  However, in Example \ref{example-replace-*}, we tried to replace $*$, a $\Prev$ position, with $\sigma$, an $\Next$ position.  It seems more likely that we can replace $*$ with $\xi$ if $o^-(\xi) = \Prev$.   As such, we have the following conjecture.

\begin{conjecture}\label{conjecture-replace-*}
	If $\xi$ is a position with $o^-(\xi) = \Prev$ and $\monoid{M}_{\cl{\xi}} \cong \monoid{M}_{\cl{*}}$, then we can replace $*$ by $\xi$ in any position which has $*$ as an option without changing the resultant \mis monoid.
\end{conjecture}

If this conjecture is true, then this gives us a tool for building positions with isomorphic monoids, as all we need do is find positions with $*$ in its options and replace it by an appropriate $\xi$.  For example, we could construct positions with monoids isomorphic to $\monoid{M}_{\cl{\rho}}$, or monoids isomorphic to $\monoid{M}_{\cl{*_2}}$ (Table \ref{table-normal-mis-outcomes}).  It may even be possible to extend this result further, i.e.\ replacing $\rho = \combgame{\{*\mid0\}}$ by $\combgame{\{\xi\mid0\}}$ for a suitable $\xi$ in positions with $\rho$ as an option.  This is definitely an area which merits further investigation.

\section{Conclusion}

In Section \ref{sec-iso-*}, much work was spent on positions with monoids isomorphic to $\monoid{M}_{\cl{*}}$.   The proofs in this section, especially that of Theorem \ref{theorem-p1-p2-*}, relied heavily on the straightforward structure of $\monoid{M}_{\cl{*}}$, that is, that $\monoid{M}_{\cl{*}}$ is isomorphic to $\mathbb{Z}_2$ as monoids.  While we did obtain a complete classification for positions with monoids isomorphic to $\monoid{M}_{\cl{*}}$, we would ideally like to extend these results for positions other than $*$.  For example, given two closed sets of positions $S_1$ and $S_2$, with \mis monoids $\monoid{M}_{S_1}$ and $\monoid{M}_{S_2}$ respectively, if
	\[ \monoid{M}_{S_1} \cong \monoid{M}_{S_2},\]
is it then true that
	\[ \monoid{M}_{S_1+S_2} \cong  \monoid{M}_{S_1},\]
where
	\[ S_1 + S_2 = \{s_1 + s_2 \mid s_1 \in S_1, s_2 \in S_2\}?\]
At first glance, it seems obvious that this must be true; since the positions have isomorphic monoids, they must behave in some similar fashion, and so the sums of positions from these two sets must also behave in the same similar fashion.  But, at the same time, how do we guarantee that some option of $s$ which is dominated when restricting ourselves to positions in one set is still dominated when we consider $s$ as an element of the sum of sets?  At this point, we are required to prove or give a counterexample.  However, even to find positions with isomorphic monoids which are neither isomorphic to $\monoid{M}_{\cl{*}}$ nor have $S_1$ and $S_2$ related in obvious ways (such as $S_1 \subset S_2$, so $S_1 + S_2 = S_2$) poses a major challenge.  An obvious approach to first try is to examine the large number of impartial examples calculated by Plambeck and Siegel and try to see there if a counterexample can be found.  If not, then more work must be done to see if the result can be proven.  At the very least, we have started the work by characterizing the $*$ positions.

\chapter{Two Examples of Partizan Heap-based \Mis Monoids}\label{chapter-heap}

\section{Introduction}

The first investigations into \mis monoids were to calculate the \mis monoids of certain impartial heap-based games \cite{TAMING}, with the general theory growing out of those initial explorations.  As such, it seems fitting that this thesis concludes with the calculation of \mis monoids for two partizan heap-based games.  Of these two examples, one yields a finite monoid, while the other yields an infinite one. To calculate these monoids, we will use the method designed by Mike Weimerskirch \cite{MIKE}.  While Weimerskirch designed his algorithm for impartial heap-based games, nothing in his algorithm is specific to impartial games.  As such, his algorithm works for partizan heap-based games as well.

As in \cite{MIKE}, we use $(x_1, x_2, \ldots, x_n)$ to denote a position, where $x_i$ is the number of heaps of size $i$.  If there are no heaps of size $k$ or higher, then we truncate the string, i.e.\ the position $(x_1, x_2)$ means there are $x_1$ heaps of size one, $x_2$ heaps of size two, and zero heaps of size three or higher.

The method of \cite{MIKE} is essentially performing induction on the outcome tables.  We find (if it exists) the periodicity of outcomes for positions in $\cl{h_1, \ldots, h_n}$.  We then add in additional heaps of size $n+1$ and search for periodicity results (if they exist) with the addition of these larger heaps.  These periodicity results determine a candidate \mis monoid.  The indistinguishability of elements is then checked to see if additional relations are required in the monoid.  If so, these relations are added and the \mis monoid is obtained.  If no additional relations are found, then the candidate \mis monoid is indeed the \mis monoid.  

In applying the algorithm, we make use of Theorem 2 of \cite{MIKE}, which is the formal statement of the preceding paragraph.  Before we can state this theorem, we need a definition.

\begin{definition}
	We say that a position $a=(a_1, a_2, \ldots, a_n)$ precedes a position $b=(b_1, b_2, \ldots, b_n)$ in the \textbf{colexicographic order} if the rightmost coordinate in which $a$ and $b$ differ is smaller in $a$.  
\end{definition}

We can now state Theorem 2 of \cite{MIKE}.

\begin{theorem}[Theorem 2 of \cite{MIKE}]\label{theorem-2-mike}
	Fix a heap-based game played under \mis play and fix values for $i$, $y_{i+1}$, $y_{i+2}$, $\ldots$, $y_n$.  Suppose for some $r_i$, $d_i$,
		\begin{enumerate}
			\item\label{item-1-mike} the outcomes for positions of the form
				\[ g = (x_1, x_2, \ldots, x_{i-1}, r_i, y_{i+1}, y_{i+2}, \ldots, y_n)\]
			agree with the outcomes of
				\[ g^* = (x_1, x_2, \ldots, x_{i-1}, r_i+d_i, y_{i+1}, y_{i+2}, \ldots, y_n)\]
			for all $x_1$, $x_2$, $\ldots$, $x_{i-1}$.
			
			\item\label{item-2-mike} In addition, suppose that the outcomes for positions of the form
				\[ k = (x_1, x_2, \ldots, x_{i-1}, r_i+u, x_{i+1}, x_{i+2}, \ldots, x_n)\]
			agree with the outcomes of
				\[ k^* = (x_1, x_2, \ldots, x_{i-1}, r_i+d_i+u, x_{i+1}, x_{i+2}, \ldots, x_n)\]
			for all $x_1$, $x_2$, $\ldots$, $x_{i-1}$, for $(x_{i+1}, x_{i+2}, \ldots, x_n)$ preceding $(y_{i+1}, y_{i+2}, \ldots, y_n)$ in the colexicographic ordering and for all $u \ge 0$.
		\end{enumerate}
	Then $o^-(g+vh_i) = o^-(g^* + v h_i)$ for all $v$, $x_1$, $x_2$, $\ldots$, $x_{i-1} \ge 0$, where $g + vh_i$ is the position
		\[ g+ vh_i = (x_1, x_2, \ldots, x_{i-1}, r_i+v, y_{i+1}, y_{i+2}, \ldots, y_n) \]
	and $g^* + v h_i$ is the position
		\[ g^* + vh_i = (x_1, x_2, \ldots, x_{i-1}, r_i+d_i+v, y_{i+1}, y_{i+2}, \ldots, y_n).\]
\end{theorem}

We will use Weimerskirch's method of \cite{MIKE} to calculate the \mis monoids of two partizan subtraction games.  One of these monoids is finite, while the other is infinite, once again demonstrating the range that can occur with \mis monoids.  Neither of these \mis monoids have an underlying monoid structure which is isomorphic to a well-known monoid.

\section{The Partizan Subtraction Game $\Left(1,2)$, $\Right(1)$}

The first heap-based game we will investigate is the subtraction game $\Left(1,2)$, $\Right(1)$. In this game, Left has two possible moves, to subtract either one or two tokens from a heap, while Right has a single possible move, to subtract one token from a heap.  

\begin{definition}
	Let $h_n$ denote a heap of size $n$.
\end{definition}

We first determine the outcome classes of $h_n$ for arbitrary $n \in \mathbb{Z}^{\ge 0}$.

\begin{proposition}
	The outcomes for $h_n$ in the game $\Left(1,2)$, $\Right(1)$ are as follows:
		\[ o^-(h_n) = \begin{cases}
			\Next &\text{if } n = 0,2;\\
			\Prev &\text{if } n=1;\\
			\Left &\text{else.}
		\end{cases}\]
\end{proposition}

\begin{proof}
	Clearly $o^-(h_0) = \Next$ as neither Left nor Right has any move from a heap of size zero.
	
	From $h_1$, both Left and Right can only move to $h_0$, an $\Next$ position.  Therefore $o^-(h_1) = \Prev$.
	
	From $h_2$, both Left and Right have the option of moving to $h_1$, a $\Prev$ position.  Therefore $o^-(h_2) = \Next$.
	
	We claim that $o^-(h_n) = \Left$ for all $n \in \mathbb{Z}^{\ge 3}$.  We proceed by induction on $n$.
	
	From $h_3$, Left takes two tokens and moves to $h_1$, a $\Prev$ position.  Right can only move to $h_2$, an $\Next$ position.  Therefore $o^-(h_3) = \Left$.  This shows the base case.
	
	Suppose true for all $3 \le n < k$ and consider $n = k$.  By induction, $o^-(h_{k-1}) = \Left$.  Left will move from $h_k$ to $h_{k-1}$, while this is Right's only available move.  Therefore $o^-(h_k) = \Left$, as required.
\end{proof}

We will now calculate the \mis monoid of this game using the method of \cite{MIKE}.  

We begin by looking at the position $(x_1)$.  The outcomes $o^-((x_1))$ are given in Table \ref{table-L(1,2)R(1)-x1}.

\begin{table}[htb]
\begin{center}
	\begin{tabular}{rccccccc}
		$x_1=$ & 0 & 1 & 2 & 3 & 4 & 5 & $\ldots$\\
		$o^-((x_1))=$&\cellcolor[gray]{0.8}$\Next$&\cellcolor[gray]{0.8}$\Prev$&\cellcolor[gray]{0.8}$\Next$&\cellcolor[gray]{0.8}$\Prev$&\cellcolor[gray]{0.8}$\Next$&\cellcolor[gray]{0.8}$\Prev$&\cellcolor[gray]{0.8}$\ldots$
	\end{tabular}	
\end{center}
\caption{The outcomes for positions $(x_1)$ in the game $\Left(1,2)$, $\Right(1)$.}
\label{table-L(1,2)R(1)-x1}
\end{table}	

It is easy to see that the pattern $o^-(x_1) = o^-(x_1 +2)$ will continue.  Using the notation of \cite{MIKE}, we let $R_1$ denote the pre-period and $D_1$ denote the period, and so $R_1 = (0)$, $D_1 = (2)$.  This gives the candidate monoid
	\begin{align*}
	\monoid{M}_{\cl{h_1}}^{\heartsuit} &= \ideal{1,a \mid a^2=1} \\
	\Next &= \{1\} \\
	\Prev &= \{a\} \\
	\Left &= \emptyset \\
	\Right &= \emptyset.
	\end{align*}
via the map
	\begin{align*}
		h_0 &\mapsto 1, \\
		h_1 &\mapsto a.
	\end{align*}
As this monoid has only two elements ($1$ and $a$) and they are clearly distinguishable (as they have different outcome classes), there are no further relations and $\monoid{M}_{\cl{h_1}}^{\heartsuit} = \monoid{M}_{\cl{h_1}}$.  

We now continue to investigate by examining positions $(x_1, x_2)$.  Table \ref{table-L(1,2)R(1)-x1.x2} gives outcomes $o^-((x_1, x_2))$.

\begin{table}[htb]
\begin{center}

\begin{tabular}{rccccc}
$x_1=$&0&1&2&3&\ldots\\
$o^-((x_1,0))=$&\cellcolor[gray]{0.8}$\Next$&\cellcolor[gray]{0.8}$\Prev$&\cellcolor[gray]{0.8}$\Next$&\cellcolor[gray]{0.8}$\Prev$&\cellcolor[gray]{0.8}$\ldots$\\
$o^-((x_1,1))=$&$\Next$&$\Left$&$\Next$&$\Left$&$\ldots$\\
$o^-((x_1,2))=$&\cellcolor[gray]{0.8}$\Left$&\cellcolor[gray]{0.8}$\Left$&\cellcolor[gray]{0.8}$\Left$&\cellcolor[gray]{0.8}$\Left$&\cellcolor[gray]{0.8}$\ldots$\\
$o^-((x_1,3))=$&$\Left$&$\Left$&$\Left$&$\Left$&$\ldots$\\
\end{tabular}

\end{center}
\caption{The outcomes for positions $(x_1,x_2)$ in the game $\Left(1,2)$, $\Right(1)$.}
\label{table-L(1,2)R(1)-x1.x2}
\end{table}	

To determine the outcomes of positions in Table \ref{table-L(1,2)R(1)-x1.x2}, we proceeded as follows:  From any position in the table, Left's moves are
	\begin{enumerate}
		\item to move one position to the left (corresponding with taking a 
		one token from a heap of size one),
		\item to move to the position diagonally right and up (corresponding with taking one token from a heap of size two), or
		\item to move one position up (corresponding with taking two tokens from a heap of size two).
	\end{enumerate}

Right's moves are
	\begin{enumerate}
		\item to move one position to the left (corresponding with taking a 
		one token from a heap of size one), or
		\item to move to the position diagonally right and up (corresponding with taking one token from a heap of size two).
	\end{enumerate}

That the first row in Table \ref{table-L(1,2)R(1)-x1.x2} is periodic follows from the periodicity of positions of the form $(x_1)$.  For the subsequent three rows, it is easy to see that once the outcomes repeat, they have become periodic as the outcome relies on the row immediately previous (which is periodic), and the preceding element in the row.  

Letting $i=2$, $r_2=2$, and $d_2 = 1$, Table \ref{table-L(1,2)R(1)-x1.x2} gives us that the outcomes of $(x_1, 2)$ and $(x_1, 3)$ agree for all $x_1$.  Thus we have satisfied condition \eqref{item-1-mike} of Theorem \ref{theorem-2-mike}.   Condition \eqref{item-2-mike} follows vacuously, so Theorem \ref{theorem-2-mike} has been satisfied.   We then have pre-period $R_2 = (0,2)$ and period $D_2 = (2,1)$.  This gives the candidate monoid
	\begin{align*}
	\monoid{M}_{\cl{h_1, h_2}}^{\heartsuit} &= \ideal{1,a, b \mid a^2=1, b^2 = b^3} \\
	\Next &= \{1,b\} \\
	\Prev &= \{a\} \\
	\Left &= \{b^2, ab, ab^2\} \\
	\Right &= \emptyset.
	\end{align*}
via the map
	\begin{align*}
		h_0 &\mapsto 1, \\
		h_1 &\mapsto a,\\
		h_2 &\mapsto b.
	\end{align*}
It remains to check if there are any further indistinguishability relations.  If two elements are in different outcome classes, then they are distinguishable.  Thus, we must only check elements which are in the same outcome classes.

The elements $1$ and $b$ are distinguished by $a$ with $1 \cdot a = a$ is a $\Prev$ position while $ab$ is an $\Left$ position.

The element $ab$ is distinguishable from both $b^2$ and $ab^2$ by $a$ since $a^2b=b$ has outcome $\Next$ while $ab^2$ and $a^2b^2 = b^2$ both have outcome $\Left$.   However $b^2$ and $ab^2$ are indistinguishable, as shown in Table \ref{table-b2,ab2-indistinguishable-L(1,2)R(1)}.

	\begin{table}[htb]
	\begin{center}
	    \begin{tabular}{llclc}
	    $\boldsymbol{x}$&$\boldsymbol{b^2x}$&$\boldsymbol{o^-(b^2x)}$&$\boldsymbol{ab^2x}$&$\boldsymbol{o^-(ab^2x)}$ 
	    \\ 
	    \cellcolor[gray]{0.8}1 &\cellcolor[gray]{0.8}$b^2$ &\cellcolor[gray]{0.8}$\Left$ &\cellcolor[gray]{0.8}$ab^2$ &\cellcolor[gray]{0.8}$\Left$ \\
	    $a$ & $ab^2$ & $\Left$ & $b^2$ & $\Left$ \\
		\cellcolor[gray]{0.8}$b$
	    &\cellcolor[gray]{0.8}$b^2$&\cellcolor[gray]{0.8}$\Left$&\cellcolor[gray]{0.8}$ab^2$&\cellcolor[gray]{0.8}$\Left$\\	    
		$b^2$
	    & $b^2$ & $\Left$ & 
	    $b^2$
	    & $\Left$\\
	    \cellcolor[gray]{0.8}$ab$
	    &\cellcolor[gray]{0.8}$ab^2$&\cellcolor[gray]{0.8}$\Left$&\cellcolor[gray]{0.8}$ab^2$&\cellcolor[gray]{0.8}$\Left$\\	    
	    $ab^2$&$ab^2$&$\Left$&$ab^2$&$\Left$\\
	    \end{tabular}
	\end{center}
	\caption{The indistinguishability of $b^2$ and $ab^2$ in $\monoid{M}_{\cl{h_1,h_2}}$ in the game $\Left(1,2)$, $\Right(1)$.}
	\label{table-b2,ab2-indistinguishable-L(1,2)R(1)}
	\end{table}		

Therefore, we add in the relation $ab^2 = b^2$ to our monoid, yielding the final monoid
	\[ \monoid{M}_{\cl{h_1,h_2}} = \ideal{1,a,b \mid a^2=1, b^2 = b^3, ab^2 = b^2}.\]

We now move on to $h_3$.  We calculate the outcome tables up to $x_3 = 2$.  Note that since $R_2 = (0,2)$ and $D_1 = (1,1)$, we only (initially) calculate up to $x_1 = 2$ and $x_2 = 3$, i.e.\ to the point where the previous heaps demonstrated their periodicity.  If the outcome tables for $x_3$ become periodic by this point, they will then remain periodic.  Table \ref{table-L(1,2)R(1)-x1.x2,x3} gives the outcome tables for $x_3=0$, $1$, and $2$.

\begin{table}[htb]
\begin{center}
	\begin{tabular}{llcccc}
		$\boldsymbol{x_3=0}$&$x_1=$&0&1&2&$\ldots$\\
			&$x_2=0$&\cellcolor[gray]{0.8}$\Next$&\cellcolor[gray]{0.8}$\Prev$&\cellcolor[gray]{0.8}$\Next$&\cellcolor[gray]{0.8}$\ldots$\\
			&$x_2=1$&$\Next$&$\Left$&$\Next$&$\ldots$\\
			&$x_3=2$&\cellcolor[gray]{0.8}$\Left$&\cellcolor[gray]{0.8}$\Left$&\cellcolor[gray]{0.8}$\Left$&\cellcolor[gray]{0.8}$\ldots$\\
			&$x_3=3$&$\Left$&$\Left$&$\Left$&$\ldots$\\
	\end{tabular}
	
	\text{ } \\ 
	\text{ } \\
	
	\begin{tabular}{llcccc}
		$\boldsymbol{x_3=1}$&$x_1=$&$0$&$1$&$2$&$\ldots$\\
			&$x_2=0$&\cellcolor[gray]{0.8}$\Left$&\cellcolor[gray]{0.8}$\Left$&\cellcolor[gray]{0.8}$\Left$&\cellcolor[gray]{0.8}$\ldots$ \\
			&$x_2=1$&$\Left$&$\Left$&$\Left$&$\ldots$ \\
			&$x_3=2$&\cellcolor[gray]{0.8}$\Left$&\cellcolor[gray]{0.8}$\Left$&\cellcolor[gray]{0.8}$\Left$&\cellcolor[gray]{0.8}$\ldots$ \\
			&$x_3=3$&$\Left$&$\Left$&$\Left$&$\ldots$ \\
	\end{tabular}

	\text{ } \\ 
	\text{ } \\
	
	\begin{tabular}{llcccc}
		$\boldsymbol{x_3=2}$&$x_1=$&$0$&$1$&$2$&$\ldots$\\
			&$x_2=0$&\cellcolor[gray]{0.8}$\Left$&\cellcolor[gray]{0.8}$\Left$&\cellcolor[gray]{0.8}$\Left$&\cellcolor[gray]{0.8}$\ldots$ \\
			&$x_2=1$&$\Left$&$\Left$&$\Left$&$\ldots$ \\
			&$x_3=2$&\cellcolor[gray]{0.8}$\Left$&\cellcolor[gray]{0.8}$\Left$&\cellcolor[gray]{0.8}$\Left$&\cellcolor[gray]{0.8}$\ldots$ \\
			&$x_3=3$&$\Left$&$\Left$&$\Left$&$\ldots$ \\
	\end{tabular}
	
\end{center}
\caption{The outcomes for positions $(x_1,x_2,x_3)$ in the game $\Left(1,2)$, $\Right(1)$.}
\label{table-L(1,2)R(1)-x1.x2,x3}
\end{table}	

We can see that outcomes for $(x_1, x_2, 1)$ agree with those for $(x_1, x_2, 2)$ and the periodicity of $x_1$ and $x_2$ remain intact.  Therefore, by Theorem \ref{theorem-2-mike}, we have pre-period $R_3 = (0,2,1)$ and period $D_3 = (2,1,1)$.  This gives the candidate quotient
	\begin{align*}
	\monoid{M}_{\cl{h_1, h_2,h_3}}^{\heartsuit} &= \ideal{1,a, b,c \mid a^2=1, b^2 = b^3, c=c^2} \\
	\Next &= \{1,b\} \\
	\Prev &= \{a\} \\
	\Left &= \{b^2, c, ab, ab^2, ac, bc, b^2c, abc, ab^2c\} \\
	\Right &= \emptyset.
	\end{align*}
via the map
	\begin{align*}
		h_0 &\mapsto 1, \\
		h_1 &\mapsto a,\\
		h_2 &\mapsto b,\\
		h_3 &\mapsto c.
	\end{align*}
We note that provided there is at least one copy of $h_3$ in the position, the position is an $\Left$ position.  

It remains to check for the distinguishability of the elements.  Again, $1$ and $b$ are distinguished by $a$.  We now have the set of $\Left$ elements to check.  As every element with $c$ is an $\Left$ position, we have that
	\[ c = ac = bc = b^2c = abc =ab^2c\]
since $o^-(cx) = o^-(xyc) = \Left$.  It remains to check whether $b^2$, $c$, $ab$, and $ab^2$ are distinguishable.  Table \ref{table-b2,c,ab2-indistinguishable-L(1,2)R(1)} shows that while $ab$ is distinguishable from the other three, the rest are indistinguishable.  

	\begin{table}[htb]
	\begin{center}
	    \begin{tabular}{llclclclc} 
	    $\boldsymbol{x}$&$\boldsymbol{b^2x}$&$\boldsymbol{o^-(b^2x)}$& $\boldsymbol{cx}$&$\boldsymbol{o^-(cx)}$&$\boldsymbol{abx}$&$\boldsymbol{o^-(abx)}$&$\boldsymbol{ab^2x}$&$\boldsymbol{o^-(ab^2x)}$
	    \\ 
	    \rowcolor[gray]{0.8}1 & $b^2$ & $\Left$ & $c$ & $\Left$ & $ab$ & $\Left$& $ab^2$ & $\Left$\\
	    $a$ & $ab^2$ & $\Left$ & $ac$ & $\Left$ & $b$ & $\Next$ & $b^2$ & $\Left$ \\
	    \rowcolor[gray]{0.8}$b$
	    & $b^2$ & $\Left$ & $bc$ & $\Left$ & $ab^2$ & $\Left$ &$ab^2$ & $\Left$\\
		$b^2$
	    & $b^2$ & $\Left$ & 
	    $b^2c$
	    & $\Left$ & $ab^2$ &$\Left$ & $ab^2$ & $\Left$\\
	    \rowcolor[gray]{0.8}$c$ & $b^2c$ & $\Left$ & $c$ & $\Left$ & $abc$ & $\Left$ & $ab^2c$ & $\Left$\\
	    $ab$
	    & $ab^2$ & $\Left$ & $abc$ & $\Left$ &$b^2$ & $\Left$ & $b^2$ & $\Left$\\
	    \rowcolor[gray]{0.8}$ab^2$&$ab^2$&$\Left$&$ab^2c$&$\Left$ & $b^2$ & $\Left$ & $b^2c$ & $\Left$\\
	    $ac$ & $ab^2c$ & $\Left$ & $ac$ & $\Left$ & $bc$ & $\Left$ & $b^2c$ & $\Left$\\
	    \rowcolor[gray]{0.8}$bc$ & $b^2c$ & $\Left$ & $bc$ & $\Left$ &$ab^2c$ & $\Left$ &$ab^2c$ & $\Left$\\
	    $b^2c$ & $b^2c$ & $\Left$ & $b^2c$ & $\Left$ & $ab^2c$ & $\Left$ & $ab^2c$ & $\Left$\\
	    \rowcolor[gray]{0.8}$abc$ & $ab^2c$ & $\Left$ & $abc$ & $\Left$ & $b^c$ & $\Left$ & $b^c$ & $\Left$\\
	    $ab^2c$ & $ab^2c$ & $\Left$ & $ab^2c$ & $\Left$ & $b^2c$ & $\Left$ & $b^2c$ & $\Left$\\
	    \end{tabular}
	\end{center}
	\caption{The indistinguishability of $b^2$, $c$, and $ab^2$ in $\monoid{M}_{\cl{h_1,h_2,h_3}}$ in the game $\Left(1,2)$, $\Right(1)$.}
	\label{table-b2,c,ab2-indistinguishable-L(1,2)R(1)}
	\end{table}

We now adjust our monoid.  Rather than
	\[ h_3 \mapsto c,\]
we change to
	\[ h_3 \mapsto b^2,\]
and we add the relation $b^2 = ab^2$ into our monoid to obtain
	\begin{align*}
	\monoid{M}_{\cl{h_1, h_2,h_3}} &= \ideal{1,a, b \mid a^2=1, b^2 = b^3, b^2 = ab^2} \\
	\Next &= \{1,b\} \\
	\Prev &= \{a\} \\
	\Left &= \{b^2, ab\} \\
	\Right &= \emptyset.
	\end{align*}
via the map
	\begin{align*}
		h_0 &\mapsto 1, \\
		h_1 &\mapsto a,\\
		h_2 &\mapsto b,\\
		h_3 &\mapsto b^2.
	\end{align*}	

When we begin calculating values with $h_4 \ge 1$, we notice that our first few values all seem have outcome $\Left$.  This pattern continues, as the following Proposition shows:

\begin{proposition}\label{prop-L(1,2)R(1)-x4-L}
	Consider position $(x_1, x_2, x_3, x_4)$ where $x_4 \ge 1$.  Then
		\[ o^-((x_1, x_2, x_3, x_4)) \in \Left.\]
\end{proposition}

\begin{proof}
	We proceed by induction on $x_4$.  When $x_4 = 1$, we have position $(x_1, x_2, x_3, 1)$.  We proceed by induction on the options of $(x_1, x_2, x_3)$.  When $(x_1, x_2, x_3) = (0,0,0)$, we have the position $(0,0,0,1)$.  Right's only move is to $(0,0,1,0)$, which is an $\Left$ position.  Left also makes that move, so 
		\[ o^-((0,0,0,1))  = \Left.\]
	Now take position $(x_1, x_2, x_3)$ and suppose that for $(y_1, y_2,y_3)$ any option of $(x_1, x_2, x_3)$ that
		\[ o^-((y_1, y_2,y_3,1)) = \Left.\]
	No matter what move Right makes, there will be at least one heap of size either three or four.  By induction, this means that Right can only move to $\Left$ positions.  Left will take one of those moves as well, and so we have our desired result.
	
	Now suppose true for all $1 \le x_4 < k$ and consider the position $(x_1, x_2, x_3, k)$.  We proceed by induction on the options of $(x_1, x_2, x_3)$.  When $(x_1, x_2, x_3) = (0,0,0)$,  Right moves to $(0,0,1,k-1)$, which is an $\Left$ position by induction.  Left makes the same move, and so we have our result.  Now suppose the result is true for all options of $(x_1, x_2, x_3)$.  Right moving first in $(x_1, x_2, x_3, k)$ has the following four possible moves:
		\begin{enumerate}
			\item $(x_1-1,x_2,x_3,k)$,
			\item $(x_1+1, x_2-1, x_3,k)$,
			\item $(x_1, x_2+1, x_3-1, k)$,
			\item $(x_1, x_2, x_3+1, k-1)$,
		\end{enumerate}
	where the first three are $\Left$ positions by induction on $(x_1, x_2, x_3)$ and the last is an $\Left$ position by induction on $k$.  Left will also make one of those moves and so $o^-((x_1, x_2, x_3, k)) = \Left$.
\end{proof}

We could have, equivalently, used the method of \cite{MIKE} rather than a tedious induction to show this result.  This should convince any doubters of why \cite{MIKE}, a much prettier method, should be used in calculations of heap-based \mis monoids.  However, whatever method we use to obtain Proposition \ref{prop-L(1,2)R(1)-x4-L}, we also obtain the following corollary.

\begin{corollary}
	In $\cl{h_1,h_2,h_3,h_4}$, we have
		\[ h_3 \equiv h_4 \imod{\cl{h_1,h_2,h_3,h_4}}.\]
\end{corollary}

\begin{proof}
	Any element which has at least one heap of size three or one heap of size four is an $\Left$ position.  Therefore $h_3$ and $h_4$ are indistinguishable.
\end{proof}

We can now easily calculate $\monoid{M}_{\cl{h_1,h_2,h_3,h_4}}$;  it will be isomorphic to $\monoid{M}_{\cl{h_1, h_2,h_3}}$ with $h_4 \mapsto b^2$, i.e.\
	\begin{align*}
	\monoid{M}_{\cl{h_1, h_2,h_3, h_4}} &= \ideal{1,a, b \mid a^2=1, b^2 = b^3} \\
	\Next &= \{1,b\} \\
	\Prev &= \{a\} \\
	\Left &= \{b^2, ab, ab^2\} \\
	\Right &= \emptyset.
	\end{align*}
via the map
	\begin{align*}
		h_0 &\mapsto 1, \\
		h_1 &\mapsto a,\\
		h_2 &\mapsto b,\\
		h_3 &\mapsto b^2,\\
		h_4 &\mapsto b^2.
	\end{align*}	

Moreover, we can repeat the same argument with $h_5$, and then $h_6$, etc.  That is,
	\begin{align*}
	\monoid{M}_{\Left(1,2),\Right(1)} &= \ideal{1,a, b \mid a^2=1, b^2 = b^3} \\
	\Next &= \{1,b\} \\
	\Prev &= \{a\} \\
	\Left &= \{b^2, ab, ab^2\} \\
	\Right &= \emptyset.
	\end{align*}
via the map
	\begin{align*}
		h_0 &\mapsto 1, \\
		h_1 &\mapsto a,\\
		h_2 &\mapsto b,\\
		h_n &\mapsto b^2 \text{ for all } n \ge 3.
	\end{align*}

\section{The Partizan Subtraction Game $\Left(1)$, $\Right(2)$}

This section gives a partizan subtraction game with an infinite \mis monoid.  
In the subtraction game $\Left(1)$, $\Right(2)$, Left's only move is to subtract one token from a heap, while Right's only move is to subtract two tokens from a heap.  

We will again proceed using the method in \cite{MIKE}.  

We begin by looking at the position $(x_1)$.  The outcomes $o^-((x_1))$ are given in Table \ref{table-L(1)R(2)-x1}.

\begin{table}[htb]
\begin{center}
	\begin{tabular}{rccccccc}
		$x_1=$&0&1&2&3&4&5&$\ldots$\\
		$o^-((x_1))=$&\cellcolor[gray]{0.8}$\Next$&\cellcolor[gray]{0.8}$\Right$&\cellcolor[gray]{0.8}$\Right$&\cellcolor[gray]{0.8}$\Right$&\cellcolor[gray]{0.8}$\Right$&\cellcolor[gray]{0.8}$\Right$&\cellcolor[gray]{0.8}$\ldots$
	\end{tabular}
\end{center}
\caption{The outcomes for positions $(x_1)$ in the game $\Left(1)$, $\Right(2)$.}
\label{table-L(1)R(2)-x1}
\end{table}	

It is easy to see that the pattern $o^-(x_1) = o^-(x_1 +1)$ will continue.  Using the notation of \cite{MIKE}, we let $R_1$ denote the pre-period and $D_1$ denote the period, and so $R_1 = (1)$, $D_1 = (1)$.  This gives the candidate monoid
	\begin{align*}
	\monoid{M}_{\cl{h_1}}^{\heartsuit} &= \ideal{1,a \mid a=a^2} \\
	\Next &= \{1\} \\
	\Prev &= \emptyset \\
	\Left &= \emptyset \\
	\Right &= \{a\}.
	\end{align*}
via the map
	\begin{align*}
		h_0 &\mapsto 1, \\
		h_1 &\mapsto a.
	\end{align*}
As this monoid has only two elements ($1$ and $a$) and they are clearly distinguishable as they have different outcome classes, there are no further relations and $\monoid{M}_{\cl{h_1}}^{\heartsuit} = \monoid{M}_{\cl{h_1}}$.  

We now continue to investigate by examining positions $(x_1, x_2)$.  Table \ref{table-L(1)R(2)-x1.x2} gives outcomes $o^-((x_1, x_2))$ with $x_1 \le 7$ and $x_2 \le 6$.  

\begin{table}[htb]
\begin{center}

\begin{tabular}{rccccccccc}
$x_1=$ & 0 &1& 2&3&4&5&6&7&$\ldots$\\ 
	$o^-((x_1,0))$= & \cellcolor[gray]{0.9}$\Next$ & $\Right$ & 
	    \cellcolor[gray]{0.8}$\Right$ & $\Right$ & 
	    \cellcolor[gray]{0.9}$\Right$ & $\Right$ & 
	    \cellcolor[gray]{0.8}$\Right$ & $\Right$ &$\ldots$ \\
	$o^-((x_1,1))$= & $\Prev$ & 
	    \cellcolor[gray]{0.9}$\Next$ & $\Right$ & 
	    \cellcolor[gray]{0.8}$\Right$ & $\Right$ & 
	    \cellcolor[gray]{0.9}$\Right$ & $\Right$ & 
	    \cellcolor[gray]{0.8}$\Right$ &$\ldots$ \\
	$o^-((x_1,2))$= & \cellcolor[gray]{0.8}$\Right$ & $\Prev$ & 
	    \cellcolor[gray]{0.9}$\Next$ & $\Right$ & 
	    \cellcolor[gray]{0.8}$\Right$ & $\Right$ & 
	    \cellcolor[gray]{0.9}$\Right$ & $\Right$ &$\ldots$ \\
	$o^-((x_1,3))$= & $\Next$ & 
	    \cellcolor[gray]{0.8}$\Right$ & $\Prev$ & 
	    \cellcolor[gray]{0.9}$\Next$ & $\Right$ & 
	    \cellcolor[gray]{0.8}$\Right$ & $\Right$ & 
	    \cellcolor[gray]{0.9}$\Right$ & $\ldots$ \\
	$o^-((x_1,4))$= & \cellcolor[gray]{0.9} $\Prev$ & $\Next$ & 
	    \cellcolor[gray]{0.8}$\Right$ & $\Prev$ & 
	    \cellcolor[gray]{0.9}$\Next$ & $\Right$ & 
	    \cellcolor[gray]{0.8}$\Right$ & $\Right$ &$\ldots$ \\
	$o^-((x_1,5))$= & $\Right$ & 
	    \cellcolor[gray]{0.9}$\Prev$ & $\Next$ & 
	    \cellcolor[gray]{0.8}$\Right$ & $\Prev$ & 
	    \cellcolor[gray]{0.9}$\Next$ & $\Right$ & 
	    \cellcolor[gray]{0.8}$\Right$ &$\ldots$ \\
	$o^-((x_1,6))$= & \cellcolor[gray]{0.8}$\Next$ & $\Right$ & 
	    \cellcolor[gray]{0.9}$\Prev$ & $\Next$ &
	    \cellcolor[gray]{0.8}$\Right$ & $\Prev$ & 
	    \cellcolor[gray]{0.9}$\Next$ & $\Right$ &$\ldots$ 
\end{tabular}
\end{center}
\caption{The outcomes for positions $(x_1,x_2)$ with $x_1 \le 7$ and $x_2 \le 6$ in the game $\Left(1,2)$, $\Right(1)$.}
\label{table-L(1)R(2)-x1.x2}
\end{table}

To determine the outcomes of positions in Table \ref{table-L(1)R(2)-x1.x2}, we proceeded as follows:  From any position in the table, Left's moves are
	\begin{enumerate}
		\item to move one position to the left (corresponding with taking a 
		one token from a heap of size one),
		\item to move to the position diagonally right and up (corresponding with taking one token from a heap of size two).
	\end{enumerate}

Right's moves are
	\begin{enumerate}
		\item to move one position up (corresponding with taking a 
		two tokens from a heap of size two).
	\end{enumerate}

We can see that the table has become, to use \cite{MIKE}'s terminology, ``diagonally" periodic, that is $o^-((x_1, x_2)) = o^-((x_1+1, x_2+1))$.  Moreover, when $x_1 > x_2$, we have $o^-((x_1,x_2)) = \Right$.  We encapsulate these results in the following proposition.

\begin{proposition}
	Let $(x_1, x_2)$ be a position in the game $\Left(1)$, $\Right(2)$.  Then
		\[ o^-((x_1, x_2)) = \begin{cases}
		\Right &\text{if } x_1 > x_2; \\
		\Next & \text{if } x_1 \le x_2,\text{ } x_1 \equiv x_2 \imod{3}; \\
		\Prev & \text{if } x_1 \le x_2,\text{ } x_1 \equiv x_2 + 2 \imod{3};\\
		\Right &\text{if } x_1 \le x_2,\text{ } x_1 \equiv x_2 + 1 \imod{3}.
		\end{cases}\]
\end{proposition}

\begin{proof}
	Follows from inspection of Table \ref{table-L(1)R(2)-x1.x2}.
\end{proof}

We now construct a candidate quotient $\monoid{M}^{\heartsuit}_{\cl{h_1,h_2}}$ via the map
	\begin{align*}
		h_0 &\mapsto 1, \\
		h_1 &\mapsto a,\\
		h_2 &\mapsto b.
	\end{align*}
We redraw Table \ref{table-L(1)R(2)-x1.x2} in terms of $a^n$ and $b^m$, giving us Table \ref{table-L(1)R(2)-x1.x2-mn}.

\begin{table}[htb]
\begin{center}
\begin{tabular}{lccccccccc}
 \cellcolor[gray]{0.9}$m \backslash n$ & 0 &1& 2&3&4&5&6&7&$\ldots$\\ 
	0 & \cellcolor[gray]{0.9}$\Next$ & $\Right$ & 
	    \cellcolor[gray]{0.8}$\Right$ & $\Right$ & 
	    \cellcolor[gray]{0.9}$\Right$ & $\Right$ & 
	    \cellcolor[gray]{0.8}$\Right$ & $\Right$ &$\ldots$ \\
	1 & $\Prev$ & 
	    \cellcolor[gray]{0.9}$\Next$ & $\Right$ & 
	    \cellcolor[gray]{0.8}$\Right$ & $\Right$ & 
	    \cellcolor[gray]{0.9}$\Right$ & $\Right$ & 
	    \cellcolor[gray]{0.8}$\Right$ &$\ldots$ \\
	2 & \cellcolor[gray]{0.8}$\Right$ & $\Prev$ & 
	    \cellcolor[gray]{0.9}$\Next$ & $\Right$ & 
	    \cellcolor[gray]{0.8}$\Right$ & $\Right$ & 
	    \cellcolor[gray]{0.9}$\Right$ & $\Right$ &$\ldots$ \\
	3 & $\Next$ & 
	    \cellcolor[gray]{0.8}$\Right$ & $\Prev$ & 
	    \cellcolor[gray]{0.9}$\Next$ & $\Right$ & 
	    \cellcolor[gray]{0.8}$\Right$ & $\Right$ & 
	    \cellcolor[gray]{0.9}$\Right$ & $\ldots$ \\
	4 & \cellcolor[gray]{0.9} $\Prev$ & $\Next$ & 
	    \cellcolor[gray]{0.8}$\Right$ & $\Prev$ & 
	    \cellcolor[gray]{0.9}$\Next$ & $\Right$ & 
	    \cellcolor[gray]{0.8}$\Right$ & $\Right$ &$\ldots$ \\
	5 & $\Right$ & 
	    \cellcolor[gray]{0.9}$\Prev$ & $\Next$ & 
	    \cellcolor[gray]{0.8}$\Right$ & $\Prev$ & 
	    \cellcolor[gray]{0.9}$\Next$ & $\Right$ & 
	    \cellcolor[gray]{0.8}$\Right$ &$\ldots$ \\
	6 & \cellcolor[gray]{0.8}$\Next$ & $\Right$ & 
	    \cellcolor[gray]{0.9}$\Prev$ & $\Next$ &
	    \cellcolor[gray]{0.8}$\Right$ & $\Prev$ & 
	    \cellcolor[gray]{0.9}$\Next$ & $\Right$ &$\ldots$ 
\end{tabular}
\end{center}
\caption{The outcomes for positions $a^nb^m$ in $\monoid{M}^{\heartsuit}_{\cl{h_1,h_2}}$with $n \le 7$ and $m \le 6$ in the game $\Left(1,2)$, $\Right(1)$.}
\label{table-L(1)R(2)-x1.x2-mn}
\end{table}

Due to the diagonal periodicity, we know that we have the indistinguishability relation $ab= 1$.  Because of this, all elements in $\monoid{M}^{\heartsuit}_{\cl{h_1,h_2}}$ are of the form $a^n$ or $b^m$ for $n, m \in \mathbb{Z}^{\ge 0}$.  Using Table \ref{table-L(1)R(2)-x1.x2-mn}, we then get the following candidate monoid and outcome tetrapartition:
	\begin{align*}
		\monoid{M}^{\heartsuit}_{\cl{h_1,h_2}} &= \ideal{1,a,b \mid 1 = ab}\\
		\Next &= \{ b^m \mid m \equiv 0 \imod{3}\} \\
		\Prev &= \{ b^m \mid m \equiv 1 \imod{3}\} \\
		\Left &= \emptyset\\
		\Right &= \{b^m \mid m \equiv 2 \imod{3}\} \cup \{a^n \mid n \in \nat\}.
	\end{align*}

We will now check that there are no indistinguishability relations other than $ab=1$.  

\begin{proposition}\label{prop-L(1)R(2)-ab=1}
	The only indistinguishability relation on $\monoid{M}^{\heartsuit}_{\cl{h_1,h_2}}$ is $ab= 1$.  Thus $\monoid{M}^{\heartsuit}_{\cl{h_1,h_2}} = \monoid{M}_{\cl{h_1,h_2}}$.
\end{proposition}

\begin{proof}
	Firstly take $j$, $k$, $m$, and $n$ such that $j < k$ and $0 <m < n$.  We will show that all elements are distinguishable.  Again, we need only check for elements which have the same outcome class as elements with differing outcome classes are distinguished by $1$.

	We divide into the three outcome cases:
		\begin{enumerate}
		 	\item Take two elements of $\monoid{M}^{\heartsuit}_{\cl{h_1,h_2}}$ which are both $\Next$, $b^{3j}$ and $b^{3k}$.  Consider element $a^{3j+2}$. Then
		 		\begin{align*}
					o^-(a^{3j+2}b^{3j}) &= o^-(a^2) = \Right \\
					o^-(a^{3j+2}b^{3k}) &= o^-(b^{3k-3j-2}) = \Prev.
		 		\end{align*} 	
		 	Therefore $b^{3j}$ and $b^{3k}$ are distinguishable.
		 	
		 	\item Take two elements of $\monoid{M}^{\heartsuit}_{\cl{h_1,h_2}}$ which are both $\Prev$, $b^{3j+1}$ and $b^{3k+1}$.  Consider element $a^{3j+3}$.  Then
		 		\begin{align*}
					o^-(a^{3j+3}b^{3j+1}) &= o^-(a^2) = \Right \\
					o^-(a^{3j+3}b^{3k+1}) &= o^-(b^{3k-3j-2}) = \Prev.
		 		\end{align*}
		 	Therefore $b^{3j+1}$ and $b^{3k+1}$ are distinguishable.
		 	
		 	\item Take two elements of $\monoid{M}^{\heartsuit}_{\cl{h_1,h_2}}$ which are both $\Right$.
		 		\begin{enumerate}
					\item Take elements $b^{3j+2}$ and $b^{3k+2}$.  Firstly, we claim that $3j +4 < 3k+2$.  Suppose not.  Then
						\[ 3j + 4 \ge 3k + 2 \implies 3j + 2 \ge 3k.\]
					But we assumed $j < k$, so $3j < 3k$, giving the following inequalities:
						\[ 3j < 3k \le 3j +2,\]
					which is a contradiction.  Therefore $3j + 4 < 3k+2$.
					
					Consider element $a^{3j+4}$.  Then
						\begin{align*}
							o^-(a^{3j+4}b^{3j+2}) &= o^-(a^2) = \Right \\
							o^-(a^{3j+4}b^{3k+2}) &= o^-(b^{3k-3j-2}) = \Prev.
						\end{align*}
					Therefore $b^{3j+2}$ and $b^{3k+2}$ are distinguishable.
					
					\item Take elements $b^{3j+2}$ and $a^m$.  We will use $a$ to show that these two elements are distinguishable.  Then
						\begin{align*}
							o^-(ab^{3j+2}) &= o^-(b^{3j+1}) = \Prev,\\
							o^-(aa^m) &= o^-(a^{m+1}) = \Right.
						\end{align*}
					Therefore $b^{3j+2}$ and $a^m$ are distinguishable.
					
					\item  Take elements $a^m$ and $a^n$.  We will use $b^m$ to show that these two elements are distinguishable.  Then
						\begin{align*}
							o^-(a^mb^m) &= o^-(1) = \Next,\\
							o^-(a^nb^m) &= o^-(a^{n-m}) = \Right.
						\end{align*}
					Therefore $a^m$ and $a^n$ are distinguishable.
		 		\end{enumerate}
		 \end{enumerate}
		
		 Thus all elements are distinguishable, so there are no more indistinguishability relations on $\monoid{M}^{\heartsuit}_{\cl{h_1,h_2}}$ and so $\monoid{M}^{\heartsuit}_{\cl{h_1,h_2}} = \monoid{M}_{\cl{h_1,h_2}}$.
\end{proof}

\begin{corollary}
	$\monoid{M}_{\cl{h_1,h_2}}$ contains an infinite number of elements.
\end{corollary}

\begin{proof}
	The proof of Proposition \ref{prop-L(1)R(2)-ab=1} showed that any two elements $x$ and $y$ such that 
		\[x,y \in \{a^n \mid n \in \nat\} \cup \{b^m \mid m \in \nat\}\]
	are distinguishable. Therefore $\monoid{M}^{\heartsuit}_{\cl{h_1,h_2}}$ has infinite cardinality.
\end{proof}

\begin{corollary}
	$\monoid{M}_{\cl{h_1,h_2, h_3, \ldots}}$ contains an infinite number of elements.
\end{corollary}

\begin{proof}
	In the proof of Proposition \ref{prop-L(1)R(2)-ab=1}, we saw that $3j$ copies of $h_2$ is distinguishable from $3k$ copies of $h_2$ for $j < k$ and the distinguishing element was $3j+2$ copies of $h_1$.  Since arbitrary sums of $h_1$ and $h_2$ are contained in $\cl{h_1, h_2, h_3, \ldots}$, $3j$ copies of $h_2$ are still distinguishable from $3k$ copies of $h_2$ for $j<k$ by $3j+2$ copies of $h_1$ in $\cl{h_1, h_2, h_3, \ldots}$.  Thus $\monoid{M}_{\cl{h_1,h_2, h_3, \ldots}}$ is of infinite cardinality.
\end{proof}

Therefore we have found a partizan subtraction game which gives an infinite monoid.

\section{Conclusion}
The rules of our two examples $\Left(1,2),\Right(1)$ and $\Left(1), \Right(2)$ are very similar, yet their \mis monoids could not be more different, with the first being finite and the second infinite.  Even the smallest change can have huge repercussions in \mis play.  

Since subtraction games are relatively easy to manipulate and to build computer programs for their analysis, partizan subtraction games are an excellent area for further investigations into \mis partizan quotients.

\chapter{Conclusion}\label{chapter-conclusion}

The thesis takes the first serious look at partizan \mis play games and gives some initial results, the most important of which is the classification of all sets of positions $\Upsilon$ such that $\monoid{M}_{\cl{\Upsilon}} \cong \monoid{M}_{\cl{*}}$.  However, work on partizan \mis play games is far from being completed.  We conclude our discussion by listing some areas for future work which appear in or are suggested by this thesis.

\begin{enumerate}
	\item In Chapter \ref{chapter-examples}, the partial orders of the \mis monoids calculated as examples were given.   However, we have yet to find any link between the partial orders and results regarding the \mis monoids.  Does knowing something about the partial order given any results about the monoid, or vice versa?  For example, if the partial order is a lattice, what does that say, if anything about the structure of the monoid?
	
	\item In Chapter \ref{chapter-cardinality-left}, we discussed the cardinality of the \mis monoids and listed some positions which force an infinite \mis monoid.  Further work on Open Problem \ref{op-chi-finite} is needed: Classify which positions $\chi$ are such that $\monoid{M}_{\cl{\chi}}$ is a finite/infinite monoid.
	
	\item In Chapter \ref{chapter-stars-0}, we showed that if $\xi$ was an all-small position, then $* + * \equiv 0 \imod{\cl{\xi}}$. Solving Open Problem \ref{op-*+*=0} is the next step, i.e.\ determining which positions $\xi$ with $* \in \cl{\xi}$ have the property that $* + * \equiv 0 \imod{\cl{\xi}}$.
	
	\item In Chapter \ref{chapter-0} we found some sets $\Upsilon$ of positions such that for all $\xi \in \Upsilon$, $\xi + \overline{\xi} \equiv 0 \imod{\Upsilon}$.  Doing such allowed us to find a Tweedledum-Tweedledee type strategy for these positions.  What other sets of positions have these properties?
	
	\item Investigate Open Problem \ref{op-ab3}:  Investigate whether $\ab{3}$ positions share any other normal play properties than having a Tweedledee-Tweedledum type strategy.
	
	\item Prove Conjecture 	\ref{conjecture-arrow-impartial}.  That is, prove the following: If $\alpha$, $\beta$, and $\gamma$ are impartial positions and $\alpha \to \beta$ and $\beta \to \gamma$ exist, then there exists an arrow $\alpha \to \gamma$ where an arrow $\delta \to \varepsilon$ exists if Left moving second can win $\varepsilon + \delta^{\circ}$ (or something similar with a slight variation of the winning condition for Left and/or Right).
	
	\item Chapter \ref{chapter-cat} was all about categories and suggested that perhaps we should be looking at taxons rather than categories.  More work must be done to see if this is the case.
	
	\item Investigate Open Problem \ref{op-L-monoid}:  Classify all positions $\xi$ such that $\monoid{M}_{\cl{\xi}} \cong \monoid{M}_{\cl{\L(\xi)}}$. 
	
	\item Prove Conjecture \ref{conjecture-replace-*}.  That is, prove the following: If $\xi$ is a position with $o^-(\xi) = \Prev$ and $\monoid{M}_{\cl{\xi}} \cong \monoid{M}_{\cl{*}}$, then we can replace $*$ by $\xi$ in any position which has $*$ as an option without changing the resultant \mis monoid.
	
	\item We would like to extend the isomorphism results of Chapter \ref{chapter-congruent} from sets isomorphic to $\monoid{M}_{\cl{*}}$ to isomorphisms for other sets.  In particular, we would like to either prove or give counterexample to the following:
		Given two closed sets of positions $S_1$ and $S_2$, with \mis monoids $\monoid{M}_{S_1}$ and $\monoid{M}_{S_2}$ respectively, if
			\[ \monoid{M}_{S_1} \cong \monoid{M}_{S_2},\]
		is it then true that
			\[ \monoid{M}_{S_1+S_2} \cong  \monoid{M}_{S_1},\]
		where
		\[ S_1 + S_2 = \{s_1 + s_2 \mid s_1 \in S_1, s_2 \in S_2\}.\]
	
	\item Given a closed set of positions $\Upsilon$ and \mis monoid $\monoid{M}_{\Upsilon}$, is there a closed set $\Psi$ with the property $\monoid{M}_{\Upsilon} \cong \monoid{M}_{\Psi}$ such that $\Psi$ is minimal in some sense?  Some possibilities for minimal include in cardinality, in terms of the birthday of the elements within, or in terms of \mis canonical forms \cite{MCF}.
\end{enumerate}

Perhaps now that the initial steps have been taken in analysing \mis play games, we can finally drop the moniker which has plagued the subject, \emph{miserable \ms}.

\appendix
\chapter{Frequently Used Positions}\label{appendixA}

\section{$0$}

\emph{Game notation}: $\combgame{\{\cdot\mid\cdot\}}$\\
\emph{Game tree}:
		\unitlength 8pt
		\begin{center}
		\begin{graph}(0,0)(1,1)
		\graphnodesize{0.4}
		\fillednodestrue
		
		\roundnode{c1}(1,1)
		
		\end{graph}
		\end{center}
\emph{\Mis Outcome}: $\Next$\\
\emph{Other}: impartial, all-small, ab0\\
\emph{First introduced}: Example \ref{example-0-1-*}

\section{$*$}
\emph{Game notation}: $\combgame{\{0\mid0\}}$\\
\emph{Game tree}:
		\unitlength 8pt
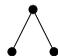
		\begin{center}
		\begin{graph}(2,2)(0,0)
		\graphnodesize{0.4}
		\fillednodestrue
		
		\roundnode{A}(1,2)
		\roundnode{B}(0,0)
		\roundnode{C}(2,0)
		
		\edge{A}{C}
		\edge{A}{B}
		
		\end{graph}
		\end{center}
\emph{\Mis Outcome}: $\Prev$\\
\emph{Other}: impartial, all-small, ab1\\
\emph{First introduced}: Example \ref{example-0-1-*}

\section{$1$}
\emph{Game notation}: $\combgame{\{0\mid\cdot\}}$\\
\emph{Game tree}:
		\unitlength 8pt
		\begin{center}
		\begin{graph}(2,2)(0,0)
		\graphnodesize{0.4}
		\fillednodestrue
		
		\roundnode{A}(1,2)
		\roundnode{B}(0,0)
		
		\edge{A}{B}
		
		\end{graph}
		\end{center}
\emph{\Mis Outcome}: $\Right$\\
\emph{Other}: partizan\\
\emph{First introduced}: Example \ref{example-0-1-*}

\section{$\overline{1}$}
\emph{Game notation}: $\combgame{\{\cdot\mid0\}}$\\
\emph{Game tree}:
		\unitlength 8pt
		\begin{center}
		\begin{graph}(2,2)(0,0)
		\graphnodesize{0.4}
		\fillednodestrue
		
		\roundnode{A}(1,2)
		\roundnode{C}(2,0)
		
		\edge{A}{C}
		
		\end{graph}
		\end{center}
\emph{\Mis Outcome}: $\Left$\\
\emph{Other}: partizan\\
\emph{First introduced}: Example \ref{example-conjugate-0-1-*}

\section{$\sigma$}
\emph{Game notation}: $\combgame{\{*\mid\cdot\}}$\\
\emph{Game tree}:
		\begin{center}
		\begin{graph}(4,4)(-1,0)
		\graphnodesize{0.4}
		\fillednodestrue

		\roundnode{A}(1,4)
		\roundnode{B}(0,2)
		\roundnode{D}(-1,0)
		\roundnode{E}(1,0)

		\edge{A}{B}
		\edge{B}{D}
		\edge{B}{E}
		
		\end{graph}
		\end{center}	
\emph{\Mis Outcome}: $\Next$\\
\emph{Other}: partizan\\
\emph{First introduced}: Definition \ref{def-sigma}

\section{$\sigmab$}
\emph{Game notation}: $\combgame{\{\cdot\mid*\}}$\\
\emph{Game tree}:
		
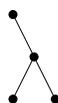
\begin{center}
		\begin{graph}(4,4)(-1,0)
		\graphnodesize{0.4}
		\fillednodestrue

		\roundnode{A}(1,4)
		\roundnode{C}(2,2)
		\roundnode{D}(3,0)
		\roundnode{E}(1,0)

		\edge{A}{C}
		\edge{C}{E}
		\edge{C}{D}
		
		\end{graph}
		\end{center}	
\emph{\Mis Outcome}: $\Next$\\
\emph{Other}: partizan\\
\emph{First introduced}: Section \ref{sec-sigma}

\section{$\rho$}
\emph{Game notation}: $\combgame{\{*\mid0\}}$\\
\emph{Game tree}:
		\unitlength 8pt
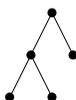
		\begin{center}
		\begin{graph}(4,4)(-1,0)
		\graphnodesize{0.4}
		\fillednodestrue

		\roundnode{A}(1,4)
		\roundnode{B}(0,2)
		\roundnode{C}(2,2)
		\roundnode{D}(-1,0)
		\roundnode{E}(1,0)
		
		\edge{A}{C}
		\edge{A}{B}
		\edge{B}{D}
		\edge{B}{E}
		
		\end{graph}
		\end{center}	
\emph{\Mis Outcome}: $\Left$\\
\emph{Other}: all-small, ab2\\
\emph{First introduced}: Definition \ref{def-rho}

\section{$\rhob$}
\emph{Game notation}: $\combgame{\{0\mid*\}}$\\
\emph{Game tree}:
		\unitlength 8pt
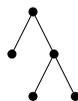
		\begin{center}
		\begin{graph}(4,4)(-1,0)
		\graphnodesize{0.4}
		\fillednodestrue

		\roundnode{A}(1,4)
		\roundnode{B}(0,2)
		\roundnode{C}(2,2)
		\roundnode{D}(3,0)
		\roundnode{E}(1,0)
		
		\edge{A}{C}
		\edge{A}{B}
		\edge{C}{D}
		\edge{C}{E}
		
		\end{graph}
		\end{center}	
\emph{\Mis Outcome}: $\Right$\\
\emph{Other}: all-small, ab2\\
\emph{First introduced}: Section \ref{sec-rho-rhob}

\section{$\tau$}
\emph{Game notation}: $\combgame{\{*\mid*\}}$\\
\emph{Game tree}:
		\unitlength 8pt
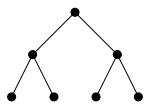
		\begin{center}
		\begin{graph}(4,4)(-1,0)
		\graphnodesize{0.4}
		\fillednodestrue

		\roundnode{A}(1,4)
		\roundnode{B}(-1,2)
		\roundnode{C}(3,2)
		\roundnode{D}(-2,0)
		\roundnode{E}(0,0)
		\roundnode{F}(2,0)
		\roundnode{G}(4,0)

		\edge{A}{C}
		\edge{A}{B}
		\edge{B}{D}
		\edge{B}{E}
		\edge{C}{F}
		\edge{C}{G}
		
		\end{graph}
		\end{center}		
\emph{\Mis Outcome}: $\Next$\\
\emph{Other}: impartial, all-small, ab2\\
\emph{First introduced}: Definition \ref{def-tau}

\section{$\tau^n$}
\emph{Game notation}: $\tau^0 = *$, $\tau^n = \combgame{\{\tau^{n-1}\mid\tau^{n-1}\}}$\\	
\emph{\Mis Outcome}: 	\[ o^-(\tau^n) = \begin{cases}
\Prev &\text{if } n \equiv 0 \imod 2;\\
\Next &\text{if } n \equiv 1 \imod 2.
\end{cases}\]
\emph{Other}: impartial, all-small, ab($k+1$)\\
\emph{First introduced}: Definition \ref{def-tau^n}

\section{$\L(\xi)$}
\emph{Game notation}: $\combgame{\{\xi\mid\cdot\}}$\\
\emph{\Mis Outcome}: 
	\[ o^-(\L(\xi)) = \begin{cases}
	\Next &\text{if } o^-(\xi) = \Prev \cup \Left;\\
	\Right &\text{if } o^-(\xi) = \Next \cup \Right.
	\end{cases}\]
\emph{First introduced}: Definition \ref{def-L(xi)}

\section{$\R(\xi)$}
\emph{Game notation}: $\combgame{\{\cdot\mid\xi\}}$\\
\emph{\Mis Outcome}: 
	\[ o^-(\R(\xi)) = \begin{cases}
	\Next &\text{if } o^-(\xi) = \Prev \cup \Right;\\
	\Left &\text{if } o^-(\xi) = \Next \cup \Left.
	\end{cases}\]
\emph{First introduced}: Definition \ref{def-L(xi)}

\section{$\overline{\xi}$}
\emph{Game notation}: For a position $\xi = \combgame{\{\xi^L\mid\xi^R\}}$, we recursively define $\overline{\xi}$ as $\overline{\xi} = \combgame{\{\overline{\xi^R}\mid\overline{\xi^L}\}}$ where $\overline{0} = 0$\\
\emph{\Mis Outcome}: 
		\begin{align*}
			o^-(\xi) = \Left &\implies o^-(\overline{\xi}) = \Right;\\
			o^-(\xi) = \Next &\implies o^-(\overline{\xi}) = \Next;\\
			o^-(\xi) = \Prev &\implies o^-(\overline{\xi}) = \Prev;\\
			o^-(\xi) = \Right &\implies o^-(\overline{\xi}) = \Left.
		\end{align*}
\emph{First introduced}: Definition \ref{def-conjugate}	
	
\section{$*_n$}
\emph{Game notation}: For $n \in \nat$, the position $*_n$ is defined recursively as follows:
		\begin{align*}
			*_n &= \combgame{\{0, *_1, *_2, \ldots, *_{n-1}\mid0, *_1, *_2, \ldots, *_{n-1}\}}.
		\end{align*}
	Generally, instead of $*_1$, we merely write $*$.\\	
\emph{\Mis Outcome}: 	\[ o^-(*_n) = \begin{cases}
\Prev &\text{if } n=1;\\
\Next &\text{if } n>1.
\end{cases}\]
\emph{Other}: impartial, all-small\\
\emph{First introduced}: 
	\begin{itemize}
		\item $*$:  Example \ref{example-0-1-*}
		\item $*_2$: Table \ref{table-normal-mis-outcomes}
		\item $*_n$: Definition \ref{def-cgstarn}
	\end{itemize}

\section{$\eta$}
\emph{Game notation}: $\eta = \combgame{\{\combgame{\{0\mid\combgame{\{*\mid0\}}\}}\mid*\}}$\\
\emph{Game tree}:
		\unitlength 8pt
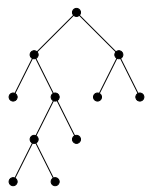
		\begin{center}
		\begin{graph}(7,10)(0,0)
		\graphnodesize{0.4}
		\fillednodestrue

		\roundnode{a1}(0,0)
		\roundnode{a2}(2,0)
		
		\roundnode{b1}(1,2)
		\roundnode{b2}(3,2)
		
		\roundnode{c1}(0,4)
		\roundnode{c2}(2,4)
		\roundnode{c3}(4,4)
		\roundnode{c4}(6,4)
		
		\roundnode{d1}(1,6)
		\roundnode{d2}(5,6)
		
		\roundnode{e1}(3,8)
		
		\edge{a1}{b1}
		\edge{a2}{b1}
		\edge{b1}{c2}
		\edge{b2}{c2}
		\edge{c1}{d1}
		\edge{c2}{d1}
		\edge{c3}{d2}
		\edge{c4}{d2}
		\edge{d1}{e1}
		\edge{d2}{e1}

		\end{graph}
		\end{center}	
\emph{\Mis Outcome}: $\Next$\\
\emph{Other}: all-small, $\ab{4}$, used to give an example of an all-small game where 
	\[* + * \not \equiv 0 \imod{\cl{\L(\eta)}}\]
\emph{First introduced}: Example \ref{example-eta}

\section{$\theta$}
\emph{Game notation}: $\theta = \combgame{\{\combgame{\{*\mid\rho\}}\mid\combgame{\{\rhob\mid*\}}\}}$\\
\emph{Game tree}:
		
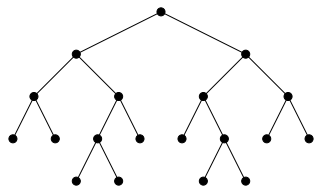
\begin{center}
		\begin{graph}(14,10)(0,0)
		\graphnodesize{0.4}
		\fillednodestrue

		\roundnode{a1}(3,0)
		\roundnode{a2}(5,0)
		\roundnode{a3}(9,0)
		\roundnode{a4}(11,0)
		
		\roundnode{b1}(0,2)
		\roundnode{b2}(2,2)
		\roundnode{b3}(4,2)
		\roundnode{b4}(6,2)
		\roundnode{b5}(8,2)
		\roundnode{b6}(10,2)
		\roundnode{b7}(12,2)
		\roundnode{b8}(14,2)

		\roundnode{c1}(1,4)
		\roundnode{c2}(5,4)
		\roundnode{c3}(9,4)
		\roundnode{c4}(13,4)
		
		\roundnode{d1}(3,6)
		\roundnode{d2}(11,6)
		
		\roundnode{e1}(7,8)
		
		\edge{b3}{a1}
		\edge{b3}{a2}
		\edge{c2}{b3}
		\edge{c2}{b4}
		\edge{b1}{c1}
		\edge{b2}{c1}
		\edge{d1}{c1}
		\edge{d1}{c2}
		\edge{c3}{b5}
		\edge{c3}{b6}
		\edge{c4}{b7}
		\edge{c4}{b8}
		\edge{b6}{a3}
		\edge{b6}{a4}
		\edge{d2}{c3}
		\edge{d2}{c4}
		\edge{e1}{d1}
		\edge{e1}{d2}
		
		\end{graph}

		\end{center}		
\emph{\Mis Outcome}: $\Next$ \\
\emph{Other}: all-small, $\ab{4}$, used as an example of a binary, all-small position with $o^-(\theta + \overline{\theta}) = \Prev$ \\
\emph{First introduced}: Example \ref{example-theta-prev}

\bibliographystyle{plain}
\bibliography{meghan}

\end{document}